\documentclass{amsart}
\usepackage{amsxtra}
\usepackage{mathdots}
\usepackage{mathrsfs}
\usepackage{verbatim}
\usepackage{calc}
\usepackage{graphicx}
\usepackage{fouridx}
\usepackage{tensor}
\usepackage[normalem]{ulem}
\usepackage{varwidth}
\usepackage{hyperref}
\usepackage[shortalphabetic,initials]{amsrefs}%,backrefs
\usepackage{bdsstandard}

\theoremstyle{plain}
\newtheorem{thm}[equation]{Theorem}
\newtheorem*{thm*}{Theorem}
\newtheorem{prop}[equation]{Proposition}
\newtheorem*{prop*}{Proposition}
\newtheorem{cor}[equation]{Corollary}
\newtheorem*{cor*}{Corollary}
\newtheorem{lem}[equation]{Lemma}
\newtheorem*{lem*}{Lemma}
\newtheorem{conj}[equation]{Conjecture}
\newtheorem*{conj*}{Conjecture}

\theoremstyle{definition}
\newtheorem{defn}[equation]{Definition}
\newtheorem*{defn*}{Definition}
\newtheorem{eg}[equation]{Example}
\newtheorem*{eg*}{Example}

\newtheorem*{ex*}{Exercise}
\newtheorem{rk}[equation]{Remark}
\newtheorem*{rk*}{Remark}

\newtheorem*{ntn*}{Notation}

\theoremstyle{plain}

\renewcommand{\H}{\ensuremath{\mathscr{H}}\xspace}
\renewcommand{\L}{\ensuremath{\mathscr{L}}\xspace}

\DeclareMathOperator{\Adm}{Adm}
\DeclareMathOperator{\Conv}{Conv}

\DeclareMathOperator{\Perm}{Perm}

\DeclareMathOperator{\SPerm}{SPerm}

\newcommand{\abs}{\ensuremath{\mathrm{abs}}\xspace}
\newcommand{\aff}{\ensuremath{\mathrm{aff}}\xspace}
\renewcommand{\Aff}{\ensuremath{\mathop{\mathrm{Aff}}}\xspace}

\renewcommand{\int}{\ensuremath{\mathop{\mathrm{int}}}\xspace}
\newcommand{\loc}{\ensuremath{\mathrm{loc}}\xspace}
\newcommand{\naive}{\ensuremath{\mathrm{naive}}\xspace}

\newcommand{\spin}{\ensuremath{\mathrm{spin}}\xspace}

\renewcommand{\vert}{\ensuremath{\mathrm{vert}}\xspace}

\entrymodifiers={+!!<0pt,\fontdimen22\textfont2>}

\begin{document}

\renewcommand{\O}{\ensuremath{\mathscr{O}}\xspace}

\title[Topological flatness of ramified unitary local models. II.]{Topological flatness of local\\ models for ramified unitary groups.\\  II.  The even dimensional case}
\author{Brian D. Smithling}
%\thanks{}
% \date{\today}
\address{University of Toronto, Department of Mathematics, 40 St.\ George St.,\ Toronto, ON  M5S 2E4, Canada}
\email{bds@math.toronto.edu}
%\subjclass[2010]{Primary 14G35; Secondary 05E15; 11G18; 17B22}
%\keywords{local model; unitary group; Iwahori-Weyl group, admissible set}

\begin{abstract}
Local models are schemes, defined in terms of linear-alge\-bra\-ic moduli problems, which give \'etale-local neighborhoods of integral models of certain $p$-adic PEL Shimura varieties defined by Rapoport and Zink.  In the case of a unitary similitude group whose localization at $\QQ_p$ is ramified, quasi-split $GU_n$, Pappas has observed that the original local models are typically not flat, and he and Rapoport have introduced new conditions to the original moduli problem which they conjecture to yield a flat scheme.  In a previous paper we proved that their new local models are topologically flat when $n$ is odd.  In the present paper we prove topological flatness when $n$ is even.  Along the way, we characterize the $\mu$-admissible set for certain cocharacters $\mu$ in types $B$ and $D$, and we show that for these cocharacters admissibility can be characterized in a vertexwise way, confirming a conjecture of Pappas and Rapoport.
\end{abstract}
\maketitle

\tableofcontents

\section{Introduction}
\numberwithin{equation}{section}

One of the basic problems in the arithmetic theory of Shimura varieties is the construction of reasonable integral models over the ring of integers \O of the reflex field, or over certain localizations of \O.  One would like to construct models that are faithfully flat with mild singularities, and in cases in which the Shimura variety can be described as a moduli space of abelian varieties with additional structure, to describe the model in terms of an integral version of the moduli problem.  A major intended application is the calculation of the Hasse-Weil zeta function of the Shimura variety.

Let $(\BG,X)$ be a Shimura datum, and let $K \subset \BG(\AA_f)$ be a sufficiently small compact open subgroup of the form $K = K^p K_p$, where $K_p \subset \BG(\QQ_p)$, $K^p \subset \BG(\AA_f^p)$, and $\AA_f^p$ denotes the ring of finite adeles over \QQ with trivial $p$-component.  Let $\BE$ denote the reflex field, let $\Sh_K(\BG,X)$ denote the canonical model of the Shimura variety over $\BE$, let $v$ be a place of $\BE$ over $p$, and let $\O_{\BE_v}$ denote the ring of integers in the $v$-adic completion $\BE_v$.  When $K_p$ is a \emph{hyperspecial} subgroup of $\BG(\QQ_p)$, smooth $\O_{\BE_v}$-models of $\Sh_K(\BG,X)$ were constructed in most cases of PEL type by Kottwitz \cite{kott92}, and, more generally, in cases of abelian type by Kisin \cite{ki10} and Vasiu (see \cite{vas08} and the references therein).

Our concern in this paper is the more general setting that $K_p$ is a \emph{parahoric,} but not necessarily hyperspecial, subgroup.  Then it is no longer reasonable to expect a smooth integral model.  For PEL Shimura varieties, when $K_p$ is a parahoric subgroup that can be described as the stabilizer of a lattice chain, Rapoport and Zink constructed natural integral models over $\O_{\BE_v}$ in \cite{rapzink96}.  When the localization $\BG_{\QQ_p}$ is unramified with simple factors of types only $A$ and $C$, G\"ortz \cite{goertz01,goertz03} showed that Rapoport and Zink's models are flat with reduced special fiber, and that the irreducible components of the special fiber are normal with rational singularities.  By contrast, when $\BG_{\QQ_p}$ is \emph{ramified,} Pappas \cite{pap00} observed that the Rapoport--Zink models may fail to be flat.

The key tool to investigate flatness of the Rapoport--Zink models, as well as other questions of a local nature, is the \emph{naive local model,} also introduced by Rapoport and Zink in \cite{rapzink96}.  The naive local model gives an \'etale-local neighborhood of the Rapoport--Zink integral Shimura model, but it is defined in terms of a purely linear-algebraic moduli problem, whereas the Shimura model is defined in terms of abelian varieties.  Therefore one expects the naive local model to be easier to study.  Here ``naive'' signifies that the naive local model, like the integral Shimura model, may fail to be flat.

In light of Pappas's observation, it is necessary to correct the definition of the integral Shimura models --- or, equivalently, of the naive local models --- to obtain flat schemes over $\O_{\BE_v}$.  The most straightforward solution is to define the corrected models and local models to be the scheme-theoretic closure of the generic fiber in the original models and local models, respectively.  Then the corrected models and local models are flat essentially by definition.  Zhu's recent proof \cite{zhu?} of the \emph{coherence conjecture} of Pappas--Rapoport \cite{paprap08} has strong implications for the geometry of the corrected models and local models whenever $\BG_{\QQ_p}$ is tamely ramified; see \cite{paprap08}*{Th.\ 11.3} for the local models we shall consider in this paper.  However, defined this way, the corrected models and local models carry the disadvantage of not admitting ready moduli-theoretic descriptions, since the process of scheme-theoretic closure can be difficult to interpret in a moduli-theoretic way.  Thus it is of interest to describe, when possible, the corrected models and local models via a refinement of the original moduli problem.  This is the problem of concern in this paper and its prequel \cite{sm11d}, in the case of a unitary similitude group attached to an imaginary quadratic number field ramified at $p \neq 2$.

More precisely, let $F/F_0$ be a ramified quadratic extension of discretely valued, non-Archimedean fields with perfect residue field of characteristic not $2$.  Endow $F^n$, $n \geq 2$, with the $F/F_0$-Hermitian form $\phi$ specified by the values $\phi(e_i,e_j) = \delta_{i,n+1-j}$ on the standard basis vectors $e_1,\dotsc$, $e_n$, and consider the reductive group $G := GU_{n}:= GU(\phi)$ over $F_0$.  In this paper we shall consider exclusively the case that $n$ is even; the case of odd $n$ was treated in \cite{sm11d}.  Let $m := n/2$.  In the even case, each conjugacy class of parahoric subgroups in $G(F_0)$ contains a unique parahoric of the form $P_I^\circ$, where $I$ is a nonempty subset of $\{0,1,\dotsc,m\}$ with the property
\begin{equation}\label{disp:I_cond}
   m+1 \in I \implies m \in I;
\end{equation}
see \s\ref{ss:parahoric} for definitions, and \cite{paprap09}*{\s1.2.3(b)} for more details.  We remark that $P_I^\circ$ is the full stabilizer of a lattice chain only when $m \in I$, owing to the existence of elements with nontrivial Kottwitz invariant.

Identifying $G \otimes_{F_0} F \iso GL_{n,F} \times \GG_{m,F}$, let $\mu_{r,s}$ denote the cocharacter $\bigl(1^{(s)},0^{(r)},1\bigr)$ of $D \times \GG_{m,F}$, where $D$ denotes the standard maximal torus of diagonal matrices in $GL_{n,F}$, and $r$ and $s$ are nonnegative integers with $r + s = n$.  Let $\{\mu_{r,s}\}$ denote the geometric conjugacy class of $\mu_{r,s}$.  In the special case that $F_0 = \QQ_p$ and $(F^n,\phi)$ is isomorphic to the $\QQ_p$-localization of a Hermitian space $(K^n,\psi)$ with $K$ an imaginary quadratic number field, the pair $(r,s)$ denotes the signature of $(K^n,\psi)$, and $\{\mu_{r,s}\}$ is the conjugacy class of minuscule cocharacters obtained in the usual way from the Shimura datum attached to the associated unitary similitude group, as in \cite[\s1.1]{paprap09}.

Let $E := F_0$ if $r = s$, and $E := F$ if $r \neq s$.  Then $E$ is the  reflex field of the conjugacy class $\{\mu_{r,s}\}$.  Attached to the triple $(G,\{\mu_{r,s}\},P_I^\circ)$ is the naive local model $M_I^\naive$ \cite{rapzink96,paprap09}, a projective $\O_E$-scheme whose explicit definition we shall recall in \s\ref{ss:naive_lm}.  As observed by Pappas \cite{pap00}, $M_I^\naive$ is not flat over $\O_E$ in general.  Let $M_I^\loc$ denote the honest local model, that is, the scheme-theoretic closure of the generic fiber of $M_I^\naive$ in $M_I^\naive$.  The papers \cite{pap00,paprap09} propose to describe $M_I^\loc$ by adding new conditions to the moduli problem defining $M_I^\naive$:  first Pappas adds the \emph{wedge condition} to define a closed subscheme $M_I^\wedge \subset M_I^\naive$, the \emph{wedge local model;} and then Pappas and Rapoport add a further condition, the \emph{spin condition,} to define a third closed subscheme $M_I^\spin \subset M_I^\wedge$, the \emph{spin local model} (see \s\ref{ss:wedge_spin_conds}).  The schemes
\[
   M_I^\loc \subset M_I^\spin \subset M_I^\wedge \subset M_I^\naive
\]
all have common generic fiber, and Pappas and Rapoport conjecture the following.

\begin{conj}[Pappas--Rapoport \cite{paprap09}*{7.3}]
$M_I^\spin$ coincides with the local model $M_I^\loc$ inside $M_I^\naive$, or in other words, $M_I^\spin$ is flat over $\O_E$.
\end{conj}

Although the full conjecture remains open, the main result of this paper is the following version of it, for even $n$.

\begin{thm}\label{st:main_thm}
Suppose that $n$ is even.  Then the scheme $M_I^\spin$ is topologically flat over $\O_E$, or in other words, the underlying topological spaces of $M_I^\spin$ and $M_I^\loc$ coincide.
\end{thm}

See \s\ref{ss:top_flat}.  Recall that a scheme over a regular, integral, $1$-dimensional base scheme is \emph{topologically flat} if its generic fiber is dense.  We proved the version of Theorem \ref{st:main_thm} for odd $n$ in \cite{sm11d}, where we also showed that the wedge local model $M_I^\wedge$ is topologically flat (though typically not flat, as its scheme structure typically differs from that of $M_I^\spin$).  For even $n$, we shall see later that $M_I^\wedge$ and $M_I^\spin$ typically do not even coincide topologically.  For example, when $m \in I$, $M_I^\wedge$ and $M_I^\spin$ coincide topologically only for the trivial signatures $(r,s) = (n,0)$ and $(r,s) = (0,n)$.

Following G\"ortz \cite{goertz01} (see also \cites{goertz03,goertz05,paprap03,paprap05,paprap08,paprap09,sm11b,sm11d}), the key technique in the proof of Theorem \ref{st:main_thm} is to embed the geometric special fiber of $M_I^\naive$ in an affine flag variety for $GU_n$, where it and the geometric special fibers of $M_I^\wedge$, $M_I^\spin$, and $M_I^\loc$ decompose (as topological spaces) into unions of finitely many Schubert varieties; see \s\ref{ss:embedding}.  We show the following.

\begin{thm}\label{st:2nd_main_thm}
With respect to the embedding into the affine flag variety specified in \s\ref{ss:embedding},
the Schubert varieties contained in the geometric special fiber of $M_I^\spin$ are indexed by the $\{\mu_{r,s}\}$-admissible set.
\end{thm}

See \s\ref{ss:adm_perm_sets} for the definition of the $\{\mu_{r,s}\}$-admissible set, and \s\ref{ss:top_flat} for the proof of the theorem.  The $\{\mu_{r,s}\}$-admissible set was defined by Kottwitz--Rapoport \cite{kottrap00} and Rapoport \cite{rap05}. It is a subset of a certain double quotient of the \emph{Iwahori-Weyl group} (see \s\ref{ss:I-W_gp}), which indexes the Schubert varieties in the relevant affine flag variety.  Pappas and Rapoport show in \cite{paprap09}*{Prop.\ 3.1} that the $\{\mu_{r,s}\}$-admissible Schubert varieties are contained in the geometric special fiber of $M_I^\loc$, which together with Theorem \ref{st:2nd_main_thm} implies Theorem \ref{st:main_thm}.

The analog of Theorem \ref{st:2nd_main_thm} for odd $n$ is proved in \cite{sm11d}.  However, the proof for even $n$ that we shall give here is considerably more laborious.  For simplicity, let us explain this in the case that our parahoric subgroup is an Iwahori subgroup.  Then the Schubert varieties in the affine flag variety are indexed by the Iwahori-Weyl group $\wt W_G$ itself.  The inclusion relations between the Schubert varieties are given by the \emph{Bruhat order $\leq$} on $\wt W_G$, which is defined in a purely combinatorial way; see \s\ref{ss:bo}.  Thus the problem of identifying which Schubert varieties are contained in the geometric special fiber of the spin local model amounts to the problem of identifying a certain finite subset of $\wt W_G$ closed under the Bruhat order.

Now let \A be an apartment in the building for the adjoint group of $G$, and let \H be the collection of affine root hyperplanes in \A.  By \cite{bourLGLA4-6}*{VI \s2.5 Prop.\ 8}, upon choosing a special vertex in \A as origin, there exists a unique reduced root system $\Sigma$ on \A such that the hyperplanes \H are precisely the affine root hyperplanes for $\Sigma$.  The Bruhat order on $\wt W_G$ is governed by the affine Weyl group for $\Sigma$, and indeed the study of $(\wt W_G,\leq )$, at least for the purposes of this paper, reduces to that of an extended affine Weyl group for $\Sigma$, as is explained in \cite{prs?}*{Rems.\ 4.4.7, 4.4.8}.  For odd $n$, $\Sigma$ is of type $C$, and Theorem \ref{st:2nd_main_thm} is proved in \cite{sm11d} by exploiting a result of Haines--Ng\^o \cite{hngo02b} relating admissible sets for $GSp_{2m}$ to \emph{permissible} sets for $GL_{2m}$.  See \s\ref{ss:adm_perm_sets} for the definition of permissibility; the force of their result is that it is relatively easy to determine when a given element is permissible, at least as compared to when an element is admissible.

By contrast, for even $n = 2m$, $\Sigma$ is of type $B_m$, and the results of Haines--Ng\^o do not apply.  We are therefore forced to work through the combinatorics largely from scratch.  To fix notation, let us make the harmless assumption that $s \leq r$, as we shall do in the body of the paper.  We shall see in \s\s\ref{ss:relation_to_B}--\ref{ss:mu_r,s-adm_set} that $\wt W_G$, equipped with its Bruhat order, identifies with the Iwahori-Weyl group $\wt W_{B_m}$ of split $GO_{2m+1}$, and that in this way the $\{\mu_{r,s}\}$-admissible set identifies with the $\mu_B$-admissible set in $\wt W_{B_m}$, where
\[
   \mu_B:= \bigl(2^{(s)},1^{(2m+1-2s)},0^{(s)}\bigr).
\]
In \s\s\ref{ss:naive_perm}--\ref{ss:spin-perm} we shall obtain simple combinatorial conditions in $\wt W_G$ that describe the Schubert varieties contained in the geometric special fiber of the spin local model.  To prove Theorem \ref{st:2nd_main_thm}, we must then show that these conditions, translated to $\wt W_{B_m}$, characterize the $\mu_B$-admissible set.

Rather than working directly with $\wt W_{B_m}$, we shall approach this problem by passing to the Iwahori-Weyl group $\wt W_{D_{m+1}}$ of split $GO_{2m+2}$.  The root datum of $GO_{2m+1}$ is a Steinberg fixed-point root datum (see \s\ref{ss:steinberg}) obtained from the root datum of $GO_{2m+2}$.  Therefore one has a natural inclusion $\wt W_{B_m} \inj \wt W_{D_{m+1}}$, and the Bruhat order on $\wt W_{D_{m+1}}$ restricts to the Bruhat order on $\wt W_{B_m}$.  The cocharacter $\mu_B$ in $\wt W_{B_m}$ has image
\[
   \mu_D := \bigl(2^{(s)},1^{(2m+2-2s)},0^{(s)}\bigr)
\]
in $\wt W_{D_{m+1}}$.  Motivated by the conditions characterizing the Schubert varieties that occur in the geometric special fiber of the spin local model, in \eqref{def:spin-perm_D} we introduce the notion of \emph{$\mu_D$-spin-permissibility} for elements in $\wt W_{D_{m+1}}$.  The fundamental combinatorial result in this paper is the following.

\begin{thm}\label{st:adm_perm_D}
An element $w \in \wt W_{D_{m+1}}$ is $\mu_D$-admissible $\iff$ $w$ is $\mu_D$-spin-permis\-si\-ble.
\end{thm}

The proof of Theorem \ref{st:adm_perm_D} will occupy almost all of \s\s\ref{ss:proper_elts}--\ref{ss:sp-perm=>adm_D}.  We shall extend it to the general parahoric case in \eqref{st:sp-perm=adm_D_parahoric}.  With Theorem \ref{st:adm_perm_D} (and its parahoric variants) in hand, the characterization of the $\mu_B$-admissible set in $\wt W_{B_m}$ (and its parahoric variants) that we need for Theorem \ref{st:2nd_main_thm} follows easily, via the formalism of Steinberg fixed-point root data.

In the course of working out the combinatorics in \s\s\ref{s:combinatorics_D}--\ref{s:combinatorics_B}, we shall also show that for the cocharacters $\mu_B$ and $\mu_D$, admissibility is equivalent to the property of \emph{vertexwise admissibility;} see \s\ref{ss:vert-adm_D} and \s\ref{ss:vert_adm_B}.  This result for $\mu_B$ immediately implies the conjecture \cite{paprap09}*{Conj.\ 4.2} of Pappas and Rapoport.  Vertexwise admissibility is defined in general in \cite{prs?}*{\s4.5}; particular instances of it appeared earlier (sometimes only implicitly) in papers of Pappas and Rapoport \cite{paprap03,paprap05,paprap09}.  The definition is motivated by Pappas and Rapoport's idea to describe the local model for general parahorics by first taking the flat closure in the maximal parahoric case, and then taking the intersection of the inverse images of these flat closures in the naive local model.  See, for example, p.\ 512 and \s4.2 in \cite{paprap09}.

We shall also obtain very explicit descriptions of the $\mu_D$-permissible set in $\wt W_{D_{m+1}}$ and its parahoric variants, and of the $\mu_B$-permissible set in $\wt W_{B_m}$ and its parahoric variants; see \eqref{st:mu-perm_D}, \eqref{st:mu-perm_D_parahoric}, and \eqref{st:mu-perm_B}.  It is a general result of Kottwitz--Rapoport \cite{kottrap00} that, for any cocharacter $\mu$ in a root datum, $\mu$-admissibility implies $\mu$-permissibility.  Haines and Ng\^o \cite{hngo02b} subsequently showed that the converse can fail.  More precisely, they showed that every irreducible root system of rank $\geq 4$ and not of type $A$ has coweights $\mu$ for which the $\mu$-permissible set is strictly bigger than the $\mu$-admissible set.  For our particular cocharacters $\mu_B$ and $\mu_D$ as above, we shall show that if $m$, $s \geq 3$, then the $\mu_B$- (resp.\ $\mu_D$-)admissible and $\mu_B$- (resp.\ $\mu_D$-)permissible sets in $\wt W_{B_m}$ (resp.\ $\wt W_{D_{m+1}}$) are not equal; see \eqref{eg:not_mu-adm_D} and \eqref{eg:perm_neq_adm_B}.  In the case of $\wt W_{B_3}$, this gives an example of a cocharacter in a rank $3$ root datum for which admissibility and permissibility do not coincide.  In the case of $\wt W_{D_{m+1}}$, we find that admissibility and permissibility do not coincide for the cocharacter
\[
   \bigl(2^{(m)},1,1,0^{(m)}\bigr) = \bigl(1^{(m+1)},0^{(m+1)}\bigr)
       + \bigl(1^{(m)},0,1,0^{(m)}\bigr),
\]
which is a sum of minuscule cocharacters for $D_{m+1}$ contained in the same closed Weyl chamber.  This gives a counterexample to Rapoport's conjecture \cite{rap05}*{\s3, p.~283} that admissibility and permissibility are equivalent for any cocharacter which is a sum of dominant minuscule cocharacters.  See \eqref{eg:not_mu-adm_D} for some further discussion.

Let us now outline the contents of the paper.  In \s\ref{s:loc_mod} we review the definitions of the various local models from \cite{paprap09}.  In \s\ref{s:GU} we review some basic group-theoretic aspects of ramified, quasi-split $GU_{2m}$.  In \s\ref{s:prelim_comb} we introduce some preliminary combinatorial considerations, essentially covering just enough so that we can translate the wedge and spin conditions into combinatorics later in \s\ref{s:schub_cells}.  In \s\ref{s:I-W_gps} we review the general formalism of Iwahori-Weyl groups, and then specialize this to the case of our group $GU_{2m}$.  In \s\ref{s:afv} we review the embedding of the geometric special fiber of the local model into an appropriate affine flag variety for $GU_{2m}$ from \cite{paprap09}.  In \s\ref{s:schub_cells} we obtain combinatorial descriptions of the Schubert varieties contained in the geometric special fibers of $M_I^\naive$, $M_{I}^\wedge$, and $M_{I}^\spin$ inside the affine flag variety, and we prove Theorems \ref{st:main_thm} and \ref{st:2nd_main_thm} modulo the later result \eqref{st:sp-perm=adm_B}, which characterizes the $\mu_B$-admissible set in $\wt W_{B_m}$.  Section \ref{s:combinatorics_D} is the technical heart of the paper.  In it we cover all of the type $D$ combinatorics, most notably proving Theorem \ref{st:adm_perm_D} and its parahoric generalization \eqref{st:sp-perm=adm_D_parahoric}.  We conclude the paper in \s\ref{s:combinatorics_B} by working through the type $B$ combinatorics, essentially all of which follows from \s\ref{s:combinatorics_D} via the formalism of Steinberg fixed-point root data.

\subsection*{Acknowledgments}
I thank Michael Rapoport for his encouragement to work on this problem and for introducing me to the subject.  
I also thank him, Thomas Haines, Robert Kottwitz, Stephen Kudla, and George Pappas for many helpful conversations.

\subsection*{Notation}
Throughout the paper $F/F_0$ denotes a ramified quadratic extension of discretely valued, non-Archimedean fields with respective rings of integers $\O_F$ and $\O_{F_0}$ and respective uniformizers $\pi$ and $\pi_0$ satisfying $\pi^2 = \pi_0$.  We assume that the residue field is perfect of characteristic not $2$, and we denote by $k$ an algebraic closure of it.  We also employ an auxiliary ramified quadratic extension $K/K_0$ of discretely valued, non-Archimedean Henselian fields with perfect residue field of characteristic not $2$, with respective rings of integers $\O_{K}$ and $\O_{K_0}$, and with respective uniformizers $u$ and $t$ satisfying $u^2 = t$; eventually $K$ and $K_0$ will be the fields of Laurent series $k((u))$ and $k((t))$, respectively.  We put $\Gamma:= \Gal(K/K_0)$, and we write $x \mapsto \ol x$ for the action of the nontrivial element of $\Gamma$ on $K$, so that $\ol u = -u$.  Abusing notation, we continue to write $x \mapsto \ol x$ for the $R$-algebra automorphism of $K \otimes_{K_0} R$ induced by any base change $K_0 \to R$.

We fix once and for all an even integer $2m \geq 2$ and nonnegative integers $r$ and $s$ with $2m = r+s$.%
\footnote{Strictly speaking, the paper of Pappas and Rapoport \cite{paprap09} restricts to unitary groups $GU_n$ for $n \geq 3$, but all of our appeals to their paper will still go through for $n = 2m = 2$.  On the other hand, local models for $GU_2$ are worked out explicitly in \cite{prs?}*{Rem.\ 2.6.13}.}
It will be easy to see later on that the cocharacters $\mu_{r,s}$ and $\mu_{s,r}$ defined above specify isomorphic local models, and that these cocharacters even have the same image in the Galois coinvariants $X_*(D \times \GG_m)_\Gamma$.  Hence for simplicity of notation we shall always assume $s \leq r$.  Although almost everything we shall do will depend on $m$, $r$, and $s$, we shall usually not embed them into the notation.  We continue to take $E := F_0$ if $r = s$ and $E:= F$ if $r \neq s$.

We relate objects by writing $\iso$ for isomorphic, $\ciso$  for canonically isomorphic, and $=$  for equal.

For $x \in \RR$, we write $\lfloor x \rfloor$ for the greatest integer $\leq x$ and $\lceil x \rceil$ for the least integer $\geq x$.

For any positive integer $n$, given a vector $v \in \RR^n$, we write $v(j)$ for the $j$th entry of $v$, $\Sigma v$ for the sum of the entries of $v$, and $v^*$ for the vector in $\RR^n$ defined by $v^*(j) = v(n+1-j)$.  Given another vector $w$, we write $v \geq w$ if $v(j) \geq w(j)$ for all $j$.  We write $\mathbf d$ for the vector $(d,d,\dotsc,d)$, leaving it to context to make clear the number of entries.  The expression $(d^{(i)}, e^{(j)}, \dotsc)$ denotes the vector with $d$ repeated $i$ times, followed by $e$ repeated $j$ times, and so on.

We write $S_n$ for the symmetric group on $1,\dotsc,$ $n$; $S_n^*$ for the subgroup
\[
   S_n^* := \bigl\{\, \sigma \in S_n \bigm| \sigma(n+1-j) = n+1 - \sigma(j)
                     \text{ for all } j \in \{1,\dotsc,n\}\,\bigr\};
\]
and $S_n^\circ$ for the subgroup
\[
   S_n^\circ := \{\, \sigma \in S_n^* \mid \text{$\sigma$ is 
                               even in $S_n$} \,\}.
\]

Given facets $\mathbf f$ and $\mathbf{f'}$ in an apartment, we write $\mathbf f \preceq \mathbf{f'}$ if $\mathbf f$ is a subfacet of $\mathbf{f'}$, that is, if $\mathbf f \subset \ol{\mathbf{f'}}$.  We say that a point $v$ is a \emph{vertex} of $\mathbf f$ if $v$ is contained in a minimal subfacet of $\mathbf f$.  In this paper, the apartments we shall consider will typically be attached to reductive groups which are not semisimple, and the minimal facets in such apartments are not themselves points.

\section{Unitary local models}\label{s:loc_mod}
\numberwithin{equation}{subsection}

We begin by recalling the definition and some of the discussion of local models for even ramified unitary groups from \cite{paprap09}.  Let $I \subset \{0,\dotsc,m\}$ be a nonempty subset satisfying property \eqref{disp:I_cond}.

\subsection{Pairings}\label{ss:pairings}
Let $V := F^{2m}$.  In this subsection we introduce some pairings on $V$ and notation related to them.

Let $e_1$, $e_2,\dotsc$, $e_{2m}$ denote the standard ordered $F$-basis in $V$, and let
\[
   \phi\colon V \times V \to F
\]
be the $F/F_0$-Hermitian form on $V$ whose matrix with respect to the standard basis is
\begin{equation}\label{disp:antidiag_1}
   \begin{pmatrix}
     &  &  1\\
     & \iddots\\
     1
   \end{pmatrix}.
\end{equation}
We attach to $\phi$ the respective alternating and symmetric $F_0$-bilinear forms $V \times V \to F_0$
\begin{equation}\label{disp:pairings}
   \langle x,y \rangle := \tfrac 1 2 \Tr_{F/F_0}\bigl( \pi^{-1}\phi(x,y) \bigr)
   \quad\text{and}\quad
   (x,y) := \tfrac 1 2 \Tr_{F/F_0}\bigl( \phi(x,y) \bigr).
\end{equation}

For any $\O_F$-lattice $\Lambda \subset V$, we denote by $\wh \Lambda$ the $\phi$-dual of $\Lambda$,
\begin{equation}\label{disp:wh_Lambda}
   \wh \Lambda := \bigl\{\, x\in V \bigm| \phi(\Lambda,x) \subset \O_F \,\bigr\}.
\end{equation}
Then $\wh\Lambda$ is also the $\langle$~,~$\rangle$-dual of $\Lambda$,
\[
   \wh\Lambda = \bigl\{\,  x\in V \bigm| \langle \Lambda,x\rangle \subset \O_{F_0} \,\bigr\};
\]
and $\wh\Lambda$ is related to the $($~,~$)$-dual $\wh\Lambda^s := \{\,  x\in V \mid ( \Lambda,x ) \subset \O_{F_0} \,\}$ by the formula $\wh\Lambda^s = \pi^{-1}\wh\Lambda$.  Both $\wh\Lambda$ and $\wh\Lambda^s$ are $\O_F$-lattices in $V$, and the forms $\langle$~,~$\rangle$ and $($~,~$)$ induce perfect $\O_{F_0}$-bilinear pairings
\begin{equation}\label{disp:perf_pairing}
   \Lambda \times \wh\Lambda \xra{\text{$\langle$~,~$\rangle$}} \O_{F_0}
   \quad\text{and}\quad
   \Lambda \times \wh\Lambda^s \xra{\text{$($~,~$)$}} \O_{F_0}
\end{equation}
for all $\Lambda$.

\subsection{Standard lattices}\label{ss:lattices}

For $i = 2mb+c$ with $0 \leq c < 2m$, we define the \emph{standard $\O_F$-lattice}
\begin{equation}\label{disp:Lambda_i}
   \Lambda_i := \sum_{j=1}^c\pi^{-b-1}\O_F e_j + \sum_{j=c+1}^{2m} \pi^{-b}\O_F e_j \subset V.
\end{equation}
Then $\wh\Lambda_i = \Lambda_{-i}$ for all $i$, and the $\Lambda_i$'s form a complete, periodic, self-dual lattice chain
\[
   \dotsb \subset \Lambda_{-2} \subset \Lambda_{-1} \subset \Lambda_0 \subset \Lambda_1 \subset \Lambda_2 \subset \dotsb,
\]
which we call the \emph{standard lattice chain}.  More generally, for our given subset $I \subset \{0,\dotsc,m\}$, we denote by $\Lambda_I$ the periodic, self-dual subchain of the standard chain consisting of all lattices of the form $\Lambda_i$ for $i \in 2m\ZZ \pm I$.  Thus $\Lambda_{\{0,\dotsc,m\}}$ denotes the standard chain itself.

The standard lattice chain admits the following obvious trivialization.  Let $\beta_i\colon \O_F^{2m} \to \O_F^{2m}$ multiply the $i$th standard basis element $\epsilon_i$ by $\pi$ and send all other standard basis elements to themselves.  Then there is a unique isomorphism of chains of $\O_F$-modules
\begin{equation}\label{disp:latt_triv}
   \vcenter{
   \xymatrix{
      \dotsb\,\, \ar@{^{(}->}[r]
         & \Lambda_0\, \ar@{^{(}->}[r] 
         & \Lambda_1\, \ar@{^{(}->}[r]
         & \,\dotsb\,\, \ar@{^{(}->}[r] 
         & \Lambda_{2m}\, \ar@{^{(}->}[r]
         & \,\dotsb\\
      \dotsb\, \ar[r]^-{\beta_{2m}}
         & \O_F^{2m} \ar[u]_\sim \ar[r]^-{\beta_1}
         & \O_F^{2m} \ar[u]_-\sim \ar[r]^-{\beta_2}
         & \,\dotsb\, \ar[r]^-{\beta_{2m}}
         & \O_F^{2m} \ar[u]_\sim \ar[r]^-{\beta_1}
         & \,\dotsb
   }
   }
\end{equation}
such that the leftmost displayed vertical arrow identifies the ordered $\O_F$-basis $\epsilon_1,\dotsc,\epsilon_{2m}$ of $\O_F^{2m}$ with the ordered basis $e_1,\dotsc,e_{2m}$ of $\Lambda_0$.  Restricting to subchains in the top and bottom rows in \eqref{disp:latt_triv}, we get an analogous trivialization of $\Lambda_I$ for any $I$.

\subsection{Naive local models}\label{ss:naive_lm}
We now review the definition of the naive local models from \cite{paprap09}*{\s1.5}.

Recall our fixed partition $2m = r + s$ with $s \leq r$.  The \emph{naive local model
$M_I^\naive$} is the following contravariant functor on the
category of $\O_E$-algebras.  Given an $\O_E$-algebra $R$, an $R$-point in $M_I^\naive$ consists of, up to an obvious notion of isomorphism,
\begin{itemize}
\item
   a functor
   \[
      \xymatrix@R=0ex{
         \Lambda_I\vphantom{(O_F \otimes_{\O_{F_0}}\O_S)} \ar[r] & {}(\text{$\O_F \otimes_{\O_{F_0}} R$-modules})\\
         \Lambda_i \ar@{|->}[r] & \F_i,
      }
   \]
   where $\Lambda_I$ is regarded as a category by taking the morphisms to be the inclusions of lattices in $V$; together with
\item
   an inclusion of $\O_F \otimes_{\O_{F_0}} R$-modules $\F_i \inj \Lambda_i \otimes_{\O_{F_0}} R$ for each $i \in 2m\ZZ \pm I$, functorial in $\Lambda_i$;
\end{itemize}
satisfying the following conditions for all $i \in 2m\ZZ \pm I$.
\begin{enumerate}
\renewcommand{\theenumi}{LM\arabic{enumi}}
\item 
   Zariski-locally on $\Spec R$, $\F_i$ embeds in $\Lambda_i \otimes_{\O_{F_0}} R$ as a direct $R$-module summand of rank $2m$.
\item\label{it:period_cond}
   The isomorphism $\Lambda_i \otimes_{\O_{F_0}} R \isoarrow \Lambda_{i - 2m} \otimes_{\O_{F_0}} R$ obtained by tensoring $\Lambda_i \xra[\sim]\pi \pi \Lambda_i = \Lambda_{i - 2m}$ identifies $\F_i$ with $\F_{i - 2m}$.
\item\label{it:perp_cond}
   The perfect $R$-bilinear pairing 
   \[
      (\Lambda_i \otimes_{\O_{F_0}} R) 
         \times (\Lambda_{-i} \otimes_{\O_{F_0}} R)
      \xra{\text{$\langle$~,~$\rangle$} \otimes R}
      R
   \]
   induced by \eqref{disp:perf_pairing} identifies $\F_i^\perp \subset \Lambda_{-i} \otimes_{\O_{F_0}} R$ with $\F_{-i}$.
\item\label{it:kottwitz_cond}
   The element $\pi \otimes 1 \in \O_F \otimes_{\O_{F_0}} R$ acts on $\F_i$ as an $R$-linear endomorphism with characteristic polynomial
   \[
      \det(\,T\cdot \id - \pi \otimes 1 \mid \F_i\,) = (T - \pi)^s (T+ \pi)^{r} \in R[T].
   \]
\end{enumerate}

When $E = F_0$, i.e.\ when $r = s = m$, the polynomial on the right-hand side of the display is to be interpreted as $(T^2 - \pi_0)^m$.  The functor $M_I^\naive$ is plainly represented by a closed subscheme, which
we again denote $M_I^\naive$, of a finite product of Grassmannians over
$\Spec \O_E$.  The $F$-generic fiber $M_I^\naive \otimes_{\O_E} F$ can be identified with the Grassmannian of $s$-planes in a $2m$-dimensional vector space; see \cite{paprap09}*{\s1.5.3}.

% As noted in \cite{paprap09}*{\s1.5.3}, the map $(\F_i)_i \mapsto \ker[\pi \otimes 1 - 1 \otimes \pi \mid \F_i]$ (for any $i$) specifies an isomorphism of the generic fiber $M_I^\naive \otimes_{\O_F} F$ onto $\Gr(s,V_0)$, the Grassmannian of $s$-planes in $V_0$, where $V_0$ is the $n$-dimensional $F$-vector space
% \[
%    V_0 := \ker[\pi\otimes 1 - 1 \otimes \pi \mid V \otimes_{F_0} F].
% \]

\subsection{Wedge and spin conditions}\label{ss:wedge_spin_conds}
In this subsection we review the wedge and spin conditions, which Pappas and Rapoport conjecture to define the honest local model inside $M_I^\naive$.

The \emph{wedge condition,} due to Pappas \cite{pap00}*{\s4}, on an $R$-point $(\F_i)_i$ of $M_I^\naive$ is the condition that
\begin{enumerate}
\setcounter{enumi}{4}
\renewcommand{\theenumi}{LM\arabic{enumi}}
\item
   if $r \neq s$, then for all $i \in 2m\ZZ \pm I$,
   \[
      \sideset{}{_R^{s+1}}{\bigwedge} (\,\pi\otimes 1 + 1 \otimes \pi \mid \F_i\,) = 0
      \quad\text{and}\quad
      \sideset{}{_R^{r+1}}{\bigwedge} (\,\pi\otimes 1 - 1 \otimes \pi \mid \F_i\,) = 0.
   \]
   (There is no condition when $r=s$.)
\end{enumerate}
We denote by $M_I^\wedge$ the subfunctor of $M_I^\naive$ of points that satisfy the wedge condition, and we call it the \emph{wedge local model}.  Plainly $M_I^\wedge$ is a closed subscheme of $M_I^\naive$.  As noted in \cite{paprap09}*{\s1.5.6}, the generic fibers of $M_I^\wedge$ and $M_I^\naive$ coincide.

To formulate the spin condition it is necessary to introduce some notation.  For brevity we shall recall only the minimum that we need from \cite{paprap09}*{\s7}; compare also \citelist{\cite{sm11b}*{\s2.3}\cite{sm11d}*{\s2.5}}.  Regarding $V$ as a $4m$-dimensional vector space over $F_0$, consider the ordered $F_0$-basis
\[
   -\pi^{-1}e_1, \dotsc, -\pi^{-1}e_m, e_{m+1}, \dotsc, e_{2m}, e_1,\dotsc, e_{m}, \pi e_{m+1}, \dotsc, \pi e_{2m},
\]
which we denote by $f_1,\dotsc$, $f_{4m}$.  Then $f_1,\dotsc,$ $f_{4m}$ is \emph{split} for the form $($~,~$)$ from \eqref{disp:pairings}, that is, we have $(f_i,f_j) = \delta_{i,4m+1-j}$ for all $i$ and $j$.

Using the basis $f_1,\dotsc,$ $f_{4m}$, we define an operator $a$ on $\bigwedge_{F_0}^{2m} V$ as follows.  For any subset $E \subset \{1,\dotsc,4m\}$ of cardinality $2m$, let
\begin{equation}\label{disp:f_E}
   f_E := f_{i_1}\wedge\dotsb\wedge f_{i_{2m}} \in \sideset{}{_{F_0}^{2m}}{\bigwedge} V,
\end{equation}
where $E = \{i_1,\dotsc,i_{2m}\}$ with $i_1 < \dotsb < i_{2m}$.  Given such $E$, we also let 
\[
   E^\perp := (4m+1-E)^c = 4m+1-E^c,
\]
where the set complements are taken in $\{1,\dotsc,4m\}$.  Then $E^\perp$ consists of the elements $j\in\{1,\dotsc,4m\}$ such that $(f_i,f_{j}) = 0$ for all $i\in E$.  We now define $a$ by defining it on the basis elements $f_E$ of $\bigwedge_{F_0}^{2m} V$ for varying $E$,
\[
   a(f_E) := \sgn(\sigma_E)f_{E^\perp},
\]
where $\sigma_E$ is the permutation on $\{1,\dotsc,4m\}$ sending $\{1,\dotsc,2m\}$ to the elements of $E$ in increasing order, and sending $\{2m+1,\dotsc,4m\}$ to the elements of $E^c$ in increasing order.

\begin{rk}\label{rk:a_sign_difference}
The definition of $a$ is in entirely the same spirit as \cite{sm11d}*{\s2.5}.  In analogy with \cite{sm11d}*{Rem.\ 2.5.2}, our $a$ agrees only up to a sign of $(-1)^m$ with the analogous operators denoted $a_{f_1\wedge \dotsb \wedge f_{4m}}$ in \cite{paprap09}*{disp.\ 7.6} and $a$ in \cite{sm11b}*{\s2.3}.
\end{rk}

It follows easily from the definition of $\sigma_E$, or directly from \cite{paprap09}*{Prop.\ 7.1} and the preceding remark, that $\sgn(\sigma_E) = \sgn(\sigma_{E^\perp})$.  Hence $a^2 = \id_{\bigwedge^{2m} V}$.  Hence $\bigwedge_{F_0}^{2m} V$ decomposes as
\[
   \sideset{}{_{F_0}^{2m}}{\bigwedge} V 
      = \Bigl(\sideset{}{_{F_0}^{2m}}{\bigwedge} V \Bigr)_{1} 
            \oplus \Bigl(\sideset{}{_{F_0}^{2m}}{\bigwedge} V \Bigr)_{-1},
\]
where
\[
   \Bigl(\sideset{}{_{F_0}^{2m}}{\bigwedge} V \Bigr)_{\pm 1} 
     := \spn_{F_0}\{f_E \pm \sgn(\sigma_E)f_{E^\perp}\}_E
\]
is the $\pm 1$-eigenspace for $a$; here $E$ ranges through the subsets of $\{1,\dotsc,4m\}$ of cardinality $2m$.

Now consider the $\O_F$-lattice $\Lambda_i \subset V$ for some $i$ and regard it as an $\O_{F_0}$-lattice.  Then $\bigwedge^{2m}_{\O_{F_0}} \Lambda_i$ is naturally an $\O_{F_0}$-lattice in $\bigwedge^{2m}_{F_0} V$.  We set
\[
   \Bigl(\sideset{}{_{F_0}^{2m}}{\bigwedge} \Lambda_i\Bigr)_{\pm 1}
      := \Bigl(\sideset{}{_{F_0}^{2m}}{\bigwedge} \Lambda_i \Bigr)
            \cap \Bigl(\sideset{}{_{F_0}^{2m}}{\bigwedge} V\Bigr)_{\pm 1}.
\]

To state the spin condition, recall our partition $2m = r + s$ and note that, given an $R$-point $(\F_i)_i$ of $M_I^\naive$, the $R$-module $\bigwedge_{R}^{2m} \F_i$ is naturally contained in $\bigwedge^{2m}_{R} (\Lambda_i \otimes_{\O_{F_0}} R)$ for all $i$.  The \emph{spin condition,} due to Pappas and Rapoport \cite{paprap09}*{\s7.2.1}, is that
\begin{enumerate}
\renewcommand{\theenumi}{LM\arabic{enumi}}
\setcounter{enumi}{5}
\item
   for all $i \in 2m\ZZ \pm I$, $\bigwedge_{R}^{2m} \F_i$ is contained in
   \[
      \im\biggl[
         \Bigl(\sideset{}{_{F_0}^{2m}}{\bigwedge} \Lambda_i\Bigr)_{(-1)^s}\otimes_{\O_{F_0}} R 
         \to \sideset{}{_{R}^{2m}}{\bigwedge} (\Lambda_i \otimes_{\O_{F_0}} R)
      \biggr].
   \]
\end{enumerate}
For fixed $i$, we say that $\F_i$ satisfies the spin condition if $\bigwedge_R^{2m} \F_i$ is contained in the displayed image.
We denote by $M_I^\spin$ the subfunctor of $M_I^\wedge$ of points satisfying the spin condition, and we call it the \emph{spin local model}.  Plainly $M_I^\spin$ is a closed subscheme of $M_I^\wedge$.  As noted in \cite{paprap09}*{\s7.2.2}, the generic fiber of $M_I^\spin$ agrees with the common generic fiber of $M_I^\naive$ and $M_I^\wedge$.

\begin{rk}
As we remarked in the odd case in \cite{sm11d}*{Rem.\ 2.5.3}, the statement of the spin condition in \cite{paprap09}*{\s7.2.1} contains a sign error which traces to the sign discrepancy observed in \eqref{rk:a_sign_difference}.  Indeed, in the even case the element $f_1\wedge\dotsb\wedge f_{2m}$ lies in the $(-1)^{m}$-eigenspace for the operator $a_{f_1\wedge\dotsb\wedge f_{4m}}$ of \cite{paprap09}.  Thus the argument of \cite{paprap09}*{\s7.2.2} shows that the sign of $(-1)^s$ in the statement of the spin condition in \cite{paprap09} should be replaced with $(-1)^{s+m}$.  For us, since $f_1\wedge\dotsb\wedge f_{2m}$ lies in the $+1$-eigenspace of our operator $a$, the same argument shows that we get a sign of $(-1)^s$.
\end{rk}

\subsection{Lattice chain automorphisms}\label{ss:latt_chain_auts}
Regarding $\Lambda_I$ as a lattice chain over $\O_{F_0}$, the perfect pairings $\langle$~,~$\rangle$ of \eqref{disp:perf_pairing} give a \emph{polarization} of $\Lambda_I$ in the sense of \cite{rapzink96}*{Def.\ 3.14} (with $B = F$, ${b}^* = \ol{b}$ in the notation of \cite{rapzink96}).  Let $\ul\Aut(\Lambda_I)$ denote the $\O_{F_0}$-group scheme of automorphisms of the lattice chain $\Lambda_I$ that preserve the pairings $\langle$~,~$\rangle$ for variable $\Lambda_i \in \Lambda_I$ up to common unit scalar.  Then $\ul\Aut(\Lambda_I)$ is smooth and affine over $\O_{F_0}$; see \cite{pap00}*{Th.\ 2.2}, which in turn relies on \cite{rapzink96}*{Th.\ 3.16}.  The base change $\ul\Aut(\Lambda_I)_{\O_E}$ acts naturally on $M_I^\naive$, and it is easy to see that this action preserves the closed subschemes $M_I^\wedge$, $M_I^\spin$, and $M_I^\loc$.  We shall return to this point in \s\ref{ss:embedding}.

\section{Unitary similitude group}\label{s:GU}
In this section we review a number of basic group-theoretic matters from \cite{paprap09}; these will become relevant when we begin to consider the affine flag variety.  We switch to working with respect to the auxiliary field extension $K/K_0$.  We write $i^* := 2m+1-i$ for $i \in \{1,\dotsc,2m\}$.

\subsection{Unitary similitudes}\label{ss:GU}
Let $h$ denote the Hermitian form on $K^{2m}$ whose matrix with respect to the standard ordered basis is \eqref{disp:antidiag_1}.  We denote by
\[
   G := GU_{2m} := GU(h)
\]
the algebraic group over $K_0$ of \emph{unitary similitudes} of $h$: for any $K_0$-algebra $R$, $G(R)$ is the group of elements $g\in GL_{n}(K \otimes_{K_0} R)$ satisfying $h_R(gx,gy) = c(g)h_R(x,y)$ for some $c(g) \in R^\times$ and all $x$, $y\in (K \otimes_{K_0} R)^{n}$, where $h_R$ is the induced form on $(K \otimes_{K_0} R)^{2m}$.  As the form $h$
is nonzero, the scalar $c(g)$ is uniquely determined, and $c$ defines an
exact sequence of $K_0$-groups
\[
   1 \to U_{2m} \to G \xra c \GG_m \to 1
\]
with evident kernel $U_{2m}:= U(h)$ the unitary group of $h$.  The displayed sequence splits (noncanonically), so that the choice of a splitting presents $G$ as a semidirect product $U_{2m} \rtimes \GG_m$.

After base change to $K$, we get the standard identification
\begin{equation}\label{disp:G_K=GL_ntimesG_m}
   G_K \xra[\sim]{(\varphi,c)} (GL_{2m})_K \times (\GG_{m})_K,
\end{equation}
where $\varphi\colon G_K \to (GL_{2m})_K$ is given on $R$-points by the map on matrix entries $K \otimes_{K_0} R \to R$ sending $x \otimes y \mapsto xy$.

\subsection{Tori}\label{ss:tori}
We denote by $S$ the standard diagonal maximal split torus in $G$:  on $R$-points,
\[
   S(R) = \bigl\{\, \diag(a_1,\dotsc,a_{2m}) \in GL_{n}(R) \bigm| 
          a_1 a_{2m} = \dotsb = a_m a_{m+1} \,\bigr\}.
\]
The centralizer $T$ of $S$ is the standard maximal torus of all diagonal matrices in $G$,
\[
   T(R) = \bigl\{\, \diag(a_1,\dotsc,a_{2m}) \in GL_n(K\otimes_{K_0} R) \bigm|
            a_1 \ol a_{2m} = \dotsb = a_m \ol a_{m+1} \,\bigr\}.
\]
The isomorphism \eqref{disp:G_K=GL_ntimesG_m} identifies $T_K$ with the split torus $D \times (\GG_{m})_K$, where $D$ denotes the standard diagonal maximal torus in $(GL_{n})_K$.  The standard identification $X_*\bigl(D \times (\GG_{m})_K\bigr) \ciso \ZZ^{2m} \times \ZZ$ then identifies the inclusion $X_*(S) \subset X_*(T)$ with
\begin{equation}\label{disp:X_*S_in_X_*T}
   \bigl\{\, (x_1,\dotsc,x_{2m},y) \bigm| x_1 + x_{2m} = \dotsb = x_m + x_{m+1} = y \,\bigr\}
   \subset \ZZ^{2m} \times \ZZ;
\end{equation}
note that this is not the description of $X_*(S)$ given in \cite{paprap09}*{\s2.4.1}, which appears to contain an error.

% It is convenient to now recall the cocharacter $\mu_{r,s} \in X_*(T)$ given in terms of our above identifications as
% \begin{equation}\label{disp:mu_s}
%    \mu_{r,s} := \bigl(1^{(s)},0^{(r)},1\bigr) \in \ZZ^{2m} \times \ZZ.
% \end{equation}

We write $\A_G$ for the standard apartment $X_*(S) \otimes_\ZZ \RR$, and we regard it as a sub\-space of $\RR^{2m} \times \RR$ via \eqref{disp:X_*S_in_X_*T}.

\subsection{Kottwitz homomorphism}\label{ss:Kott_hom}
Let $L$ be the completion of a maximal unramified extension of $K_0$, let $\ol L$ be an algebraic closure of $L$, and let $\I := \Gal(\ol L / L)$.  Let $\G$ be a connected reductive group defined over $L$.  In \cite{kott97}*{\s7} Kottwitz defines a surjective functorial surjection
\[
   \kappa_\G \colon \G(L) \surj \pi_1(\G)_\I,
\]
where the target is the \I-coinvariants of the fundamental group in the sense of Borovoi \cite{bor98}.

In the case of our group $G$, it harmless to work over $K_0$ itself, and we recall from \cite{paprap09}*{\s1.2.3(b)} that $\kappa_G$ may be described as the surjective map
\[
   \kappa_G\colon G(K_0) \surj \ZZ \oplus (\ZZ/2\ZZ)
\]
given on the first factor by $g \mapsto \ord_t c(g)$ and on the second factor by $g \mapsto \ord_u x \bmod 2$; here $\det(g)c(g)^{-m} \in K^\times$ has norm $1$ and, using Hilbert's Theorem 90, $x \in K^\times$ is any element satisfying $x\ol x^{-1} = \det(g)c(g)^{-m}$.

\subsection{Parahoric subgroups}\label{ss:parahoric}
We next recall the description of the parahoric subgroups of $G(K_0)$ from \cite{paprap09}.  In analogy with \s\ref{ss:lattices}, for $i = 2mb+c$ with $0 \leq c < 2m$, we define the $\O_{K}$-lattice
\begin{equation}\label{disp:lambda_i}
   \lambda_i := \sum_{j=1}^c u^{-b-1}\O_K e_j + \sum_{j=c+1}^{2m} u^{-b}\O_K e_j \subset K^{2m},
\end{equation}
where now $e_1,\dotsc,e_{2m}$ denotes the standard ordered basis in $K^{2m}$.  For any nonempty subset $I \subset \{0,\dotsc,m\}$, we write $\lambda_I$ for the chain consisting of all lattices $\lambda_i$ for $i \in 2m\ZZ \pm I$, and we define
\begin{align*}
   P_I &:= \bigl\{\, g\in G(K_0) \bigm| g\lambda_i = \lambda_i \text{ for all } i \in 2m\ZZ \pm I \,\bigr\}\\
       &\phantom{:}= \bigl\{\, g\in G(K_0) \bigm| g\lambda_i = \lambda_i \text{ for all } i \in I \,\bigr\}.
\end{align*}
In contrast with the odd case \cite{paprap09}*{\s1.2.3(a)}, $P_I$ is not a parahoric subgroup in general since it may contain elements with nontrivial Kottwitz invariant. Let
\[
   P_I^\circ := P_I \cap \ker\kappa_G.
\]

\begin{prop}[Pappas--Rapoport \cite{paprap09}*{\s1.2.3(b)}]\label{st:parahoric_description}
$P_I^\circ$ is a parahoric subgroup of $G(K_0)$, and every parahoric subgroup of $G(K_0)$ is conjugate to $P_I^\circ$ for a unique nonempty $I \subset \{0,\dotsc,m\}$ satisfying \eqref{disp:I_cond}.  For such $I$, we have $P_I^\circ = P_I$ exactly when $m \in I$.  The special maximal parahoric subgroups are exactly those conjugate to $P_{\{m\}}^\circ = P_{\{m\}}$.\qed
\end{prop}

\subsection{Affine roots}\label{ss:affine_roots}
In terms of the identification \eqref{disp:X_*S_in_X_*T}, the \emph{relative roots} of $S$ in $G$ consist of the maps $\{\pm\alpha_{i,j}\}_{i < j < i^*} \cup \{\pm\alpha_{i,i^*}\}_{i \leq m}$ on $X_*(S)$, where
\[
   \alpha_{i,j} \colon (x_1,\dotsc,x_{2m},y) \mapsto
   x_i - x_j.
\]
The \emph{affine roots} are calculated in \cite{paprap09}*{\s2.4.1} (modulo the description of $X_*(S)$ there), but in a way that implicitly chooses coordinates on $\A_G$ such that the origin is a special vertex.  
% Let $\A := \RR^{2m}$ denote the standard apartment of $GL_{2m}$, and regard \aaa as a subspace of \A by projecting onto the first $2m$ factors.  
Later on we shall wish to identify the action of $G$ on its building with the action of $G$ on lattices given by its standard representation, and to do so in the most straightforward way will require a different choice of coordinates:  \eqref{st:parahoric_description} says that the parahoric subgroup attached to $\lambda_m$, and not to $\lambda_0$ (which will correspond to the origin), is special.

We shall calculate the affine roots with regard to the following choices. 
Let $\nu \colon T(K_0) \to \A_G$ be the unique map satisfying
\[
   \chi\bigl(\nu(x)\bigr) = \ord_t \bigl(\chi(x)\bigr)
   \quad\text{for all}\quad
   \text{$x \in T(K_0)$ and $\chi \in X^*(T)$};
\]
here $\nu$ is the opposite of the map defined in Tits's article \cite{tits79}*{\s1.2}, but this will have no impact on the calculation of the affine roots.  Let $N$ be the normalizer of $S$ in $G$.  Then $N$ is the scheme of monomial matrices in $G$, and we have an internal semidirect product $N(K_0) = T(K_0) \rtimes S_{2m}^*$, where we identify $S_{2m}^*$ with the permutation matrices in $G(K_0)$.  We let $N(K_0)$ act by affine transformations on $\A_G$ by letting $T(K_0)$ act by translations via $\nu$, and by letting $S_{2m}^*$ act via its natural linear action on $X_*(S)$.  With respect to this choice of action, a calculation completely analogous to the one for quasi-split $SU_n$ with $n$ odd in \cite{tits79}*{\s1.15} shows that the affine roots for $G$ consist of the functions
\begin{equation}\label{disp:aff_roots}
   \pm\alpha_{i,j} + \frac 1 2 \ZZ
   \quad\text{for}\quad
   i < j <i^*
   \quad\text{and}\quad
   \pm \alpha_{i,i^*} + \frac 1 2 + \ZZ
   \quad\text{for}\quad
   i < m+1.
\end{equation}
% Thus the affine root hyperplanes consist of the zero loci of the affine functions
% \begin{equation}\label{disp:B_m_roots_unscaled}
%    \begin{gathered}
%       \pm2\alpha_{i,j} + \ZZ
%       \quad\text{for}\quad
%       i < j <i^*,\ j \neq m+1;\quad\text{and}\\
%       \pm \alpha_{i,i^*} + \tfrac 1 2 + \ZZ
%       \quad\text{for}\quad
%       i < m+1.
%    \end{gathered}
% \end{equation}
% Relative to the choice of a special vertex as origin, the functions \eqref{disp:B_m_roots_unscaled} may be regarded as the affine roots attached to a root system of type $B_m$.

\section{Preliminary combinatorics}\label{s:prelim_comb}
In this section we introduce some preliminary combinatorial considerations, covering just enough to be able to translate the wedge and spin conditions into combinatorics later in \s\ref{s:schub_cells}.  Let $n$ be a positive integer, and for $j \in \{1,\dotsc,n\}$, let $j^* := n + 1 - j$.

\subsection{The group \texorpdfstring{$\wt W_n$}{W\~{}\_n}}
Let
\begin{equation}\label{disp:X_*n}
   X_{*n} := \{\,v \in \ZZ^n \mid v + v^* = \mathbf d \text{ for some } d\in \ZZ \,\}.
\end{equation}
Then the group $S_n^*$ acts on $X_{*n}$ by permuting the factors of $\ZZ^n$, and we form the semidirect product
\begin{equation}\label{disp:wtW_n}
   \wt W_n := X_{*n} \rtimes S_n^*.
\end{equation}
Note that if $n$ is odd, then the integer $d$ appearing the definition of $X_{*n}$ must be even, since $\lceil n/2 \rceil^* = \lceil n/2 \rceil$.

\subsection{Faces of type $(n,I)$}\label{ss:faces_type_I}
Let $I$ be a nonempty subset of $\{1,\dotsc,\lfloor n/2\rfloor\}$.

\begin{defn}\label{def:face}
Let $d \in \ZZ$.  A \emph{$d$-face of type $(n,I)$} is a family $\mathbf v = (v_i)_{i\in n\ZZ\pm I}$ of vectors $v_i \in \ZZ^n$ such that
\begin{enumerate}
\renewcommand{\theenumi}{F\arabic{enumi}}
\item\label{it:face_per_cond}
   $v_{i+n} = v_i - \mathbf 1$ for all $i \in n\ZZ \pm I$;
\item
   $v_i \geq v_j$ for all $i$, $j \in n\ZZ \pm I$ with $i \leq j$;
\item\label{it:face_adjcy_cond}
   $\Sigma v_i - \Sigma v_j = j-i$ for all $i$, $j \in n\ZZ \pm I$; and
\item\label{it:face_dlty_cond}
   $v_i + v_{-i}^* = \mathbf d$ for all $i \in n\ZZ \pm I$.
\end{enumerate}
A family of vectors $\mathbf v = (v_i)_{i\in n\ZZ \pm I}$ is a \emph{face of type $(n,I)$} if it is a $d$-face of type $(n,I)$ for some $d$.
\end{defn}

When the integer $n$ is clear from context, we typically say just ``face of type $I$.''  Faces of type $I$ were defined by Kottwitz--Rapoport in \cite{kottrap00}*{\s\s9--10}, at least for even $n$, where they are called ``$G$-faces of type $n\ZZ\pm I$.'' (Kottwitz--Rapoport use the term ``face of type $n\ZZ \pm I$'' for the notion \eqref{def:face} without condition \eqref{it:face_dlty_cond}.)  As in \cite{sm11b,sm11d}, we have adopted a different convention from Kottwitz--Rapoport in formulating \eqref{def:face}, corresponding to a different choice of base alcove, to better facilitate working with the affine flag variety later on.

Note that if $n$ is odd, then it is an easy consequence of \eqref{it:face_adjcy_cond} and \eqref{it:face_dlty_cond} that $d$-faces only occur for even $d$.

For $i = bn + c$ with $b \in \ZZ$ and $0 \leq c < n$, let
\begin{equation}\label{disp:omega_i}
   \omega_i := \bigl((-1)^{(c)},0^{(n-c)}\bigr) - \mathbf b.
\end{equation}
Then the family
\begin{equation}\label{disp:omega_I}
   \omega_I := (\omega_i)_{i\in n\ZZ \pm I}
\end{equation}
is a $0$-face of type $I$, which we call the \emph{standard face of type $I$}.

The group $\wt W_n$ acts naturally on $\ZZ^n$ by affine transformations, and this induces an action of $\wt W_n$ on faces of type $I$,
\[
   w \cdot (v_i)_{i\in n\ZZ \pm I} = (wv_i)_{i\in n\ZZ\pm I},
\]
which one easily sees to be transitive.  In the ``Iwahori'' case $I = \{0,\dotsc,\lfloor n/2\rfloor\}$, $\wt W_n$ acts simply transitively.  For nonempty $J \subset I$, there is a natural $\wt W_n$-equivariant surjection
\[
   \xymatrix@R=0ex{
      \{\text{faces of type } I\} \ar@{->>}[r]  & \{\text{faces of type } J\}\\
      (v_i)_{i\in n\ZZ \pm I}  \ar@{|->}[r]  &  (v_i)_{i \in n\ZZ \pm J}.
   }
\]
% \textbf{DO I REALLY WANT TO SAY THE FOLLOWING?  FOR EXAMPLE, DON'T I NEED TO BE CAREFUL ABOUT THE KOTTWITZ INVARIANT WHEN DEFINING $W_{n,I}$?  MIGHT BE BETTER TO JUST SAY THAT WE GET A SIMPLY TRANSITIVE ACTION IN THE IWAHORI CASE.}
% We denote by $W_{n,I}$ the stabilizer of $\omega_I$ in $\wt W_n$, so that we obtain a bijection
% \[
%    \xymatrix@R=0ex{
%       \wt W_n/W_{n,I} \ar[r]^-\sim  & \{\text{faces of type } I\}\\
%       w  \ar@{|->}[r]  & w\omega_I.
%    }
% \]
% In the ``Iwahori'' case $I = \{0,\dotsc,\lfloor n/2\rfloor\}$, the stabilizer of $\omega_I$ is trivial and we get an identification of $\wt W_n$ itself with the set of faces of type $I$.
% 
% \textbf{MAYBE SAY SOMETHING ABOUT SURJECTIVE PROJECTION FOR $J \subset I$?}

\subsection{The vector \texorpdfstring{$\mu_i^{\mathbf v}$}{mu\_i\^{}v}}\label{ss:mu_i}

We continue with $I$ a nonempty subset of $\{0,\dotsc,\lfloor n/2 \rfloor\}$.  Given a face $\mathbf v = (v_i)_i$ of type $I$, let
\begin{equation}\label{disp:mu_i^v}
   \mu_i^{\mathbf v} := v_i - \omega_i
   \quad\text{for}\quad
   i \in n\ZZ \pm I.
\end{equation}
Then condition \eqref{it:face_per_cond} is equivalent to the periodicity relation
\begin{equation}\label{disp:mu_per_cond}
   \mu_i^{\mathbf v} = \mu_{i+n}^{\mathbf v} \quad \text{for all}\quad i,
\end{equation}
and \eqref{it:face_adjcy_cond} is equivalent to
\begin{equation}\label{disp:mu_adjcy_cond}
   \Sigma\mu_i^{\mathbf v} = \Sigma\mu_j^{\mathbf v} \quad \text{for all}\quad i,\ j.
\end{equation}
If $\mathbf v$ is a $d$-face, then condition \eqref{it:face_dlty_cond} is equivalent to
\begin{equation}\label{disp:mu_dlty_cond}
   \mu_i^{\mathbf v} + (\mu_{-i}^{\mathbf v})^* = \mathbf d \quad \text{for all}\quad i,
\end{equation}
from which we also determine
\begin{equation}\label{disp:Sigma_mu_val}
   \Sigma \mu_i^{\mathbf v} = nd/2.
\end{equation}

We now prove a couple of lemmas.  For $i \in I$, we define subsets of $\{1,\dotsc,n\}$
\begin{equation}\label{disp:A_i_B_i}
   A_i := \{1,2,\dotsc,i,i^*,i^*+1,\dotsc,n\}
   \quad\text{and}\quad
   B_i := \{i+1,i+2,\dotsc,n-i\}.
\end{equation}
For arbitrary $i \in \ZZ$, we define $A_i := A_{i'}$ and $B_i := B_{i'}$, where $i'$ is the unique element in $\{0,\dotsc,\lfloor n/2 \rfloor\}$ that is congruent mod $n$ to $i$ or $-i$.

\begin{lem}[Basic inequalities]\label{st:basic_ineqs}
Let $\mathbf v$ be a $d$-face of type $I$.  Then for any $i \in I$, we have
\[
   \begin{gathered}
      d \leq \mu_i^{\mathbf v}(j) + \mu_i^{\mathbf v}(j^*) \leq d+1, \\
      d-1 \leq \mu_{-i}^{\mathbf v}(j) + \mu_{-i}^{\mathbf v}(j^*) \leq d
   \end{gathered}
   \quad\text{for}\quad j \in A_i
\]
and
\[
   \begin{gathered}
      d-1 \leq \mu_i^{\mathbf v}(j) + \mu_i^{\mathbf v}(j^*) \leq d, \\
      d \leq \mu_{-i}^{\mathbf v}(j) + \mu_{-i}^{\mathbf v}(j^*) \leq d+1
   \end{gathered}
   \quad\text{for}\quad j \in B_i.
\]
\end{lem}

\begin{proof}
Since faces of type $\{0,\dotsc,\lfloor n/2 \rfloor\}$ surject onto faces of type $I$, it suffices to assume that $I = \{0,\dotsc,\lfloor n/2 \rfloor\}$.  Modulo conventions related to the choice of base alcove, the inequalities in this case for $\mu_i^{\mathbf v}(j) + \mu_i^{\mathbf v}(j^*)$ are proved for $n$ even and $d = 0$ in \cite{sm11c}*{Lem.\ 4.4.1}, and the argument there works just as well for arbitrary $n$ and $d$.  The inequalities for $\mu_{-i}^{\mathbf v}(j) + \mu_{-i}^{\mathbf v}(j^*)$ then follow from \eqref{disp:mu_dlty_cond}.
\end{proof}

Of course, periodicity gives analogous inequalities for $\mu_i^{\mathbf v}(j) + \mu_i^{\mathbf v}(j^*)$ for any $i \in n\ZZ \pm I$.

\begin{rk}\label{rk:whatevs}
Suppose that $n$ is odd. Then $d$ is even, and the basic inequalities immediately imply that $\mu_i^{\mathbf v}\bigl(\lceil n/2 \rceil \bigr) = d /2$ for all $i$.
\end{rk}

\begin{defn}
Let $\mathbf v$ be a $d$-face of type $I$, and let $i \in n\ZZ \pm I$.  We say that $\mu_i^{\mathbf v}$ is \emph{self-dual} if $\mu_i^{\mathbf v} = \mu_{-i}^{\mathbf v}$, or in other words, if $\mu_i^{\mathbf v} + (\mu_i^{\mathbf v})^* = \mathbf d$.
\end{defn}

\begin{lem}\label{st:self-dual_mu}
Let $\mathbf v$ be a $d$-face of type $I$.  Let $i \in n\ZZ \pm I$, and suppose that the equality $\mu_i^{\mathbf v}(j) + \mu_i^{\mathbf v}(j^*) = d$ holds for all $j \in A_i$ or for all $j \in B_i$.  Then $\mu_i^{\mathbf v}$ is self-dual.
\end{lem}

\begin{proof}
By \eqref{disp:Sigma_mu_val} we have $\Sigma \mu_i^{\mathbf v} = nd/2$, and the conclusion is then an easy consequence of the basic inequalities.
\end{proof}

\subsection{Further lemmas}\label{ss:further_lems}
We continue with $I \subset \{0,\dotsc,\lfloor n/2 \rfloor\}$ nonempty.  The aim of this subsection is to prove \eqref{st:c_equiv_defs} below, which we will have occasion to use at a number of points later on in our study of admissibility and spin-permissibility.  For $S$ a subset of $\{1,\dotsc,n\}$, we write $S^*:= n + 1 - S$.

\begin{lem}\label{st:cardE=cardG}
Let $\mathbf v$ be a $2$-face of type $I$, and let $i \in I$.  Suppose that $\mathbf 0 \leq \mu_i^{\mathbf v} \leq \mathbf 2$.  Then
\[
   \#\bigl\{\, j \bigm| \mu_i^{\mathbf v}(j) = 2 \text{ and } \mu_i^{\mathbf v}(j^*) = 1 \,\bigr\}
   =
   \#\bigl\{\, j \bigm| \mu_i^{\mathbf v}(j) = 0 \text{ and } \mu_i^{\mathbf v}(j^*) = 1 \,\bigr\}.
\]
\end{lem}

\begin{proof}
Let
% Note: probably best to leave these sets inside this proof, since it's really best to study them only under the assumpiton of (SP1).
{\allowdisplaybreaks
\begin{equation}\label{disp:EFGH}
\begin{aligned}
   E &:= \bigl\{\, j \bigm| \mu_i^{\mathbf v}(j) = 2 \text{ and } \mu_i^{\mathbf v}(j^*) = 1 \,\bigr\},\\
   F &:= \bigl\{\, j \bigm| \mu_i^{\mathbf v}(j) = 0 \text{ and } \mu_i^{\mathbf v}(j^*) = 1 \,\bigr\},\\
   G &:= \bigl\{\, j \bigm| \mu_i^{\mathbf v}(j) = 2 \text{ and } \mu_i^{\mathbf v}(j^*) = 0 \,\bigr\}, \text{ and}\\
   H &:= \bigl\{\, j \bigm| \mu_i^{\mathbf v}(j) = 1 \text{ and } \mu_i^{\mathbf v}(j^*) = 1 \,\bigr\},
\end{aligned}
\end{equation}}%
and let $e$, $f$, $g$, and $h$ denote the respective cardinalities of these sets.  The basic inequalities \eqref{st:basic_ineqs} and our assumptions on $\mathbf v$ imply that
\[
   \{1,\dotsc,n\} = E \amalg F \amalg G \amalg H \amalg E^* \amalg F^* \amalg G^*.
\]
Counting cardinalities on both sides of the display, we get
\[
   n = 2e+ 2f + 2g + h.
\]
On the other hand, by \eqref{disp:Sigma_mu_val}
\[
   n = \Sigma \mu_i^{\mathbf v} = 3e + f + 2g + h.
\]
Hence $e = f$.\end{proof}

\begin{lem}\label{st:c_equiv_defs}
Under the same assumptions as in the previous lemma,
\[
   \#\bigl\{\,j \bigm| \mu_i^{\mathbf v}(j) = 2\,\bigr\} = \#\bigl\{\,j \bigm| \mu_i^{\mathbf v}(j) = 0\,\bigr\}.
\]
\end{lem}

\begin{proof}
In the notation of \eqref{disp:EFGH}, we have
\[
   \bigl\{\,j \bigm| \mu_i^{\mathbf v}(j) = 2\,\bigr\} = E \amalg G
   \quad\text{and}\quad
   \bigl\{\,j \bigm| \mu_i^{\mathbf v}(j) = 0\,\bigr\} = F \amalg G^*.
\]
Now use \eqref{st:cardE=cardG}.
\end{proof}

\section{Iwahori-Weyl groups}\label{s:I-W_gps}
In this section we review some generalities on Iwahori-Weyl groups and the $\{\mu\}$-admissible set, and we make these generalities explicit in the case of our group $G$.  We then show that that the Iwahori-Weyl group for $G$, equipped with its Bruhat order, is isomorphic to the Iwahori-Weyl group for split $GO_{2m+1}$.  The significance of the Iwahori-Weyl group (and certain parahoric variants of it) is that, in the function field case, it indexes the Schubert varieties in the affine flag variety, and the Bruhat order gives the inclusion relations between the Schubert varieties.  We take \cite{prs?}*{\s4} as a general reference.  We write $i^* := 2m + 1 - i$.

\subsection{Iwahori-Weyl groups}\label{ss:I-W_gp}
Let $L$ be a complete, discretely valued, strictly Hen\-sel\-ian field with algebraic closure $\ol L$ and Galois group $\I := \Gal(\ol L/L)$.  Let \G be a connected reductive group over $L$.  Let \S be a maximal split torus in \G, and let \T and \N be its respective normalizer and centralizer in \G.  Then \T is a maximal torus since by Steinberg's theorem \G is quasi-split.  The \emph{Iwahori-Weyl group} of \G with respect to \S is the group
\[
   \wt W := \N(L)/\T(L)_1,
\]
where $\T(L)_1$ is the kernel of the Kottwitz homomorphism
\[
   \kappa_\T\colon \T(L) \surj X_*(\T)_\I.  
\]
The evident exact sequence
\[
   1 \to \T(L)/\T(L)_1 \to \wt W \to \N(L)/\T(L) \to 1
\]
splits by Haines--Rapoport \cite{hrap08}*{Prop.\ 13}.  Hence, via the Kottwitz homomorphism, $\wt W$ is expressible as a semidirect product
\[
   \wt W \iso X_*(\T)_\I \rtimes W,
\]
where $W:= \N(L)/\T(L)$ is the relative Weyl group of \S in \G.  Given $\mu \in X_*(\T)_\I$, we usually write $t_\mu$ when we wish to regard $\mu$ as an element of $\wt W$, and we refer to such elements in $\wt W$ as \emph{translation elements.}

If $\R = (X^*,X_*,\Phi,\Phi^\vee)$ is a root datum, then we define its \emph{extended affine Weyl group} $\wt W_\R := X_* \rtimes W_\R$, where $W_\R$ denotes the Weyl group of \R.  If \G is split with root datum \R attached to \S, then we obtain a canonical isomorphism $\wt W \ciso \wt W_\R$.

To consider local models with general parahoric, not just Iwahori, level structure, it is necessary to consider certain double coset variants of $\wt W$.  Let $\A := X_*(\S) \otimes \RR$ denote the apartment attached to \S.  The group $\wt W$ acts on $\A$ by affine transformations in a way that is prescribed uniquely up to isomorphism in \cite{tits79}*{\s1.2}.  Having chosen such an action, we may then consider the affine root system $\Phi_\aff$ and the attendant affine root hyperplanes.  Fix an alcove $\mathbf a$ in $\A$, and let $\Pi_{\mathbf a}$ denote the set of reflections across the walls of $\mathbf a$.  The \emph{affine Weyl group $W_\aff$} is the group of affine transformations on $\A$ generated by $\Pi_{\mathbf a}$; it lifts canonically to a normal subgroup of $\wt W$.  In fact the Kottwitz homomorphism $\kappa_\G$, restricted to $\N(L)$, induces an isomorphism
\[
   \wt W /W_\aff \xra[\sim]{\kappa_\G} \pi_1(\G)_\I,
\]
as is explained in \cite{hrap08}*{p.\ 196} or \cite{prs?}*{\s4.2}.  For $\mathbf f$ a subfacet of $\mathbf a$, we denote by $W_{\mathbf f}$ the common stabilizer and pointwise fixer of $\mathbf f$ in $\wt W_\aff$.  Then $W_{\mathbf f}$ is Coxeter group, generated by the reflections across the walls of $\mathbf a$ containing $\mathbf f$ \cite{bourLGLA4-6}*{V \s3.3 Prop.\ 2}.  If $P$ is the parahoric subgroup of $\G(L)$ corresponding to $\mathbf f$, then
\[
   W_{\mathbf f} = \bigl(\N(L) \cap P \bigr)\big/ \T(L)_1;
\]
see \cite{prs?}*{\s4.2}.  The double coset variants of $\wt W$ that we shall be interested in are those of the form $W_{\mathbf{f_1}} \bs \wt W / W_{\mathbf{f_2}}$ for $\mathbf{f_1}$, $\mathbf{f_2} \preceq \mathbf a$.

\subsection{Bruhat order}\label{ss:bo}
We continue with the notation of the previous subsection.  Let $\Omega_{\mathbf a}$ denote the stabilizer of $\mathbf a$ in $\wt W$.  The affine Weyl group $W_\aff$ acts simply transitively on the alcoves in \A, and therefore $\wt W$ decomposes as a semidirect product
\[
   \wt W = W_\aff \rtimes \Omega_{\mathbf a}.
\]
The pair $(W_\aff,\Pi_{\mathbf a})$ is a Coxeter system.  Hence $W_\aff$ is endowed with a length function $\ell$ and Bruhat order $\leq$, which extend to $\wt W$ via the rules $\ell(w\tau) = \ell(w)$ and $w\tau \leq w'\tau'$ if $\tau = \tau'$ and $w \leq w'$, where $w$, $w' \in W_\aff$ and $\tau$, $\tau' \in \Omega_{\mathbf a}$.

For subfacets $\mathbf{f_1}$, $\mathbf{f_2} \preceq \mathbf a$, the Bruhat order extends to $W_{\mathbf{f_1}} \bs \wt W / W_{\mathbf{f_2}}$ as follows.  Each double coset $w \in W_{\mathbf{f_1}} \bs \wt W / W_{\mathbf{f_2}}$ is known to contain a unique element $\tensor[^{\mathbf{f_1}}]w{^{\mathbf{f_2}}} \in \wt W$ of minimal length, and one then defines $w \leq x$ in $W_{\mathbf{f_1}} \bs \wt W / W_{\mathbf{f_2}}$ if $\tensor[^{\mathbf{f_1}}]w{^{\mathbf{f_2}}} \leq \tensor[^{\mathbf{f_1}}]x{^{\mathbf{f_2}}}$ in $\wt W$.  The Bruhat order has the property that if $\wt w \leq \wt x$ in $\wt W$, then $W_{\mathbf{f_1}}\wt w W_{\mathbf{f_2}} \leq W_{\mathbf{f_1}}\wt x W_{\mathbf{f_2}}$.

\subsection{Admissible and permissible sets}\label{ss:adm_perm_sets}
We now recall the key notion of $\{\mu\}$-admis\-si\-bil\-i\-ty, and the related notion of $\{\mu\}$-permissibility.  We continue with the notation of the previous two subsections.

Let $\{\mu\}$ denote a $\G(\ol L)$-conjugacy class of cocharacters of \G, or what amounts to the same, a $W^\abs$-conjugacy class in $X_*(\T)$, where $W^\abs := \N(\ol L)/\T(\ol L)$ is the absolute Weyl group of \T in \G.  Let $\wt \Lambda_{\{\mu\}} \subset \{\mu\}$ denote the subset of cocharacters contained in the closure of some absolute Weyl chamber corresponding to a Borel subgroup of \G containing \T and defined over $L$.  Such Borel subgroups exist since \G is quasi-split, and the elements of $\wt\Lambda_{\{\mu\}}$ form a single $W$-orbit, since all such Borel subgroups are $W$-conjugate.  Let $\Lambda_{\{\mu\}}$ denote the image of $\wt\Lambda_{\{\mu\}}$ in the coinvariants $X_*(\T)_\I$.  Finally, let $\mathbf{f_1}$, $\mathbf{f_2} \preceq \mathbf{a}$.

\begin{defn}[Kottwitz--Rapoport \cite{kottrap00}*{Intro.}; Rapoport \cite{rap05}*{disp.\ 3.6}]\label{def:adm}
An element $w \in W_{\mathbf{f_1}} \bs \wt W / W_{\mathbf{f_1}}$ is \emph{$\{\mu\}$-admissible} if there exists $\lambda \in \Lambda_{\{\mu\}}$ such that $w \leq W_{\mathbf{f_1}} t_\lambda W_{\mathbf{f_2}}$.
We write $\Adm_{\mathbf{f_1}, \mathbf{f_2}}(\{\mu\})$, or $\Adm_{\mathbf{f_1}, \mathbf{f_2}, \G}(\{\mu\})$ when we wish to emphasize \G, for the set of $\{\mu\}$-admissible elements in $W_{\mathbf{f_1}} \bs \wt W / W_{\mathbf{f_1}}$.  We abbreviate this to $\Adm_\G(\{\mu\}) = \Adm(\{\mu\}) \subset \wt W$ when $\mathbf{f_1} = \mathbf{f_2} = \mathbf{a}$.
\end{defn}

In the special case that \G is split, we have $\wt \Lambda_{\{\mu\}} = \{\mu\} = W^\abs\mu = W\mu$ for any $\mu \in \{\mu\}$.  We then typically speak of $\mu$-admissibility in place of $\{\mu\}$-admissibility, and we write $\Adm_{\mathbf{f_1}, \mathbf{f_2}}(\mu)$.  If \R is a root datum, then we analogously define $\Adm_{\mathbf{f_1}, \mathbf{f_2}, \R}(\mu) \subset W_{\mathbf{f_1}} \bs \wt W_\R / W_{\mathbf{f_2}}$.

Let $\{\ol\mu\}$ denote the image of the conjugacy class $\{\mu\}$ in $X_*(\T)_\I$.  It is conjectured in \cite{prs?}*{Conj.\ 4.5.3} that $\Lambda_{\{\mu\}}$ consists of exactly the subset of elements in $\{\ol\mu\}$ that are maximal for the Bruhat order.  Note that if this conjecture holds true, then the notion of $\{\mu\}$-admissibility is just that $w \leq W_{\mathbf{f_1}} t_{\ol\mu} W_{\mathbf{f_2}}$ for some $\ol\mu \in \{\ol\mu\}$.  We shall see explicitly in \s\ref{ss:mu_r,s-adm_set} that the conjecture holds for the pair $(G,\{\mu_{r,s}\})$, and it is not hard to see that furthermore the conjecture holds for any geometric conjugacy class for $G$.

% It is straightforward to verify the conjecture, for any $\{\mu\}$, for the group $G_{K_0^\un}$, where $G$ is the quasi-split unitary similitude group over $K_0$ defined in \s\ref{ss:GU} and $K_0^\un$ is a maximal unramified algebraic extension of $K_0$.

Now let $\Conv(\Lambda_{\{\mu\}})$ denote the convex hull of the image of $\Lambda_{\{\mu\}}$ in $X_*(\T)_\I \otimes \RR \ciso \A$, and recall from \cite{rap05}*{Lem.\ 3.1} that all elements of $\Lambda_{\{\mu\}}$, regarded as elements of $\wt W$, are congruent mod $W_\aff$.  Suppose furthermore that $\mathbf{f_2} \preceq \mathbf{f_1} \preceq \mathbf{a}$.

\begin{defn}[Kottwitz-Rapoport \cite{kottrap00}*{Intro.}; Rapoport \cite{rap05}*{disp.\ 3.10}]\label{def:mu-perm}
An element $w \in W_{\mathbf{f_1}} \bs \wt W / W_{\mathbf{f_1}}$ is \emph{$\{\mu\}$-permissible} if $w \equiv t_\lambda \bmod W_\aff$ for one, hence any, $\lambda \in \Lambda_{\{\mu\}}$ and $wv - v \in \Conv(\Lambda_{\{\mu\}})$ for all $v \in \mathbf{f_2}$.  We write $\Perm_{\mathbf{f_1}, \mathbf{f_2}}(\{\mu\})$, or $\Perm_{\mathbf{f_1}, \mathbf{f_2}, \G}(\{\mu\})$ when we wish to emphasize \G, for the set of $\{\mu\}$-permissible elements in $W_{\mathbf{f_1}} \bs \wt W / W_{\mathbf{f_1}}$.  We abbreviate this to $\Perm_\G(\{\mu\}) = \Perm(\{\mu\}) \subset \wt W$ when $\mathbf{f_1} = \mathbf{f_2} = \mathbf{a}$.
\end{defn}

It is clear that $\{\mu\}$-permissibility is well-defined on right $W_{\mathbf{f_2}}$-cosets.  That it is also well-defined on left $W_{\mathbf{f_1}}$-cosets is shown immediately after display 3.10 in \cite{rap05}.  By convexity, the condition
\[
   \text{$wv - v \in \Conv(\Lambda_{\{\mu\}})$ for all $v \in \mathbf{f_2}$}
\]
is equivalent to the condition
\[
   \text{for all minimal subfacets $\mathbf f \preceq \mathbf{f_2}$, $wv-v \in \Conv(\Lambda_{\{\mu\}})$ for one, hence any, $v \in \mathbf f$.}
\]
In the split or root datum case, we typically speak of $\mu$-permissibility, for any $\mu \in \{\mu\}$, in place of $\{\mu\}$-permissibility, and we write $\Perm_{\mathbf{f_1}, \mathbf{f_2}}(\mu)$.

% \textbf{LEAVE THE FOLLOWING IN?  AND THEN SAY REVERSE INCLUSION CAN FAIL BY HAINES--NGO, AND THAT WE'LL GIVE ANOTHER EXAMPLE OF THIS FAILURE LATER ON.  WAIT TO SEE HOW THE INTRODUCTION GOES.}
% 
% It is a general result of Kottwitz--Rapoport \cite{kottrap00}*{11.2} that
% \[
%    \Adm_{\mathbf{f_1},\mathbf{f_2}}(\{\mu\}) \subset \Perm_{\mathbf{f_1},\mathbf{f_2}}(\{\mu\}).
% \]

\subsection{Case of \texorpdfstring{$GU_{2m}$}{GU\_2m}}
We now begin to apply the generalities of the three previous subsections to our even unitary similitude group $G$ over $K_0$ and its maximal split torus $S$ from \s\ref{s:GU}.  We shall assume throughout the rest of \s\ref{s:I-W_gps} that the residue field of $K_0$ is algebraically closed.

We first compute the Galois coinvariants of $X_*(T)$.  Since $G$ splits over $K$, the action of the Galois group on $X_*(T)$ factors through $\Gamma$.  In terms of our identification $X_*(T) \iso \ZZ^{2m} \times \ZZ$ from \s\ref{ss:tori}, the nontrivial element in $\Gamma$ acts on $X_*(T)$ by sending
\[
   (x_1,\dotsc,x_{2m},y) \mapsto (y-x_{2m},\dotsc,y-x_1,y).
\]
Let
\[
   X := \{\, v\in\ZZ^{2m} \mid v + v^* = \mathbf d \text{ for some } d\in 2\ZZ \,\}.
\]
It follows that the surjective map
\begin{equation}\label{disp:coinvar_map}
   \begin{matrix}
   \xymatrix@R=0ex{
      X_*(T) \ar@{->>}[r]  &  X\\
      (x_1,\dotsc,x_{2m},y) \ar@{|->}[r]  & (x_1 - x_{2m} +y, x_2 - x_{2m+1} + y,\dotsc, x_{2m} - x_1 + y)
   }
   \end{matrix}
\end{equation}
identifies
\begin{equation}\label{disp:coinvar_X}
   X_*(T)_\Gamma \isoarrow X.
\end{equation}
In this way the Kottwitz map for $T$ is given by
\[
   \kappa_T\colon
   \xymatrix@R=0ex{
      T(K_0) \ar@{->>}[r]  &  X\\
      \diag(a_1,\dotsc,a_{2m}) \ar@{|->}[r]  &  \bigl(\ord_u (a_1),\dotsc, \ord_u (a_{2m})\bigr).
   }
\]

Now consider the apartment $\A_G = X_*(S) \otimes \RR \ciso X_*(T)_\Gamma \otimes \RR$, which we identify with $X \otimes \RR$ via \eqref{disp:coinvar_X}.  We declare that $\wt W_G$ acts on $\A_G$ via the semidirect product decomposition $N(K_0) = T(K_0) \rtimes S_{2m}^*$ from \s\ref{ss:affine_roots}, with $T(K_0)$ acting by translations on $X_*(T)_\Gamma$ via the Kottwitz homomorphism, and the permutation matrices $S_{2m}^*$ acting via their natural linear action on $X_*(T)_\Gamma$.

Note that the map \eqref{disp:coinvar_map} identifies $X_*(S)$ with its image in $X$,
\begin{equation}\label{disp:new_X_*(S)_coords}
   X_*(S) \isoarrow \{\, v\in X \mid \text{the entries of $v$ are even} \,\}.
\end{equation}
\emph{These coordinates on $X_*(S)$ are not the ones given earlier in \eqref{disp:X_*S_in_X_*T}}.  From now on we shall work with respect to these new coordinates \eqref{disp:new_X_*(S)_coords}, and therefore we must re-express the affine roots \eqref{disp:aff_roots} in terms of them.  We get that the affine roots consist of the functions on $X \otimes \RR$
\begin{equation}\label{disp:aff_roots_new}
   \frac 1 2 \alpha_{i,j} + \frac 1 2 \ZZ
   \quad\text{for}\quad
   j \neq i, i^*
   \quad\text{and}\quad
   \frac 1 2 \alpha_{i,i^*} + \frac 1 2 + \ZZ
   \quad\text{for}\quad
   1 \leq i \leq 2m.
\end{equation}
Regarding $X\otimes \RR$ as a subspace of $\RR^{2m}$, we then take as our base alcove
\[
   A_G := \biggl\{(x_1,\dotsc,x_{2m}) \in \RR^{2m} \biggm|
      \begin{varwidth}{\textwidth}
         \centering
         $x_1 + x_{2m} = \dotsb = x_m + x_{m+1}$ and\\
         $x_1 < x_2 < x_3 < \dotsb < x_{m-1} < x_m, x_{m+1}$
      \end{varwidth}\biggr\}.
\]
Of course, this choice of base alcove endows $\wt W_G$ with a Bruhat order, as discussed in \s\ref{ss:bo}.

\subsection{Faces of type $I$}\label{ss:GU_faces}
The action of $\wt W_G$ on $\A_G$ described in the previous subsection identifies $\wt W_G$ with $X \rtimes S_{2m}^*$ inside the group of affine transformations $\Aff(X \otimes \RR)$.  In this way we may regard $\wt W_G$ as a subgroup of $\wt W_{2m}$ \eqref{disp:wtW_n}, and we attach to each $w \in \wt W_G$ a face $\mathbf v$ of type $(2m,\{0,\dotsc,m\})$ by letting $w$ act on the standard face $\omega_{\{0,\dotsc,m\}}$, as in \s\ref{ss:faces_type_I}.  We write $\mu_i^w$ for the vector $\mu_i^{\mathbf{v}}$ attached to this face,
\begin{equation}\label{disp:something}
   \mu_i^w := w\omega_i - \omega_i
   \quad\text{for}\quad
   i \in \ZZ.
\end{equation}

For each $\O_K$-lattice $\lambda$ in $K^{2m}$ of the form $\sum_{i=1}^{2m} u^{j_i}\O_Ke_i$, let
\[
   T(\lambda) := (j_1,\dotsc,j_{2m}) \in \ZZ^{2m}.
\]
Then $T$ is equivariant in the sense that for $n \in N(K_0)$,
\[
   T(n\lambda) = \ol n \cdot T(\lambda),
\]
where $\ol n$ denotes the image of $n$ in $\wt W_G$, which acts via the action of $\wt W_{2m}$ on $\ZZ^{2m}$.

Recalling the lattice $\lambda_i$ from \eqref{disp:lambda_i}, we have
\[
   T(\lambda_i) = \omega_i
   \quad\text{for all}\quad
   i \in \ZZ.
\]
The points $\frac{\omega_i + \omega_{-i}}2$ are contained in $\ol A_G$ for all $i$, and we conclude that $P_{\{0,\dotsc,m\}}$ is precisely the Iwahori subgroup of $G$ fixing the alcove $A_G$ in the building.  For nonempty $I \subset \{0,\dotsc,m\}$ satisfying \eqref{disp:I_cond}, we define
\[
   W_{I,G} := \bigl( N(K_0) \cap P_I^\circ \bigr) \big/ T(K_0)_1
    \subset \wt W_G.
\]
Then, as discussed in \s\ref{ss:I-W_gp}, $W_{I,G}$ is a subgroup of the affine Weyl group inside $\wt W_G$.  It is generated by the reflections across the walls of $A_G$ containing the points $\frac{\omega_i + \omega_{-i}} 2$ for all $i \in I$.

By definition of $P_I^\circ$, elements of $W_{I,G}$ fix $\omega_i$ for all $i \in 2m\ZZ \pm I$.  Thus to each $w \in \wt W_G/ W_{I,G}$ we get a well-defined face $w\cdot \omega_I$ of type $(2m,I)$ and vector $\mu_i^w = w\omega_i - \omega_i$ for $i \in 2m\ZZ \pm I$.

\subsection{Kottwitz homomorphism on \texorpdfstring{$\wt W_G$}{W\~{}\_G}}
Let $W_{\aff,G}$ denote the affine Weyl group inside $\wt W_G$.  Then the Kottwitz homomorphism $\kappa_G$ (see \s\ref{ss:Kott_hom}) identifies
\[
   \wt W_G/ W_{\aff,G} \isoarrow \ZZ \oplus (\ZZ/2\ZZ).
\]
Explicitly, for $w \in \wt W_G$, the Kottwitz homomorphism sends
\begin{equation}\label{disp:explicit}
   w \mapsto \bigl(d, {\textstyle \sum_{i=1}^m \mu_m^w(i)} \bmod 2\bigr),
\end{equation}
where $w$ defines a $d$-face of type $\{0,\dotsc,m\}$.

\subsection{Relation to type \texorpdfstring{$B_m$}{B\_m}}\label{ss:relation_to_B}
As noted in the Introduction, upon choosing a special vertex in $\A_G$ as origin, there exists a unique reduced root system $\Sigma$ on $\A_G$ such that the affine root hyperplanes for $G$ are precisely the affine root hyperplanes for $\Sigma$, that is, the vanishing loci of the functions $\alpha + \ZZ$ for $\alpha \in \Sigma$.
In this subsection we shall show that $\Sigma$ is of type $B_m$.  In fact, we shall construct an explicit isomorphism between the triple $(\wt W_G, \A_G, \H_G)$, where $\H_G$ denotes the system of affine root hyperplanes in $\A_G$, and the analogous triple attached to split $GO_{2m+1}$.

The point $\omega_m \in X \otimes \RR$ is a special vertex for $A_G$, and we shall begin by transporting it to the origin.  Rather than simply translating by $-\omega_m$, let
\[
   \delta := (1,m+1)(2,m+2)\dotsm (m,2m) \in S_{2m}^*,
\]
where $(i,j)$ denotes the transposition interchanging $i$ and $j$ and fixing the other elements of $\{1,\dotsc,2m\}$.  Of course $\delta^{-1} = \delta$.  Let
\[
   z := \delta t_{-\omega_m} \in \Aff(X\otimes \RR).
\]
Then
\begin{equation}\label{disp:z}
   X \otimes \RR \xra[\sim]{z} X \otimes \RR
\end{equation}
sends $\omega_m \mapsto \mathbf 0$ and our base alcove $A_G$ to the set
\begin{equation}\label{disp:A_G_changed}
   \biggl\{(x_1,\dotsc,x_{2m}) \in \RR^{2m} \biggm|
      \begin{varwidth}{\textwidth}
         \centering
         $x_1 + x_{2m} = \dotsb = x_m + x_{m+1}$ and\\
         $x_1, x_{2m} - 1 < x_2 < x_3 < \dotsb < x_m < x_{m+1}$
      \end{varwidth}\biggr\}.
\end{equation}
Each affine root $\wt \alpha$, viewed as a function on the source in \eqref{disp:z}, is carried by $z$ to the function on the target
\[
   z_*\wt\alpha = \wt \alpha \circ z^{-1} = \wt\alpha \circ (t_{\omega_m}\delta).
\]
Recalling the explicit form \eqref{disp:aff_roots_new} for the affine roots, we then have
\[
   z_*\biggl(\frac 1 2 \alpha_{i,j} + \frac d 2\biggr) = \frac 1 2 \alpha_{\delta(i),\delta(j)} + \frac{\alpha_{i,j}(\omega_m) + d}2
   \quad
   (j\neq i,i^*,\ d\in \ZZ)
\]
and
\[
   z_*\biggl(\frac 1 2 \alpha_{i,i^*} + \frac 1 2 + d\biggr)
   = \frac 1 2 \alpha_{\delta(i),\delta(i^*)} + \frac{\alpha_{i,i^*}(\omega_m) + 1}2 + d
   \quad
   (d \in \ZZ).
\]
Since $\alpha_{i,i^*}(\omega_m) = \pm 1$, we conclude that $z$ carries the affine root hyperplanes in the source of \eqref{disp:z} to the zero loci of the functions
\begin{equation}\label{disp:Sigma_fcns}
   \alpha_{i,j} + \ZZ
   \quad\text{for}\quad
   j \neq i,i^*
   \quad\text{and}\quad
   \frac 1 2 \alpha_{i,i^*} + \ZZ
   \quad\text{for}\quad
   1 \leq i \leq 2m.
\end{equation}
The functions \eqref{disp:Sigma_fcns} are the affine roots for the root system
\[
   \Sigma := \biggl\{\, \alpha_{i,j}, \frac 1 2\alpha_{i,i^*} \biggm| 1 \leq i,j\leq 2m,\ j \neq i,i^*\,\biggr\}.
\]

Now consider the isomorphism
\[
   f\colon \xymatrix@R=0ex{
      X  \ar[r]^-\sim  &  X_{*2m+1}\\
      (x_1,\dotsc,x_{2m})  \ar@{|->}[r]  &  (x_1,\dotsc,x_m,c/2,x_{m+1},\dotsc,x_{2m}),
   }
\]
where $c$ denotes the common integer $x_1 + x_{2m} = \dotsb = x_m + x_{m+1}$.  The group $X_{*2m+1}$ is the cocharacter group for the root datum $\R_{B_m}$ defined later in \s\ref{ss:R_B_m}, which is the root datum for split $GO_{2m+1}$.  It is immediate that the roots $\Phi_{B_m}$ \eqref{disp:roots_B} for $\R_{B_m}$ pull back under $f$ to the functions $\Sigma$ on $X$.  Furthermore $f$ induces an isomorphism $S_{2m}^* \isoarrow S_{2m+1}^*$ between groups acting faithfully on source and target, respectfully.  Hence $f$ induces an isomorphism, which we still denote $f$,
\[
   f\colon X \rtimes S_{2m}^* \isoarrow X_{*2m+1} \rtimes S_{2m+1}^*;
\]
here the target is the extended affine Weyl group $\wt W_{B_m}$ of the root datum $\R_{B_m}$.  Since $f \otimes \RR$ plainly identifies the set \eqref{disp:A_G_changed} with the alcove $A_{B_m}$ \eqref{disp:A_B_m}, and since $X \rtimes S_{2m}^*$ is easily seen to be stable under conjugation by $z$ inside $\Aff(X \otimes \RR)$, we conclude the following.

\begin{lem}\label{st:wtW_G->wtW_B_m}
The composite
\[
   X \otimes \RR \xra[\sim]z X \otimes \RR \xra[\sim]{f \otimes \RR} X_{*2m+1} \otimes \RR
\]
induces a composite isomorphism
\[
   \wt W_G \ciso X\rtimes S_{2m}^* \xra[\sim]{\int(z)}
   X \rtimes S_{2m}^* \xra[\sim] f
   X_{*2m+1} \rtimes S_{2m+1}^* = \wt W_{B_m}
\]
which identifies the $A_G$-Bruhat order on $\wt W_G$ with the $A_{B_m}$-Bruhat order on $\wt W_{B_m}$.  Here $\int(z)$ is the conjugation map $w \mapsto zwz^{-1}$.\qed
\end{lem}

The lemma immediately allows us to transfer all questions we shall have about $\wt W_G$ to $\wt W_{B_m}$.  To do so in an effective way, we shall need to explicitly address the intervention of $\int(z)$ in the displayed isomorphism.  The following will take care of everything we need.

\begin{lem}\label{st:int(z)_lem}\hfill
\begin{enumerate}
\renewcommand{\theenumi}{\roman{enumi}}
\item\label{it:z.omega_i}
   $z\omega_i = \omega_{i-m}$ for all $i \in \ZZ$.
\item\label{it:mu^zwz^-1_fmla}
   $\mu_i^{zwz^{-1}} = \delta \mu_{i+m}^w$ for all $i \in \ZZ$ and $w \in \wt W_G$.
\end{enumerate}
\end{lem}

\begin{proof}
The first assertion is an exercise, and the second is an easy consequence of the first.
\end{proof}

It follows easily from the first part of the lemma that, for any nonempty $I \subset\{0,\dotsc,m\}$ satisfying \eqref{disp:I_cond}, the composite isomorphism $\wt W_G \isoarrow \wt W_{B_m}$ in \eqref{st:wtW_G->wtW_B_m} sends the subgroup $W_{I,G}$ isomorphically to the subgroup $W_{m-I,B_m}$ defined in \eqref{disp:W_I,B_m}.

\subsection{The \texorpdfstring{$\{\mu_{r,s}\}$}{mu\_$\{r,s\}$}-admissible set}\label{ss:mu_r,s-adm_set}

Recalling our identification $X_*(T) \iso \ZZ^{2m} \times \ZZ$ from \s\ref{ss:tori}, let
\[
   \mu_{r,s} := \bigl(1^{(s)},0^{(2m-s)},1\bigr) \in X_*(T),
\]
and let
\[
   \{\mu_{r,s}\} := S_{2m} \cdot \mu_{r,s}
\]
denote the geometric conjugacy class of $\mu_{r,s}$, where $S_{2m}$ acts by permuting the first $2m$ entries.  The image of $\{\mu_{r,s}\}$ in $X$ under the map \eqref{disp:coinvar_map} is the set
\[
   \{\ol\mu_{r,s}\} = S_{2m}^*\cdot \bigl(2^{(s)},1^{(2m-2s)},0^{(s)}\bigr) \amalg S_{2m}^* \cdot \bigl(2^{(s-2)},1^{(2m-2s+4)},0^{(s-2)}\bigr) \amalg \dotsb,
\]
where the union terminates with $\{\mathbf{1}\}$ or $S_{2m}^* \cdot \bigl(2,1^{(2m-1)},0\bigr)$ according as $s$ is even or odd.  Regarding $\{\ol\mu_{r,s}\}$ as a subset of the translation subgroup of $\wt W_G$, the image of $\{\ol\mu_{r,s}\}$ in $\wt W_{B_m}$ under the isomorphism in \eqref{st:wtW_G->wtW_B_m} consists of the translation elements
\[
   S_{2m+1}^*\cdot \bigl(2^{(s)},1^{(2m-2s+1)},0^{(s)}\bigr) \amalg S_{2m+1}^* \cdot \bigl(2^{(s-2)},1^{(2m-2s+5)},0^{(s-2)}\bigr) \amalg \dotsb.
\]
These elements are all congruent modulo coroots, and the maximal elements among them in the Bruhat order are the elements $S_{2m+1}^*\cdot \bigl(2^{(s)},1^{(2m-2s+1)},0^{(s)}\bigr)$.  Hence the elements $S_{2m}^* \cdot \bigl(2^{(s)},1^{(2m-2s)},0^{(s)}\bigr)$ are maximal in $\{\ol\mu_{r,s}\}$.  This confirms, for $G$ and the conjugacy class $\{\mu_{r,s}\}$, the conjecture \cite{prs?}*{Conj.\ 4.5.3} we discussed in \s\ref{ss:adm_perm_sets}, since in this case the set $\Lambda_{\{\mu_{r,s}\}}$ is by definition the image of $S_{2m}^* \cdot\mu_{r,s}$ in the coinvariants.

Now let $\Adm_{J,I,G}(\{\mu_{r,s}\})$ denote the $\{\mu_{r,s}\}$-admissible set in $W_{J,G} \bs \wt W_G / W_{I,G}$, and for any cocharacter $\mu \in X_{*2m+1}$, let $\Adm_{J,I,B_m}(\mu)$ denote the $\mu$-admissible set in $W_{J,B_m} \bs \wt W_{B_m} / W_{I,B_m}$.  We have shown the following.

\begin{lem}\label{st:id_adm_sets}
For any nonempty subsets $I$, $J \subset \{0,\dotsc,m\}$ satisfying property \eqref{disp:I_cond}, the isomorphism $W_{J,G} \bs \wt W_G/W_{I,G} \isoarrow \wt W_{m-J,B_m} \bs \wt W_{B_m}/W_{m-I,B_m}$ induced by \eqref{st:wtW_G->wtW_B_m} identifies
\[
   \Adm_{J,I,G}(\{\mu_{r,s}\}) \isoarrow \Adm_{m-J,m-I,B_m}(\mu),
\]
where $\mu$ is the cocharacter $\bigl(2^{(s)},1^{(2m-2s+1)},0^{(s)}\bigr) \in X_{*2m+1}$.\qed
\end{lem}

\section{Affine flag variety}\label{s:afv}
From now on we take $K_0 = k((t))$, $\O_{K_0} = k[[t]]$, $K = k((u))$, and $\O_K = k[[u]]$, where we recall that $k$ is an algebraic closure of the common residue field of $F_0$ and $F$.  In this section we review the affine flag variety attached to a parahoric subgroup of $G$, and the embedding of the geometric special fiber of the naive local model into it, from \cite{paprap09}*{\s3}.  Let $I \subset \{0,\dotsc,m\}$ be a nonempty subset satisfying property \eqref{disp:I_cond}.

\subsection{Affine flag variety}\label{ss:afv}
Let $P$ be a parahoric subgroup of $G(K_0)$.  Then Bruhat-Tits theory attaches to $P$ a smooth affine $\O_{K_0}$-group scheme whose generic fiber identifies with $G$, whose special fiber is connected, and whose group of $\O_{K_0}$-points identifies with $P$; abusing notation, we denote this group scheme again by $P$.  The \emph{affine flag variety $\F_P$ relative to $P$} is the fpqc quotient of functors on the category of $k$-algebras,
\[
   \F_P := LG/L^+P,
\]
where $LG$ is the \emph{loop group} $LG\colon R \mapsto G\bigl(R((t))\bigr)$ and $L^+P$ is the \emph{positive loop group}  $L^+P\colon R \mapsto G\bigl(R[[t]]\bigr)$.  See \cite{paprap08}.  The affine flag variety is an ind-$k$-scheme of ind-finite type.

\subsection{Lattice-theoretic description}\label{ss:lattice_theoretic}
In this subsection we give a slight variant of (and make a minor correction to) the description in \cite{paprap09}*{\s3.2} of the affine flag variety in terms of lattice chains in $K^{2m}$.  %Actually, we will give a slight variant of the description there, though it will be easy to see the equivalence.

Let $R$ be a $k$-algebra.  Recall that an \emph{$R[[u]]$-lattice in $R((u))^{2m}$} is an $R[[u]]$-submodule $L \subset R((u))^{2m}$ which is free as an $R[[u]]$-module Zariski-locally on $\Spec R$, and such that the natural arrow $L \otimes_{R[[u]]} R((u)) \to R((u))^{2m}$ is an isomorphism.  Borrowing notation from \eqref{disp:wh_Lambda}, given an $R[[u]]$-lattice $L$, we write $\wh L$ for the dual lattice
\[
   \wh L := \bigl\{\, x \in R((u))^{2m} \bigm| h_{R((u))}(L,x) \subset R[[u]] \, \bigr\},
\]
where $h_{R((u))} := h \otimes_K R((u))$ is the induced form on $R((u))^{2m}$. A collection $\{L_i\}_i$ of $R[[u]]$-lattices in $R((u))^{2m}$ is a \emph{chain} if it is totally ordered under inclusion and all successive quotients are locally free $R$-modules (necessarily of finite rank).  A lattice chain is \emph{periodic} if $u L$ is in the chain for every lattice $L$ in the chain.  We write $\L\bigl(R((u))^{2m}\bigr)$ for the category whose objects are the $R[[u]]$-lattices in $R((u))^{2m}$ and whose morphisms are the natural inclusions of lattices.  Of course, any $R[[u]]$-lattice chain may be regarded as a full subcategory of $\L\bigl(R((u))^{2m}\bigr)$.

We define $\F_I$ to be the functor on $k$-algebras that assigns to each $R$ the set of all functors $L\colon\lambda_I \to \L\bigl(R((u))^{2m}\bigr)$ such that
\begin{enumerate}
\renewcommand{\theenumi}{C}
\item\label{it:afv_chain_cond}
   (chain) the image $L(\lambda_I)$ is a lattice chain in $R((u))^{2m}$;
   %need to leave this in to ensure that succ quots are vector bundles
\renewcommand{\theenumi}{P}
\item\label{it:afv_periodicity_cond}
   (periodicity) $L(u\lambda_i) = uL(\lambda_i)$ for all $i \in {2m}\ZZ \pm I$, so that the chain $L(\lambda_I)$ is periodic;
\renewcommand{\theenumi}{R}
\item\label{it:afv_rank_cond}
   (rank) $\dim_k \lambda_i/\lambda_j = \rank_R L(\lambda_i) / L(\lambda_j)$ for all $j  < i$; and
\renewcommand{\theenumi}{D}
\item\label{it:afv_dlty_cond}
   (duality) Zariski-locally on $\Spec R$, there exists $\alpha \in R((t))^\times \subset R((u))^\times$ such that $\wh{L(\lambda_i)} = \alpha L\bigl(\wh\lambda_i\bigr)$ for all $i \in {2m}\ZZ \pm I$.
\end{enumerate}

If $L \in \F_I(R)$ globally admits a scalar $\alpha$ as in \eqref{it:afv_dlty_cond}, then $\alpha$ is well-defined modulo the group
\[
   \bigl\{\, a \in R((t))^\times \bigm| a L(\lambda_{-i}) = L(\lambda_{-i}) \,\bigr\}
   = R[[t]]^\times
\]
(independent of $i$).  It is then an easy exercise to check that the map specified locally by
\[
   L \mapsto \Bigl(\bigl( L(\lambda_i) \bigr)_{i \in I},\ \alpha \bmod R[[t]]^\times \Bigr)
\]
defines an isomorphism from the functor $\F_I$ as we've defined it to the functor $\F_I$ as defined in \cite{paprap09}*{\s3.2}, except that the functor in \cite{paprap09} should only require that $\alpha \bmod R[[t]]^\times$ be given Zariski-locally.

Now recall the subgroup $P_I \subset G(K_0)$, which is defined to be the lattice-wise stabilizer of the chain $\lambda_I$; it equals the parahoric subgroup $P_I^\circ$ when $m \in I$, contains $P_I^\circ$ as a subgroup of index $2$ when $m \notin I$, and in all cases is the full stabilizer in $G(K_0)$ of the facet corresponding to $P_I^\circ$.  As for parahoric subgroups, Bruhat-Tits theory attaches to $P_I$ a canonical smooth affine group scheme over $\O_{K_0}$ with generic fiber $G$ and whose $\O_{K_0}$ points identify with $P_I$; abusing notation as in \s\ref{ss:afv}, we continue to denote this group scheme by $P_I$.  We may then consider the positive loop group $L^+P_I$, which assigns each $k$-algebra $R$ to the lattice-wise stabilizer in $G\bigl(R((t))\bigr)$ of the lattice chain $\lambda_I \otimes_{\O_K} R[[u]]$.

The loop group $LG$ acts on $\F_I$ via the natural representation of $G\bigl(R((t))\bigr)$ on $R((u))^{2m}$, and it follows that the $LG$-equivariant map $LG \to \F_I$ specified by taking the tautological inclusion $\bigl(\lambda_I \inj \L(K^{2m})\bigr) \in \F_I(k)$ as basepoint defines an $LG$-equivariant isomorphism
\begin{equation}\label{disp:whatever}
   LG/L^+P_I \isoarrow \F_I.
\end{equation}
When $m \in I$, this identifies $\F_I$ with the affine flag variety $\F_{P_I^\circ}$.  When $m \notin I$, we recall that the Kottwitz homomorphism
\[
   \kappa_G \colon G(K_0) \to \ZZ \oplus (\ZZ/2\ZZ)
\]
extends to a homomorphism of functors
\[
   \kappa_G \colon LG \to \ZZ \oplus (\ZZ/2\ZZ);
\]
here the target even classifies the connected components of $LG$ \cite{paprap08}*{Th.\ 0.1}.  Let $H$ denote the kernel of the map to the second factor,
\[
   H := \ker[LG \to \ZZ/2\ZZ],
\]
and let $\tau$ denote the block diagonal matrix in $LG(k)$
\begin{equation}\label{disp:tau}
   \tau :=
   \begin{pmatrix}
      \Id_{m-1}\\
        & 0 & 1\\
        & 1 & 0\\
        & & & \Id_{m-1}
   \end{pmatrix}.
\end{equation}
Then $LG = H \amalg \tau H$, and
\begin{equation}\label{disp:hrnt}
   \F_{P_I^\circ} = (H/L^+P_I^\circ) \amalg (\tau H/L^+P_I^\circ)
      \isoarrow LG/L^+P_I \amalg LG/L^+P_I \isoarrow \F_I \amalg \F_I.
\end{equation}

% We shall always identify $\F_{P_I^\circ}$ and $\F_I$ in this way.

\subsection{Schubert cells and varieties}
Consider the parahoric subgroup scheme $P_I^\circ$ over $\O_{K_0}$ and its associated affine flag variety $\F_{P_I^\circ}$.  For $n \in N(K_0)$, the associated \emph{Schubert cell} is the reduced $k$-subscheme
\[
   L^+P_I^\circ \cdot n \subset \F_{P_I^\circ}.
\]
The Schubert cell depends only on the image $w$ of $n$ in $W_{I,G} \bs \wt W_G / W_{I,G}$, and we denote the Schubert cell by $C_w$.  By Haines--Rapoport \cite{hrap08}*{Prop.\ 8}, the inclusion $N(K_0) \subset G(K_0)$ induces a bijection
\[
   W_{I,G} \bs \wt W_G / W_{I,G} 
   \isoarrow P_I^\circ(\O_{K_0}) \bs G(K_0) / P_I^\circ(\O_{K_0}),
\]
so that the Schubert cells are indexed by precisely the elements of $W_{I,G} \bs \wt W_G / W_{I,G}$ and give a stratification of all of $\F_{P_I^\circ}$.  Note that in the special case $I = \{0,\dotsc,m\}$, $P_I^\circ$ is an Iwahori subgroup, the group $W_{I,G}$ is trivial, and the Schubert cells are indexed by $\wt W_G$ itself.

For $w\in W_{I,G} \bs \wt W_G / W_{I,G}$, the associated \emph{Schubert variety $S_w$} is the reduced closure of $C_w$ in $\F_{P_I^\circ}$.  The closure relations between Schubert cells are given by the Bruhat order:  for $w$, $w' \in W_{I,G} \bs \wt W_G / W_{I,G}$, we have $S_w \subset S_{w'}$ in $\F_{P_I^\circ}$ $\iff$ $w \leq w'$ in $W_{I,G} \bs \wt W_G / W_{I,G}$.  See Richarz \cite{rich?}*{Prop.\ 2.8}.

\subsection{Embedding the geometric special fiber}\label{ss:embedding}
We now recall from \cite{paprap09}*{\s3.3} the embedding of the geometric special fiber $M_{I,k}^\naive := M_I^\naive \otimes_{\O_E} k$ of the naive local model into the affine flag variety $\F_{P_I^\circ}$.

The embedding makes use of the lattice-theoretic description of the affine flag variety in terms of $\F_I$ from \s\ref{ss:lattice_theoretic}.  First note that the $\O_K$-lattice chain $\lambda_I$ admits a trivialization in obvious analogy with the trivialization of $\Lambda_I$ specified by \eqref{disp:latt_triv}, where $\lambda_i$ replaces $\Lambda_i$, $\O_K$ replaces $\O_F$, and $u$ replaces $\pi$.  Upon identifying $\O_K \otimes_{\O_{K_0}} k \isoarrow \O_F \otimes_{\O_{F_0}} k$ by sending the $k$-basis elements $1 \otimes 1 \mapsto 1 \otimes 1$ and $u \otimes 1 \mapsto \pi \otimes 1$, our trivializations then yield an identification of chains of $k$-vector spaces
\begin{equation}\label{disp:chain_isom}
   \lambda_i \otimes_{\O_{K_0}} k \iso \Lambda_i \otimes_{\O_{F_0}} k;
\end{equation}
this is even an isomorphism of $k[u]/(u^2)$-modules, where $u$ acts on the right-hand side as multiplication by $\pi \otimes 1$.

Now let $R$ be a $k$-algebra.  Given an $R$-point $\{\F_i\}_i$ in $M_{I,k}^\naive$, for each $i$, let $L_i \subset \lambda_i \otimes_{\O_{K_0}} R[[t]]$ denote the inverse image of
\[
   \F_i \subset \Lambda_i \otimes_{\O_{F_0}} R \iso \lambda_i \otimes_{\O_{K_0}} R
\]
under the reduction-mod-$t$-map
\[
   \lambda_i \otimes_{\O_{K_0}} R[[t]]
   \surj \lambda_i \otimes_{\O_{K_0}} R.
\]
Then $L_i$ is naturally a lattice in $R((u))^{2m}$, and the functor $\lambda_I \to \L\bigl( R((u))^{2m} \bigr)$ sending $\lambda_i \mapsto L_i$ determines a point in $\F_I(R)$.  In this way we get a monomorphism
\begin{equation}\label{disp:embedding}
   \vcenter{
   \xymatrix@R=0ex{
      M_{I,k}^\naive\,  \ar@{^{(}->}[r]  &  \F_I\\
      \{\F_i\}_i  \ar@{|->}[r]  & (\lambda_i \shortmapsto L_i)
   }
   }.
\end{equation}
Since $M_{I,k}^\naive$ is proper, \eqref{disp:embedding} is a closed immersion of ind-schemes.  When $m \in I$, this embeds $M_{I,k}^\naive$ in $\F_{P_I^\circ} = \F_{P_I}$ via the isomorphism \eqref{disp:whatever}.  When $m \notin I$, $\F_{P_I^\circ}$ is a disjoint union of two copies of $\F_I$ via \eqref{disp:hrnt}, and we use \eqref{disp:embedding} to embed $M_{I,k}^\naive$ in the copy marked by $t_{\mu_{r,s}} \in \wt W_G$.  From now on we shall often identify $M_{I,k}^\naive$ with its image in $\F_{P_I^\circ}$.

The embedding $M_{I,k}^\naive \inj \F_{P_I^\circ}$ is \emph{$L^+P_I^\circ$-equivariant} with respect to the following left $L^+P_I^\circ$-actions on source and target; compare \citelist{\cite{paprap03}*{\s3}\cite{paprap05}*{\s6, \s11}\cite{paprap08}*{\s11}\cite{paprap09}*{\s3.3}\cite{sm11d}*{\s4.4}}.  On $\F_{P_I^\circ}$ we just take the natural left $L^+P_I^\circ$-action.  For $M_{I,k}^\naive$, $L^+P_I^\circ$ acts naturally on $\lambda_I \otimes_{\O_{K_0}} k$ via the tautological action of $L^+P_I$ on $\lambda_I$. The chain isomorphism \eqref{disp:chain_isom} then induces a homomorphism
\[
   L^+P_I^\circ \to \ul\Aut(\Lambda_I)_{\O_E} \otimes_{\O_{E}} k,
\]
where we recall the $\O_{E}$-group scheme $\ul\Aut(\Lambda_I)_{\O_E}$ from \s\ref{ss:latt_chain_auts}.  The $\ul\Aut(\Lambda_I)_{\O_E}$-action on $M_I^\naive$ now furnishes the desired $L^+P_I^\circ$-action on $M_{I,k}^\naive$.  In this way $L^+P_I^\circ$ also acts on $M_{I,k}^\wedge$, $M^\spin_{I,k}$, and $M_{I,k}^\loc$.

For nonempty $J \subset I \subset \{0,\dotsc,m\}$ satisfying \eqref{disp:I_cond}, we get a diagram
\[
   \xymatrix{
      M_{I,k}^\naive\, \ar@{^{(}->}[r] \ar[d]  &  \F_{P_I^\circ} \ar[d]\\
      M_{J,k}^\naive\, \ar@{^{(}->}[r]  &  \F_{P_J^\circ},
   }
\]
where the vertical arrows are the natural projections and the horizontal arrows are the embeddings just defined.  If $m \in J$ or $m \notin I$, then this diagram commutes, but if $m \in I \smallsetminus J$, then it does not, as will follow from our description of the Schubert varieties in $M_{I,k}^\naive$ and $M_{J,k}^\naive$ in the next section.  On the other hand, if we replace the geometric special fiber of the naive local model with that of the spin local model everywhere, then the diagram does commute for all nonempty $J \subset I$ satisfying \eqref{disp:I_cond}.

\section{Schubert cells in the geometric special fiber}\label{s:schub_cells}
We continue to take $I$ to be a nonempty subset of $\{0,\dotsc,m\}$ satisfying property \eqref{disp:I_cond}, and to identify $M_{I,k}^\naive$ with its image in the affine flag variety $\F_{P_I^\circ}$ under the embedding described in \s\ref{ss:embedding}.  It follows from $L^+P_I^\circ$-equivariance that the underlying topological spaces of $M^\naive_{I,k}$, $M_{I,k}^\wedge$, $M^\spin_{I,k}$, and $M_{I,k}^\loc$ are all unions of Schubert varieties in $\F_{P_I^\circ}$.  In this subsection we shall describe the Schubert varieties that are contained in each, and we shall prove Theorems \ref{st:main_thm} and \ref{st:2nd_main_thm} modulo the main combinatorial results of the next two sections.  We write $i^* := 2m+1-i$.

\subsection{The image of the geometric special fiber}\label{ss:im_spec_fib}
Let $R$ be a $k$-algebra.  It is clear from the definition of the embedding $M_{I,k}^\naive \inj \F_I$ \eqref{disp:embedding} and the various conditions in the definition of $M_I^\naive$ that the image of $M^\naive_{I,k}(R)$ in $\F_I(R)$ consists precisely of the functors $\lambda_i \mapsto L_i$ in $\F_I(R)$ such that, for all $i \in 2m\ZZ \pm I$,
\begin{enumerate}
\item\label{it:latt_incl_cond}
   $\lambda_{i,R[[t]]} \supset L_i \supset t\lambda_{i,R[[t]]}$, where $\lambda_{i,R[[t]]} := \lambda_i \otimes_{\O_{K_0}} R[[t]]$;
\item\label{it:rank_n_cond}
   the $R$-module $\lambda_{i,R[[t]]}/L_i$ is locally free of rank $2m$; and
\item\label{it:t^-1_cond}
   $\wh L_i = t^{-1} L_{-i}$.
\end{enumerate}

We remark that condition \eqref{it:t^-1_cond} is actually redundant; that is, for any functor $\lambda_I \to \L \bigl(R((u))^{2m} \bigr)$ satisfying conditions \eqref{it:afv_chain_cond}, \eqref{it:afv_periodicity_cond}, \eqref{it:afv_rank_cond}, and \eqref{it:afv_dlty_cond} from \s\ref{ss:lattice_theoretic} and conditions \eqref{it:latt_incl_cond} and \eqref{it:rank_n_cond} above, the scalar $\alpha$ appearing in \eqref{it:afv_dlty_cond} must be congruent to $t^{-1} \bmod R[[t]]^\times$.  Since we won't need to use this fact anywhere in the paper, we leave the details as an exercise to the reader.

% It is an immediate consequence of the $L^+P_I$-equivariance of the embedding \eqref{disp:embedding} that \emph{the underlying topological spaces of $M^\naive_{I,k}$, $M_{I,k}^\wedge$, $M^\spin_{I,k}$, and $M_{I,k}^\loc$ are all unions of Schubert varieties in $\F_I$.}  One of our essential goals for the rest of the paper is to obtain a good description of the Schubert varieties that occur in $M_{I,k}^\wedge$ and $M^\spin_{I,k}$.

\subsection{Naive permissibility}\label{ss:naive_perm}
Let $w \in W_{I,G} \bs \wt W_G / W_{I,G}$.  Then the condition that the Schubert variety $S_w$ be contained in $M_{I,k}^\naive$ (resp.\ $M_{I,k}^\wedge$, resp.\ $M_{I,k}^\spin$) amounts to the condition that for one, hence any, representative $\wt w$ of $w$ in $\wt W_G/W_{I,G}$, the point $\wt w \cdot \lambda_I \in \F_I(k)$ be contained in $M_{I,k}^\naive(k)$ (resp.\ $M_{I,k}^\wedge(k)$, resp.\ $M_{I,k}^\spin(k)$) and, in case $m \notin I$, that $\wt w$ mark the same copy of $\F_I$ inside $\F_{P_I^\circ}$ as $t_{\mu_{r,s}}$ does.  We shall find it convenient to express these conditions in terms of faces of type $I$, beginning in this subsection with containment in $M_{I,k}^\naive$.

\begin{defn}
Let $\wt w \in \wt W_{G}/W_{I,G}$.  We say that $\wt w$ is \emph{naively permissible} if the Schubert variety in $\F_{P_I^\circ}$ attached to $\wt w$ is contained in $M_{I,k}^\naive$.
\end{defn}

To translate the notion of naive permissibility into combinatorics, let $\wt w \in \wt W_G/W_{I,G}$, and consider the face $(\wt w \omega_i)_{i\in 2m\ZZ \pm I}$ of type $(2m,I)$ attached to $\wt w$, as in \s\ref{ss:GU_faces}.  Then it is clear from the definitions and from \s\ref{ss:im_spec_fib} that $\wt w$ is naively permissible $\iff$
\begin{enumerate}
\item\label{it:ineq_cond}
   $\omega_i \leq \wt w \omega_i \leq \omega_i + \mathbf 2$ for all $i \in 2m\ZZ \pm I$;
\item\label{it:size_cond}
   $\Sigma (\wt w\omega_i) = 2m - i$ for one, hence every, $i \in 2m\ZZ \pm I$; and
\item
   if $m \notin I$, then $\wt w$ and $t_{\mu_{r,s}}W_{I,G}$ have the same image under the composite map
   \[
      \wt W_G/W_{I,G} \xra{\kappa_G} \ZZ \oplus (\ZZ/2\ZZ)
         \xra{\pr_{\ZZ/2\ZZ}} \ZZ/2\ZZ.
   \]
\end{enumerate}
Recasting this in terms of the vector $\mu_i^{\wt w} = \wt w \omega_i - \omega_i$ \eqref{disp:something}, we obtain the following.

\begin{prop}
$\wt w \in \wt W_G/W_{I,G}$ is naively permissible $\iff$
\begin{enumerate}
\renewcommand{\theenumi}{P\arabic{enumi}}
\item\label{it:P1}
   $\mu_i^{\wt w} + (\mu_{-i}^{\wt w})^* = \mathbf 2$ and $\mathbf 0 \leq \mu_i^{\wt w} \leq \mathbf 2$ for all $i \in I$, and if $m \notin I$, then furthermore $\wt w \equiv t_{\mu_{r,s}} \bmod W_{\aff,G}$.
\end{enumerate}
\end{prop}

\begin{proof}
Note that condition \eqref{it:size_cond} above just says that $\wt w$ defines a $2$-face.  In light of the explicit description of the Kottwitz homomorphism \eqref{disp:explicit}, the implication $\Longrightarrow$ is then clear, and the implication $\Longleftarrow$ is clear from \eqref{disp:mu_per_cond} and \eqref{disp:mu_dlty_cond}.
\end{proof}

Given a naively permissible $\wt w$, the point $\wt w\cdot \lambda_I$ in $\F_I(k)$ corresponds to a point $(\F_i \subset \Lambda_i \otimes_{\O_{F_0}} k)_{i\in 2m\ZZ \pm I} \in M_{I,k}^\naive(k)$ of a rather special sort.  Namely, identifying $\Lambda_i \otimes_{\O_{F_0}} k$ with $\O_F^{2m} \otimes_{\O_{F_0}} k$ via \eqref{disp:latt_triv}, we have
\begin{enumerate}
\renewcommand{\theenumi}{S}
\item\label{it:S_fixed}
   for all $i$, $\F_i$, regarded as a subspace in $\O_F^{2m} \otimes_{\O_{F_0}} k$, is $k$-spanned by $2m$ of the elements $\epsilon_1\otimes 1,\dotsc$, $\epsilon_{2m}\otimes 1$, $\pi\epsilon_1 \otimes 1,\dotsc$, $\pi\epsilon_{2m} \otimes 1$,
\end{enumerate}
where we recall that $\epsilon_1,\dotsc,\epsilon_{2m}$ denotes the standard $\O_F$-basis in $\O_F^n$.  On the other hand, for any point $(\F_i)_i$ in $M_{I,k}^\naive(k)$, let us say that $(\F_i)_i$ is an \emph{$S$-fixed point} if it satisfies \eqref{it:S_fixed}; it is easy to check that the $S$-fixed points are exactly the points in $M_k^\naive(k)$ fixed by $L^+S(k)$.  In this way, we get a bijection between the naively permissible $\wt w\in \wt W_G/W_{I,G}$ and the $S$-fixed points in $M_{I,k}^\naive(k)$, which we denote by $\wt w \mapsto (\F_i^{\wt w})_i$.  Explicitly, for naively permissible $\wt w$, in terms of the vector $\mu_i^{\wt w}$, we have
\begin{equation}\label{disp:F_i}
   \F_i^{\wt w} = \sum_{\mu^{\wt w}_i(j) = 0} k\cdot (\epsilon_j\otimes 1)
                          +\sum_{\mu^{\wt w}_i(j) = 0,1} k \cdot (\pi\epsilon_j \otimes 1) \subset \O_F^{2m} \otimes_{\O_{F_0}} k.
\end{equation}

\subsection{Wedge-permissibility}\label{ss:wedge_spin_perm}
We continue with the notation of the previous subsection.

\begin{defn}
We say that $\wt w \in \wt W_G/W_{I,G}$ is \emph{wedge-permissible} if the Schubert variety in $\F_{P_I^\circ}$ attached to $\wt w$ is contained in $M_{I,k}^\wedge$.
\end{defn}

Of course wedge-permissibility depends only on the image of $\wt w$ in $W_{I,G} \bs \wt W_G / W_{I,G}$.
By definition, for naively permissible $\wt w$, we have
\[
   (\F^{\wt w}_i)_i \in M_{I,k}^\wedge(k) \iff
   \begin{varwidth}{\textwidth}
      \centering
      for all $i \in 2m\ZZ \pm I$, $\bigwedge_k^{s+1} (\,\pi\otimes 1 \mid \F^{\wt w}_i\,) = 0$\\
      and $\bigwedge_k^{r+1} (\,\pi\otimes 1 \mid \F^{\wt w}_i\,) = 0$,
   \end{varwidth}
\]
where we recall our fixed partition $2m = s + r$ with $s < r$.  For fixed $i$, the second equality on the right-hand side of the display is implied by the first.  We then have the following.

\begin{prop}\label{st:wedge-perm}
$\wt w \in \wt W_G/W_{I,G}$ is wedge-permissible $\iff$ $\wt w$ is naively permissible and
\begin{enumerate}
\renewcommand{\theenumi}{P2}
\item
   for all $i \in I$, $\#\{\, j\mid \mu^{\wt w}_i(j) = 0\, \} \leq s$.\qed
\end{enumerate}
\end{prop}

\begin{proof}
The implication $\Longrightarrow$ is immediate from the description of $\F_i^{\wt w}$ \eqref{disp:F_i}, and the implication $\Longleftarrow$ follows from this, \eqref{disp:mu_per_cond}, \eqref{disp:mu_dlty_cond}, and \eqref{st:c_equiv_defs}.
\end{proof}

\subsection{Spin-permissibility}\label{ss:spin-perm}
We continue with the notation of the previous subsections. We also freely use the notation of \s\ref{ss:wedge_spin_conds}, and we recall the sets $A_i$ and $B_i$ (taken with $n = 2m$) from \eqref{disp:A_i_B_i}.

\begin{defn}
We say that $\wt w \in \wt W/W_{I,G}$ is \emph{spin-permissible} if the Schubert variety in $\F_{P_I^\circ}$ attached to $\wt w$ is contained in $M_{I,k}^\spin$.
\end{defn}

As for naive and wedge-permissibility, spin-permissibility depends only on the image of $\wt w$ in $W_{I,G} \bs \wt W_G / W_{I,G}$.  Our aim in this subsection is to characterize the set of spin-permissible $\wt w$ in terms of the vector $\mu_i^{\wt w}$, for $i \in 2m\ZZ \pm I$.

Let $i\in I$, and suppose that $\wt w$ is wedge-permissible.  Continuing to identify $\O_F^{2m}$ with $\Lambda_i$ (resp.\ $\Lambda_{2m-i}$) via \eqref{disp:latt_triv}, the $2m$ elements in $\O_F^{2m}$
\begin{equation}\label{disp:F_i_spanners}
\begin{aligned}
   \epsilon_j \quad &\text{for}\quad \mu^{\wt w}_i(j) = 0  & &(\text{resp.}\ \mu_{2m-i}^{\wt w}(j) = 0); \quad \text{and}\\
   \pi\epsilon_j \quad &\text{for}\quad \mu^{\wt w}_i(j) = 0,1 & &(\text{resp.}\ \mu_{2m-i}^{\wt w}(j) = 0,1)
\end{aligned}
\end{equation}
span an $\O_{F_0}$-submodule whose image under the reduction-mod-$\pi_0$ map
\[
   \O_F^{2m}  \surj \O_F^{2m} \otimes_{\O_{F_0}} k
\]
is $\F^{\wt w}_i$ (resp.\ $\F_{2m-i}^{\wt w}$).  Take the wedge product (in any order) of the elements \eqref{disp:F_i_spanners} in $\bigwedge_{\O_{F_0}}^{2m} \O_F^{2m}$, and let $g_i \in \bigwedge_{\O_{F_0}}^{2m} \Lambda_i$ (resp.\ $g_{2m-i} \in \bigwedge_{\O_{F_0}}^{2m} \Lambda_{2m-i}$) denote the image of this element under the isomorphism $\bigwedge_{\O_{F_0}}^{2m} \O_F^{2m} \isoarrow \bigwedge_{\O_{F_0}}^{2m} \Lambda_i$ (resp.\ $\bigwedge_{\O_{F_0}}^{2m} \O_F^{2m} \isoarrow \bigwedge_{\O_{F_0}}^{2m} \Lambda_{2m-i}$) induced by \eqref{disp:latt_triv}.  Then, up to sign, and in terms of the notation \eqref{disp:f_E}, $g_i$ equals
\[
   \pi_0^{a_i} f_{E_i} \in \sideset{}{_{F_0}^{2m}}{\bigwedge} V,
\]
where $E_i \subset \{1,\dotsc,4m\}$ is the subset of cardinality $2m$
{\allowdisplaybreaks
\begin{equation}\label{disp:E_i}
\begin{aligned}
   E_i := {}&\bigl\{\, j \in \{1,\dotsc,i\} \bigm| \mu^{\wt w}_i(j) = 0 \,\bigr\}\\
                   &\quad\amalg \bigl\{\, j \in \{i+1,\dotsc,m\} \bigm| \mu^{\wt w}_i(j) = 0,1 \,\bigr\}\\
                   &\quad\amalg \bigl\{\, j \in \{m+1,\dotsc,2m\} \bigm| \mu^{\wt w}_i(j) = 0 \,\bigr\}\\
                   &\quad\amalg\bigl\{\, 2m+j \in \{2m+1,\dotsc,2m+i\} \bigm| \mu^{\wt w}_i(j) = 0,1 \,\bigr\}\\
                   &\quad\amalg\bigl\{\, 2m+j \in \{2m+i+1,\dotsc,3m\} \bigm| \mu^{\wt w}_i(j) = 0 \,\bigr\}\\
                   &\quad\amalg\bigl\{\, 2m+j \in \{3m+1,\dotsc,4m\} \bigm| \mu^{\wt w}_i(j) = 0,1 \,\bigr\},
\end{aligned}
\end{equation}
}%
and
\[
   a_i := \#\bigl(E_i \cap \{i+1,\dotsc,m\}\bigr) = \#\bigl\{\, j \in \{i+1,\dotsc,m\} \bigm| \mu^{\wt w}_i(j) = 0,1 \,\bigr\};
\]
and again up to sign, $g_{2m-i}$ equals
\[
   \pi_0^{a_{2m-i}} f_{E_{2m-i}} \in \sideset{}{_{F_0}^{2m}}{\bigwedge} V,
\]
where $E_{2m-i} \subset \{1,\dotsc,4m\}$ is the subset of cardinality $2m$
{\allowdisplaybreaks
\begin{align*}
   E_{2m-i} := {}&\bigl\{\, j \in \{1,\dotsc,m\} \bigm| \mu^{\wt w}_{2m-i}(j) = 0 \,\bigr\}\\
                   &\quad\amalg \bigl\{\, j \in \{m+1,\dotsc,2m-i\} \bigm| \mu^{\wt w}_{2m-i}(j) = 0,1 \,\bigr\}\\
                   &\quad\amalg \bigl\{\, j \in \{i^*,\dotsc,2m\} \bigm| \mu^{\wt w}_{2m-i}(j) = 0 \,\bigr\}\\
                   &\quad\amalg\bigl\{\, 2m+j \in \{2m+1,\dotsc,3m\} \bigm| \mu^{\wt w}_{2m-i}(j) = 0,1 \,\bigr\}\\
                   &\quad\amalg\bigl\{\, 2m+j \in \{3m+1,\dotsc,4m-i\} \bigm| \mu^{\wt w}_{2m-i}(j) = 0 \,\bigr\}\\
                   &\quad\amalg\bigl\{\, 2m+j \in \{2m+i^*,\dotsc,4m\} \bigm| \mu^{\wt w}_{2m-i}(j) = 0,1 \,\bigr\},
\end{align*}
}%
and
\begin{align*}
   a_{2m-i} &:= -\#\bigl(E_{2m-i} \cap \{3m+1,\dotsc,4m-i\}\bigr)\\
            &\phantom{:}= -\#\bigl\{\, j \in \{m+1,\dotsc,2m-i\} \bigm| \mu^{\wt w}_{2m-i}(j) = 0 \,\bigr\}.
\end{align*}

To study the spin condition for $\F^{\wt w}_i$, we shall also need the set $E_i^\perp = (4m+1-E_i)^c$, which is given by
{\allowdisplaybreaks
\begin{align*}
   E_i^\perp = {}&\bigl\{\, j \in \{1,\dotsc,m\} \bigm| \mu^{\wt w}_i(j^*) \neq 0,1 \,\bigr\}\\
          &\quad\amalg \bigl\{\, j \in \{m+1,\dotsc,2m-i\} \bigm| \mu^{\wt w}_i(j^*) \neq 0 \,\bigr\}\\
          &\quad\amalg \bigl\{\, j \in \{i^*,\dotsc,2m\} \bigm| \mu^{\wt w}_i(j^*) \neq 0,1 \,\bigr\}\\
          &\quad\amalg\bigl\{\, 2m+j \in \{2m+1,\dotsc,3m\} \bigm| \mu^{\wt w}_i(j^*) \neq 0 \,\bigr\}\\
          &\quad\amalg\bigl\{\, 2m+j \in \{3m+1,\dotsc,4m-i\} \bigm| \mu^{\wt w}_i(j^*) \neq 0,1 \,\bigr\}\\
          &\quad\amalg\bigl\{\, 2m+j \in \{2m+i^*,\dotsc,4m\} \bigm| \mu^{\wt w}_i(j^*) \neq 0 \,\bigr\}.
\end{align*}
}%
We then have the following.

\begin{lem}\label{st:laundry_list}\hfill
\begin{enumerate}
\renewcommand{\theenumi}{\roman{enumi}}
\item\label{it:E_i^perp}
   $E_i^\perp = E_{2m-i}$.
\item\label{it:a_i^perp}
   Let $a_i^\perp$ (resp.\ $a_{2m-i}^\perp$) denote the unique integer such that
   \[
      \pi_0^{a_i^\perp} f_{E_i^\perp} \in \sideset{}{_{\O_{F_0}}^{2m}}\bigwedge \Lambda_i \smallsetminus \pi_0\sideset{}{_{\O_{F_0}}^{2m}}\bigwedge \Lambda_i
   \]
   (resp.
   \[
      \pi_0^{a_{2m-i}^\perp} f_{E_{2m-i}^\perp} \in \sideset{}{_{\O_{F_0}}^{2m}}\bigwedge \Lambda_{2m-i} \smallsetminus \pi_0\sideset{}{_{\O_{F_0}}^{2m}}\bigwedge \Lambda_{2m-i} \text{).}
   \]
   Then
   \begin{align*}
      a_i^\perp &= \#(E_i^\perp \cap \{i+1,\dotsc,m\})\\
                &= \#\bigl\{\, j \in \{i+1,\dotsc,m\} \bigm| \mu^{\wt w}_i(j^*) \neq 0,1 \,\bigr\}
   \end{align*}
   and
   \begin{align*}
      a_{2m-i}^\perp &= -\#\bigl(E_{2m-i}^\perp \cap \{3m+1,\dotsc,4m-i\}\bigr)\\
               &= -\#\bigl\{\, j \in \{m+1,\dotsc,2m-i\} \bigm| \mu^{\wt w}_{2m-i}(j^*) \neq 0 \,\bigr\}.
   \end{align*}
\item\label{it:a_i_ineq}
   $a_i \geq a_i^\perp$ and $a_{2m-i} \geq a_{2m-i}^\perp$.
\item\label{it:a_i-a_2m-i}
   $a_i - a_{2m-i} = a_i^\perp - a_{2m-i}^\perp = m - i$.
\end{enumerate}
\end{lem}

\begin{proof}
Assertion \eqref{it:E_i^perp} follows immediately from our explicit expressions for $E_{2m-i}$ and $E_i^\perp$, using that $\mu_i^{\wt w} + (\mu_{2m-i}^{\wt w})^* = \mathbf 2$ and $\mathbf 0 \leq \mu_i^{\wt w} \leq \mathbf 2$ by naive permissibility.  Assertion \eqref{it:a_i^perp} is clear.  Assertion \eqref{it:a_i_ineq} follows from our explicit expressions for the quantities involved, naive permissibility, and the basic inequalities \eqref{st:basic_ineqs}.  To prove \eqref{it:a_i-a_2m-i}, by naive permissibility we have
\[
   a_{2m-i} = -\#\bigl\{\, j \in \{i+1,\dotsc,m\} \bigm| \mu_i^{\wt w}(j) = 2\,\bigr\}.
\]
Hence
\[
   a_i - a_{2m-i} = \#\{i+1,\dotsc,m\} = m - i.
\]
The proof that $a_i^\perp - a_{2m-i}^\perp = m-i$ is similar.
\end{proof}

As we shall see in a moment, one of the upshots of \eqref{st:laundry_list} is that $\F_i^{\wt w}$ satisfies the spin condition $\iff$ $\F_{2m-i}^{\wt w}$ does.  Before explaining this, let us first prove another lemma.

\begin{lem}\label{st:E_equiv_conds}
Consider the following conditions.
\begin{enumerate}
\renewcommand{\theenumi}{\roman{enumi}}
\item\label{it:E_i=E_i^perp}
   $E_i = E_i^\perp$.
\item\label{it:a_i=a_i^perp}
   $a_i = a_i^\perp$.
\item\label{it:mu(j)=0,2}
   $\mu_i^{\wt w}(j) \in \{0,2\}$ for all $j \in B_i$.
\item\label{it:mu_self_dual}
   $\mu_i^{\wt w}$ is self-dual.
\end{enumerate}
Then \eqref{it:E_i=E_i^perp} $\iff$ \eqref{it:a_i=a_i^perp} $\iff$ \eqref{it:mu(j)=0,2} $\implies$ \eqref{it:mu_self_dual}.
\end{lem}

\begin{proof}
The implication \eqref{it:mu(j)=0,2} $\implies$ \eqref{it:mu_self_dual} is given by \eqref{st:self-dual_mu}.
   
We now show \eqref{it:a_i=a_i^perp} $\iff$ \eqref{it:mu(j)=0,2}.  Suppose that $a_i = a_i^\perp$, or in other words, since $\wt w$ is naively permissible, that
\[
   \#\bigl\{\, j \in \{i+1,\dotsc,m\} \bigm| \mu_i^{\wt w}(j) = 0,1 \,\bigl\}
   =
   \#\bigl\{\, j \in \{i+1,\dotsc,m\} \bigm| \mu_i^{\wt w}(j^*) = 2 \,\bigl\}.
\]
For $j$ contained in the set on the right-hand side, the basic inequalities \eqref{st:basic_ineqs} imply that $\mu_i^{\wt w}(j) = 0$.  Hence $j$ is contained in the set on the left-hand side.  Since the two sets have the same cardinality, they must then be equal.  The basic inequalities then imply \eqref{it:mu(j)=0,2}.  Conversely, it is clear that \eqref{it:mu(j)=0,2} implies \eqref{it:a_i=a_i^perp}.

We finally show $\eqref{it:E_i=E_i^perp}$ $\iff$ $\eqref{it:mu(j)=0,2}$.  Suppose that \eqref{it:mu(j)=0,2} holds.  If $j \in A_i$, then
\begin{alignat*}{2}
   j \in E_i
      &\iff \mu_i^{\wt w}(j) = 0\\
      &\iff \mu_i^{\wt w}(j^*) = 2 & & \quad \text{(by \eqref{it:mu_self_dual})}\\
      &\iff 2m+j^* \notin E_i.
\end{alignat*}
If $j \in B_i$, then
\begin{alignat*}{2}
   j \in E_i
      &\iff \mu_i^{\wt w}(j) = 0 & &\quad \text{(by \eqref{it:mu(j)=0,2})}\\
      &\iff \mu_i^{\wt w}(j^*) = 2 & &\quad \text{(by \eqref{it:mu_self_dual})}\\
      &\iff 2m+j^* \notin E_i & &\quad \text{(by \eqref{it:mu(j)=0,2}).}
\end{alignat*}
Hence $E_i = E_i^\perp$.  Conversely, suppose that there exists $j \in B_i$ such that $\mu_i^{\wt w}(j) = 1$.  Then $\mu_i^{\wt w}(j^*) \in \{0,1\}$ by the basic inequalities.  Let $l := \min\{j,j^*\}$.  Then $l$, $2m+l^* \in E_i$.  Hence $E_i \neq E_i^\perp$.
\end{proof}

We are now ready to translate the spin condition for $\F_i^{\wt w}$ and $\F_{2m-i}^{\wt w}$ into combinatorics.  Following \eqref{st:laundry_list}\eqref{it:a_i_ineq}, we consider separately the cases $a_i > a_i^\perp$ and $a_i = a_i^\perp$.

First suppose that $a_i > a_i^\perp$, or equivalently, by \eqref{st:laundry_list}\eqref{it:a_i-a_2m-i}, that $a_{2m-i} > a_{2m-i}^\perp$.  Consider the elements
\[
   h_i := \pi_0^{a_i} f_{E_i} + (-1)^s\pi_0^{a_i - a_i^\perp} \pi_0^{a_i^\perp} \sgn(\sigma_{E_i}) f_{{E_i^\perp}}
   \in\Bigl(\sideset{}{_{\O_F}^{2m}}{\bigwedge} \Lambda_i\Bigr)_{(-1)^s}
\]
and
\begin{multline*}
   h_{2m-i} := \pi_0^{a_{2m-i}} f_{E_{2m-i}} + (-1)^s\pi_0^{a_{2m-i} - a_{2m-i}^\perp} \pi_0^{a_{2m-i}^\perp} \sgn(\sigma_{E_{2m-i}}) f_{{E_{2m-i}^\perp}}\\
   \in\Bigl(\sideset{}{_{\O_F}^{2m}}{\bigwedge} \Lambda_{2m-i}\Bigr)_{(-1)^s}.
\end{multline*}
Then the image of $h_i$ in $\bigwedge_{k}^{2m} (\Lambda_i \otimes_{\O_{F_0}} k)$ spans $\bigwedge_k^{2m} \F_i^{\wt w}$, and analogously with $i$ replaced by $2m-i$.  Hence $\F_i^{\wt w}$ and $\F_{2m-i}^{\wt w}$ satisfy the spin condition.

Now suppose that $a_i = a_i^\perp$, or equivalently by \eqref{st:laundry_list}\eqref{it:a_i-a_2m-i}, that $a_{2m-i} = a_{2m-i}^\perp$.  Then $E_i = E_i^\perp = E_{2m-i} = E_{2m-i}^\perp$ by \eqref{st:laundry_list} and \eqref{st:E_equiv_conds}.  Thus we see easily that $\F_i^{\wt w}$ and $\F_{2m-i}^{\wt w}$ simultaneously satisfy or fail the spin condition according as
\[
   f_{E_i} \in \Bigl(\sideset{}{_{F_0}^{2m}}\bigwedge V\Bigr)_{(-1)^s}
   \quad\text{or}\quad
   f_{E_i} \in \Bigl(\sideset{}{_{F_0}^{2m}}\bigwedge V\Bigr)_{(-1)^{s+1}},
\]
or equivalently, according as
\[
   \sgn(\sigma_{E_i}) = (-1)^s
   \quad\text{or}\quad
   \sgn(\sigma_{E_i}) = (-1)^{s+1}.
\]

\begin{lem}
Let $E \subset \{1,\dotsc,4m\}$ be a subset of cardinality $2m$ such that $E = E^\perp$.  Then $\sigma_E \in S_{4m}^*$, that is, $\sigma_E(4m+1-j) = 4m + 1 - \sigma_E(j)$ for all $j \in \{1,\dotsc,4m\}$.
\end{lem}

\begin{proof}
Obvious from the definition of $\sigma_E$.
\end{proof}

It follows from the lemma that $\sigma_{E_i}$ is even or odd according as the number of elements $j \in \{1,\dotsc,2m\}$ such that $\sigma_{E_i}(j) > 2m$ is even or odd.  By definition of $\sigma_{E_i}$, this is in turn equal to the parity of the cardinality of $E_i \cap \{2m+1,\dotsc,4m\}$.  By \eqref{st:E_equiv_conds}, if $\mu_i^{\wt w}(j) = 1$, then $j \in A_i$.  So we see from our explicit description of $E_i$ \eqref{disp:E_i} that
\[
   \#(E_i \cap\{2m+1,\dotsc,4m\}) = \#\bigl\{\, j \bigm| \mu_i^{\wt w}(j) = 0 \,\bigr\}
                                  + \#\bigl\{\, j \bigm| \mu_i^{\wt w}(j) = 1\,\bigr\}.
\]
Since $\mu_i^{\wt w}$ is self-dual, the second term on the right-hand side is even.  We conclude that under our assumption $a_i = a_i^\perp$,
\[
   \F_i^{\wt w}\ \text{and}\ \F_{2m-i}^{\wt w}\ \text{satisfy the spin condition}
   \iff
   \#\{\, j \mid \mu_i^{\wt w}(j) = 0 \,\} \equiv s \bmod 2.
\]

Rephrasing slightly, we have now proved the following.

\begin{prop}\label{st:wedge_iff_spin-perm}
$\wt w \in \wt W_G/W_{I,G}$ is spin-permissible $\iff$  $\wt w$ is wedge-permissible and
\begin{enumerate}
\renewcommand{\theenumi}{P3}
\item\label{it:P3}
   for all $i \in I$, if $\mu_i^{\wt w}$ is self-dual, then $\#\{\, j \mid \mu_i^{\wt w}(j) = 0\,\} \equiv s \bmod 2$ or there exists $j \in B_i$ such that $\mu_i^{\wt w}(j) = 1$.\qed
\end{enumerate}
\end{prop}

\begin{rk}
Suppose that $\wt w$ is spin-permissible and $m \in I$.  Then $\mu_m^{\wt w}$ is self-dual and $B_m = \emptyset$.  Hence \eqref{it:P3} and the condition $\mu_m^{\wt w} + (\mu_m^{\wt w})^* = \mathbf 2$ in \eqref{it:P1} imply that $\wt w \equiv t_{\mu_{r,s}} \bmod W_{\aff, G}$.  On the other hand, it really is necessary to \emph{impose} the condition $\wt w \equiv t_{\mu_{r,s}} \bmod W_{\aff, G}$ when $m \notin I$.  Indeed, for such $I$ and any $\wt w \in \wt W_G/W_{I,G}$, we have $\mu_i^{\wt w} = \mu_i^{\wt w\tau}$ for all $i \in 2m \pm I$, where $\tau$ is the matrix \eqref{disp:tau}, regarded as an element of the image of $S_{2m}^*$ in $\wt W_G/W_{I,G}$.  Here $\tau$ has nontrivial Kottwitz invariant, so that $\wt w \not\equiv \wt w\tau \bmod W_{\aff,G}$.
\end{rk}

\begin{rk}
Suppose that $\wt w$ is wedge-permissible and $0 \in I$.  Then we claim that $\F_0^{\wt w}$ automatically satisfies the spin condition.  Indeed, $B_0 = \{1,\dotsc,2m\}$.  So if $\mu_0^{\wt w}(j) \neq 1$ for all $j \in B_0$, then by wedge-permissibility we must have $\mu_0^{\wt w}(j) \in \{0,2\}$ for all $j$ and $s = m$.  Hence $\#\{\, j \mid \mu_0^{\wt w}(j) = 0\,\} = s$ on the nose.  For $m \geq 2$, it follows that $M_{\{0\}}^\wedge$ and $M_{\{0\}}^\spin$ coincide as topological spaces.  Once we complete the proof in the next subsection that $M_I^\spin$ is topologically flat, it will furthermore follow that $M_{\{0\}}^\wedge$ is topologically flat, in support of Pappas's conjecture \cite{pap00}*{\s4 p.\ 594} that $M_{\{0\}}^\wedge$ is flat over $\O_E$.
\end{rk}

\subsection{Topological flatness of \texorpdfstring{$M_I^\spin$}{M\_I\^{}spin}}\label{ss:top_flat}

We now come to the main algebro-geometric results of the paper.  The key combinatorial fact we shall need to prove that $M_I^\spin$ is topologically flat is the following.

\begin{thm}\label{st:adm<=>perm_I}
Let $\wt w \in \wt W_G/W_{I,G}$.  Then $\wt w$ is spin-permissible $\iff$ $\wt w$ is $\{\mu_{r,s}\}$-admissible.
\end{thm}

\begin{proof}
We exploit the isomorphism
\[
   \wt W_G/W_{I,G} \isoarrow \wt W_{B_m}/W_{m-I,B_m}
\]
induced by the composite isomorphism $\wt W_G \isoarrow \wt W_{B_m}$ in \eqref{st:wtW_G->wtW_B_m}.  By \eqref{st:id_adm_sets} this identifies the $\{\mu_{r,s}\}$-admissible set in $\wt W_G/W_{I,G}$ with the $\mu$-admissible set in $\wt W_{B_m}/W_{m-I,B_m}$, where $\mu$ denotes the cocharacter $\bigl(2^{(s)},1^{(2m-2s+1)},0^{(s)}\bigr)$.  On the other hand, it follows easily from \eqref{st:int(z)_lem}\eqref{it:mu^zwz^-1_fmla} that the displayed isomorphism identifies the set of spin-permissible elements in $\wt W_G/W_{I,G}$ with the set of \emph{$\mu$-spin-permissible} elements in $\wt W_{B_m}/W_{m-I,B_m}$ defined in \eqref{def:spin-permissible_B_m}.  So the result follows from the equivalence of $\mu$-admissibility and $\mu$-spin-permissibility in $\wt W_{B_m}/W_{m-I,B_m}$, which we shall prove later in \eqref{st:sp-perm=adm_B}.
\end{proof}

As corollaries, we deduce Theorems \ref{st:main_thm} and \ref{st:2nd_main_thm} from the Introduction.

\begin{proof}[Proof of \eqref{st:2nd_main_thm}]
By definition, the set of spin-permissible elements in $\wt W_G/W_{I,G}$ surjects onto the set of $w \in W_{I,G} \bs \wt W_G / W_{I,G}$ such that the associated Schubert variety $S_w$ is contained in $M_{I,k}^\spin$; and the set of $\{\mu_{r,s}\}$-admissible elements in $\wt W_G/W_{I,G}$ surjects onto the set of $\{\mu_{r,s}\}$-admissible elements in $W_{I,G} \bs \wt W_G / W_{I,G}$.  So the theorem is immediate from \eqref{st:adm<=>perm_I}.
\end{proof}

\begin{proof}[Proof of \eqref{st:main_thm}]
Without loss of generality, we may assume that $\O_E$ has algebraically closed residue field.  Then, since the set of $\{\mu_{r,s}\}$-admissible elements in $\wt W_G$ surjects onto the set of $\{\mu_{r,s}\}$-admissible elements in $W_{I,G}\bs\wt W_G/W_{I,G}$, \cite{paprap09}*{Prop.\ 3.1} asserts exactly that the Schubert varieties $S_w$ indexed by $\{\mu_{r,s}\}$-admissible $w$ are contained in $M_{I,k}^\loc$.  Now use \eqref{st:2nd_main_thm}.
\end{proof}

\section{Type $D$ combinatorics}\label{s:combinatorics_D}
In this section we work through the paper's combinatorics in type $D$. Our main aim is to prove Theorem \ref{st:adm_perm_D} from the Introduction and its parahoric generalization \eqref{st:sp-perm=adm_D_parahoric}.  We shall also prove results, for cocharacters of the form $\mu$ below, on permissible sets (see \s\ref{ss:mu-perm} and \s\ref{ss:mu-perm_parahoric}) and vertexwise admissibility (see \s\ref{ss:vert-adm_D}).  For most of the section we shall work just in the Iwahori case; beginning in \s\ref{ss:parahoric_D} we shall turn to the parahoric case.  Aside from the preliminaries in \s\ref{s:prelim_comb} (which we shall always apply with $n = 2m$), this section is independent of everything else in the paper.

For convenience, we shall make some notational changes from earlier in the paper.  We fix an integer $m \geq 2$.  Except where noted to the contrary, throughout this section we denote by $\mu$ the cocharacter
\setcounter{equation}{0}
\begin{equation}\label{disp:mu_D}
   \mu := \bigl(2^{(q)},1^{(2m-2q)},0^{(q)}\bigr),\quad 0 \leq q \leq m-1.
\end{equation}
For $i \in \{1,\dotsc,2m\}$, we write $i^* := 2m + 1 - i$.  For any nonempty $E \subset \{1,\dotsc,2m\}$, we put $E^* := 2m + 1 - E$.  We denote by $e_1,\dotsc,$ $e_{2m}$ the standard basis in $\ZZ^{2m}$.  We no longer use the symbol $k$ to denote an algebraic closure of the common residue field of $F_0$ and $F$; instead we shall use $k$ to denote integers.  By a \emph{half-integer} we shall mean an element of $\frac 1 2 \ZZ \smallsetminus \ZZ$.

\subsection{The root datum \texorpdfstring{$\R_{D_m}$}{R\_$\{$D\_m$\}$}}\label{ss:type_D_variant}
In this subsection we define the root datum $\R_{D_m}$.  Let
\begin{align*}
   X_{*D_m} &:= X_{*2m}\\ 
            &\phantom{:}= \bigl\{\,(x_1,\dotsc,x_{2m}) \in \ZZ^{2m} \bigm|
                        x_1+x_{2m} = x_2+x_{2m-1} = \dots = x_m+x_{m+1}\,\bigr\},
\end{align*}
where we recall $X_{*2m}$ from \eqref{disp:X_*n},
and
\[
   X^*_{D_m} := \ZZ^{2m}\big/
       \bigl\{(x_1,\dotsc,x_{m-1},\textstyle{-\sum_{i=1}^{m-1}x_i},
       \textstyle{-\sum_{i=1}^{m-1}x_i},x_{m-1},\dotsc,x_1)\bigr\}.
\]
Then the standard dot product on $\ZZ^{2m}$ induces a perfect pairing $X^*_{D_m} \times X_{*D_m} \to \ZZ$.  For $1 \leq i,j \leq 2m$ with $j \neq i$, $i^*$, let
\begin{equation}\label{disp:alpha_i,j}
   \begin{matrix}
   \alpha_{i,j}\colon
   \xymatrix@R=0ex{
      X_{*D_m} \ar[r] & \ZZ\\
      (x_1,\dotsc,x_{2m}) \ar@{|->}[r] & x_i - x_j
   }
   \end{matrix}
\end{equation}
and
\begin{equation}\label{disp:alpha^vee_i,j}
   \alpha_{i,j}^\vee := e_i-e_j+e_{j^*}-e_{i^*} \in X_{*D_m}.
\end{equation}
Then we may regard $\alpha_{i,j}$ as an element of $X^*_{D_m}\ciso \Hom(X_{*D_m},\ZZ)$, and we let
\[
   \Phi_{D_m}:= \{\alpha_{i,j}\}_{j\neq i,i^*} \subset X^*_{D_m}
   \quad\text{and}\quad
   \Phi^\vee_{D_m} := \{\alpha^\vee_{i,j}\}_{j\neq i,i^*} \subset X_{*D_m}.
\]
The objects so defined form a root datum
\[
   \R_{D_m} := (X^*_{D_m},X_{*D_m},\Phi_{D_m},\Phi^\vee_{D_m}),
\]
which is the root datum for split $GO_{2m}$.  We take as simple roots
\begin{equation}\label{disp:simple_roots_D}
   \alpha_{1,2}, \alpha_{2,3},\dotsc,\alpha_{m-1,m},\alpha_{m-1,m+1}.
\end{equation}

The Weyl group of $\R_{D_m}$ identifies canonically with $S_{2m}^\circ$, as defined at the end of the Introduction.  As discussed in \s\ref{ss:I-W_gp}, we then have the extended affine Weyl group
\[
   \wt W_{D_m} := X_{*D_m} \rtimes S_{2m}^\circ.
\]
The affine Weyl group $W_{\aff,D_m}$ is the subgroup $Q^\vee_{D_m} \rtimes S_{2m}^\circ \subset \wt W_{D_m}$, where $Q^\vee_{D_m} \subset X_{*D_m}$ is the coroot lattice.  Explicitly,
\[
   Q^\vee_{D_m} = \biggl\{(x_1,\dotsc,x_{2m}) \in X_{*D_m} \biggm|
   \begin{varwidth}{\textwidth}
      \centering
      $x_1 + x_{2m} = \dots = x_m + x_{m+1} = 0$\\
      and $x_1 + \dotsb + x_m$ is even
   \end{varwidth}\biggr\}.
\]

Let
\[
   \A_{D_m}:= X_{*D_m} \otimes_\ZZ \RR \subset \RR^{2m}.
\]
We take as positive Weyl chamber the chamber in $\A_{D_m}$ on which the simple roots are all positive, and we say that a cocharacter is dominant if it is contained in the closure of this chamber.
We take as our base alcove
\[
   A_{D_m} := \biggl\{(x_1,\dotsc,x_{2m}) \in \RR^{2m} \biggm|
      \begin{varwidth}{\textwidth}
         \centering
         $x_1 + x_{2m} = \dotsb = x_m + x_{m+1}$ and\\
         $x_1, x_{2m}-1 < x_2 < x_3 < \dotsb < x_{m-1} < x_m, x_{m+1}$
      \end{varwidth}\biggr\};
\]
note that this is the unique alcove contained in the Weyl chamber \emph{opposite} the positive chamber and whose closure contains the origin.  The minimal facets of $A_{D_m}$ are the lines
\[
   a + \RR\cdot \bigl(1,\dotsc,1\bigr)
\]
for $a$ one of the points
\begin{equation}\label{disp:D_m_verts}
\begin{aligned}
   a_k &:= \bigl((-\tfrac 1 2)^{(k)},0^{(2m + 1 - 2k)},(\tfrac 1 2)^{(k)}\bigr)
   \quad\text{for}\quad k = 0,2,3,\dotsc,m-2,m ,\\
   a_{0'} &:= \bigl(-1,0^{{2m-1}},1\bigr),\\
   a_{m'} &:= \bigl((-\tfrac 1 2)^{(m-1)},\tfrac 1 2, -\tfrac 1 2, (\tfrac 1 2)^{(m-1)}\bigr).
\end{aligned}
\end{equation}
As discussed in \s\ref{ss:bo}, our choice of $A_{D_m}$ endows $W_{\aff,D_m}$ and $\wt W_{D_m}$ with Bruhat orders.

The extended affine Weyl group $\wt W_{D_m}$ is naturally a subgroup of $\wt W_{2m}$ \eqref{disp:wtW_n}.  Thus for each $w \in \wt W_{D_m}$ we get we get a face $\mathbf v$ of type $(2m,\{0,\dotsc,m\})$ by letting $w$ act on the standard face $\omega_{\{0,\dotsc,m\}}$, as in \s\ref{ss:faces_type_I}.  We write $\mu_k^w$ for the vector $\mu_k^{\mathbf{v}}$ attached to this face,
\begin{equation}\label{disp:mu_i^w}
   \mu_k^w := w\omega_k - \omega_k
   \quad\text{for}\quad
   k \in \ZZ.
\end{equation}
The rule $w \mapsto w\cdot\omega_{\{0,\dotsc,m\}}$ identifies $\wt W_{D_m}$ with the faces $\mathbf v$ of type $(2m,\{0,\dotsc,m\})$ such that $\mu_0^{\mathbf{v}} \equiv \mu_m^{\mathbf{v}} \bmod Q^\vee_{D_m}$.

\subsection{\texorpdfstring{$\mu$}{mu}-spin-permissibility}\label{ss:mu-spin-perm_D}
In this subsection we define \emph{$\mu$-spin-permissibility} in $\wt W_{D_m}$, for $\mu$ the cocharacter \eqref{disp:mu_D}, and we formulate its equivalence with $\mu$-admissibility.

\begin{defn}\label{def:spin-perm_D}
Let $\mu \in X_{*D_m}$ be the cocharacter \eqref{disp:mu_D}.  We say that $w \in \wt W_{D_m}$ is \emph{$\mu$-spin-permissible} if it satisfies the following conditions for all $0 \leq k \leq m$.
\begin{enumerate}
\renewcommand{\theenumi}{SP\arabic{enumi}}
\item\label{it:SP1} $\mu_k^w + (\mu_{-k}^w)^* = \mathbf 2$ and $\mathbf 0 \leq \mu_k^w \leq \mathbf{2}$.
\item\label{it:SP2} $\#\{\, j\mid \mu^w_k(j) = 2\,\} \leq q$.
\item\label{it:SP3} (spin condition) If $\mu_k^w$ is self-dual and $\mu_k^w \not\equiv \mu \bmod Q^\vee_{D_m}$, then there exist elements $j_1 \in A_k$ and $j_2 \in B_k$ such that $\mu_k^w(j_1) = \mu_k^w(j_2) = 1$, with the sets $A_k$ and $B_k$ as in \eqref{disp:A_i_B_i}.
\end{enumerate}
For fixed $k$ with $0 \leq k \leq m$, we say that $\mu_k^w$ is \emph{$\mu$-spin-permissible} if it satisfies \eqref{it:SP1}, \eqref{it:SP2}, and \eqref{it:SP3}.  We say that $w$ is \emph{naively $\mu$-permissible} if it satisfies \eqref{it:SP1} for all $0 \leq k \leq m$.
\end{defn}

Note that the condition $\mu_k^w + (\mu_{-k}^w)^* = \mathbf 2$ in \eqref{it:SP1} holds for all $k$ as soon as it holds for a single $k$, and says just that $w$ defines a $2$-face.  By \eqref{st:c_equiv_defs}, when \eqref{it:SP1} is satisfied, the quantity $\#\{\, j\mid \mu^w_k(j) = 2\,\}$ appearing in \eqref{it:SP2} is also equal to $\#\{\, j\mid \mu^w_k(j) = 0\,\}$; we shall study it in more detail in \s\ref{ss:c_i^w}.  Also note that the $k = 0$ and $k = m$ cases of \eqref{it:SP3} require respectively that $\mu_0^w \equiv \mu$ and $\mu_m^w \equiv \mu \bmod Q^\vee_{D_m}$, since $\mu_0^w$ and $\mu_m^w$ are always self-dual and $A_0 = B_m = \emptyset$.  Of course, the condition that $\mu_k^w$ be self-dual is exactly the condition that it be contained in $X_{*D_m}$.

The key result we shall prove in \s\ref{s:combinatorics_D} is Theorem \ref{st:adm_perm_D} from the Introduction, which we formulate in our present notation as follows.  Recall from \s\ref{ss:adm_perm_sets} that an element $w \in \wt W_{D_m}$ is $\mu$-admissible if $w \leq t_\lambda$ for some $\lambda \in S_{2m}^\circ \cdot \mu$.

\begin{thm}\label{st:sp-perm=adm_D}
Let $\mu$ be the cocharacter \eqref{disp:mu_D}.  Then $w \in \wt W_{D_m}$ is $\mu$-admissible $\iff$ w is
$\mu$-spin-permissible.
\end{thm}

We shall generalize the theorem to the general parahoric case in \eqref{st:sp-perm=adm_D_parahoric}.  We emphasize that the theorem does not include the case $\mu = \bigl(2^{(m)},0^{(m)}\bigr)$; we shall not consider in this paper how to characterize the admissible set for this cocharacter.  We shall prove the implication $\Longrightarrow$ in \s\ref{ss:adm=>sp-perm} and the implication $\Longleftarrow$ in \s\ref{ss:sp-perm=>adm_D}; the intervening sections shall large serve to lay the groundwork.  The implication $\Longleftarrow$ is by far the harder of the two.  The strategy we shall use to prove it is, in essence, that of Kottwitz--Rapoport \cite{kottrap00}, which is also the strategy used in \cites{sm11b,sm11c}.  To explain it in the present situation, let us first prove a baby case of the theorem.

\begin{lem}\label{st:transl_adm}
Let $w = t_\nu$ be a translation element in $\wt W_{D_m}$, and suppose that $w$ is $\mu$-spin-permissible.  Then $w$ is $\mu$-admissible.
\end{lem}

\begin{proof}
Let $\nu^+$ denote the dominant Weyl conjugate of $\nu$.  Then it is immediate from the definition of $\mu$-spin-permissibility that $\nu^+ \preccurlyeq \mu$ in the dominance order.  Hence $w$ is $\mu$-admissible.
\end{proof}

To now explain our strategy to prove the implication $\Longleftarrow$ in \eqref{st:sp-perm=adm_D}, suppose that $w \in \wt W_{D_m}$ is $\mu$-spin-permissible.  If $w$ is a translation element, then $w$ is $\mu$-admissible by the lemma.  If $w$ is not a translation element, then our task will be to find an affine root $\wt\alpha$ for $\R_{D_m}$ such that, for the associated reflection $s_{\wt\alpha} \in W_{\aff,D_m}$, $s_{\wt\alpha} w$ is again $\mu$-spin-permissible and $w < s_{\wt\alpha} w$ in the Bruhat order.  For then, repeating the argument as needed, we obtain a chain $w < s_{\wt\alpha} w < \dotsb$ of $\mu$-spin-permissible elements which must terminate in a translation element, since the set of $\mu$-spin-permissible elements is manifestly finite.

\subsection{Proper elements}\label{ss:proper_elts}
Let $w = t_{\mu_0^w}\sigma \in \wt W_{D_m}$, with $\mu_0^w \in X_{*D_m}$ and $\sigma\in S_{2m}^\circ$.

\begin{defn}
We say the element $j \in \{1,\dotsc,2m\}$ is \emph{proper} if $\sigma(j) \neq j$.
\end{defn}

Of course $j$ is proper $\iff$ $j^*$ is proper.  The formula \eqref{disp:mu_i^w} for $\mu_k^w$ gives the obvious recursion relation
\begin{equation}\label{disp:mu_recursion}
   \mu^w_k = \mu^w_{k-1} + e_k - e_{\sigma(k)},\quad
   1 \leq k \leq 2m.
\end{equation}
When $j$ is proper, it follows from this that $\mu_k^w(j)$ takes exactly two values as $k$ varies between $0$ and $2m$ and $j$ remains fixed, and that these two values differ by $1$.

\begin{defn}\label{def:upper_value}
For $j \in \{1,\dotsc,2m\}$ proper, we define the \emph{upper value}
\[
   u(j):= \max\{\mu_k^w(j)\}_{0\leq k \leq 2m}.
\]
\end{defn}

Of the course the upper value depends on $w$ as well as $j$, but to avoid clutter, we do not embed $w$ in the notation.  If $j$ is proper and $w$ defines a $d$-face, then it is clear from the formula \eqref{disp:mu_dlty_cond} that $u(j) + u(j^*) = d + 1$.  In particular, when $w$ satisfies \eqref{it:SP1}, one of every pair $j$, $j^*$ of proper elements has upper value $2$, and the other has upper value $1$.

\subsection{The integer \texorpdfstring{$c^w_k$}{c\_k\^{}w}}\label{ss:c_i^w}
We continue with $w = t_{\mu_0^w} \sigma \in \wt W_{D_m}$.  In this subsection we study the quantity $\#\{\, j \mid \mu_k^w(j) = 2 \,\}$ appearing in condition \eqref{it:SP2}.  In light of \eqref{st:c_equiv_defs}, we make the following definition.

\begin{defn}\label{def:c_i^w}
Let $0 \leq k \leq m$, and suppose that $\mu_k^w$ satisfies \eqref{it:SP1}.  We define $c^w_k$ to be the common integer
\[
   c^w_k := \#\bigl\{\,j \bigm| \mu^w_k(j) = 2\,\bigr\} = \#\bigl\{\,j \bigm| \mu^w_k(j) = 0\,\bigr\}.
\]
\end{defn}
Thus we write $c_k^w \leq q$ for condition \eqref{it:SP2} in the definition of $\mu$-spin-permissible; it will be useful to have both interpretations of $c_k^w$ in the display.  Another use for $c_k^w$ is the following, which is relevant for the spin condition.

\begin{lem}\label{st:c_i_equiv_q}
Suppose that $\mu_k^w$ satisfies \eqref{it:SP1} and is self-dual.  Then $\mu^w_k \equiv \mu \bmod Q^\vee_{D_m}$ $\iff$ $c^w_k \equiv q \bmod 2$.
\end{lem}

\begin{proof}
The map $X_{*D_m} \to \ZZ \oplus \ZZ/2\ZZ$ sending
\[
   (x_1,\dotsc,x_{2m}) \mapsto (d, x_1 + x_2 + \dotsb + x_m \bmod 2),
\]
where $d$ denotes the common integer $x_1 + x_{2m} = \dotsb = x_m + x_{m+1}$, induces an isomorphism $X_{*D_m}/Q^\vee_{D_m} \isoarrow \ZZ \oplus \ZZ/2\ZZ$.  This renders the lemma transparent.
\end{proof}

% \begin{rk}
% Note that the lemma makes crucial use of the fact that $q < m$, and indeed becomes false when $q = m$: for example, the vectors $\bigl(2^{(m)},0^{(m)})$... wait what i'm saying is wrong!!!
% \end{rk}

We conclude the subsection with a couple of simple lemmas.

\begin{lem}\label{st:c_i_change_lem}
Let $1 \leq k \leq m$, and suppose that $\mu_k^w$ and $\mu_{k-1}^w$ satisfy \eqref{it:SP1}.  Then $c_k^w = c_{k-1}^w + u(k) - u\bigl(\sigma(k)\bigr)$.
\end{lem}

\begin{proof}
This is easily checked using the recursion relation \eqref{disp:mu_recursion}.
\end{proof}

Our second lemma will prove to be a very useful tool later on in the proof of the key proposition \eqref{st:big_D_prop}.

\begin{lem}\label{st:c_i_neq_c_p}
Let $0 \leq i < p \leq m$, and suppose that $\mu_i^w$ and $\mu_p^w$ satisfy \eqref{it:SP1}.
\begin{enumerate}
\renewcommand{\theenumi}{\roman{enumi}}
\item\label{it:c_i<c_p}
   If $c_i^w < c_p^w$, then there exists $j \in \{i+1,\dotsc, p\}$ such that $\mu^w_i(j) = 1$ and $\mu^w_p(j) = 2$.
\item\label{it:c_i>c_p}
   If $c_i^w > c_p^w$, then there exists $j \in \{i+1,\dotsc, p\}$ such that $\mu_i^w(j) = 0$ and $\mu_p^w(j) = 1$.
\end{enumerate}
\end{lem}

\begin{proof}
The hypothesis in \eqref{it:c_i<c_p} implies that there exists a $j$ such that $\mu_p^w(j) = 2$ and $\mu_i^w(j) = 1$.  The recursion relation \eqref{disp:mu_recursion} then forces $j \in \{i+1,\dotsc,p\}$.  Assertion \eqref{it:c_i>c_p} is proved in a similar way, using that $c_i^w$ also equals the number of entries of $\mu_i^w$ equal to $0$.
\end{proof}

\subsection{The vector \texorpdfstring{$\nu^w_k$}{nu\_k\^{}w}}\label{ss:nu_i^w}
We continue with our element $w = t_{\mu_0^w}\sigma \in \wt W_{D_m}$.  In this subsection we introduce the vector $\nu_k^w$, which for many purposes is better to work with than $\mu_k^w$.

In addition to the vertices \eqref{disp:D_m_verts}, let us define the points in $\A_{D_m}$
\begin{equation}\label{disp:a_k}
   a_1 := \bigl(-\tfrac 1 2,0^{(2m-2)},\tfrac 1 2\bigr)
   \quad\text{and}\quad
   a_{m-1} := \bigl((-\tfrac 1 2)^{(m-1)},0,0,(\tfrac 1 2)^{(m-1)}\bigr).
\end{equation}
Then $a_1$, $a_{m-1} \in \ol{A_{D_m}}$, but neither is a vertex.  Instead the points $a_0$, $a_1,\dotsc$, $a_m$ are vertices for an alcove for the symplectic group.  Since this symplectic alcove is contained in $A_{D_m}$, we shall find the $a_k$'s quite suitable for our purposes.

\begin{defn}\label{def:nu_i^w}
For $0 \leq k \leq m$, we define
\[
   \nu^w_k : = wa_k - a_k = \frac{\mu_k^w + \mu_{-k}^w} 2 = \frac{\mu^w_k +\mu^w_{2m-k}} 2.
\]
\end{defn}

Note that $\nu_k^w = \mu_k^w$ whenever $\mu_k^w$ is self-dual; in particular $\nu_0^w = \mu_0^w$ and $\nu_m^w = \mu_m^w$.  If $w$ defines a $d$-face, then \eqref{disp:mu_dlty_cond} gives
\begin{equation}\label{disp:random_nu_i_fmla}
   \nu_k^w = \frac{\mu_k^w - (\mu_k^w)^* + \mathbf d} 2.
\end{equation}
Since $\mu_k^w$ has only integer entries, the basic inequalities \eqref{st:basic_ineqs} then allow us to recover $\mu_k^w$ uniquely from $\nu_k^w$.  Using the recursion relation \eqref{disp:mu_recursion} for $\mu_k^w$, we get the recursion relation for $\nu_k^w$, valid for $1 \leq k \leq m$,
\begin{equation}\label{disp:nu_recursion}
\begin{aligned}
   \nu_k^w &= \frac{\mu^w_{k-1} + e_k - e_{\sigma(k)}
                    + e_{\sigma(k^*)} - e_{k^*} + \mu^w_{k^*}}2\\
           &= \nu^w_{k-1} +
   \begin{cases}
      0, & \sigma(k) = k\\
      e_k - e_{k^*}, & \sigma(k) = k^*\\
      \frac 1 2 \alpha_{k,\sigma(k)}^\vee, & \sigma(k) \neq k, k^*.
%      -\frac 1 2 \alpha_{\sigma(k),k}^\vee, & k > \sigma(k) \neq k^*.
   \end{cases}
\end{aligned}
\end{equation}
It is clear from the definition of $\nu_k^w$ that
\[
   \nu_k^w(j) \in \bigl\{u(j) - 1, u(j) - \tfrac 1 2, u(j)\bigr\}
   \quad\text{for all}\quad
   0 \leq k \leq m \text{ and } 1 \leq j \leq 2m.
\]

We shall devote the rest of the subsection to expressing conditions \eqref{it:SP1}--\eqref{it:SP3} in terms of the vector $\nu_k^w$.

\begin{lem}\label{it:self_dual_conds}
Let $0 \leq k \leq m$.  The following are equivalent.
\begin{enumerate}
\renewcommand{\theenumi}{\roman{enumi}}
\item\label{it:mu_self-dual}
   $\mu_k^w$ is self-dual.
\item\label{it:mu_i=nu_i}
   $\nu_k^w = \mu_k^w$.
\item\label{it:nu_in_ZZ^2m}
   $\nu_k^w \in \ZZ^{2m}$ (as opposed to $\frac 1 2 \ZZ^{2m}$).
% \item\label{it:nu(j)=0,2_A}
%    $\nu_k^w(j) \in \{0,2\}$ for all $j \in A_k$.
% \item\label{it:nu(j)=0,2_B}
%    $\nu_k^w(j) \in \{0,2\}$ for all $j \in B_k$.
% \item
%    The set $A_i$ is $\sigma$-stable.
% \item
%    The set $B_i$ is $\sigma$-stable.
\end{enumerate}
\end{lem}

\begin{proof}
The implications \eqref{it:mu_self-dual} $\implies$ \eqref{it:mu_i=nu_i} and \eqref{it:mu_i=nu_i} $\implies$ \eqref{it:nu_in_ZZ^2m} are trivial.  Now assume \eqref{it:nu_in_ZZ^2m}.  To deduce \eqref{it:mu_self-dual}, we must show that $\mu_k^w = \mu_{-k}^w$, or equivalently, if $w$ defines a $d$-face, that $\mu_k^w = \mathbf d - (\mu_k^w)^*$.  This is a straightforward consequence of the formula \eqref{disp:random_nu_i_fmla} and the basic inequalities \eqref{st:basic_ineqs}.
\end{proof}

Note that if $\nu_k^w \in \ZZ^{2m}$, then $\nu_k^w = \mu_k^w \in X_{*D_m}$.

\begin{lem}\label{st:c_i_nu_interp}
Let $0 \leq k \leq m$, and suppose that $\mu_k^w$ satisfies \eqref{it:SP1}.  Then
\begin{align*}
   c_k^w &= \#\bigl\{\,j \bigm| \nu^w_k(j) = 2\,\bigr\} + \frac{\#\bigl\{\,j \bigm| \nu^w_k(j) \notin \ZZ\,\bigr\}}4\\
         &= \#\bigl\{\,j \bigm| \nu^w_k(j) = 0\,\bigr\} + \frac{\#\bigl\{\,j \bigm| \nu^w_k(j) \notin \ZZ\,\bigr\}}4.
\end{align*}
\end{lem}

\begin{proof}
Since $\mu_k^w$ satisfies \eqref{it:SP1}, we have $\nu_k^w + (\nu_k^w)^* = \mathbf 2$.  This implies the second equality.  To prove the first, consider the sets defined in \eqref{disp:EFGH}, where we replace $\mu_i^{\mathbf v}$ with $\mu_k^w$.  We have
\[
   \bigl\{\,j \bigm| \nu^w_k(j) = 2\,\bigr\} = G
   \quad\text{and}\quad
   \bigl\{\,j \bigm| \nu^w_k(j) \notin \ZZ\,\bigr\} = E \amalg E^* \amalg F \amalg F^*.
\]
By \eqref{st:cardE=cardG} $\#E = \#F$, and the lemma follows.
\end{proof}

We now introduce the following conditions on $\nu_k^w$ for $0 \leq k \leq m$.

\begin{enumerate}
\renewcommand{\theenumi}{SP\arabic{enumi}$'$}
\item\label{it:SP1'}
   $\nu_k^w + (\nu_k^w)^* = \mathbf 2$ and $\mathbf 0 \leq \nu_k^w \leq \mathbf{2}$.
\item\label{it:SP2'}
   $\#\{\, j\mid \nu^w_k(j) = 2\,\} + \#\{\, j \mid \nu^w_k(j) \notin \ZZ\, \}/4 \leq q$.
\item\label{it:SP3'}
   (spin condition) If $\nu_k^w \in \ZZ^{2m}$ and $\nu_k^w \not\equiv \mu \bmod Q^\vee_{D_m}$, then there exist elements $j_1 \in A_k$ and $j_2 \in B_k$ such that $\nu_k^w(j_1) = \nu_k^w(j_2) = 1$.
\end{enumerate}

\begin{lem}\label{st:SP_SP'_equiv}
Let $0 \leq k \leq m$.
\begin{enumerate}
\renewcommand{\theenumi}{\roman{enumi}}
\item\label{it:SP1_equiv}
   $\mu_k^w$ satisfies \eqref{it:SP1} $\iff$ $\nu_k^w$ satisfies \eqref{it:SP1'}.
\item\label{it:SP2_equiv}
   Suppose that the equivalent conditions in \eqref{it:SP1_equiv} hold.  Then $\mu_k^w$ satisfies \eqref{it:SP2} $\iff$ $\nu_k^w$ satisfies \eqref{it:SP2'}.
\item\label{it:SP3_equiv}
   $\mu_k^w$ satisfies \eqref{it:SP3} $\iff$ $\nu_k^w$ satisfies \eqref{it:SP3'}.
\end{enumerate}
\end{lem}

Thus $w$ is $\mu$-spin-permissible $\iff$ \eqref{it:SP1'}--\eqref{it:SP3'} hold for all $0 \leq k \leq m$.  We shall find this %$\nu_i^w$-theoretic 
formulation of $\mu$-spin-permissibility %in terms of the $\nu_i^w$'s 
more convenient to work with than the original one \eqref{def:spin-perm_D}.  We remind the reader that when the equivalent conditions \eqref{it:SP1} and \eqref{it:SP1'} hold and $\nu_k^w \in \ZZ^{2m}$, the condition $\nu_k^w \not\equiv \mu \bmod Q^\vee_{D_m}$ is equivalent to $c_k^w \not\equiv q \bmod 2$ by \eqref{st:c_i_equiv_q}.  We shall usually work with this latter formulation of the spin condition in practice.

\begin{proof}[Proof of \eqref{st:SP_SP'_equiv}]
Assertion \eqref{it:SP2_equiv} is proved by \eqref{st:c_i_nu_interp}, and assertion \eqref{it:SP3_equiv} is proved by \eqref{it:self_dual_conds}.  For \eqref{it:SP1_equiv}, suppose that $w$ defines a $d$-face.  Then $\mu_k^w + (\mu_{-k}^w)^* = \mathbf d$, and
\[
   \nu_k^w + (\nu_k^w)^* = \frac{\mu_k^w + \mu_{-k}^w + (\mu_k^w)^* + (\mu_{-k}^w)^*}2 = \mathbf d.
\]
So the conditions $\mu_k^w + (\mu_{-k}^w)^* = \mathbf 2$ and $\nu_k^w + (\nu_k^w)^* = \mathbf 2$ are equivalent.  Assume that they hold.  If $\mathbf 0 \leq \mu_k^w \leq \mathbf 2$, then this assumption implies that $\mathbf 0 \leq \mu_{-k}^w \leq \mathbf 2$.  Hence $\mathbf 0 \leq \nu_k^w \leq \mathbf 2$.  Conversely, if $\mathbf 0 \leq \nu_k^w \leq \mathbf 2$, then writing $\nu_k^w$ as in \eqref{disp:random_nu_i_fmla} and combining with the basic inequalities \eqref{st:basic_ineqs} (applied with $d = 2$), we easily get $\mathbf 0 \leq \mu_k^w \leq \mathbf 2$.
\end{proof}

\subsection{\texorpdfstring{$\mu$}{mu}-permissibility}\label{ss:mu-perm}
We continue with our element $w = t_{\nu_0^w} \sigma \in \wt W_{D_m}$.  In this subsection we shall characterize the \emph{$\mu$-permissible set} of Kottwitz--Rapoport (see \s\ref{ss:adm_perm_sets}) inside $\wt W_{D_m}$ in terms of a subset of the conditions that define the $\mu$-spin-permissible set.  We begin with a lemma which gives a fairly direct geometric interpretation of conditions \eqref{it:SP1'} and \eqref{it:SP2'}.  Let $\Conv(S_{2m}^\circ \mu)$ denote the convex hull in $\A_{D_m}$ of the Weyl group orbit $S_{2m}^\circ \mu$.

\begin{lem}\label{st:SP1-2'_conv_cond}
Let $0 \leq k \leq m$.  Then $\nu_k^w$ satisfies \eqref{it:SP1'} and \eqref{it:SP2'} $\iff$ $\nu_k^w \in \Conv(S_{2m}^\circ\mu)$.
\end{lem}

\begin{proof}
This is proved in the proof of \cite{sm11d}*{Prop.\ 6.6.2}, but for convenience we shall make the argument explicit here.  Recall the group $S_{2m}^*$ from the Introduction.  Since $q < m$, we have $S_{2m}^\circ \mu = S_{2m}^*\mu$.  For $1 \leq i \leq m$, let $\lambda_i = \bigl(1^{(i)},0^{(2m-i)}\bigr)$.  Then the convex hull $\Conv(S_{2m}^\circ \mu) = \Conv(S_{2m}^* \mu)$ admits the description
\begin{equation}\label{disp:Conv_explicit}
   \Conv(S_{2m}^\circ \mu) =
   \biggl\{v \in \A_{D_m} \biggm|
   \begin{varwidth}{\textwidth}
      \centering
      $v + v^* = \mathbf 2$ and $\lambda \cdot v \leq \lambda_i \cdot \mu$\\
      for all $1 \leq i \leq m$ and all $\lambda \in S_{2m}^*\lambda_i$
   \end{varwidth}\biggr\},
\end{equation}
where we use the standard dot product on $\RR^{2m}$; see for example \cite{sm11d}*{Lem.\ 6.6.1}.

Now suppose that $\nu_k^w$ satisfies \eqref{it:SP1'} and \eqref{it:SP2'}.  Then, using \eqref{st:SP_SP'_equiv}, $\mu_k^w$ and $\mu_{-k}^w$ each have all entries contained in $\{0,1,2\}$ and at most $q$ entries equal to $2$.  Hence
\[
   \lambda \cdot \mu_k^w,\ \lambda \cdot \mu_{-k}^w \leq \lambda_i \cdot \mu
   \quad\text{for all}\quad
   1 \leq i \leq m\ \text{and}\ \lambda \in S_{2m}^*\lambda_i.
\]
Hence
\[
   \lambda \cdot \nu_k^w = \lambda \cdot \biggl(\frac{\mu_k^w + \mu_{-k}^w} 2\biggr) \leq \lambda_i \cdot \mu
   \quad\text{for all}\quad
   1 \leq i \leq m\ \text{and}\ \lambda \in S_{2m}^* \lambda_i.  
\]
Hence $\nu_k^w \in \Conv(S_{2m}^\circ \mu)$.

Conversely, suppose that $\nu_k^w \in \Conv(S_{2m}^\circ\mu)$.  Then $\nu_k^w$ clearly satisfies \eqref{it:SP1'}.  Next recall the sets $E$, $F$, $G$, and $H$ from \eqref{disp:EFGH}, where $\mu_k^w$ replaces $\mu_i^{\mathbf v}$, and again denote their respective cardinalities by $e$, $f$, $g$, and $h$.  By \eqref{st:cardE=cardG} $e = f$, and by definition $c_k^w = e + g = f + g$.  Let $l := e + f + g$.  If $l < q$, then trivially $c_k^w < q$.  If $l \geq q$, then
\begin{align*}
   \sum_{j \in E \amalg F^* \amalg G}\nu_k^w(j) 
      = \frac 3 2 e + \frac 3 2 f + 2 g &= 3e + 2g\\
      &\leq \lambda_{l} \cdot \mu = 2q + (l - q) = q + 2e + g,
\end{align*}
where the inequality in the middle holds because $\nu_k^w \in \Conv(S_{2m}^\circ \mu)$.  Hence
\[
   c_k^w = e + g \leq q,
\]
as desired.
\end{proof}

% Let us now recall the general definition of $\mu$-permissibility in extended affine Weyl groups from \cite{kottrap00}*{p.\ 404}.  For the moment, let $\mu$ denote an arbitrary cocharacter in a root datum \R, and let $W$, $W_\aff$, and $\wt W$ respectively denote the Weyl group, affine Weyl group, and extended affine Weyl group.  Let \A denote the apartment for \R, and fix an alcove $\mathbf a$ in \A.  Then an element $x \in \wt W$ is $\mu$-permissible if $x \equiv t_{\mu} \bmod W_\aff$ and $xv - v \in \Conv(W\mu)$ for all $v \in \mathbf a$, where $\Conv(W\mu)$ is the convex hull in \A of the Weyl orbit $W\mu$.  Equivalently, $x$ is $\mu$-permissible if $x \equiv t_{\mu} \bmod W_\aff$ and for all subfacets $\mathbf f$ of $\mathbf a$ of minimal dimension, $xv - v \in \Conv(W\mu)$ for one, hence any, $v \in \mathbf f$.

Now recall the vertices $a_{0'}$ and $a_{m'}$ from \eqref{disp:D_m_verts}, and in analogy with \eqref{def:nu_i^w}, define
\[
   \nu_{0'}^w := wa_{0'} - a_{0'}
   \quad\text{and}\quad
   \nu_{m'}^w := wa_{m'} - a_{m'}.
\]

\begin{lem}\label{st:sickofnaminglemmas}\hfill
\begin{enumerate}
\renewcommand{\theenumi}{\roman{enumi}}
\item\label{it:0'}
   $\nu_0^w$ and $\nu_1^w$ are $\mu$-spin-permissible $\iff$ $w \equiv t_{\mu} \bmod W_{\aff,D_m}$ and $\nu_0^w$, $\nu_{0'}^w \in \Conv(S_{2m}^\circ \mu)$.
\item\label{it:m'}
   $\nu_{m-1}^w$ and $\nu_m^w$ are $\mu$-spin-permissible $\iff$ $w \equiv t_{\mu} \bmod W_{\aff,D_m}$ and $\nu_m^w$, $\nu_{m'}^w \in \Conv(S_{2m}^\circ \mu)$.
\end{enumerate}
\end{lem}

\begin{proof}
We shall just prove \eqref{it:0'}; \eqref{it:m'} is proved in a completely analogous way, with $\nu_{m-1}^w$ in the role of $\nu_1^w$, $\nu_m^w$ in the role of $\nu_0^w$, and $\nu_{m'}^w$ in the role of $\nu_{0'}^w$.  Let
\[
  \varepsilon := e_1 - e_{\sigma(1)} + e_{\sigma(2m)} - e_{2m}.
\]
Then
\[
   \nu_{0'}^w = \nu_0^w + \varepsilon
   \quad\text{and}\quad
   \nu_1^w = \nu_0^w + \frac 1 2 \varepsilon.
\]
The congruence $w \equiv t_\mu \bmod W_{\aff, D_m}$ is equivalent to $\nu_0^w \equiv \mu \bmod Q^\vee_{D_m}$, and we shall assume that both hold throughout the proof.

Let us first prove the implication $\Longleftarrow$.  Lemma \ref{st:SP1-2'_conv_cond} gives us what we need to conclude that $\nu_0^w$ is $\mu$-spin-permissible.  Since $\nu_1^w$ is the midpoint of $\nu_0^w$ and $\nu_{0'}^w$, we have $\nu_1^w \in \Conv(S_{2m}^\circ \mu)$, and we conclude from \eqref{st:SP1-2'_conv_cond} that $\nu_1^w$ satisfies \eqref{it:SP1'} and \eqref{it:SP2'}.

To see that $\nu_1^w$ satisfies the spin condition, suppose that $\nu_1^w \in \ZZ^{2m}$.  Then, since $\nu_0^w \in \ZZ^{2m}$, we must have $\varepsilon = 0$ or $\varepsilon = 2e_1 - 2e_{2m}$.  In the former case $\nu_1^w = \nu_0^w$ certainly satisfies \eqref{it:SP3'}.  In the latter case, since $\nu_0^w$, $\nu_{0'}^w \in \Conv(S_{2m}^\circ)$, we must have $\nu_0^w(1) = 0$ and $\nu_{0'}^w(1) = 2$.  Hence $\nu_1^w(1) = 1$.  Since $q < m$, there also exists a $j$ such that $\nu_0^w(j) = 1$.  Evidently $j \neq 1$, $2m$, whence $\nu_1(j) = 1$.  Since $1 \in A_1$ and $j \in B_1$, we conclude that $\nu_1^w$ satisfies \eqref{it:SP3'}.

We now prove the implications $\Longrightarrow$.  Lemma \ref{st:SP1-2'_conv_cond} immediately gives $\nu_0^w \in \Conv(S_{2m}^\circ\mu)$.  It remains to show that $\nu_{0'}^w \in \Conv(S_{2m}^\circ\mu)$.  For this, since plainly $\nu_{0'}^w + (\nu_{0'}^w)^* = \mathbf 2$ and $\nu_{0'}^w \in \ZZ^{2m}$, it is immediate from the description of the convex hull \eqref{disp:Conv_explicit} that it suffices to show that $\mathbf 0 \leq \nu_{0'}^w \leq \mathbf 2$ and $c_{0'}^w \leq q$, where
\[
   c_{0'}^w := \bigl\{\, j \bigm| \nu_{0'}^w(j) = 2 \,\bigr\}.
\]
To do so, we shall enter into a case analysis based on the possibilities for $\sigma(1)$.

If $\sigma(1) = 1$, then $\varepsilon = 0$ and $\nu_{0'}^w = \nu_0^w \in \Conv(S_{2m}^\circ \mu)$.  If $\sigma(1) = 2m$, then $\varepsilon = 2e_1 - 2e_{2m}$.  Since $\nu_1^w = \nu_0^w + e_1 - e_{2m}$ satisfies \eqref{it:SP1'}, we must have $\nu_0^w(1) \neq 2$.  We claim that furthermore $\nu_0^w(1) \neq 1$.  For evidently $\nu_1^w \in \ZZ^{2m}$, and $c_1^w = c_0^w \pm 1 \not\equiv q \bmod 2$ by \eqref{st:c_i_equiv_q}.  If $\nu_0^w(1) = 1$, then $\nu_1^w(1) = 2$ and $\nu_1^w(2m) = 0$, so that $\nu_1^w(j) \neq 1$ for all $j \in A_1$, in violation of the spin condition.  Thus the only possibility is that $\nu_0^w(1) = 0$, and we see by inspection that $\mathbf 0 \leq \nu_{0'}^w \leq \mathbf 2$ and $c_{0'}^w = c_0^w \leq q$.

Finally suppose that $\sigma(1) \neq 1,$ $2m$.  Since $\nu_1^w$ satisfies \eqref{it:SP1'}, we have $\nu_0^w(1) \in \{0,1\}$ and $\nu_0^w\bigl(\sigma(1)\bigr) \in \{1,2\}$.  Hence $\mathbf 0 \leq \nu_{0'}^w \leq \mathbf 2$.  Moreover, by \eqref{st:c_i_change_lem} $c_1^w \in \{c_0^w-1, c_0^w, c_0^w+1\}$, and one sees easily that $c_{0'}^w = c_0^w + 2(c_1^w - c_0^w)$.  So if $c_1^w \leq c_0^w$, then certainly $c_{0'}^w \leq c_0^w \leq q$.  If $c_1^w = c_0^w + 1$, then the conditions $c_1^w \leq q$ and $c_0^w \equiv q \bmod 2$ give us $c_0^w + 1 = c_1^w < q$.  Hence $c_{0'}^w = c_0^w + 2 \leq q$, as desired.
\end{proof}

Our characterization of the $\mu$-permissible set in $\wt W_{D_m}$ is as follows.

\begin{prop}\label{st:mu-perm_D}
Let $\mu$ be the cocharacter \eqref{disp:mu_D}.  Then
$w$ is $\mu$-permissible $\iff$ conditions \eqref{it:SP1'} and \eqref{it:SP2'} hold for all $0 \leq k \leq m$, and condition \eqref{it:SP3'} holds for $k = 0$, $1$, $m-1$, $m$.
\end{prop}

\begin{proof}
Since each minimal facet of $A_{D_m}$ contains exactly one of the vertices $a_0$, $a_{0'}$, $a_2$, $a_3,\dotsc,$ $a_{m-2}$, $a_m$, $a_{m'}$, this is obvious from \eqref{st:SP1-2'_conv_cond} and \eqref{st:sickofnaminglemmas}.
\end{proof}

\begin{eg}\label{eg:not_mu-adm_D}
The proposition affords us a ready supply of elements in $\wt W_{D_m}$ that are $\mu$-permissible but not $\mu$-spin-permissible.  Once we have shown later in \eqref{st:mu-adm=>mu-spin-perm} that $\mu$-admissibility implies $\mu$-spin-permissibility, this will give us examples of elements that are $\mu$-permissible but not $\mu$-admissible.  For example, take $m = 4$ and $\mu = \bigl(2^{(3)},1,1,0^{(3)}\bigr)$, and let $w = t_{(2,1^{(6)},0)}(27)(36)$, where we use cycle notation to denote elements in the Weyl group.  Then one readily verifies, using \eqref{st:mu-perm_D}, that $w$ is $\mu$-permissible.  However
\[
   \nu_2^w = (2,2,1,1,1,1,0,0) \not\equiv \mu \bmod Q^\vee_{D_m},
\]
in violation of the spin condition.  This example generalizes in an obvious way to show that the $\mu$-admissible and $\mu$-permissible sets in $\wt W_{D_m}$ differ whenever $m \geq 4$ and $q \geq 3$ (still with $q < m$, of course).

Note that for $m \geq 4$,
\[
   \bigl(2^{(m-1)},1,1,0^{(m-1)}\bigr) = \bigl(1^{(m)},0^{(m)}\bigr)
       + \bigl(1^{(m-1)},0,1,0^{(m-1)}\bigr)
\]
is a sum of two dominant minuscule cocharacters.  Thus we get a counterexample to Rapoport's conjecture \cite{rap05}*{\s3, p.~283} that $\mu$-admissibility and $\mu$-permissibility are equivalent whenever $\mu$ is a sum of dominant minuscule cocharacters.

Two further remarks on this point are in order.  First, our counterexample aside, Rapoport's conjecture is known to hold true when $\mu$ \emph{itself} is minuscule and the root datum involves only types $A$, $B$, $C$ and $D$; see \cite{prs?}*{Prop.\ 4.4.5} or \cite{sm11c}*{Th.\ 2.2.1} for an overview.  Moreover, no counterexamples to the conjecture in other types are known when $\mu$ is minuscule.  Second, something of a replacement for Rapoport's conjecture in general is proposed in \cite{prs?}*{Conj.\ 4.5.3} in terms of the notion of \emph{vertexwise admissibility}.  See the discussion in \cite{prs?}*{\s4.5}.  We shall show in \eqref{st:vert-adm_D} that $\mu$-admissibility is equivalent to $\mu$-vertexwise admissibility for the cocharacter $\mu$ \eqref{disp:mu_D}.
\end{eg}

As a counterpoint to the example just given, we have the following result. 

\begin{prop}\label{st:exceptional_adm=perm}
Let $\mu$ be the cocharacter \eqref{disp:mu_D}, and suppose that $q \leq 2$.  Then $w$ is $\mu$-admissible $\iff$ $w$ is $\mu$-permissible.
\end{prop}

\begin{proof}
By Kottwitz--Rapoport \cite{kottrap00}*{11.2}, we need only prove the implication $\Longleftarrow$.  We shall do so by assuming the equivalence of $\mu$-admissibility and $\mu$-spin-permis\-sibil\-i\-ty asserted in \eqref{st:sp-perm=adm_D}; the proof of this equivalence will be completed in \s\ref{ss:sp-perm=>adm_D} and will make no use of the result we are proving now.  So suppose that $w$ is $\mu$-permissible.  If $q = 0$, then the $\mu$-admissible and $\mu$-permissible sets each consist of a single element, namely $t_{(1^{(2m)})}$, and the result is trivial.  If $q$ equals $1$ or $2$, then we must show that $\nu_k^w$ satisfies the spin condition for all $2 \leq k \leq m - 2$.  Fix such a $k$.  Since $\nu_k^{w}$ satisfies \eqref{it:SP1'} and \eqref{it:SP2'}, we immediately reduce to considering the case that $\nu_k^w \in \ZZ^{2m}$ and $c_k^w < q$, i.e.\ $c_k^w$ equals $0$ or $1$.  But our restriction on $k$ means that $A_k$ and $B_k$ each have cardinality at least $4$.  The constraint on $c_k^w$ then forces each to contain an element $j$ such that $\nu_k^w(j) = 1$.
\end{proof}

\begin{rk}
For $m = 2$, $3$, we conclude that $\mu$-permissibility and $\mu$-admissibility in $\wt W_{D_m}$ are equivalent.  Here the root system $\Phi_{D_m}$ is of type $A_1 \times A_1$ and $A_3$, respectively.  Thus this equivalence is a special case of the general result of Haines--Ng\^o \cite{hngo02b}*{Th.\ 1} that $\mu$-admissibility and $\mu$-permissibility are equivalent for all cocharacters $\mu$ in root data of type $A$.
\end{rk}

\subsection{Intervals}\label{ss:intervals}
We continue with our element $w = t_{\nu_0^w}\sigma \in \wt W_{D_m}$.  By the periodicity relation \eqref{disp:mu_per_cond} (applied with $n = 2m$), we may regard the vector $\mu_k^w$ as indexed by $k \in \ZZ/2m\ZZ$.

\begin{defn}
For $j \in \{1,\dotsc,2m\}$ proper, we define $K_j$ to be the set of all $k \in \ZZ/2m\ZZ$ for which $\mu_k^w(j)$ takes its lower value,%
\footnote{This definition is the analog of the set denoted $K_j$ in \cite{kottrap00}*{5.2}, but, in terms of the terminology we've introduced, Kottwitz and Rapoport define $K_j$ to be the set of all $k$ where $\mu_k^w(j)$ takes its \emph{upper} value.  This discrepancy is due to differing choices of base alcove:  our $a_k$'s are located in the closure of the Weyl chamber for $GL_{2m}$ \emph{opposite} the one containing Kottwitz and Rapoport's base alcove.  On the other hand, our definition of $K_j$ is in obvious analogy with the definition of $K_j$ in \cite{sm11b}*{\s8.3}.}
i.e.\ the value $u(j) - 1$. 
\end{defn}

For $j$ proper, the set $K_j$ is an example of an \emph{interval} in $\ZZ/2m\ZZ$, in the terminology of Kottwitz--Rapoport \cite{kottrap00}*{5.4}.  Recall from their paper that if $a$ and $b$ are distinct in $\ZZ/2m\ZZ$, then the interval $[a,b) \subset \ZZ/2m\ZZ$ is defined in the following way.  Let $\wt a$ denote any representative of $a$ in \ZZ, and let $\wt b$ denote the unique representative of $b$ in \ZZ satisfying $\wt a < \wt b < \wt a + 2m$.  Then $[a,b)$ is the image of $\{\, k\in \ZZ \mid \wt a \leq k < \wt b\,\}$ in $\ZZ/2m\ZZ$.  For $j$ proper, it is immediate from the recursion relation \eqref{disp:mu_recursion} that $K_j$ equals the interval $[\sigma^{-1}(j),j)$ in $\ZZ/2m\ZZ$; here and elsewhere we use the same symbols to denote integers and their residues mod $2m\ZZ$.  Of course $K_{j^*} = [\sigma^{-1}(j^*),j^*) = [\sigma^{-1}(j)^*,j^*)$.

Note that if $i$, $i' \in \ZZ/2m\ZZ$ are distinct from $j$, then one of the intervals $[i,j)$ and $[i',j)$ is always contained in the other.

Let us extend the definition of $\nu_k^w$ to all $k \in \ZZ/2m\ZZ$ in the obvious way, by setting $\nu_k^w = \frac 1 2 (\mu_k^w + \mu_{-k}^w)$.  Then for proper fixed $j$, the sets of $k$ for which $\nu_k^w(j)$ takes its various values are easily described in terms of intervals.  Indeed, the formula \eqref{disp:random_nu_i_fmla} makes clear that
\begin{gather*}
   \bigl\{\, k \in \ZZ/2m\ZZ \bigm| \nu_k^w(j) = u(j) - 1 \,\bigl\} = K_j \cap K_{j^*}^c = \bigl[\sigma^{-1}(j),j\bigr) \cap \bigl[j^*,\sigma^{-1}(j)^*\bigr),\\
   \bigl\{\, k \in \ZZ/2m\ZZ \bigm| \nu_k^w(j) = u(j) \,\bigl\} = K_j^c \cap K_{j^*} = \bigl[j,\sigma^{-1}(j)\bigr) \cap \bigl[\sigma^{-1}(i)^*, i^*\bigr),\\
\intertext{and}
   \begin{aligned}
   \biggl\{\, k \in \ZZ/2m\ZZ \biggm| \nu_k^w(j) = u(j) - \frac 1 2 \,\biggl\} &= (K_j \cup K_{j^*}^c) \smallsetminus (K_j \cap K_{j^*}^c)\\
   &= (K_j^c \cup K_{j^*}) \smallsetminus (K_j^c \cap K_{j^*}),
   \end{aligned}
\end{gather*}
where we use a superscript $c$ to denote set complements in $\ZZ/2m\ZZ$.  Of course, some of these sets may be empty.

In practice, rather than having to consider $\nu_k^w(j)$ for all $k \in \ZZ/2m\ZZ$, we shall wish to restrict just to $k$ in the range $0 \leq k \leq m$.  The following is an immediate consequence of our above discussion; it also follow easily from the recursion relation \eqref{disp:nu_recursion}.

\begin{lem}\label{st:nu(j)_vals}
Let $j \in \{1,\dotsc,2m\}$.  Then exactly one of the following three possibilities holds.

\begin{enumerate}
\item
   $j$ is not proper, and $\nu_k^w(j)$ is constant as $k$ varies in the range $0 \leq k \leq m$.
\item
   $\sigma(j) = j^*$, and $\nu_k^w(j)$ takes exactly the two values $u(j)$ and $u(j) - 1$ as $k$ varies in the range $0 \leq k \leq m$.  There is exactly one $k$ in the range $1 \leq k \leq m$ with the property that $\nu_k^w(j) \neq \nu_{k-1}^w(j)$, namely $k = \min\{j,j^*\}$.
\item
   $\sigma(j) \neq j$, $j^*$, and $\nu_k^w(j)$ takes the value $u(j) - \frac 1 2$ for some $k$ in the range $0 \leq k \leq m$.  There are exactly two $k$ in the range $1 \leq k \leq m$ with the property that $\nu_k^w(j) \neq \nu_{k-1}^w(j)$, namely $k = \min\{j,j^*\}$ and $k = \min\{\sigma^{-1}(j), \sigma^{-1}(j)^*\}$.  For each of these two $k$, the values $\nu_{k-1}^w(j)$ and $\nu_k^w(j)$ differ by $\pm \frac 1 2$.\qed
\end{enumerate}
\end{lem}

\subsection{Reflections}
The \emph{affine roots} in $\R_{D_m}$ are the functions
\begin{equation}\label{disp:wtalpha_i,j;d}
   \begin{matrix}
   \wt\alpha_{i,j;d}\colon
   \xymatrix@R=0ex{
      X_{*D_m} \ar[r] & \ZZ\\
      (x_1,\dotsc,x_{2m}) \ar@{|->}[r] & x_i - x_j - d
   }
   \end{matrix}
\end{equation}
for $1 \leq i,j\leq 2m$ with $j \neq i$, $i^*$, and $d \in \ZZ$.  Plainly $\wt\alpha_{i,j;d} = \wt\alpha_{j^*,i^*;d}$.  Attached to the affine root $\wt\alpha = \wt\alpha_{i,j;d}$ is its linear part $\alpha := \alpha_{i,j}$ \eqref{disp:alpha_i,j}; the coroot $\alpha^\vee := \alpha^\vee_{i,j}$ \eqref{disp:alpha^vee_i,j}; and the reflection $s_{\wt\alpha} \in W_{\aff,D_m}$ sending
\[
   v \mapsto v - \langle \wt\alpha, v\rangle\alpha^\vee,\quad v \in \A_{D_m};
\]
here of course $\langle\wt\alpha, v\rangle = \langle \alpha,v\rangle - d$, where $\langle \alpha,v \rangle$ is defined by \RR-linearly extending $\alpha$ from a function on $X_{*D_m}$ to $\A_{D_m}$.  We also write $s_{i,j;d} := s_{\wt\alpha}$, and we define $s_{i,j} := s_{i,j;0} = s_{\alpha}$.  Note that $s_{i,j}$ acts on vectors as the permutation $(ij)(i^*j^*)$ (expressed in cycle notation) on entries.

For $w \in \wt W_{D_m}$, it will be important for us to relate the vectors $\nu_k^w$ and $\nu_k^{s_{\wt\alpha}w}$.  In general, for any $v \in \A_{D_m}$, we have
\begin{equation}\label{disp:nu^sw_fmla1}
   s_{\wt\alpha}wv -v = wv -v - \langle \wt\alpha, wv \rangle \alpha^\vee. %= wv-v - \bigl(\langle \wt\alpha,wv-v\rangle - \langle\alpha,v\rangle\bigr)\alpha^\vee
\end{equation}
Hence
\begin{equation}\label{disp:ss_fmla}
\begin{aligned}
   s_{\alpha}(s_{\wt\alpha}wv -v) &= wv - v - \langle\wt\alpha, wv\rangle \alpha^\vee - \bigl( \langle \alpha,wv \rangle - \langle \alpha,v\rangle - 2\langle \wt\alpha,wv\rangle \bigr)\alpha^\vee\\
   &= wv-v + \langle\wt\alpha, v \rangle\alpha^\vee.
\end{aligned}
\end{equation}
Taking $v = a_k$, we conclude
\begin{equation}\label{disp:nu^w'_fmla}
   \nu_k^{s_{\wt\alpha}w} = s_{\alpha}\bigl(\nu_k^w + \langle\wt\alpha, a_k \rangle\alpha^\vee\bigr),\quad 0 \leq k \leq m.
\end{equation}

\subsection{Reflections and the Bruhat order}
We continue with our element $w \in \wt W_{D_m}$ from previous subsections.  Let $\wt\alpha$ be an affine root with linear part $\alpha$.  In this subsection we give a characterization of when $w < s_{\wt\alpha}w$ in the Bruhat order.  We shall simply recall the result from \cite{sm11c}, noting that that argument given there is completely general.  Recall that $A_{D_m}$ denotes our base alcove in $\A_{D_m}$.

\begin{lem}[\cite{sm11c}*{Lem.\ 5.3.2}]\label{st:bo}
Let $S$ be any subset of $\ol{A_{D_m}}$, and suppose that
\begin{itemize}
\item
   $|\langle\alpha, wv -v\rangle| < \bigl|\bigl\langle\alpha, wv -v + \langle \wt \alpha, v \rangle \alpha^\vee \bigr\rangle\bigr|$ for some $v \in S$; or
\item
   there exists $v \in S$ such that $\langle \wt \alpha, v \rangle = 0$ and $\langle \alpha, wv-v \rangle$ is positive or negative according as $\wt \alpha$ is positive or negative on $A_{D_m}$.
\end{itemize}
Then $w < s_{\wt\alpha}w$.  The converse holds if $S$ is not contained in a single proper subfacet of $A_{D_m}$.  In particular, $w < s_{\wt\alpha}w$ $\iff$ for some $0 \leq k \leq m$,
\begin{enumerate}
\item\label{it:crit_1}
   $|\langle\alpha, \nu_k^w\rangle| < \bigl|\bigl\langle\alpha, \nu_k^w + \langle \wt \alpha, a_k \rangle \alpha^\vee \bigr\rangle\bigr|$; or
\item\label{it:crit_2}
   $\langle \wt \alpha, a_k \rangle = 0$ and $\langle \alpha, \nu_k^w \rangle$ is positive or negative according as $\wt \alpha$ is positive or negative on $A_{D_m}$.\qed
\end{enumerate}
\end{lem}

\subsection{\texorpdfstring{$\mu$}{mu}-admissibility implies \texorpdfstring{$\mu$}{mu}-spin-permissibility}\label{ss:adm=>sp-perm}
It is a general result of Kottwitz--Rapoport \cite{kottrap00}*{11.2} that for any cocharacter $\mu$ in any root datum, $\mu$-admissibility implies $\mu$-permissibility for elements in the extended affine Weyl group.  In this subsection we shall show that for the cocharacter $\mu$ \eqref{disp:mu_D} in $\R_{D_m}$, we have the stronger result that $\mu$-admissibility implies $\mu$-spin-permissibility in $\wt W_{D_m}$.  This will prove the ``easy'' implication in \eqref{st:sp-perm=adm_D}.

We first prove that the set of $\mu$-spin-permissible elements is closed in the Bruhat order.

\begin{prop}\label{st:SPerm_bruhat_closed}
Let $w$, $x \in \wt W_{D_m}$.  Suppose that $w \leq x$ in the Bruhat order and that $x$ is $\mu$-spin-permissible.  Then $w$ is $\mu$-spin-permissible.
\end{prop}

\begin{proof}
Since the $\mu$-permissible set is closed in the Bruhat order by \cite{kottrap00}*{Lem.\ 11.3}, our characterization of $\mu$-permissibility \eqref{st:mu-perm_D} immediately reduces us, via an obvious induction argument, to proving the following lemma.   
%    
% By Kottwitz and Rapoport's result, $w$ is $\mu$-permissible.  So by \eqref{st:mu-perm_D}, conditions \eqref{it:SP1'} and \eqref{it:SP2'} hold for all $0 \leq k \leq m$.  We must show that $\nu_k^w$ satisfies the spin condition for all $k$.  Of course, \eqref{st:mu-perm_D} also does this for $k = 0$, $1$, $m-1$, $m$, but we shall give a uniform argument that works for all $k$.  Indeed, since the translation elements $t_{\mu'}$ for $\mu' \in S_{2m}^\circ \mu$ are evidently $\mu$-spin-permissible, the following lemma gives us what we need.
\end{proof}

\begin{lem}
Let $0 \leq k \leq m$, and let $\wt\alpha := \wt\alpha_{i,j;d}$ for $i \neq j$, $j^*$ be an affine root such that $w < s_{\wt\alpha}w$.  Suppose that $\nu_k^w$ satisfies \eqref{it:SP1'} and that $\nu_k^{s_{\wt\alpha}w}$ satisfies \eqref{it:SP1'} and the spin condition.  Then $\nu_k^w$ satisfies the spin condition.
\end{lem}

\begin{proof}
Suppose that $\nu_k^w \in \ZZ^{2m}$.  We must show that if $\nu_k^w(l) \in \{0,2\}$ for all $l \in A_k$ or for all $l \in B_k$, then $\nu_k^w \equiv \mu \bmod Q^\vee_{D_m}$.  Let $x := s_{\wt\alpha} w$.  By \cite{kottrap00}*{Lem.\ 11.4}, since $w < x$, the vector $\nu_k^w$ lies on the line segment between $\nu_k^x$ and $s_{i,j}\nu_k^x$.  Hence for any $l$,
\begin{equation}\label{disp:adm=>spin-perm_ineqs}
   \nu_k^x(l) \leq \nu_k^w(l) \leq (s_{i,j}\nu_k^x)(l)
   \quad\text{or}\quad
   \nu_k^x(l) \geq \nu_k^w(l) \geq (s_{i,j}\nu_k^x)(l).
\end{equation}

To continue with the proof, let us suppose that $\nu_k^w(l) \in \{0,2\}$ for all $l \in A_k$; the case that $B_k$ replaces $A_k$ is entirely similar.  We consider subcases based on the containment of $i$ and $j$ in $A_k$.  First suppose that $i \in A_k$.  Then $\nu_k^w(i) \in \{0,2\}$.  Since $\nu_k^x$ satisfies \eqref{it:SP1'}, the inequalities \eqref{disp:adm=>spin-perm_ineqs} then imply that $\nu_k^w(i) = \nu_k^x(i)$ or $\nu_k^w(i) = (s_{i,j}\nu_k^x)(i)$.  If $\nu_k^w(i) = \nu_k^x(i)$, then by examining the $i$th entries on both sides of the equality
\[
   \nu_k^w - \nu_k^x = \langle \wt\alpha, wa_k\rangle \alpha^\vee_{i,j},
\]
which is obtained from \eqref{disp:nu^sw_fmla1}, we see that $\langle \wt\alpha, wa_k\rangle = 0$.  This and the spin condition for $\nu_k^x$ then yield $\nu_k^w = \nu_k^x \equiv \mu \bmod Q^\vee_{D_m}$, as desired.  If $\nu_k^w(i) = (s_{i,j}\nu_k^x)(i)$, then by examining the $i$th entries on both sides of
\[
   s_{i,j}\nu_k^x - \nu_k^w = \langle \wt\alpha, a_k \rangle \alpha^\vee_{i,j},
\]
which is obtained from \eqref{disp:ss_fmla}, we see that $\langle \wt\alpha, a_k \rangle = 0$ and $\nu_k^w = s_{i,j}\nu_k^x$.  Since $\langle \wt\alpha, a_k \rangle \in \ZZ$, we must have $j \in A_k$ as well.  Thus $A_k$ is $s_{i,j}$-stable and $\nu_k^x(l) = (s_{i,j}\nu_k^w)(l) \in \{0,2\}$ for all $l \in A_k$.  Since $\nu_k^x$ satisfies the spin condition, we conclude $\nu_k^w = s_{i,j}\nu_k^x \equiv \nu_k^x \equiv \mu \bmod Q^\vee_{D_m}$, as desired.

If $j \in A_k$ then one proves the lemma in a completely similar way.

Finally suppose that $i$, $j \in B_k$.   Then $\langle \wt\alpha, wa_k\rangle \in \ZZ$, so that $\nu_k^w \equiv \nu_k^x \bmod Q^\vee_{D_m}$; and for $l \in A_k$, the inequalities \eqref{disp:adm=>spin-perm_ineqs} are all equalities, so that $\nu_k^x(l) = \nu_k^w(l) \in \{0,2\}$.  Hence $\nu_k^x \equiv \mu \bmod Q^\vee_{D_m}$ by the spin condition, and we're done.
  % 
  % Hence, which by the spin condition is $\equiv \mu$, as desired.
\end{proof}

\begin{cor}\label{st:mu-adm=>mu-spin-perm}
Let $w \in \wt W_{D_m}$.  Then $w$ is $\mu$-admissible $\implies$ $w$ is $\mu$-spin-per\-mis\-sible.
\end{cor}

\begin{proof}
Since the translation elements $t_{\lambda}$ for $\lambda \in S_{2m}^\circ \mu$ are evidently $\mu$-spin-permis\-sible, this follows from \eqref{st:SPerm_bruhat_closed}.
\end{proof}

\subsection{Lemmas on reflections and \texorpdfstring{$\mu$}{mu}-spin-permissibility}

In this subsection we collect a number of technical lemmas for use in the next subsection when we prove that $\mu$-spin-permissibility implies $\mu$-admissibility.  We continue with our element $w = t_{\nu_0^w}\sigma \in \wt W_{D_m}$.

% The big thing (I think) is that for the reflection $s_{\wt\alpha}$, one wants $wa_i-a_i + \wt\alpha(a_i)\alpha^\vee \in \Conv(W\mu)$ for all $i$.  Obviously will need to recheck this!  Yes, this is ok:  see \s\ref{ss:nu_i}.
% The point is that for testing $\mu$-permissibility or Bruhat stuff, $wa_i-a_i + \wt\alpha(a_i)\alpha^\vee$ is just as good as $s_{\wt\alpha}wa_i-a_i$.

\begin{lem}\label{st:c^w'_change_lem}
Let $\wt\alpha = \wt\alpha_{i,j;d}$ for some $j \neq i$, $i^*$ be an affine root.  Let $0 \leq k \leq m$, and suppose that $\nu_k^w$ and $\nu_k^{s_{\wt\alpha}w}$ satisfy \eqref{it:SP1'}.
\begin{enumerate}
\renewcommand{\theenumi}{\roman{enumi}}
\item\label{it:all_integers}
   If $\nu_k^w(i)$, $\nu_k^w(j)$, and $\langle \wt\alpha, a_k\rangle$ are all integers, then
   \[
      c_k^{s_{\wt\alpha}w} \in \{c_k^w-2, c_k^w, c_k^w + 2\}.
   \]
\item\label{it:some_half-integer}
   If any one of  $\nu_k^w(i)$, $\nu_k^w(j)$, or $\langle \wt\alpha, a_k\rangle$ is a half-integer, then
   \[
      c_k^{s_{\wt\alpha}w} \in \{c_k^w-1, c_k^w, c_k^w + 1\}.
   \]
\end{enumerate}
\end{lem}

\begin{proof}
This is easily verified by explicit case analysis using the formula $\nu_k^{s_{\wt\alpha}w} = s_{i,j}\bigl(\nu_k^w + \langle \wt\alpha, a_k\rangle \alpha^\vee_{i,j}\bigr)$ from \eqref{disp:nu^w'_fmla} and our various equivalent expressions for $c_k^w$ in \s\s\ref{ss:c_i^w}--\ref{ss:nu_i^w}.
\end{proof}

The following lemma will be a key tool for us in the next subsection.  It is an analog of \cite{sm11b}*{Lem.\ 8.5.2}, but rather than taking the time to build up to it systematically, as is done to some extent in \cite{sm11b}, we shall give a narrow statement and proof tailored just to what we need later on.  Recall our discussion of intervals in \s\ref{ss:intervals}.

\begin{lem}\label{st:wt_alpha_i,sigma(i)}
Suppose that $w$ is naively $\mu$-permissible, and let $i \in \{1,\dotsc,2m\}$ be such that $i < \sigma(i) < i^*$.
\begin{itemize}
\renewcommand{\theenumi}{\roman{enumi}}
\item
   Suppose $u(i) = u\bigl(\sigma(i)\bigr)$.  If $K_i \subset [\sigma(i),i)$, then let $\wt\alpha := \wt\alpha_{i,\sigma(i);0}$.  If $[\sigma(i),i) \subset K_i$, then let $\wt\alpha := \wt\alpha_{i,\sigma(i);-1}$.
\item
   Suppose $u(i) = 2$ and $u\bigl(\sigma(i)\bigr) = 1$.  If $K_i \subset [\sigma(i),i)$, then let $\wt\alpha := \wt\alpha_{i,\sigma(i);1}$.  If $[\sigma(i),i) \subset K_i$, then let $\wt\alpha := \wt\alpha_{i,\sigma(i);-1}$.
\item
   Suppose $u(i) = 1$ and $u\bigl(\sigma(i)\bigr) = 2$.  If $K_i \subset [\sigma(i),i)$, then let $\wt\alpha := \wt\alpha_{i,\sigma(i);0}$.  If $[\sigma(i),i) \subset K_i$, then let $\wt\alpha := \wt\alpha_{i,\sigma(i);-2}$.
\end{itemize}
Then in every case, $s_{\wt\alpha}w$ is naively $\mu$-permissible and $w < s_{\wt\alpha}w$ in the Bruhat order.  If $u(i) = u\bigl(\sigma(i)\bigr)$, then moreover $c_k^{s_{\wt\alpha}w} = c_k^w$ for all $k$.
\end{lem}

Note that one of the inclusions $K_i \subset [\sigma(i),i)$ or $[\sigma(i),i) \subset K_i$ always holds.  So for any naively $\mu$-permissible $w$ and element $i$ such that $i < \sigma(i) < i^*$, the lemma gives us an $\wt\alpha$ such that $s_{\wt\alpha}w$ is naively $\mu$-permissible and $w < s_{\wt\alpha}w$.  To be clear, when there is an \emph{equality} $K_i = [\sigma(i),i)$, the conclusion of the lemma holds for both choices of $\wt\alpha$.

\begin{proof}
We shall proceed by a fairly explicit case analysis, making repeated use of the formula $\nu_k^{s_{\wt\alpha}w} = s_{i,\sigma(i)}\bigl(\nu_k^w + \langle \wt\alpha, a_k \rangle \alpha^\vee_{i,\sigma(i)}\bigr)$ \eqref{disp:nu^w'_fmla}.  Let $n := \min\{\sigma(i),\sigma(i)^*\}$.  It is natural to consider $\nu_k^{s_{\wt\alpha}w}$ and $\nu_k^w$ for $k$ separately in the ranges $0 \leq k < i$, $i \leq k < n$, $n \leq k \leq m$ with $n = \sigma(i)$, and $n \leq k \leq m$ with $n = \sigma(i)^*$.  In all cases, we have
{\allowdisplaybreaks
\begin{alignat*}{2}
   \nu_k^w\bigl(\sigma(i)\bigr) &= u\bigl(\sigma(i)\bigr)
      & &\quad\text{for}\quad
      0 \leq k < i;\\
   \nu_k^w\bigl(\sigma(i)\bigr) &= u\bigl(\sigma(i)\bigr) - \frac 1 2
      &&\quad\text{for}\quad
      i \leq k < n;\\
   \nu_k^w\bigl(\sigma(i)\bigr) &= u\bigl(\sigma(i)\bigr)
      &&\quad\text{for}\quad
      \sigma(i) \leq k \leq m;\ \text{and}\\
   \nu_k^w\bigl(\sigma(i)\bigr) &= u\bigl(\sigma(i)\bigr) - 1
      &&\quad\text{for}\quad
      \sigma(i)^* \leq k \leq m.
\end{alignat*}}

First suppose that $K_i \subset [\sigma(i),i)$.  Then for $0 \leq k < i$ or $\sigma(i) \leq k \leq m$, there is no condition on $\nu_k^w(i)$, i.e.\ $\nu_k^w(i) \in \bigl\{ u(i)-1, u(i) - \frac 1 2, u(i) \bigr\}$.  For $i \leq k < n$, this assumption implies that $\mu_k^w(i) = u(i)$.  Hence $\nu_k^w(i) \in \bigl\{ u(i) - \frac 1 2, u(i) \bigr\}$.  And for $\sigma(i)^* \leq k \leq m$, 
%$n \leq k \leq m$ with $n = \sigma(i)^*$, 
this assumption implies that $\mu_k^w(i) = \mu_{2m-k}^w(i) = u(i)$.  Hence $\nu_k^w(i) = u(i)$.

If $u(i) = u\bigl(\sigma(i)\bigr)$, or if $u(i) = 1$ and $u\bigl(\sigma(i)\bigr) = 2$, then let $\wt\alpha := \wt\alpha_{i,\sigma(i);0}$.  Then for $0 \leq k < i$ or $\sigma(i) \leq k \leq m$, we have $\langle \wt\alpha, a_k\rangle = 0$ and $\nu_k^{s_{\wt\alpha}w} = s_{i,\sigma(i)} \nu_k^w$.  Hence $\nu_k^{s_{\wt\alpha}w}$ satisfies \eqref{it:SP1'} because $\nu_k^w$ does, and $c_k^{s_{\wt\alpha}w} = c_k^w$.  Moreover, by the recursion relation \eqref{disp:nu_recursion} $\nu_{i-1}^w(i) = \nu_i^w(i) - \frac 1 2 \leq u(i) - \frac 1 2$.  Since $\nu_k^w\bigl(\sigma(i)\bigr) = u\bigl(\sigma(i)\bigr)$, we conclude that $\nu_{i-1}^w(i) < \nu_{i-1}^w\bigl(\sigma(i)\bigr)$.  Hence $w < s_{\wt\alpha}w$ by taking $k = i-1$ in criterion \eqref{it:crit_2} in \eqref{st:bo}.  For $i \leq k < n$, we have $\langle \wt\alpha , a_k \rangle = - \frac 1 2$ and $\nu_k^{s_{\wt\alpha}w} = s_{i,\sigma(i)} \bigl(\nu_k^w - \frac 1 2 \alpha^\vee_{i,\sigma(i)}\bigr)$.  Since $\nu_k^w(i) \in \bigl\{ u(i) - \frac 1 2, u(i) \bigr\}$ and $\nu_k^w\bigl(\sigma(i)\bigr) = u\bigl(\sigma(i)\bigr) - \frac 1 2$, we conclude that $\nu_k^{s_{\wt\alpha}w}$ satisfies \eqref{it:SP1'}.  If moreover $u(i) = u\bigl(\sigma(i)\bigr)$, then we see by inspection that $c_k^{s_{\wt\alpha}w} = c_k^w$.  For $\sigma(i)^* \leq k \leq m$, we have $\nu_k^w(i) = u(i)$, $\nu_k^w\bigl(\sigma(i)\bigr) = u\bigl(\sigma(i)\bigr) - 1$, $\langle \wt\alpha, a_k \rangle = -1$, and $\nu_k^{s_{\wt\alpha}w} = s_{i,\sigma(i)}( \nu_k^w - \alpha^\vee_{i,\sigma(i)})$.  So $\nu_k^{s_{\wt\alpha}w}$ satisfies \eqref{it:SP1'} by inspection.  And if moreover $u(i) = u\bigl(\sigma(i)\bigr)$, then $\nu_k^{s_{\wt\alpha}w} = \nu_k^w$, so that certainly $c_k^{s_{\wt\alpha}w} = c_k^w$.

If $u(i) = 2$ and $u\bigl(\sigma(i)\bigr) = 1$, then let $\wt\alpha := \wt\alpha_{i,\sigma(i);1}$.  Then for $0 \leq k < i$ or $\sigma(i) \leq k \leq m$, we have $\langle \wt\alpha, a_k\rangle = -1$ and $\nu_k^{s_{\wt\alpha}w} = s_{i,\sigma(i)}(\nu_k^w - \alpha^\vee_{i,\sigma(i)})$. So by our assumption on $u$-values, $\nu_k^{s_{\wt\alpha}w}$ satisfies \eqref{it:SP1'}.  Moreover, $\nu_{i-1}^w\bigl(\sigma(i)\bigr) = 1$ and, by the same reasoning as in the previous paragraph, $\nu_{i-1}^w(i) \in \bigl\{1, \frac 3 2\bigr\}$.  Hence $w < s_{\wt\alpha}w$ by taking $k = i-1$ in criterion \eqref{it:crit_1} in \eqref{st:bo}.  For $i \leq k < n$, we have $\langle \wt\alpha, a_k \rangle = - \frac 3 2$ and $\nu_k^{s_{\wt\alpha}w} = s_{i,\sigma(i)}\bigl(\nu_k^w - \frac 3 2 \alpha^\vee_{i,\sigma(i)}\bigr)$.  Since $\nu_k^w(i) \in \bigl\{ \frac 3 2, 2 \bigr\}$ and $\nu_k^w\bigl(\sigma(i)\bigr) = \frac 1 2$, we conclude that $\nu_k^{s_{\wt\alpha}w}$ satisfies \eqref{it:SP1'}.  For $\sigma(i)^* \leq k \leq m$, we have $\nu_k^w(i) = 2$, $\nu_k^w\bigl(\sigma(i)\bigr) = 0$, $\langle \wt\alpha, a_k \rangle = -2$, and $\nu_k^{s_{\wt\alpha}w} = s_{i,\sigma(i)}(\nu_k^w - 2\alpha^\vee_{i,\sigma(i)})$.  Hence $\nu_k^{s_{\wt\alpha}w}$ satisfies \eqref{it:SP1'} by inspection.  This completes the proof when $K_i \subset [\sigma(i),i)$.

Now suppose that $[\sigma(i),i) \subset K_i$.  Then for $0 \leq k < i$ or $\sigma(i) \leq k \leq m$, this assumption implies that $\mu_k^w(i) = \mu_{2m-k}^w(i) = u(i) - 1$.  Hence $\nu_k^w(i) = u(i) - 1$.  For $i \leq k < n$, our assumption implies that $\mu_{2m-k}^w(i) = u(i) - 1$.  Hence $\nu_k^w(i) \in \bigl\{ u(i) - 1, u(i) - \frac 1 2 \bigr\}$.  For $\sigma(i)^* \leq k \leq m$, there is no constraint on $\nu_k^w(i)$.

If $u(i) = u\bigl(\sigma(i)\bigr)$, or if $u(i) = 2$ and $u\bigl(\sigma(i)\bigr) = 1$, then let $\wt\alpha := \wt\alpha_{i,\sigma(i);-1}$.  Then for $0 \leq k < i$ or $\sigma(i) \leq k \leq m$, we have $\nu_k^w(i) = u(i) - 1$, $\nu_k^w\bigl(\sigma(i)\bigr) = u\bigl(\sigma(i)\bigr)$, $\langle \wt\alpha, a_k\rangle = 1$, and $\nu_k^{s_{\wt\alpha}w} = s_{i,\sigma(i)}(\nu_k^w + \alpha^\vee_{i,\sigma(i)})$.  So $\nu_k^{s_{\wt\alpha}w}$ satisfies \eqref{it:SP1'} by inspection.  And if moreover $u(i) = u\bigl(\sigma(i)\bigr)$, then $\nu_k^{s_{\wt\alpha}w} = \nu_k^w$, so that certainly $c_k^{s_{\wt\alpha}w} = c_k^w$.  For $i \leq k < n$, we have $\langle \wt\alpha, a_k\rangle = \frac 1 2$ and $\nu_k^{s_{\wt\alpha}w} = s_{i,\sigma(i)}\bigl(\nu_k^w + \frac 1 2 \alpha^\vee_{i,\sigma(i)}\bigr)$.  Since $\nu_k^w(i) \in \bigl\{ u(i) - 1, u(i) - \frac 1 2\bigr\}$ and $\nu_k^w\bigl(\sigma(i)\bigr) = u\bigl(\sigma(i)\bigr) - \frac 1 2$, we conclude that $\nu_k^{s_{\wt\alpha}w}$ satisfies \eqref{it:SP1'}.  Moreover, since $\nu_{i-1}^w(i) = u(i) - 1$, we have $\nu_i^w(i) = u(i) - \frac 1 2$ by \eqref{disp:nu_recursion}.  Hence $w < s_{\wt\alpha}w$ by taking $k = i$ in criterion \eqref{it:crit_1} in \eqref{st:bo}.  If furthermore $u(i) = u\bigl(\sigma(i)\bigr)$, then we also see by inspection that $c_k^{s_{\wt\alpha}w} = c_k^w$.  For $\sigma(i)^* \leq k \leq m$, we have $\langle \wt\alpha, a_k \rangle = 0$ and $\nu_k^{s_{\wt\alpha}w} = s_{i,\sigma(i)}\nu_k^w$.  Hence $\nu_k^{s_{\wt\alpha}w}$ satisfies \eqref{it:SP1'} because $\nu_k^w$ does, and $c_k^{s_{\wt\alpha}w} = c_k^w$.

If $u(i) = 1$ and $u\bigl(\sigma(i)\bigr) = 2$, then let $\wt\alpha := \wt\alpha_{i,\sigma(i);-2}$.  Then for $0 \leq k < i$ or $\sigma(i) \leq k \leq m$, we have $\nu_k^w(i) = 0$, $\nu_k^w\bigl(\sigma(i)\bigr) = 2$, $\langle \wt\alpha, a_k \rangle = 2$, and $\nu_k^{s_{\wt\alpha}w} = s_{i,\sigma(i)}(\nu_k^w + 2 \alpha^\vee_{i,\sigma(i)})$.  Hence $\nu_k^{s_{\wt\alpha}w}$ satisfies \eqref{it:SP1'}.  For $i \leq k < n$, we have $\langle \wt\alpha, a_k\rangle = \frac 3 2$ and $\nu_k^{s_{\wt\alpha}w} = s_{i,\sigma(i)}\bigl(\nu_k^w + \frac 3 2 \alpha^\vee_{i,\sigma(i)}\bigr)$.  Since $\nu_k^w(i) \in \bigl\{ 0, \frac 1 2 \bigr\}$ and $\nu_k^w\bigl(\sigma(i)\bigr) = \frac 3 2$, we conclude that $\nu_k^{s_{\wt\alpha}w}$ satisfies \eqref{it:SP1'}.  Moreover, by the same reasoning as in the previous paragraph, $\nu_i^w(i) = \frac 1 2$.  Hence $w < s_{\wt\alpha}w$ by taking $k = i$ in criterion \eqref{it:crit_1} in \eqref{st:bo}.  Finally, for $\sigma(i)^* \leq k \leq m$ we have $\langle \wt\alpha, a_k \rangle = 1$ and $\nu_k^{s_{\wt\alpha}w} = s_{i,\sigma(i)}(\nu_k^w + \alpha^\vee_{i,\sigma(i)})$.  So $\nu_k^{s_{\wt\alpha}}$ satisfies \eqref{it:SP1'} by our assumption on $u$-values.  This exhausts all the cases we have to consider and completes the proof of the lemma.
\end{proof}

\begin{lem}\label{st:i_sigma(i)_entries_fractional}
Suppose that $w$ is naively $\mu$-permissible, and let $i$ be as in the preceding lemma.  Let $\wt\alpha$ be an affine root prescribed by the preceding lemma such that $s_{\wt\alpha}w$ is naively $\mu$-permissible and $w < s_{\wt\alpha}w$.  Suppose that $u(i) \neq u\bigl(\sigma(i)\bigr)$ and that $\nu_k^w(i)$ and $\nu_k^w\bigl(\sigma(i)\bigr)$ are half-integers for some $k$.  Then $c_k^{s_{\wt\alpha}w} = c_k^w + 1$.
\end{lem}

\begin{proof}
Our assumption $u(i) \neq u\bigl(\sigma(i)\bigr)$ requires that for one element $a \in \{i,\sigma(i)\}$ we have $\nu_i^w(a) = \frac 3 2$, and for the other element $b \in \{i,\sigma(i)\}$ we have $\nu_i^w(b) = \frac 1 2$.  Since $\nu_k\bigl(\sigma(i)\bigr)$ is a half-integer we moreover have $i \leq k < \min\{\sigma(i),\sigma(i)^*\}$.  Thus we see from the formula \eqref{disp:nu^w'_fmla} and the explicit possibilities for $\wt\alpha$ in \eqref{st:wt_alpha_i,sigma(i)} that
\[
   \nu_k^{s_{\wt\alpha}w} = s_{a,b}(\nu_k^w + \varepsilon\alpha^\vee_{a,b})
   \quad\text{for}\quad
   \varepsilon \in \biggl\{ -\frac 3 2, \frac 1 2\biggr\}.
\]
The conclusion now follows by inspection.
\end{proof}

\begin{lem}\label{st:spin_fail_domain}
Let $0 \leq k \leq m$, and suppose that $\nu_k^w$ satisfies \eqref{it:SP1'} and \eqref{it:SP3'}.  For $j \neq i,$ $i^*$, let $\wt\alpha = \wt\alpha_{i,j;d}$ be an affine root such that $\nu_k^{s_{\wt\alpha}w}$ satisfies \eqref{it:SP1'}.  Suppose that $\nu_k^{s_{\wt\alpha}w}$ fails \eqref{it:SP3'} and $\langle \wt\alpha, a_k\rangle \in \ZZ$.  Then
\[
   \nu_k^w \in \ZZ^{2m},\quad
   \langle \wt\alpha, a_k \rangle = \pm 1,\quad
   \nu_k^w(i) = \nu_k^w(j) = 1,\quad
   \text{and}\quad
   c_k^w \not\equiv q \bmod 2.
\]
\end{lem}

\begin{proof}
For $\nu_k^{s_{\wt\alpha}w}$ to fail the spin condition, we must have $\nu_k^{s_{\wt\alpha}w} \in \ZZ^{2m}$ and $c_k^{s_{\wt\alpha}w} \not\equiv q \bmod 2$.  Since $\langle \wt\alpha, a_k\rangle \in \ZZ$, the formula \eqref{disp:ss_fmla} makes clear that $\nu_k^w \in \ZZ^{2m}$, and moreover that
\[
   \nu_k^w \equiv s_{i,j}\nu_k^{s_{\wt\alpha}w} \equiv \nu_k^{s_{\wt\alpha}w} \bmod Q^\vee_{D_m}.
\]
Hence $c_k^w \not\equiv q \bmod 2$ by \eqref{st:c_i_equiv_q}.  Since $\nu_k^w$ satisfies the spin condition, there then exist $l_1 \in A_k$ and $l_2 \in B_k$ such that $\nu_k^w(l_1) = \nu_k^w(l_2) = 1$.  Since $\langle \wt\alpha , a_k\rangle \in \ZZ$, we also see that $i$ and $j$ are both in $A_k$ or both in $B_k$.  Without loss of generality, let us assume the former.  Then $\nu_k^{s_{\wt\alpha}w}(j_2) = 1$, since $\nu_k^w$ and $\nu_k^{s_{\wt\alpha}w}$ can differ only in their $i$, $i^*$, $j$, and $j^*$ entries.  Since $\nu_k^{s_{\wt\alpha}w}$ fails the spin condition, we conclude that $\nu_k^{s_{\wt\alpha}w}(l) \in \{0,2\}$ for all $l \in A_k$.  Hence $l_1 \in \{i,i^*,j,j^*\}$ and $\langle \wt\alpha, a_k \rangle \neq 0$.  Without loss of generality, let us assume $l_1 = i$.  In view of \eqref{disp:nu^w'_fmla}, the fact that $\nu_k^{s_{\wt\alpha}w}$ satisfies \eqref{it:SP1'} forces $\langle \wt\alpha, a_k\rangle = \pm 1$.  And since $\nu_k^{s_{\wt\alpha}w}(j) \in \{0,2\}$, we conclude from this that $\nu_k^w(j) = 1$ as well, which completes the proof.
\end{proof}

\begin{lem}\label{st:fail_entries}
Let $0 \leq k \leq m$, and suppose that $\nu_k^w$ is $\mu$-spin-permissible.  For $j \neq i$, $i^*$, let $\wt\alpha := \wt\alpha_{i,j;d}$ be an affine root such that $\nu_k^{s_{\wt\alpha}w}$ satisfies \eqref{it:SP1'} but not \eqref{it:SP2'} or \eqref{it:SP3'}.  Then
\[
   \nu_k^w(i), \nu_k^w(j) \in \biggl\{\frac 1 2, 1, \frac 3 2 \biggr\}.
\]
\end{lem}

\begin{proof}
We shall show that if $\nu_k^w(i) \in \{0,2\}$ or $\nu_k^w(j) \in \{0,2\}$, then $\nu_k^{s_{\wt\alpha}w}$ is $\mu$-spin-permissible.  If $\langle \wt\alpha, a_k\rangle \in \ZZ$, then the previous lemma shows that $\nu_k^w$ satisfies \eqref{it:SP3'}.  And if $\langle \wt\alpha, a_k\rangle$ is a half-integer, then \eqref{disp:nu^w'_fmla} shows that $\nu_k^{s_{\wt\alpha}w} \notin \ZZ^{2m}$ so that it trivially satisfies \eqref{it:SP3'}.  So it remains to show that $\nu_k^w$ satisfies \eqref{it:SP2'}.  But under our present hypotheses, inspection of \eqref{disp:nu^w'_fmla} yields easily that $c_k^{s_{\wt\alpha}w} \leq c_k^w$, which is $\leq q$ by assumption.
\end{proof}

\begin{lem}\label{st:useful_reflections_lem}
Assume that $\nu_k^w$ satisfies \eqref{it:SP1'} and \eqref{it:SP2'} for all $0 \leq k \leq m$.  Let $i < j < j^* < i^*$ be elements in $\{1,\dotsc,2m\}$ such that for all $i \leq k < j$,
\begin{enumerate}
\item\label{it:u_vals}
   $\nu_k^w(i) \leq 1$ and $\nu_k^w(j) \geq 1$;
\item\label{it:nu_constraints}
   $\nu_k^w(i) \geq \frac 1 2$  and $\nu_k(j) \leq \frac 3 2$; and
\item\label{it:c_constraints}
   $c_k^w < q$.
\end{enumerate}
Let $\wt\alpha := \wt\alpha_{i,j;0}$.  Then $\nu_k^{s_{\wt\alpha}w}$ satisfies \eqref{it:SP1'} and \eqref{it:SP2'} for all $0 \leq k \leq m$, and $w < s_{\wt\alpha}w$.  If moreover $\nu_k^w$ satisfies \eqref{it:SP3'} for $k$ in the range $0 \leq k < i$ or $j \leq k \leq m$, then so does $\nu_k^{s_{\wt\alpha}w}$.
% Is there something about switching the roles of $i$ and $j$ here?
\end{lem}

Note that if $i$ (resp. $j$) is proper, then the inequality $\nu_k^w(i) \leq 1$ (resp. $\nu_k^w(j) \geq 1$) in condition \eqref{it:u_vals} is implied by $u(i) = 1$ (resp. $u(j) = 2$).

\begin{proof}
We shall appeal to the formula \eqref{disp:nu^w'_fmla} for $\nu_k^{s_{\wt\alpha}w}$.  If $0 \leq k < i$ or $j \leq k \leq m$, then $\langle \wt\alpha, a_k \rangle = 0$ and $\nu_k^{s_{\wt\alpha}w} = s_{i,j}\nu_k^w$, which satisfies \eqref{it:SP1'} and \eqref{it:SP2'} because $\nu_k^w$ does.  If $\nu_k^w$ satisfies \eqref{it:SP3'}, then this also implies that $\nu_k^{s_{\wt\alpha}w}$ satisfies \eqref{it:SP3'} by \eqref{st:spin_fail_domain}.  If $i \leq k < j$, then $\langle \wt\alpha, a_k\rangle = - \frac 1 2$ and $\nu_k^{s_{\wt\alpha}w} = s_{i,j}\bigl(\nu_k^w - \frac 1 2\alpha^\vee_{i,j}\bigr)$.  So we see immediately from hypothesis \eqref{it:nu_constraints} that $\nu_k^w$ satisfies \eqref{it:SP1'}.  And since $\langle \wt\alpha, a_k \rangle$ is a half-integer, \eqref{st:c^w'_change_lem}\eqref{it:some_half-integer} and hypothesis \eqref{it:c_constraints} together imply $c_k^{s_{\wt\alpha}w} \leq q$, so that $\nu_k^{s_{\wt\alpha}w}$ satisfies \eqref{it:SP2'}.  To finish, hypothesis \eqref{it:u_vals} immediately implies that criterion \eqref{it:crit_1} in \eqref{st:bo} is satisfied for every $k$ in the range $i \leq k < j$.  So $w < s_{\wt\alpha}w$.
\end{proof}

\subsection{\texorpdfstring{$\mu$}{mu}-spin-permissibility implies \texorpdfstring{$\mu$}{mu}-admissibility}\label{ss:sp-perm=>adm_D}
As discussed at the end of \s\ref{ss:mu-spin-perm_D}, the implication $\mu$-spin-permissible $\implies$ $\mu$-admissible in \eqref{st:sp-perm=adm_D} is an immediate consequence of part \eqref{it:part_i} of the following proposition, whose proof we are now ready to give.

\begin{prop}\label{st:big_D_prop}
Suppose that $w \in \wt W_{D_m}$ is $\mu$-spin-permissible and not a translation element.
\begin{enumerate}
\renewcommand{\theenumi}{\roman{enumi}}
\item\label{it:part_i}
   There exists an affine reflection $s_{\wt\alpha}$ such that $s_{\wt\alpha}w$ is $\mu$-spin-permissible and $w < s_{\wt\alpha}w$.
\item\label{it:part_ii}
   If moreover $\nu_0^w(m) = 1$, then the reflection $s_{\wt\alpha}$ in \eqref{it:part_i} can be chosen such that $\nu_0^{s_{\wt\alpha}w}(m) = 1$.
\end{enumerate}
\end{prop}

Note that part \eqref{it:part_ii} of the proposition is not needed for the proof of \eqref{st:sp-perm=adm_D}.  We shall use it to deduce \eqref{st:sp-perm=adm_B} below, which is the type $B$ analog of \eqref{st:sp-perm=adm_D}.

\begin{proof}
Let us begin by reviewing notation.  Attached to $w$ are the vectors $\mu_k^w$ \eqref{disp:mu_i^w} for $0 \leq k \leq 2m$ and $\nu_k^w$ \eqref{def:nu_i^w} for $0 \leq k \leq m$.  We decompose $w$ as a product $t_{\nu^w_0}\sigma$, where $t_{\nu^w_0}$ is the translation part of $w$ on the left and $\sigma$ is the linear part of $w$, which we regard as a permutation of $\{1,\dotsc,2m\}$.  We recall the sets $A_k$ and $B_k$ \eqref{disp:A_i_B_i} for $0 \leq k \leq m$; the upper value $u(i)$ \eqref{def:upper_value} for $1 \leq i \leq 2m$; the integer $c_k^w$ \eqref{def:c_i^w} for $0 \leq k \leq m$; the point $a_k$ for $0 \leq k \leq m$ (\ref{disp:D_m_verts}, \ref{disp:a_k}); and the affine root $\wt\alpha_{i,j;d}$ \eqref{disp:wtalpha_i,j;d} and coroot $\alpha^\vee_{i,j}$ \eqref{disp:alpha^vee_i,j} for $1 \leq i,j \leq 2m$ and $j \neq i$, $i^*$.  Given an affine root $\wt\alpha$ with linear part $\alpha$, we shall make repeated use of the formula $\nu_k^{s_{\wt\alpha}w} = s_{\alpha}\bigl(\nu_k^w + \langle \wt\alpha, a_k\rangle \alpha^\vee\bigr)$ \eqref{disp:nu^w'_fmla}.

Since $w$ is not a translation element, there exists a proper element in $\{1,\dotsc,2m\}$.  Let $i$ denote the \emph{minimal} proper element.  Then
\[
   \nu_0^w = \dotsb = \nu_{i-1}^w \neq \nu_i^w.
\]
Throughout the proof, let
\[
   n := \min\bigl\{\sigma(i),\sigma(i)^*\bigr\}.
\]
Before continuing, we record a subsidiary lemma for later use.

\begin{lem}\label{st:c_i<q}
The upper values $u(i)$ and $u\bigl(\sigma(i)\bigr)$ are distinct $\iff$ $c_i^w \not\equiv q \bmod 2$.  In particular, if these equivalent conditions hold then $c_i^w < q$.
\end{lem}

\begin{proof}
Everything follows from the facts \eqref{st:c_i_change_lem}, $c_{i-1}^w = c_0^w \equiv q \bmod 2$ \eqref{it:SP3'}, and $c_i^w \leq q$ \eqref{it:SP2'}.
\end{proof}

We return to the proof of \eqref{st:big_D_prop}.  We shall first prove part \eqref{it:part_i}.  Of course either $\sigma(i) = i^*$ or $\sigma(i) \neq i^*$, and we shall divide the proof into these two cases.

\ssk

\noindent\emph{Case I: $\sigma(i) = i^*$.}  Then $\nu_i^w = \nu_0^w + e_i -e_{i^*} \in \ZZ^{2m}$ by \eqref{disp:nu_recursion}, and $c_i^w \not\equiv q \bmod 2$ by \eqref{st:c_i<q}.  We shall consider the subcases $u(i) = 1$ and $u(i) = 2$.  In each subcase we shall use in a crucial way that $\nu_i^w$ satisfies the spin condition.

First suppose that $u(i) = 1$.  Then $\nu_k^w(i) = 0$ for $0 \leq k < i$ and $\nu_k^w(i) = 1$ for $i \leq k \leq m$.  By the spin condition, there exists a minimal $j \in B_i$ such that $\nu_i^w(j) = 1$.  Of course $n \neq j$, $j^*$, and therefore $\nu_k(j) = \nu_k(j^*) = 1$ for all $0 \leq k \leq i$.  Take
\[
   \wt\alpha := \wt\alpha_{i,j^*;-1}.
\]
We must show that $\nu_k^{s_{\wt\alpha}w}$ is $\mu$-spin-permissible for all $0 \leq k \leq m$.  First let $0 \leq k < i$.  Then $\nu_k^w(i) = 0$, $\nu_k^w(j^*) = 1$, $\langle\wt\alpha, a_k\rangle = 1$, and $\nu_k^{s_{\wt\alpha}w} = s_{i,j^*}(\nu_k^w + \alpha^\vee_{i,j^*}) = \nu_k^w$, which is $\mu$-spin-permissible.  Next let $i \leq k < j$.  Then $\nu_k^w(i) = 1$, $\nu_k^w(j^*) \in \bigl\{\frac 1 2, 1, \frac 3 2\bigr\}$ by \eqref{st:nu(j)_vals}, $\langle\wt\alpha, a_k\rangle = \frac 1 2$, and $\nu_k^{s_{\wt\alpha}w} = s_{i,j^*}(\nu_k^w + \frac 1 2 \alpha^\vee_{i,j^*})$.  So $\nu_k^{s_{\wt\alpha}w}$ satisfies \eqref{it:SP1'} by inspection, and $c_k^{s_{\wt\alpha}w} \leq c_k^w + 1$ by \eqref{st:c^w'_change_lem}\eqref{it:some_half-integer}.  Now, since $k < j$ and $j$ was taken to be \emph{minimal,} it follows from \eqref{st:c_i_neq_c_p}\eqref{it:c_i<c_p} that $c_i^w \geq c_k^w$.  Hence
\[
   c_k^{s_{\wt\alpha}w} \leq c_k^w + 1 \leq c_i^w + 1 \leq q,
\]
where the last inequality uses \eqref{st:c_i<q}.  So $\nu_k^{s_{\wt\alpha}w}$ satisfies \eqref{it:SP2'}.  Moreover $\nu_k^{s_{\wt\alpha}w}$ trivially satisfies the spin condition since, for example, $\nu_k^{s_{\wt\alpha}w}(j^*) = \frac 3 2$ is a half-integer.  In addition, since $\nu_i^w(i) = \nu_i^w(j^*) = 1$, we see by taking $k = i$ in criterion \eqref{it:crit_1} in \eqref{st:bo} that $w < s_{\wt\alpha}w$.  Finally let $i \leq k \leq m$.  Then $\langle \wt\alpha, a_k \rangle = 0$ and $\nu_k^{s_{\wt\alpha}w} = s_{i,j^*}\nu_k^w$, which satisfies \eqref{it:SP1'} and \eqref{it:SP2'} by inspection and \eqref{it:SP3'} by \eqref{st:spin_fail_domain}.

Now suppose that $u(i) = 2$.  Then $\nu_k^w(i) = 1$ for $0 \leq k < i$ and $\nu_k^w(i) = 2$ for $i \leq k \leq m$.  By the spin condition, there exists an element $j \leq i$ such that $\nu_i^w(j) = 1$.  Evidently $j \neq i$, and therefore $\nu_k(j) = 1$ for all $k$ since $i$ is the minimal proper element.  Take
\[
   \wt\alpha := \wt\alpha_{j,i;0}.
\]
We shall show that $w < s_{\wt\alpha}w$ and that $\nu_k^{s_{\wt\alpha}w}$ satisfies \eqref{it:SP1'} and \eqref{it:SP2'} for all $0 \leq k \leq m$ by applying \eqref{st:useful_reflections_lem} (with the roles of $i$ and $j$ reversed).  Indeed, for $j \leq k < i$ we have $\nu_k^w(j) = \nu_k^w(i) = 1$, so that conditions \eqref{it:u_vals} and \eqref{it:nu_constraints} in \eqref{st:useful_reflections_lem} are satisfied.  And by \eqref{st:c_i_change_lem} $c_k^w < c_k^w + 1 = c_i^w \leq q$, so that condition \eqref{it:c_constraints} is satisfied as well.  It remains to show that $\nu_k^{s_{\wt\alpha}w}$ satisfies the spin condition for all $k$.  For $0 \leq k < j$ and $i \leq k \leq m$, this is done by \eqref{st:spin_fail_domain}.  And for $j \leq k < i$, we have $\langle \wt\alpha, a_k \rangle = - \frac 1 2$ and $\nu_k^{s_{\wt\alpha}w} = s_{i,j}\bigl(\nu_k^w - \frac 1 2 \alpha^\vee_{i,j}\bigr)$.  Hence $\nu_k^{s_{\wt\alpha}w}(j) = \frac 1 2$ is a half-integer and $\nu_k^{s_{\wt\alpha}w}$ trivially satisfies the spin condition.  

This completes the proof in Case I.  We now turn to Case II, which will be much more involved.

\ssk

\noindent\emph{Case II: $\sigma(i) \neq i^*$.}  In this case, by minimality of $i$, we have $i < \sigma(i),\sigma(i)^* < i^*$ and
\[
   \nu_0^w(i) = \dotsb = \nu_{i-1}^w(i) = u(i) - 1
   \quad\text{and}\quad
   \nu_i^w(i) = u(i) - \frac 1 2.
\]
Moreover $\nu_k^w\bigl(\sigma(i)\bigr)$ takes the constant value $u\bigl(\sigma(i)\bigr)$ equal to $1$ or $2$ for $0 \leq k < i$; the constant value $u\bigl(\sigma(i)\bigr) - \frac 1 2$ for $i \leq k < n$; and the constant value $u\bigl(\sigma(i)\bigr)$ or $u\bigl(\sigma(i)\bigr)-1$ for $n \leq k \leq m$ according as $n = \sigma(i)$ or $n = \sigma(i)^*$.

By \eqref{st:wt_alpha_i,sigma(i)}, the element $w' := s_{\wt\alpha'}w$ is at least naively $\mu$-permissible and satisfies $w < w'$ for some $\wt\alpha' := \wt\alpha_{i,\sigma(i);d}$ with $d \in \{-1,0\}$ if $u(i) = u\bigl(\sigma(i)\bigr)$; $d \in \{-1,1\}$ if $u(i) > u\bigl(\sigma(i)\bigr)$; and $d \in \{-2, 0\}$ if $u(i) < u\bigl(\sigma(i)\bigr)$.  If $w'$ is $\mu$-spin-permissible then we are done.  So let us suppose it is not.  Our problem is to find another affine root $\wt\alpha$ such that $s_{\wt\alpha}w$ is $\mu$-spin-permissible and $w < s_{\wt\alpha}w$.  Before continuing with the main proof in Case II, we shall pause to prove some subsidiary lemmas.

\begin{lem}\label{st:min_bad_p}
$\nu_k^{w'}$ is $\mu$-spin-permissible for all $0 \leq k \leq i$.  If moreover $u(i) = u\bigl(\sigma(i)\bigr)$, then $\nu_k^{w'}$ is $\mu$-spin-permissible for all $n \leq k \leq m$.
\end{lem}

\begin{proof}
First let $0 \leq k < i$.  Then $\langle \wt\alpha', a_k\rangle = -d \in \ZZ$ and $c_k^w = c_0^w \equiv q \bmod 2$.  So $\nu_k^{w'}$ satisfies the spin condition by \eqref{st:spin_fail_domain}.  So it remains to show that $\nu_k^{w'}$ satisfies \eqref{it:SP2'}, that is, that $c_k^{w'} \leq q$.  We have
\[
   \nu_k^{w'} = s_{i,\sigma(i)}(\nu_k^w - d\alpha^\vee_{i,\sigma(i)}).
\]
Since $\nu_k^w \in \ZZ^{2m}$, we have $c_k^{w'} \in \{c_k^w -2, c_k^w, c_k^w + 2\}$ by \eqref{st:c^w'_change_lem}\eqref{it:all_integers}.  Thus $c_k^{w'} \leq c_k^w \leq q$ unless $c_k^{w'} = c_k^w + 2$, which occurs when $d = \pm 1$ and $\nu_k^w(i) = \nu_k^w\bigl(\sigma(i)\bigr) = 1$.  In this case $u(i) = 2$ and $u\bigl(\sigma(i)\bigr) = 1$, and $c_k^w + 2 = c_i^w + 1 \leq q$ by \eqref{st:c_i_change_lem} and \eqref{st:c_i<q}, as desired.

The case $k = i$ requires a separate argument.  If $u(i) = u\bigl(\sigma(i)\bigr)$, then $c_i^{w'} = c_i^w \equiv q \bmod 2$ by \eqref{st:wt_alpha_i,sigma(i)} and \eqref{st:c_i<q}.  Hence $\nu_i^{w'}$ satisfies \eqref{it:SP2'} and \eqref{it:SP3'}.  If $u(i) \neq u\bigl(\sigma(i)\bigr)$, then $c_i^{w'} = c_i^w + 1$ by \eqref{st:i_sigma(i)_entries_fractional}.  It follows immediately from this and \eqref{st:c_i<q} that $c_k^{w'} \leq q$ and $c_k^{w'} \equiv q \bmod 2$, so that $\nu_i^{w'}$ satisfies \eqref{it:SP2'} and \eqref{it:SP3'}.

Finally suppose that $u(i) = u\bigl(\sigma(i)\bigr)$, and let $n \leq k \leq m$.  By \eqref{st:wt_alpha_i,sigma(i)}, $c_k^{w'} = c_k^w \leq q$, so that $\nu_k^{w'}$ satisfies \eqref{it:SP2'}.  Moreover $\langle \wt\alpha', a_k\rangle \in \ZZ$, and it follows from \eqref{st:spin_fail_domain} that if $\nu_k^{w'}$ fails \eqref{it:SP3'}, then $c_k^{w'} = c_k^w + 2$.  Since $c_k^{w'} = c_k^w$, we conclude that \eqref{it:SP3'} holds.
\end{proof}

\begin{lem}\label{st:bad_k_c_i=c_k}
Suppose that $c_i^w = c_k^w$ for some $0 \leq k \leq m$, and that $u(i) = u\bigl(\sigma(i)\bigr)$ or $k < n$.  Then $\nu_k^{w'}$ is $\mu$-spin-permissible.
\end{lem}

\begin{proof}
First suppose that $u(i) = u\bigl(\sigma(i)\bigr)$.  Then by \eqref{st:wt_alpha_i,sigma(i)} we have $c_k^{w'} = c_k^w = c_i^w$.  This is $\leq q$ because $\nu_i^w$ satisfies \eqref{it:SP2'}, and $\equiv q \bmod 2$ by \eqref{st:c_i<q}.  Thus $\nu_k^{w'}$ is $\mu$-spin-permissible.

Now suppose that $u(i) \neq u\bigl(\sigma(i)\bigr)$ and $k < n$.  For $k \leq i$ the conclusion is given by \eqref{st:min_bad_p}, so assume $i < k < n$.  Then $\langle \wt\alpha',a_k\rangle$ is a half-integer and $c_k^{w'} \in \{c_k^w-1, c_k^w, c_k^w+1\}$ by \eqref{st:c^w'_change_lem}\eqref{it:some_half-integer}.  Since $c_k^w = c_i^w < q$ by \eqref{st:c_i<q}, we conclude that $\nu_k^{w'}$ satisfies \eqref{it:SP2'}.  To see that $\nu_k^{w'}$ also satisfies the spin condition, assume that it has only integer entries.  Then, since $\langle \wt\alpha', a_k\rangle$ is a half-integer, \eqref{disp:nu^w'_fmla} requires that $\nu_k^w(i)$ and $\nu_k^w\bigl(\sigma(i)\bigr)$ are half-integers.  Hence $c_k^{w'} = c_k^w + 1$ by \eqref{st:i_sigma(i)_entries_fractional}.  But $c_k^w = c_i^w \not\equiv q \bmod 2$ by \eqref{st:c_i<q}, and therefore $c_k^{w'} \equiv q \bmod 2$.  So $\nu_k^{w'}$ satisfies \eqref{it:SP3'}.
\end{proof}

\begin{lem}
Suppose that $\nu_p^{w'}$ is not $\mu$-spin-permissible for some $0 \leq p \leq m$.
\begin{enumerate}\label{st:bad_k_c_compare_lem}
\renewcommand{\theenumi}{\roman{enumi}}
\item\label{it:c_i>c_k}
   If $c_i^w > c_p^w$, then $i < p < n$ and $\nu_p^{w'}$ satisfies \eqref{it:SP2'} but not \eqref{it:SP3'}.
\item\label{it:c_i_leq_c_k}
   If $c_i^w \leq c_p^w$, then $c_i^w < q$.
\item\label{it:k_geq_n}
   If $p \geq n$, then $\nu_p^w\bigl(\sigma(i)\bigr) = 1$, and
   \[
      u\bigl(\sigma(i)\bigr) = 1 \text{ and } n = \sigma(i)
      \quad\text{or}\quad
      u\bigl(\sigma(i)\bigr) = 2 \text{ and } n = \sigma(i)^*.
   \]
   If furthermore $c_i^w = c_p^w$, then $\nu_p^w(i) = 1$.
\end{enumerate}
\end{lem}

% \textbf{MUST CAREFULLY GO THROUGH REST OF PROOF TO MAKE SURE CORRECTED LEMMA DOESN'T IMPACT ANYTHING.}

\begin{proof}
We first prove \eqref{it:c_i>c_k}.  To see that $\nu_p^{w'}$ satisfies \eqref{it:SP2'}, note that if $u(i) = u \bigl(\sigma(i)\bigr)$, then by \eqref{st:wt_alpha_i,sigma(i)} $c_p^{w'} = c_p^w \leq q$.  And if $u(i) \neq u \bigl(\sigma(i)\bigr)$, then $c_p^w < c_i^w < q$ by \eqref{st:c_i<q}.  So by \eqref{st:c^w'_change_lem} $c_p^{w'} \leq c_p^w + 2 \leq q$.  So either way \eqref{it:SP2'} is satisfied.

It remains to show that $i < p < n$.  For this, we already know from \eqref{st:min_bad_p} that $i < p$.  So to finish the proof of \eqref{it:c_i>c_k}, it suffices to show that if $n \leq k \leq m$ and $c_i^w > c_k^w$, then $\nu_k^{w'}$ is $\mu$-spin-permissible.  If $u(i) = u\bigl(\sigma(i)\bigr)$, then this is done by \eqref{st:min_bad_p}.  So assume $u(i) \neq u\bigl(\sigma(i)\bigr)$.  We just showed in the previous paragraph that $\nu_k^{w'}$ satisfies \eqref{it:SP2'}.  To see that $\nu_k^{w'}$ also satisfies the spin condition, note that since $n \leq k$, we have $\langle \wt\alpha', a_k\rangle \in \ZZ$.  So by \eqref{st:spin_fail_domain}, we might as well assume that
\[
   \nu_k^w \in \ZZ^{2m},\quad
   \nu_k^w(i) = \nu_k^w\bigl(\sigma(i)\bigr) = 1,\quad
   \text{and}\quad
   c_k^w \not\equiv q \bmod 2.
\]
Since $\nu_k^w$ satisfies the spin condition and $i$, $i^*$, $\sigma(i)$, $\sigma(i)^* \in A_k$, there must exist an element $j \in B_k$ such that $\nu_k^w(j) = 1$.  Now, since $u(i) \neq u\bigl(\sigma(i)\bigr)$, we also have $c_i^w \not\equiv q \bmod 2$ by \eqref{st:c_i<q}.  Since $c_i^w > c_k^w$, we conclude that $c_i^w \geq c_k^w + 2$.  Let us now compare the entries of $\mu_i^w$ and $\mu_k^w$.  If $u(i) = 2$ and $u\bigl(\sigma(i)\bigr) = 1$, then
\[
   \mu_i^w(i) = 2,\quad
   \mu_i^w(i^*) = 1,\quad
   \mu_i^w\bigl(\sigma(i)\bigr) = 0,\quad
   \text{and}\quad
   \mu_i^w\bigl(\sigma(i)^*\bigr) = 1.
\]
If $u(i) = 1$ and $u\bigl(\sigma(i)\bigr) = 2$, then
\[
   \mu_i^w(i) = 1,\quad
   \mu_i^w(i^*) = 2,\quad
   \mu_i^w\bigl(\sigma(i)\bigr) = 1,\quad
   \text{and}\quad
   \mu_i^w\bigl(\sigma(i)^*\bigr) = 0.
\]
Either way, since $\mu_k^w(i) = \mu_k^w(i^*) = \mu_k^w\bigl(\sigma(i)\bigr) = \mu_k^w\bigl(\sigma(i)^*\bigr) = 1$ and $c_i^w \geq c_k^w + 2$, there must exist another element $l \notin \{i,i^*,\sigma(i),\sigma(i^*)\}$ such that $\mu_i^w(l) = 0$ and $\mu_k^w(l) = 1$.  By the recursion relation \eqref{disp:mu_recursion} $l \in \{i+1,\dotsc,k\} \subset A_k$.  And since $\nu_k^w \in \ZZ^{2m}$, we must have $\nu_k^w(l) = \mu_k^w(l) = 1$.  Since $\nu_k^{w'}(l) = \nu_k^w(l) = 1$ and $\nu_k^{w'}(j) = \nu_k^w(j) = 1$, we conclude that $\nu_k^{w'}$ satisfies \eqref{it:SP3'}, which completes \eqref{it:c_i>c_k}.

We next prove \eqref{it:c_i_leq_c_k}.  If $c_i^w < c_p^w$ then the conclusion is clear since $c_p^w \leq q$.  And if $c_i^w = c_p^w$ then $u(i) \neq u\bigl(\sigma(i)\bigr)$ by \eqref{st:bad_k_c_i=c_k}.  Hence \eqref{st:c_i<q} gives the conclusion.

We finally prove \eqref{it:k_geq_n}.  For $p \geq n$ the $\sigma(i)$-entry of $\nu_p^w$ is an integer, which, since $\nu_p^{w'}$ is not $\mu$-spin-permissible, must equal $1$ by \eqref{st:fail_entries}.  This immediately implies that $u\bigl(\sigma(i)\bigr) = 1$ and $n = \sigma(i)$, or $u\bigl(\sigma(i)\bigr) = 2$ and $n = \sigma(i)^*$.

To complete \eqref{it:k_geq_n}, note that if $\nu_p^{w'}$ fails \eqref{it:SP3'}, then \eqref{st:spin_fail_domain} gives $\nu_p^w(i) = 1$ independently of the values of $c_i^w$ and $c_p^w$.  If moreover $c_i^w = c_p^w$, then $u(i) \neq u\bigl(\sigma(i)\bigr)$ by \eqref{st:bad_k_c_i=c_k}.  Hence $c_p^w \not\equiv q \bmod 2$ by \eqref{st:c_i<q}.  If $\nu_p^{w'}$ fails \eqref{it:SP2'}, then this, \eqref{st:c^w'_change_lem}, and the inequality $c_p^w \leq q$ imply that $c_p^w = q - 1$ and $c_p^{w'} = q + 1$.  By \eqref{st:c^w'_change_lem} and \eqref{st:fail_entries} this requires $\nu_p^w(i) = 1$.
\end{proof}

Our subsidiary lemmas now dispensed with, we return to the main proof of the proposition in Case II.  Since $w'$ is naively $\mu$-permissible but not $\mu$-spin-permissible, there exists a minimal $p \in \{1,\dotsc,m\}$ such that $\nu_p^{w'}$ fails \eqref{it:SP2'} or \eqref{it:SP3'}.  By \eqref{st:min_bad_p} $p > i$.  Of course $c_i^w \leq c_p^w$ or $c_i^w > c_p^w$, and we shall consider each of these possibilities separately.

First suppose $c_i^w \leq c_p^w$.  We claim that there exists $j' \in \{i+1,\dotsc,p\}$ such that $\mu_i(j') = 1$ and $\mu_p(j') = 2$.  If $c_i^w < c_p^w$ then \eqref{st:c_i_neq_c_p}\eqref{it:c_i<c_p} furnishes exactly the claim.  If $c_i^w = c_p^w$ then note that $\mu_i^w(i) = 2$ or $\mu_i^w(i^*) = 2$ according as $u(i) = 2$ or $u(i) = 1$.  By \eqref{st:bad_k_c_i=c_k} $p \geq n$, and then by \eqref{st:bad_k_c_compare_lem}\eqref{it:k_geq_n} $\mu_p^w(i) = \mu_p^w(i^*) = 1$.  Since $c_i^w = c_p^w$, there therefore exists a $j'$ such that $\mu_i^w(j') = 1$ and $\mu_p^w(j') = 2$.  By the recursion relation \eqref{disp:mu_recursion} such a $j'$ must be contained in $\{i+1,\dotsc,p\}$, which proves the claim.

Now, $j'$ is clearly an element of the set
\begin{equation}\label{disp:S_set}
   S := \bigl\{\, j' \in \{i+1,\dotsc,p\} \bigm| \mu_i^w(j') = 1 \text{ and } \mu_{j'}^w(j') = 2\,\bigr\}.
\end{equation}
Let $j$ denote the \emph{minimal} element of $S$.  Evidently $u(j) = 2$, so that $u(j^*) = 1$.  Since $i$ is the minimal proper element and $j \neq n$, we have $\nu_k^w(j) = \nu_k^w(j^*) = 1$ for all $0 \leq k \leq i$.

We claim that $j \neq n$.  To prove this we might as well assume $n \leq j$.  Then $n \leq p$.  Then by \eqref{st:bad_k_c_compare_lem}\eqref{it:k_geq_n} $u(n) = 1$.  But $u(j) = 2$, which proves the claim.

Still supposing $c_i^w \leq c_p^w$, we are now ready to prescribe our affine root $\wt\alpha$, although to do so we will need to consider some further subcases.  First suppose $j < n$ and $u(i) = 2$.  Then we take
\[
   \wt\alpha := \wt\alpha_{i,j^*;-1}.
\]
To see that $s_{\wt\alpha}w$ is $\mu$-spin-permissible, first let $0 \leq k < i$.  Then $\nu_k^w(i) = \nu_k^w(j^*) = 1$, $\langle\wt\alpha , a_k\rangle = 1$, and $\nu_k^{s_{\wt\alpha}w} = s_{i,j^*}(\nu_k^w + \alpha^\vee_{i,j^*})$.  Thus by inspection $\nu_k^{s_{\wt\alpha}w}$ satisfies \eqref{it:SP1'} and $c_k^{s_{\wt\alpha}w} = c_k^w + 2$, and of course $\nu_k^{s_{\wt\alpha}w}$ satisfies the spin condition by \eqref{st:spin_fail_domain}.  To see that $\nu_k^{s_{\wt\alpha}w}$ also satisfies \eqref{it:SP2'}, consider the inequalities
\[
   c_k^w = c_{i-1}^w \leq c_i^w < q;
\]
here the first inequality uses that $\mu_{i-1}^w(i) = 1$ and $\mu_i^w(i) = 2$, and the second is \eqref{st:bad_k_c_compare_lem}\eqref{it:c_i_leq_c_k}.  Hence $c_k^{s_{\wt\alpha}w} = c_k^w + 2 \leq q$ by parity.  So indeed $\nu_k^{s_{\wt\alpha}w}$ satisfies \eqref{it:SP2'}.  Now let $i \leq k < j$.  Then $\langle \wt\alpha, a_k \rangle = \frac 1 2$, $\nu_k^{s_{\wt\alpha}w} = s_{i,j^*}\bigl(\nu_k^w + \frac 1 2 \alpha^\vee_{i,j^*}\bigr)$, and $c_k^{s_{\wt\alpha}w} \leq c_k^w + 1$ by \eqref{st:c^w'_change_lem}\eqref{it:some_half-integer}.  Since $\nu_p^w(i) \in \bigl\{1, \frac 3 2\bigr\}$ by \eqref{st:fail_entries}, we must have $\nu_k^w(i) \in \bigl\{1, \frac 3 2\bigr\}$ by \eqref{st:nu(j)_vals}.  And since $\nu_i^w(j^*) = 1$, $u(j^*) = 1$, and $k < j$, we must have $\nu_k^w(j^*) \in \bigl\{\frac 1 2, 1\bigr\}$ again by \eqref{st:nu(j)_vals}.  Therefore we see the following by inspection: that $\nu_k^{s_{\wt\alpha}w}$ satisfies \eqref{it:SP1'}; that $\nu_k^{s_{\wt\alpha}w}\bigl(\sigma(i)\bigr) = \nu_k^w\bigl(\sigma(i)\bigr)$ is a half-integer, so that $\nu_k^{s_{\wt\alpha}w}$ satisfies \eqref{it:SP3'}; and that
\[
   \nu_i^w(i) = \frac 3 2,\quad
   \nu_i^w(j^*) = 1,\quad
   \nu_i^{s_{\wt\alpha}w}(i) = 2,\quad\text{and\quad}
   \nu_i^{s_{\wt\alpha}w}(j^*) = \frac 1 2,
\]
so that $w < s_{\wt\alpha}w$ by criterion \eqref{it:crit_1} in \eqref{st:bo}.  We must still show that $\nu_k^{s_{\wt\alpha}w}$ satisfies \eqref{it:SP2'}.  For this, since $i \leq k < j$ and $j$ is the \emph{minimal} element  $\geq i+1$ such that $\mu_i(j) = 1$ and $\mu_j(j) = 2$, we have $c_i^w \geq c_k^w$.  Hence
\[
   c_k^{s_{\wt\alpha}w} \leq c_k^w + 1 \leq c_i^w + 1 \leq q,
\]
where the last inequality uses \eqref{st:bad_k_c_compare_lem}\eqref{it:c_i_leq_c_k}.  So indeed $\nu_k^{s_{\wt\alpha}w}$ satisfies \eqref{it:SP2'}.  Finally let $j \leq k \leq m$.  Then $\langle \wt\alpha,a_k\rangle = 0$ and $\nu_k^{s_{\wt\alpha}w} = s_{i,j^*}\nu_k^w$ is clearly $\mu$-spin-permissible, using \eqref{st:spin_fail_domain} for \eqref{it:SP3'}.  Thus $\wt\alpha$ solves our problem when $j < n$, $u(i) = 2$, and $c_i^w \leq c_p^w$.

Now suppose $j < n$ and $u(i) = 1$.  Then we take
\[
   \wt\alpha := \wt\alpha_{i,j;0}.
\]
We shall show that $w < s_{\wt\alpha}w$ and that $\nu_k^{s_{\wt\alpha}w}$ satisfies \eqref{it:SP1'} and \eqref{it:SP2'} for all $0 \leq k \leq m$ by applying \eqref{st:useful_reflections_lem}.  We must show that conditions \eqref{it:u_vals}--\eqref{it:c_constraints} are satisfied.  Condition \eqref{it:u_vals} is satisfied because $u(i) = 1$ and $u(j) = 2$.  For condition \eqref{it:nu_constraints}, note that $\nu_i^w(i) = \frac 1 2$ and $\nu_p^w(i) \in \bigl\{\frac 1 2, 1\bigr\}$ by \eqref{st:fail_entries} (applied with $\wt\alpha = \wt\alpha'$).  Therefore, by \eqref{st:nu(j)_vals}, $\nu_k^w(i) \in \bigl\{\frac 1 2, 1\bigr\}$ for all $i \leq k \leq p$, and in particular for all $i \leq k < j$.  And since $\nu_i^w(j) = 1$, we have again by \eqref{st:nu(j)_vals} that $\nu_i^w(j) \leq \frac 3 2$ for all $i \leq k < j$.  So indeed condition \eqref{it:nu_constraints} holds.  For condition \eqref{it:c_constraints}, let $i \leq k < j$.  Then $c_k^w \leq c_i^w$ by minimality of $j$, and $c_i^w < q$ by \eqref{st:bad_k_c_compare_lem}\eqref{it:c_i_leq_c_k}, which gives us exactly what we need.  So \eqref{st:useful_reflections_lem} applies.  To show that $\wt\alpha$ solves our problem, it remains to show that $\nu_k^{s_{\wt\alpha}w}$ satisfies the spin condition for all $k$.  As always, by \eqref{st:spin_fail_domain} we may restrict our attention to $i \leq k < j$.  Since by assumption $j < n$, we have that $\nu_k^{s_{\wt\alpha}w}\bigl(\sigma(i)\bigr) = \nu_k^w\bigl(\sigma(i)\bigr)$ and that $\nu_k^w\bigl(\sigma(i)\bigr)$ is a half-integer.  Hence $\nu_k^{s_{\wt\alpha}w}$ trivially satisfies \eqref{it:SP3'}.  This completes the subcase $j < n$, $u(i) = 1$, and $c_i^w \leq c_p^w$.

Still supposing $c_i^w \leq c_p^w$, now suppose $j > n$.  Then $p > n$, so that $u(i) \neq u\bigl(\sigma(i)\bigr)$ by \eqref{st:min_bad_p}.  We take
\[
   \wt\alpha := \wt\alpha_{n,j;0}.
\]
As in the subcase considered in the previous paragraph, we shall first apply \eqref{st:useful_reflections_lem} to show that $w < s_{\wt\alpha}w$ and that $\nu_k^{s_{\wt\alpha}w}$ satisfies \eqref{it:SP1'} and \eqref{it:SP2'} for all $k$.  Condition \eqref{it:u_vals}  in \eqref{st:useful_reflections_lem} is satisfied because $u(j) = 2$ and because \eqref{st:bad_k_c_compare_lem}\eqref{it:k_geq_n} implies that $u(n) = 1$.  For condition \eqref{it:nu_constraints}, the value $\nu_k^w(n)$ is constant for varying $k$ with $n \leq k \leq m$.  Hence $\nu_k^w(n) = \nu_p^w(n)$, which in turn equals $1$ by \eqref{st:bad_k_c_compare_lem}\eqref{it:k_geq_n}, for $n \leq k < j$.  And since $\nu_i^w(j) = 1$, we have by \eqref{st:nu(j)_vals} that $\nu_i^w(j) \leq \frac 3 2$ for all $i \leq k < j$, and in particular for all $n \leq k < j$.  So indeed condition \eqref{it:nu_constraints} holds.  For condition \eqref{it:c_constraints}, since $p$ is \emph{minimal} such that $\nu_p^{w'}$ is not $\mu$-spin-permissible, we have $c_k^{w'} \leq q$ for all $k < p$, and in particular for all $n \leq k < j$.  To make use of this, let us compare $c_k^w$ and $c_k^{w'}$.  Since $u(i) \neq u\bigl(\sigma(i)\bigr)$, by \eqref{st:wt_alpha_i,sigma(i)} and \eqref{st:bad_k_c_compare_lem}\eqref{it:k_geq_n} we have the possibilities $u\bigl(\sigma(i)\bigr) =1$, $\sigma(i) \leq m$, and $d \in \{-1,1\}$; or $u\bigl(\sigma(i)\bigr) = 2$, $\sigma(i) > m$, and $d \in \{-2,0\}$.  Either way $\langle \wt\alpha', a_k\rangle = \pm 1$ and $\nu_k^{w'} = s_{i,\sigma(i)}(\nu_k^w \pm \alpha^\vee_{i,\sigma(i)})$ for $n \leq k \leq m$, and in particular for $n \leq k < j$.  Now, $\nu_i(i) \in \bigl\{\frac 1 2, \frac 3 2\bigr\}$, and by \eqref{st:fail_entries} $\nu_p^w(i) \in \bigl\{\frac 1 2, 1,\frac 3 2\bigr\}$.  Hence by \eqref{st:nu(j)_vals} $\nu_k^w(i) \in \bigl\{\frac 1 2, 1,\frac 3 2\bigr\}$ for all $i \leq k \leq p$, and in particular for all $n \leq k < j$.  And we have already observed that $\nu_k^w(n) = 1$ for $n \leq k < j$.  Thus we see explicitly that $c_k^{w'} \in \{c_k^w + 1, c_k^w + 2\}$.  Hence $c_k^w < c_k^{w'} \leq q$.  So condition \eqref{it:c_constraints} holds and \eqref{st:useful_reflections_lem} applies.  To show that $\wt\alpha$ solves our problem, it remains to show that $\nu_k^{s_{\wt\alpha}w}$ satisfies the spin condition for all $n \leq k < j$.  But for such $k$ we have $\nu_k^w(n) = 1$, $\langle \wt\alpha, a_k\rangle = -\frac 1 2$, and $\nu_k^{s_{\wt\alpha}w} = s_{n,j}\bigl(\nu_k^w - \frac 1 2 \alpha^\vee_{n,j}\bigr) \not\in \ZZ^{2m}$.  So $\nu_k^{s_{\wt\alpha}w}$ trivially satisfies \eqref{it:SP3'}.  With this we have solved our problem in all subcases for which $c_i^w \leq c_p^w$.

Now suppose $c_i^w > c_p^w$.  Then by \eqref{st:bad_k_c_compare_lem}\eqref{it:c_i>c_k} $i < p < n$ and by \eqref{st:c_i_neq_c_p}\eqref{it:c_i>c_p} there exists a minimal element $j \in \{i+1,\dotsc,p\}$ such that $\mu_i(j) = 0$ and $\mu_p(j) = 1$.  Note that evidently $\nu_k^w(j) = 0$ and $\nu_k^w(j^*) = 2$ for all $0 \leq k \leq i$.  Moreover, since $\langle \wt\alpha', a_p \rangle$ is a half-integer, \eqref{st:c^w'_change_lem}\eqref{it:some_half-integer} gives
\[
   c_p^{w'} \leq c_p^w + 1 \leq c_i^w \leq q.
\]
Hence $\nu_p^w$ satisfies \eqref{it:SP2'}.  So it must fail \eqref{it:SP3'}.  Since this requires $\nu_p^{w'} = s_{i,\sigma(i)}(\nu_p^w + \langle \wt\alpha', a_p\rangle \alpha^\vee_{i,\sigma(i)})$ to have only integer entries, we conclude that $\nu_p^w(k)$ is a half-integer for $k \in \{i,i^*,\sigma(i),\sigma(i)^*\}$ and an integer for all other $k$.  In particular, since $\mu_p^w(j) = 1$, we must have $\nu_p^w(j) = \nu_p^w(j^*) = 1$.  We shall now prescribe $\wt\alpha$ separately in the subcases $u(i) = 2$ and $u(i) = 1$.

If $u(i) = 2$ then we take
\[
   \wt\alpha := \wt\alpha_{i,j^*;-1}.
\]
For $0 \leq k < i$, we have $\langle \wt\alpha,a_k\rangle = 1$, $\nu_k^w(i) = 1$, and $\nu_k^w(j^*) = 2$.  Hence $\nu_k^{s_{\wt\alpha}w} = s_{i,j^*}(\nu_k^w + \alpha^\vee_{i,j^*}) = \nu_k^w$ is $\mu$-spin-permissible.  Next let $i \leq k < j$.  Then $\langle\wt\alpha, a_k\rangle = \frac 1 2$.  Since $u(i) = 2$ we have $\nu_i(i) = \frac 3 2$.  And since $\nu_p^w(i)$ is a half-integer we must also have $\nu_p^w(i) = \frac 3 2$.  Hence by \eqref{st:nu(j)_vals} $\nu_k(i) = \frac 3 2$.  Since $\nu_i^w(j^*) = 2$, \eqref{st:nu(j)_vals} also implies $\nu_k^w(j^*) \in \bigl\{\frac 3 2, 2\bigr\}$.  Hence by inspection $\nu_k^{s_{\wt\alpha}w} = s_{i,j^*}\bigl(\nu_k^w + \frac 1 2 \alpha^\vee_{i,j^*}\bigr)$ satisfies \eqref{it:SP1'}.  And $c_k^{s_{\wt\alpha}w} = c_k^w$, so that $\nu_k^{s_{\wt\alpha}w}$ satisfies \eqref{it:SP2'}.  Moreover $\nu_k^{s_{\wt\alpha}w}$ trivially satisfies \eqref{it:SP3'} since $\nu_k^{s_{\wt\alpha}w}\bigl(\sigma(i)\bigr) = \nu_k^w\bigl(\sigma(i)\bigr)$ is a half-integer.  Finally let $j \leq k \leq m$.  Then $\langle \wt\alpha,a_k\rangle = 0$ and $\nu_k^{s_{\wt\alpha}w} = s_{i,j^*}\nu_k^w$, which is clearly $\mu$-spin-permissible, using as always \eqref{st:spin_fail_domain} to address \eqref{it:SP3'}.  Since $\nu_p^w(i) = \frac 3 2$ and $\nu_p^w(j^*) = 1$, we see moreover from criterion \eqref{it:crit_2} in \eqref{st:bo} that $w < s_{\wt\alpha}w$.  Thus $\wt\alpha$ solves our problem when $u(i) = 2$ and $c_i^w > c_p^w$.

If $u(i) = 1$ then we take
\[
   \wt\alpha := \wt\alpha_{i,j;0}.
\]
Our argument will parallel the case $u(i) = 2$ just discussed.  For $0 \leq k < i$ or $j \leq k \leq m$, we have $\langle \wt\alpha,a_k\rangle = 0$.  Hence $\nu_k^{s_{\wt\alpha}w} = s_{i,j}\nu_k^w$, which is clearly $\mu$-spin-permissible.  Now let $i \leq k < j$.  Then $\langle\wt\alpha, a_k\rangle = - \frac 1 2$.  Since $u(i) = 1$ we have $\nu_i(i) = \frac 1 2$.  And since $\nu_p^w(i)$ is a half-integer we must also have $\nu_p^w(i) = \frac 1 2$.  Hence by \eqref{st:nu(j)_vals} $\nu_k(i) = \frac 1 2$.  Since $\nu_i^w(j) = 0$, \eqref{st:nu(j)_vals} also implies $\nu_k^w(j) \in \bigl\{0, \frac 1 2\bigr\}$.  Hence by inspection $\nu_k^{s_{\wt\alpha}w} = s_{i,j}\bigl(\nu_k^w - \frac 1 2 \alpha^\vee_{i,j}\bigr)$ satisfies \eqref{it:SP1'}.  And $c_k^{s_{\wt\alpha}w} = c_k^w$, so that $\nu_k^{s_{\wt\alpha}w}$ satisfies \eqref{it:SP2'}.  Moreover $\nu_k^{s_{\wt\alpha}w}$ trivially satisfies \eqref{it:SP3'} since $\nu_k^{s_{\wt\alpha}w}\bigl(\sigma(i)\bigr) = \nu_k^w\bigl(\sigma(i)\bigr)$ is a half-integer.  Since $\nu_p^w(i) = \frac 1 2$ and $\nu_p^w(j) = 1$, we see moreover from criterion \eqref{it:crit_2} in \eqref{st:bo} that $w < s_{\wt\alpha}w$.  Thus $\wt\alpha$ solves our problem when $u(i) = 1$ and $c_i^w > c_p^w$.

This finishes Case II, and with it the proof of part \eqref{it:part_i} of the Proposition is complete.

\medskip

It remains to prove part \eqref{it:part_ii}.  So suppose $\nu_0^w(m) = 1$.  Since $\nu_0^w \equiv \nu_m^w \bmod Q^\vee_{D_m}$ for any $w \in \wt W_{D_m}$, we must have $i < m$.  If $\sigma(i) = i^*$, then a quick review of Case I above reveals that the affine root $\wt\alpha$ we prescribed there satisfies $\nu_0^{s_{\wt\alpha}w} = \nu_0^w$.  So $\wt\alpha$ certainly solves our problem in part \eqref{it:part_ii}.

% it is trivial to verify that the affine root $\wt\alpha$ specified earlier in Case I has the additional property required in part \eqref{it:part_ii}.

In the rest of the proof we shall assume that $\sigma(i) \neq i^*$.  This assumption places us in the setting of Case II in the proof of part \eqref{it:part_i}, where we used \eqref{st:wt_alpha_i,sigma(i)} to produce an affine root $\wt\alpha' = \wt\alpha_{i,\sigma(i);d}$ such that $w' = s_{\wt\alpha'}w$ is naively $\mu$-permissible and $w < w'$.  We shall now consider cases based on the possibilities for $w'$.

\ssk

\noindent\emph{Case A:  $w'$ is not $\mu$-spin-permissible.}  Then in every subcase considered in Case II above, we prescribed another affine root $\wt\alpha$ that solved the problem posed in part \eqref{it:part_i}.  We claim that $\wt\alpha$ also has the additional property required in part \eqref{it:part_ii}.  As in Case II, let $p$ be minimal in $\{0,\dotsc,m\}$ such that $\nu_p^{w'}$ is not $\mu$-spin-permissible.  A quick review of Case II reveals that $\wt\alpha$ is of the form $\wt\alpha_{i,j;0}$ or $\wt\alpha_{i,j^*;-1}$ for some $i < j \leq p$, or $\wt\alpha_{n,j;0}$ for some $n < j \leq p$.  Thus to prove the claim it suffices to show that $j \neq m$, since then $\nu_k^{s_{\wt\alpha}w}(m) = \nu_k^w(m)$ for all $k$.

To show that $j \neq m$, first note that if $p < m$, then of course $j < m$.  So we might as well assume $p = m$.  Then $c_p^w \equiv q \bmod 2$ and $\langle \wt\alpha', a_p \rangle \in \ZZ$.  Hence by \eqref{st:spin_fail_domain} $\nu_p^{w'}$ satisfies the spin condition.  So $\nu_p^{w'}$ must fail \eqref{it:SP2'}.  Hence
\[
   q < c_p^{w'} \leq c_p^w + 2 \leq q + 2,
\]
where the middle inequality uses \eqref{st:c^w'_change_lem}.  Since $c_p^w \equiv q \bmod 2$, we conclude that $c_p^w = q$.  On the other hand, \eqref{st:min_bad_p} implies that $u(i) \neq u\bigl(\sigma(i)\bigr)$, so that $c_i^w < q$ by \eqref{st:c_i<q}.  We conclude that $c_i^w < c_p^w$.  

To complete the argument, we compare the entries of $\nu_i^w$ and $\nu_p^w$. Since $\nu_p^w = \nu_m^w \in \ZZ^{2m}$, \eqref{st:fail_entries} implies that $\mu_p^w(i) = \mu_p^w(i^*) = 1$.  On the other hand, we have $\mu_i^w(i) = 2$ or $\mu_i^w(i^*) = 2$ according as $u(i) = 2$ or $u(i) = 1$.  Either way, since $c_i^w < c_p^w$, we see that there exist at least \emph{two} elements $j' \in \{i+1,\dotsc,p\}$ such that $\mu_i^w(j') = 1$ and $\mu_p^w(j') = 2$.  Such elements are clearly elements of the set $S$ defined in \eqref{disp:S_set}.  But when $c_i^w \leq c_p^w$, we defined $j$ to be the \emph{minimal} element of $S$.  Therefore $j < m$, which completes Case A.

\ssk

\noindent\emph{Case B: $w'$ is $\mu$-spin-permissible and $\sigma(i) \notin \{m,m+1\}$.}  Then $\nu_k^{s_{\wt\alpha}w}(m) = \nu_k^w(m)$ for all $k$, so we may take $\wt\alpha := \wt\alpha'$.

\ssk

\noindent\emph{Case C: $w'$ is $\mu$-spin-permissible and $\sigma(i) \in \{m,m+1\}$.}  In this case the $m$th entries of $\nu_0^w$ and $\nu_0^{w'}$ may differ, so the simple argument of Case B need not apply.  Evidently $\mu_0^w\bigl(\sigma(i)\bigr) = 1$ and $\mu_i^w\bigl(\sigma(i)\bigr) = 0$, so that $u\bigl(\sigma(i)\bigr) = 1$; and $n = m$.  We consider the subcases $u(i) = 1$ and $u(i) = 2$ separately.

First suppose $u(i) = 1$.  Then $\nu_0^w(i) = 0$ and $\wt\alpha'$ equals $\wt\alpha_{i,\sigma(i);-1}$ or $\wt\alpha_{i,\sigma(i);0}$.  If $\wt\alpha' = \wt\alpha_{i,\sigma(i);-1}$, then $\langle \wt\alpha', a_0\rangle = 1$ and $\nu_0^{w'} = s_{i,\sigma(i)}(\nu_0^w + \alpha^\vee_{i,\sigma(i)}) = \nu_0^w$.  Hence $\wt\alpha'$ solves our problem.  If $\wt\alpha' = \wt\alpha_{i,\sigma(i);0}$, then $\langle \wt\alpha', a_0 \rangle = 0$ and $\nu_0^{w'}\bigl(\sigma(i)\bigr) = \nu_0^w(i) = 0$.  So we must find a different affine root $\wt\alpha$ to solve our problem.

Now, note that \eqref{st:wt_alpha_i,sigma(i)} allows us to take $\wt\alpha' = \wt\alpha_{i,\sigma(i);0}$ exactly when 
%there is an inclusion of intervals 
$K_i \subset [\sigma(i),i)$.  If this inclusion is an equality $K_i = [\sigma(i),i)$, then \eqref{st:wt_alpha_i,sigma(i)} also allows us to take $\wt\alpha_{i,\sigma(i);-1}$ for our $\wt\alpha'$.  And if $s_{i,\sigma(i);-1}w$ is $\mu$-spin-permissible, then $\wt\alpha_{i,\sigma(i);-1}$ solves our problem exactly as in the previous paragraph.  If $s_{i,\sigma(i);-1}w$ is not $\mu$-spin-permissible, then we can run through the argument of Case II with $\wt\alpha_{i,\sigma(i);-1}$ in place of $\wt\alpha_{i,\sigma(i);0}$ to produce a new affine root $\wt\alpha$ that solves the problem posed in part \eqref{it:part_i}.  Case A then shows that this $\wt\alpha$ also solves our problem in part \eqref{it:part_ii}.

Still assuming $u(i) = 1$, we therefore reduce to solving our problem under the assumption $K_i \varsubsetneq [\sigma(i),i)$.  Since $\sigma(i) \in \{m,m+1\}$, this assumption implies that $m \notin K_i$.  Hence $\mu_m^w(i) = 1$, which, since $\nu_m^w \in \ZZ^{2m}$, in turn implies that $\nu_m^w(i) = 1$.  Let
\[
   \wt\beta := \wt\alpha_{i,\sigma(i);-1}.
\]
Let us investigate the extent to which $\wt\beta$ solves our problem.  Let $0 \leq k < i$.  Then as before $\nu_k^w(i) = 0$, $\nu_k^w\bigl(\sigma(i)\bigr) = 1$, $\bigl\langle \wt\beta, a_k\bigr\rangle = 1$, and $\nu_k^{s_{\wt\beta}w} = s_{i,\sigma(i)}(\nu_k^w + \alpha^\vee_{i,\sigma(i)}) = \nu_k^w$, which is $\mu$-spin-permissible.  Next let $i \leq k < m$.  Then $\nu_k^w\bigl(\sigma(i)\bigr) = \frac 1 2$ and $\bigl\langle \wt\beta, a_k \bigr\rangle = \frac 1 2$.  Since $\nu_i^w(i) = \frac 1 2$ and $\nu_m^w(i) = 1$, \eqref{st:nu(j)_vals} implies that $\nu_k^w(i) \in \bigl\{\frac 1 2, 1 \bigr\}$.  Hence by inspection $\nu_k^{s_{\wt\beta}w} = s_{i,\sigma(i)}\bigl(\nu_k^w + \frac 1 2 \alpha^\vee_{i,\sigma(i)} \bigr)$ satisfies \eqref{it:SP1'}, and $c_k^{s_{\wt\beta}w}$ equals $c_k^w$ or $c_k^w + 1$ according as $\nu_k^w(i)$ equals $\frac 1 2$ or $1$.  Moreover, taking $k = i$ in criterion \eqref{it:crit_1} in \eqref{st:bo}, we see that $w < s_{\wt\beta}w$.  Finally let $k = m$.  If $m = \sigma(i)$, then $\nu_m^w\bigl(\sigma(i)\bigr) = 1$ and $\bigl\langle \wt\beta, a_m \bigr\rangle = 1$.  Since $\nu_m^w(i) = 1$, we see explicitly that $\nu_m^{s_{\wt\beta}w} = s_{i,\sigma(i)}(\nu_m^w + \wt\alpha^\vee_{i,\sigma(i)})$ satisfies \eqref{it:SP1'} and $c_m^{s_{\wt\beta}w} = c_m^w + 2$.  On the other hand, if $m = \sigma(i)^*$, then $\bigl\langle \wt\beta, a_m \bigr\rangle = 0$ and $\nu_m^{s_{\wt\beta}w} = s_{i,\sigma(i)}\nu_m^w$ is plainly $\mu$-spin-permissible.

We conclude that $s_{\wt\beta}w$ is naively $\mu$-permissible and $w < s_{\wt\beta}w$.  If moreover $s_{\wt\beta}w$ is $\mu$-spin-permissible, then as before $\wt\beta$ solves our problem.  Otherwise, quite similarly to Case II above, we shall study the way in which $s_{\wt\beta}w$ fails to be $\mu$-spin-permissible to produce the desired affine root $\wt\alpha$.  Let $p \in \{0,\dotsc,m\}$ be minimal such that $\nu_p^{s_{\wt\beta}w}$ is not $\mu$-spin-permissible.  It is an immediate consequence of our analysis in the previous paragraph that $i < p$.

The following will be the main tool to solve our problem in present situation.

\begin{lem}
Let $j \in \{i+1,\dotsc,p\}$ be such that $\mu_i^w(j) = 1$ and $\mu_p^w(j) = 2$, and let $\wt\alpha := \wt\alpha_{i,j^*;-1}$.  Then $s_{\wt\alpha}w$ is $\mu$-spin-permissible and $\nu_0^{s_{\wt\alpha}w}(m) = 1$.
\end{lem}

\begin{proof}
First note that $j < m$: this is clear if $p < m$, and if $p = m$ then our above analysis shows that $m = \sigma(i)$ and $\mu_p^w(m) = \mu_p^w\bigl(\sigma(i)\bigr) = 1 \neq \mu_p^w(j)$.  Also note that evidently $\nu_k^w(j) = \nu_k^w(j^*) = 1$ for all $0 \leq k \leq i$, and $u(j^*) = 1$.

To prove the lemma, first let $0 \leq k < i$.  Then $\nu_k^w(i) = 0$, $\nu_k^w(j^*) = 1$, $\langle \wt\alpha, a_k\rangle = 1$, and $\nu_k^{s_{\wt\alpha}w} = s_{i,j^*}(\nu_k^w + \alpha^\vee_{i,j^*}) = \nu_k^w$.  Hence $\nu_k^{s_{\wt\alpha}w}$ is $\mu$-spin-permissible with $m$th entry equal to $1$.

Next let $i \leq k < j$.  Then $\nu_k^w(i) \in \bigl\{\frac 1 2, 1\bigr\}$, $\langle \wt\alpha , a_k \rangle = \frac 1 2$, and $\nu_k^{s_{\wt\alpha}w} = s_{i,j^*}\bigl(\nu_k^w + \frac 1 2 \alpha^\vee_{i,j^*}\bigr)$.  Since $\nu_i^w(j^*) = 1$ and $u(j^*) = 1$, we conclude from \eqref{st:nu(j)_vals} that $\nu_k^w(j^*) \in \bigl\{ \frac 1 2, 1 \bigr\}$.  Hence by inspection $\nu_k^{s_{\wt\alpha}w}$ satisfies \eqref{it:SP1'}, and $c_k^{s_{\wt\alpha}w}$ equals $c_k^w$ or $c_k^w + 1$ according as $\nu_k^w(i)$ equals $\frac 1 2$ or $1$.  If $\nu_k^w(i) = \frac 1 2$ then certainly $\nu_k^{s_{\wt\alpha}w}$ satisfies \eqref{it:SP2'}.  If $\nu_k^w(i) = 1$, then recall from above that $c_k^{s_{\wt\beta}w} = c_k^w + 1$.  Since $k < p$, minimality of $p$ ensures that this is $\leq q$, and again $\nu_k^{s_{\wt\alpha}w}$ satisfies \eqref{it:SP2'}.  Furthermore, since $j < m$ we have $\nu_k^{s_{\wt\alpha}w}\bigl(\sigma(i)\bigr) = \nu_k^w\bigl(\sigma(i)\bigr) = \frac 1 2$, so that $\nu_k^{s_{\wt\alpha}w}$ trivially satisfies \eqref{it:SP3'}.

Finally let $j \leq k \leq m$.  Then $\langle \wt\alpha , a_k \rangle = 0$ and $\nu_k^{s_{\wt\alpha}w} = s_{i,j^*}\nu_k^w$ is clearly $\mu$-spin-permissible. 
\end{proof}

To enable ourselves to apply the lemma, let us first show that $c_i^w \leq c_p^w$.  For let $k > i$ be such that $c_i^w > c_k^w$.  If $i < k < m$, then we showed above that $c_k^{s_{\wt\beta}w} \leq c_k^w + 1$.  By hypothesis this is $\leq c_i^w$, whence $\nu_k^{s_{\wt\beta}w}$ satisfies \eqref{it:SP2'}.  To see that $\nu_k^{s_{\wt\beta}w}$ also satisfies \eqref{it:SP3'}, suppose that $\nu_k^{s_{\wt\beta}w} \in \ZZ^{2m}$.  Then $\nu_k^w(i) = \nu_k^w\bigl(\sigma(i)\bigr) = \frac 1 2$, $\nu_k^w(i^*) = \nu_k^w\bigl(\sigma(i)^*\bigr) = \frac 3 2$, and all other corresponding entries of $\nu_k^w$ and $\nu_k^{s_{\wt\beta}w}$ are equal and in \ZZ.  Now, by \eqref{st:c_i_neq_c_p}\eqref{it:c_i>c_p} there exists $j \in \{i+1,\dotsc,k\}$ such that $\mu_i^w(j) = 0$ and $\mu_k^w(j) = 1$.  Since $\nu_k^w(j)$ is evidently an integer, it then also equals $1$.  On the other hand, the formula $\nu_k^{s_{\wt\beta}w} = s_{i,\sigma(i)}\bigl(\nu_k^w + \frac 1 2 \alpha^\vee_{i,\sigma(i)}\bigr)$ gives explicitly that $\nu_k^{s_{\wt\beta}w}\bigl(\sigma(i)\bigr) = 1$.  Since $j \in A_k$ and $\sigma(i) \in B_k$, we conclude that $\nu_k^{s_{\wt\beta}w}$ satisfies \eqref{it:SP3'}.  Thus $\nu_k^{s_{\wt\beta}w}$ is $\mu$-spin-permissible and $p \neq k$.  Now suppose $k = m$.  Then $c_k^w \equiv q \equiv c_i^w \bmod 2$ by \eqref{st:c_i<q}.  Since $c_i^w > c_k^w$, we must then have $c_i^w \geq c_k^w + 2$.  Hence by \eqref{st:c^w'_change_lem}\eqref{it:all_integers} $c_k^{s_{\wt\beta}w} \leq c_i^w$ and $\nu_k^{s_{\wt\beta}w}$ satisfies \eqref{it:SP2'}.  As always $\nu_k^{s_{\wt\beta}w}$ satisfies \eqref{it:SP3'} by \eqref{st:spin_fail_domain}, and we again conclude that $p \neq k$.  Thus $c_i^w \leq c_p^w$.

To complete the solution to our problem under our present assumptions, it remains to produce a $j$ as in the lemma and to show that, in the notation of the lemma, $w < s_{\wt\alpha}w$.  Of course the vector $\nu_p^{s_{\wt\beta}w}$ fails \eqref{it:SP2'} or \eqref{it:SP3'}, and we shall treat these cases separately.

Suppose that $\nu_p^{s_{\wt\beta}w}$ fails \eqref{it:SP2'}.  Then certainly $c_p^{s_{\wt\beta}w} > c_p^w$.  If $i < p < m$, then we saw earlier that this requires $\nu_p^w(i) = 1$.  If $p = m$, then we have already observed that $\nu_p^w(i) = 1$.  Either way, we see in particular that $\mu_p^w(i^*) = 1$.  Since $\mu_i^w(i^*) = 2$ and $c_i^w \leq c_p^w$, there then exists an element $j$ such that $\mu_i^w(j) = 1$ and $\mu_p^w(j) = 2$.  As usual, the recursion relation \eqref{disp:mu_recursion} implies that any such $j$ must be contained in $\{i+1,\dotsc,p\}$, as desired.  To see that, in the notation of the lemma, $w < s_{\wt\alpha}w$, note that $\nu_p^w(j^*) \in \bigl\{0,\frac 1 2\bigr\}$ since $\mu_p^w(j) = 2$.  So, since $\langle\wt\alpha, a_p\rangle = 0$ and $\nu_p^w(i) = 1$, we get the conclusion by taking $k = p$ in criterion \eqref{it:crit_2} in \eqref{st:bo}.

Now suppose that $\nu_p^{s_{\wt\beta}w}$ fails \eqref{it:SP3'}.  Then $i < p < m$, $\nu_p^{s_{\wt\beta}w} \in \ZZ^{2m}$, and $c_p^{s_{\wt\beta}w} \not\equiv q \bmod 2$.  Our earlier analysis then shows that $\nu_p^w(i) = \frac 1 2$ and $c_p^{s_{\wt\beta}w} = c_p^w$.  Since $c_i^w \equiv q \bmod 2$ by \eqref{st:c_i<q} and $c_i^w \leq c_p^w$ by our claim, we deduce the strict inequality $c_i^w < c_p^w$.  Then \eqref{st:c_i_neq_c_p}\eqref{it:c_i<c_p} gives us our $j$. To see that, in the notation of the lemma, $w < s_{\wt\alpha}w$, we again use that $\mu_p^w(j) = 2$.  Since this time we know that $\nu_p^w(j^*)$ is an integer, we conclude $\nu_p^w(j^*) = 0$.  Since $\langle \wt\alpha , a_p\rangle = 0$ and $\nu_p^w(i) = \frac 1 2$, we again get the conclusion by taking $k = p$ in criterion \eqref{it:crit_2} in \eqref{st:bo}.  With this we have completely solved our problem when $u(i) = 1$.

Now suppose $u(i) = 2$.  Then $\nu_0^w(i) = 1$ and $\wt\alpha'$ equals $\wt\alpha_{i,\sigma(i);- 1}$ or $\wt\alpha_{i,\sigma(i);1}$.  Neither possibility for $\wt\alpha'$ solves our problem, since $\nu_0^{s_{\wt\alpha'}w}\bigl(\sigma(i)\bigr)$ equals $0$ or $2$.  We are therefore led to consider
\[
   \wt\beta := \wt\alpha_{i,\sigma(i)^*;0}.
\]
Indeed, for all $0 \leq k < i$, and in particular for $k = 0$, we have $\bigl\langle \wt\beta, a_k\bigr\rangle = 0$ and $\nu_k^{s_{\wt\beta}w} = s_{i,\sigma(i)^*}\nu_k^w = \nu_k^w$, which by hypothesis has $m$th entry equal to $1$.  So $\wt\beta$ solves our problem provided $s_{\wt\beta}w$ is $\mu$-spin-permissible and $w < s_{\wt\beta}w$.  Let us investigate the extent to which this is the case.  We have just shown that $\nu_k^{s_{\wt\beta}w} = \nu_k^w$ is $\mu$-spin-permissible for $0 \leq k < i$.  For $i \leq k < m$, we have $\nu_k^w(i) \in \bigl\{1, \frac 3 2, 2 \bigr\}$, $\nu_k^w\bigl(\sigma(i)^*\bigr) = \frac 3 2$, and $\langle\wt\beta, a_k \rangle = -\frac 1 2$.  Hence by inspection $\nu_k^{s_{\wt\beta}w} = s_{i,\sigma(i)^*}\bigl(\nu_k^w - \frac 1 2 \alpha^\vee_{i,\sigma(i)^*}\bigr)$ satisfies \eqref{it:SP1'}, and $c_k^{s_{\wt\beta}w}$ equals $c_k^w$ or $c_k^w + 1$ according as $c_k^w \in \bigl\{\frac 3 2, 2\bigr\}$ or $c_k^w = 1$.  If $\nu_k^w(i) \in \bigl\{ \frac 3 2, 2 \bigr\}$ then certainly $\nu_k^{s_{\wt\beta}w}$ satisfies \eqref{it:SP2'}.  If $\nu_k^w(i) = 1$, then we use that $\nu_k^{w'} = s_{i,\sigma(i)}\bigl(\nu_k^w + \langle \wt\alpha',a_k\rangle \alpha^\vee_{i,\sigma(i)}\bigr)$ is $\mu$-spin-permissible.  Indeed, condition \eqref{it:SP1'} forces $\wt\alpha' = \wt\alpha_{i,\sigma(i);-1}$ and $\langle \wt\alpha', a_k \rangle = \frac 1 2$.  Since $\nu_k^w\bigl(\sigma(i)\bigr) = \frac 1 2$, we see explicitly that $c_k^{w'} = c_k^w + 1$, which is $\leq q$ by condition \eqref{it:SP2'}.  Hence $\nu_k^{s_{\wt\beta}w}$ satisfies \eqref{it:SP2'} as well.  Furthermore, since $\nu_i^w(i) = \nu_i^w\bigl(\sigma(i)^*\bigr) = \frac 3 2$, we get $w < s_{\wt\beta}w$ by taking $k = i$ in criterion \eqref{it:crit_1} in \eqref{st:bo}.  Finally let $k = m$.  Then $\nu_m^w(i) \in \{1,2\}$.  If $m = \sigma(i)^*$, then $\bigl\langle \wt\beta, a_m \bigr\rangle = 0$ and $\nu_m^{s_{\wt\beta}w} = s_{i,\sigma(i)^*}\nu_m^w$ is plainly $\mu$-spin-permissible.  On the other hand, if $m = \sigma(i)$, then $\nu_m^w\bigl(\sigma(i)\bigr) = \nu_m^w\bigl(\sigma(i)^*\bigr) = 1$ and $\bigl\langle \wt\beta, a_m \bigr\rangle = -1$.  Hence $\nu_m^{s_{\wt\beta}w} = s_{i,\sigma(i)^*}(\nu_m^w - \alpha^\vee_{i,\sigma(i)^*})$ satisfies \eqref{it:SP1'}, and $c_m^{s_{\wt\beta}w}$ equals $c_m^w$ or $c_m^w + 2$ according as $\nu_m^w(i)$ equals $2$ or $1$.  If $\nu_m^w(i) = 2$ then certainly $\nu_m^{s_{\wt\beta}w}$ satisfies \eqref{it:SP2'}.  If $\nu_m^w(i) = 1$ then we once more use that $\nu_m^{w'}$ is $\mu$-spin-permissible.  Indeed, since $m = \sigma(i)$ we have $\langle \wt\alpha', a_m\rangle = \pm 1$ and $\nu_m^{w'} = s_{i,\sigma(i)}(\nu_m^w \pm \alpha^\vee_{i,\sigma(i)})$.  Hence $c_m^{w'} = c_m^w + 1$, which is $\leq q$ by condition \eqref{it:SP2'}.  Hence $\nu_m^{s_{\wt\beta}w}$ satisfies \eqref{it:SP2'} as well.

We conclude from the previous paragraph that $s_{\wt\beta}w$ fails to solve our problem if and only if $\nu_p^{s_{\wt\beta}w}$ fails the spin condition for some $p$.  In fact it suffices to consider just $p = i$:  first let $i \leq p < m$ and suppose that $\nu_p^{s_{\wt\beta}w} \in \ZZ^{2m}$.  Then
\[
   \nu_p^w(i) = \nu_p^w\bigl(\sigma(i)^*\bigr) = \frac 3 2
   \quad\text{and}\quad
   \nu_p^w\bigl(\sigma(i)\bigr) = \nu_p^w(i^*) = \frac 1 2;
\]
all other entries of $\nu_p^w$ are integers;
\[
   \nu_p^{s_{\wt\beta}w}(i) = 2,\quad
   \nu_p^{s_{\wt\beta}w}(i^*) = 0,\quad
   \nu_p^{s_{\wt\beta}w}\bigl(\sigma(i)\bigr) = 1,\quad
   \text{and}\quad
   \nu_p^{s_{\wt\beta}w}\bigl(\sigma(i)^*\bigr) = 1;
\]
and all other corresponding entries of $\nu_p^{s_{\wt\beta}w}$ and $\nu_p^w$ are equal.  Hence $c_p^{s_{\wt\beta}w} = c_p^w$.  If this is $\not\equiv q \bmod 2$, then, since $\sigma(i) \in \{m,m+1\} \subset B_p$, we conclude that $\nu_p^{s_{\wt\beta}w}$ satisfies the spin condition $\iff$ there exists $j \in A_p$ such that $\nu_p^{s_{\wt\beta}w}(j) = 1$; or equivalently, since such a $j$ is evidently distinct from $i$ and $i^*$, there exists $j \in A_p$ such that $\nu_p^w(j) = 1$.  In particular, our discussion applies when $p = i$, using \eqref{st:c_i<q} for the relation $c_i^w \not\equiv q \bmod 2$.  If such a $j$ exists in this case, then $\nu_k^{s_{\wt\beta}w}(j) = \nu_k^w(j) = \nu_i^w(j) = \nu_i^{s_{\wt\beta}w} = 1$ for all $k$ since $i$ and $i^*$ are the only proper elements in $A_i$.  And then $\nu_p^{s_{\wt\beta}w}$ satisfies the spin condition for all $i \leq p < m$ since $j \in A_p$ and $\sigma(i) \in B_p$.  We conclude that $s_{\wt\beta}w$ is $\mu$-spin-permissible $\iff$ $\nu_i^{s_{\wt\beta}w}$ satisfies the spin condition.

So, to complete the proof, let us suppose that $\nu_i^{s_{\wt\beta}w}$ fails the spin condition; or equivalently, that $\nu_i^w(k) \in \{0,2\}$ for all $k \in A_i \smallsetminus \{i,i^*\}$.  We shall make crucial use of the fact that $q < m$.  By explicit calculation, we have $\nu_i^{w'}(k) \in \{0,2\}$ for all $k \in \{i,i^*,\sigma(i),\sigma(i^*)\}$.  Hence $\nu_i^{w'}(k) \in \{0,2\}$ for all $k \in A_i\cup\{m,m+1\}$.  Since $q < m$, there must exist a $j$, which we may moreover take to be minimal, such that $i < j < m$ and $\nu_i^{w'}(j) = \nu_i^w(j) = 1$.  Of course $\nu_k^w(j) = 1$ for all $0 \leq k \leq i$.  Let
\[
   \wt\gamma := \wt\alpha_{i,j^*;-1}.
\]
For $0 \leq k < i$, we have $\nu_k^w(i) = \nu_k^w(j) = \langle \wt\gamma, a_k \rangle = 1$ and $\nu_k^{s_{\wt\gamma}w} = s_{i,j^*}(\nu_k^w + \alpha^\vee_{i,j^*})$.  Hence $\nu_k^{s_{\wt\gamma}w}$ satisfies \eqref{it:SP1'} and $c_k^{s_{\wt\gamma}w} = c_k^w + 2$.  But
\[
   c_k^w + 2 = c_{i-1}^w + 2 = c_i^w + 1 \leq q,
\]
where the second equality uses \eqref{st:c_i_change_lem}.  So $\nu_k^{s_{\wt\gamma}w}$ satisfies \eqref{it:SP2'}.  And of course $\nu_k^{s_{\wt\gamma}w}$ satisfies the spin condition by \eqref{st:spin_fail_domain}.  Thus $\nu_k^{s_{\wt\gamma}w}$ is $\mu$-spin-permissible.  Next let $j \leq k \leq m$.  Then $\langle \wt\gamma, a_k \rangle = 0$ and $\nu_k^{s_{\wt\gamma}w} = s_{i,j^*}\nu_k^w$ is clearly $\mu$-spin-permissible.  Finally let $i \leq k < j$.  Then $\nu_k^w(i) \in \bigl\{1, \frac 3 2, 2 \bigr\}$, $\nu_k^w(j^*) \in \bigl\{\frac 1 2, 1, \frac 3 2 \bigr\}$ by \eqref{st:nu(j)_vals}, $\langle \wt\gamma, a_k \rangle = \frac 1 2$, and $\nu_k^{s_{\wt\gamma}w} = s_{i,j^*}\bigl(\nu_k^w + \frac 1 2 \alpha^\vee_{i,j^*}\bigr)$.  Hence $\nu_k^{s_{\wt\gamma}w}$ satisfies \eqref{it:SP1'} provided $\nu_k^w(i) \neq 2$.  Let us suppose for the moment that this is the case.  Since $\wt\alpha' = \wt\alpha_{i,\sigma(i);\pm 1}$ and $\nu_k^w\bigl(\sigma(i)\bigr) = \frac 1 2$, we see explicitly from the formula $\nu_k^{w'} = s_{i,\sigma(i)}\bigl(\nu_k^w + \langle \wt\alpha', a_k\rangle \alpha^\vee_{i,\sigma(i)}\bigr)$ that $c_k^{w'} = c_k^w + 1$, which is $\leq q$ by condition \eqref{it:SP2'}.  Hence
\[
   c_k^{s_{\wt\gamma}w} \leq c_k^w + 1 \leq q,
\]
where the first inequality uses \eqref{st:c^w'_change_lem}\eqref{it:some_half-integer}.  Hence $\nu_k^{s_{\wt\gamma}w}$ satisfies \eqref{it:SP2'}.  And $\nu_k^{s_{\wt\gamma}w}$ trivially satisfies \eqref{it:SP3'} because $\nu_k^{s_{\wt\gamma}w}\bigl(\sigma(i)\bigr) = \nu_k^w\bigl(\sigma(i)\bigr) = \frac 1 2$ is a half-integer.  Furthermore, since $\nu_i^w(i) = \frac 3 2$ and $\nu_i^w(j^*) = 1$, we see by taking $k = i$ in criterion \eqref{it:crit_1} in \eqref{st:bo} that $w < s_{\wt\gamma}w$.

In conclusion, we have shown that if $\nu_k^w(i) \neq 2$ for all $i \leq k < j$, then $s_{\wt\gamma}w$ is $\mu$-spin-permissible and $w < s_{\wt\gamma}w$.  Moreover $\wt\gamma$ then completely solves our problem, since $\nu_k^{s_{\wt\gamma}w}(m) = \nu_k^w(m)$ for all $k$.  So it remains to find a solution to our problem under the assumption that $\nu_k^w(i) = 2$ for some $i < k < j$, which occurs exactly when $i < \sigma^{-1}(i^*) < j$.  Minimality of $j$ forces $\nu_i^w\bigl(\sigma^{-1}(i^*)\bigr)$ to equal $0$ or $2$, and we shall complete the proof by considering these cases separately.

If $\nu_i^w\bigl(\sigma^{-1}(i^*)\bigr) = 0$ then $\nu_i^w\bigl(\sigma^{-1}(i)\bigr) = 2$, and we take
\[
   \wt\alpha := \wt\alpha_{i,\sigma^{-1}(i);-1}.
\]
To see that $\wt\alpha$ solves our problem, first let $0 \leq k < i$.  Then $\nu_k^w(i) = 1$, $\nu_k^w\bigl(\sigma^{-1}(i)\bigr) = 2$, $\langle \wt\alpha, a_k\rangle = 1$, and $\nu_k^{s_{\wt\alpha}w} = s_{i,\sigma^{-1}(i)}(\nu_k^w + \alpha^\vee_{i,\sigma^{-1}(i)}) = \nu_k^w$.  Hence $\nu_k^{s_{\wt\alpha}w}$ is $\mu$-spin-permissible.  Next let $i \leq k < \sigma^{-1}(i^*)$.  Then $\nu_k^w(i) = \frac 3 2$, $\nu_k^w\bigl(\sigma^{-1}(i)\bigr) \in \bigl\{ \frac 3 2, 2\bigr\}$ by \eqref{st:nu(j)_vals}, $\langle \wt\alpha, a_k \rangle = \frac 1 2$, and $\nu_k^{s_{\wt\alpha}w} = s_{i,\sigma^{-1}(i)}\bigl(\nu_k^w + \frac 1 2 \alpha^\vee_{i,\sigma^{-1}(i)}\bigr)$.  Hence by inspection $\nu_k^{s_{\wt\alpha}w}$ satisfies \eqref{it:SP1'} and $c_k^{s_{\wt\alpha}w} = c_k^w$, so that $\nu_k^{s_{\wt\alpha}w}$ also satisfies \eqref{it:SP2'}.  Moreover $\nu_k^{s_{\wt\alpha}w}$ trivially satisfies \eqref{it:SP3'} since $\nu_k^{s_{\wt\alpha}w}\bigl(\sigma(i)\bigr) = \nu_k^w\bigl(\sigma(i)\bigr) = \frac 1 2$ is a half-integer.  Finally let $\sigma^{-1}(i^*) \leq k \leq m$.  Then $\langle \wt\alpha, a_k \rangle = 0$ and $\nu_k^{s_{\wt\alpha}w} = s_{i,\sigma^{-1}(i)}\nu_k^w$ is clearly $\mu$-spin-permissible.  Moreover we have $\nu_{\sigma^{-1}(i^*)}^w(i) = 2$ and $\mu_{\sigma^{-1}(i^*)}^w\bigl(\sigma^{-1}(i^*)\bigr) = 1$, so that $\nu_{\sigma^{-1}(i^*)}^w\bigl(\sigma^{-1}(i)\bigr) \in \bigl\{1, \frac 3 2\bigr\}$.  Hence $w < s_{\wt\alpha}w$ by taking $k = \sigma^{-1}(i^*)$ in criterion \eqref{it:crit_2} in \eqref{st:bo}.  Of course $\nu_k^{s_{\wt\alpha}w}(m) = \nu_k^w(m)$ for all $k$, and we conclude that $\wt\alpha$ solves our problem.

If $\nu_i^w\bigl(\sigma^{-1}(i^*)\bigr) = 2$ then $\nu_i^w\bigl(\sigma^{-1}(i)\bigr) = 0$, and we take
\[
   \wt\alpha := \wt\alpha_{i,\sigma^{-1}(i);1}.
\]
To see that $\wt\alpha$ solves our problem, first let $0 \leq k < i$.  Then $\nu_k^w(i) = 1$, $\nu_k^w\bigl(\sigma^{-1}(i)\bigr) = 0$, $\langle \wt\alpha, a_k\rangle = -1$, and $\nu_k^{s_{\wt\alpha}w} = s_{i,\sigma^{-1}(i)}(\nu_k^w - \alpha^\vee_{i,\sigma^{-1}(i)}) = \nu_k^w$.  Hence $\nu_k^{s_{\wt\alpha}w}$ is $\mu$-spin-permissible.  Next let $i \leq k < \sigma^{-1}(i^*)$.  Then $\nu_k^w(i) = \frac 3 2$, $\nu_k^w\bigl(\sigma^{-1}(i)\bigr) \in \bigl\{ 0, \frac 1 2\bigr\}$ by \eqref{st:nu(j)_vals}, $\langle \wt\alpha, a_k \rangle = -\frac 3 2$, and $\nu_k^{s_{\wt\alpha}w} = s_{i,\sigma^{-1}(i)}\bigl(\nu_k^w - \frac 3 2 \alpha^\vee_{i,\sigma^{-1}(i)}\bigr)$.  Hence by inspection $\nu_k^{s_{\wt\alpha}w}$ satisfies \eqref{it:SP1'} and $c_k^{s_{\wt\alpha}w} \leq c_k^w + 1$.  To see that $\nu_k^{s_{\wt\alpha}w}$ satisfies \eqref{it:SP2'}, we use, as in earlier parts of the proof, that $\nu_k^{w'}$ does:  indeed $q \geq c_k^{w'} = c_k^w + 1$ by \eqref{st:i_sigma(i)_entries_fractional} (applied with $\wt\alpha = \wt\alpha'$), which gives us what we need.  Moreover $\nu_k^{s_{\wt\alpha}w}$ trivially satisfies \eqref{it:SP3'} since $\nu_k^{s_{\wt\alpha}w}\bigl(\sigma(i)\bigr) = \nu_k^w\bigl(\sigma(i)\bigr) = \frac 1 2$ is a half-integer.  Finally let $\sigma^{-1}(i^*) \leq k \leq m$.  Then $\nu_k^w(i) = 2$, $\nu_k^w\bigl(\sigma^{-1}(i)\bigr) = 0$, $\langle \wt\alpha, a_k \rangle = -2$, and $\nu_k^{s_{\wt\alpha}w} = s_{i,\sigma^{-1}(i)}(\nu_k^w - 2\alpha^\vee_{i,\sigma^{-1}(i)}) = \nu_k^w$.  Hence $\nu_k^{s_{\wt\alpha}}$ is $\mu$-spin-permissible.  To see that $w < s_{\wt\alpha}w$, note that $\sigma^{-1}(i)$ is evidently proper.  So there must exist a $k$ such that $\nu_k^w\bigl(\sigma^{-1}(i)\bigr) \neq 0$.  By what we have just seen, such a $k$ must satisfy $i \leq k < \sigma^{-1}(i^*)$, $\nu_k^w\bigl(\sigma^{-1}(i)\bigr) = \frac 1 2$, and $\nu_k^w(i) = \frac 3 2$.  Hence $w < s_{\wt\alpha}w$ by criterion \eqref{it:crit_1} in \eqref{st:bo}.  Finally $\nu_k^{s_{\wt\alpha}w}(m) = \nu_k^w(m)$ for all $k$, and we conclude that $\wt\alpha$ solves our problem.  This completes the proof in Case C when $u(i) = 2$, and with it the entire proposition is proved.
\end{proof}

This finally completes the proof of \eqref{st:sp-perm=adm_D}, and, equivalently, of Theorem \ref{st:adm_perm_D}.

\subsection{Parahoric variants and faces of type $I$}\label{ss:parahoric_D}
Our aims for the rest of \s\ref{s:combinatorics_D} are to extend \eqref{st:sp-perm=adm_D} and \eqref{st:mu-perm_D} to the general parahoric case (in \s\ref{ss:mu-adm_mu-sp-perm_parahoric} and \s\ref{ss:mu-perm_parahoric}, respectively) and to characterize $\mu$-admissibility in terms of $\mu$-vertexwise admissibility (in \s\ref{ss:vert-adm_D}).   
In this subsection we shall lay some of the needed groundwork.

For $F$ a subfacet of the base alcove $A_{D_m}$, let $W_{F,D_m}$ denote the common stabilizer and pointwise fixer of $F$ in $W_{\aff,D_m}$, as in \s\ref{ss:I-W_gp}.  To simplify matters for ourselves, we introduce the affine transformations on $\A_{D_m}$
\begin{equation}\label{disp:thetas}
\begin{aligned}
   \theta_0 &\colon (x_1,\dotsc,x_{2m}) \mapsto (x_{2m} - 1,x_2,\dotsc,x_{2m-1}, x_1 + 1),\\
   \theta_m &\colon (x_1,\dotsc,x_{2m}) \mapsto (x_1,\dotsc,x_{m-1},x_{m+1},x_m,x_{m+2},\dotsc,x_{2m}).
\end{aligned}
\end{equation}
Then $\theta_0$ and $\theta_m$ induce automorphisms of the affine root system of $\R_{D_m}$, and $\theta_0$ (resp.\ $\theta_m$) interchanges the points $a_0$ and $a_{0'}$ (resp.\ $a_m$ and $a_{m'}$) and fixes $a_k$ for all other $k \in \{0,0',2,\dotsc,m-2,m,m'\}$.  In particular, $\theta_0$ and $\theta_m$ send $A_{D_m}$ to itself.  It follows that both $\theta_0$ and $\theta_m$ induce automorphisms of $\wt W_{D_m}$ which 
preserve the $A_{D_m}$-Bruhat order.  The force of this in the next two subsections will be that, by applying $\theta_0$ or $\theta_m$ (or both) as needed, we shall incur no loss of generality by considering only subfacets $F$ of $A_{D_m}$ with the property that
\begin{equation}\label{disp:allowable_F_for_D_m}
   a_{0'} \in \ol F \implies a_0 \in \ol F
   \quad\text{and}\quad
   a_{m'} \in \ol F \implies a_m \in \ol F.
\end{equation}

Given a subfacet $F$ of $A_{D_m}$ satisfying \eqref{disp:allowable_F_for_D_m}, let
\[
   I := \bigl\{\, k\in\{0,\dotsc,m\} \bigm| a_k \in \ol F \,\bigr\}.
\]
Then $I$ is a nonempty subset of $\{0,\dotsc,m\}$ satisfying
\begin{equation}\label{disp:allowable_I_for_D_m}
   1 \in I \implies 0 \in I
   \quad\text{and}\quad
   m-1 \in I \implies m \in I.
\end{equation}
Indeed, if $a_1 \in \ol F$ (resp.\ $a_{m-1} \in \ol F$), then so is $a_0$ (resp.\ $a_m$), simply because $a_0$ (resp.\ $a_m$) is in the closure of the facet containing $a_1$ (resp.\ $a_{m-1}$).  Moreover $I$ uniquely determines $F$, since $a_{0'}$ (resp.\ $a_{m'}$) is also in the closure of the facet containing $a_1$ (resp.\ $a_{m-1}$), and any subfacet of $A_{D_m}$ is specified by which of the vertices $a_0$, $a_{0'}$, $a_2,\dotsc,$ $a_{m-2}$, $a_m$, $a_{m'}$ are in its closure.  In this way we obtain a bijection between the subfacets $F$ of $A_{D_m}$ satisfying \eqref{disp:allowable_F_for_D_m} and the nonempty subsets $I$ of $\{0,\dotsc,m\}$ satisfying \eqref{disp:allowable_I_for_D_m}.  Given such $I$, let
\begin{equation}\label{disp:W_I,D_m}
   W_{I,D_m} := \{\, w \in W_{\aff,D_m} \mid wa_k = a_k \text{ for all } k \in I\,\}.
\end{equation}
Of course
\[
   W_{I,D_m} = W_{F,D_m}
\]
for the unique subfacet $F$ of $A_{D_m}$ determined by $I$ in the way we have just described.

Our aim for the rest of the subsection is to interpret the set $\wt W_{D_m}/W_{I,D_m}$ in terms of faces of type $I$, for nonempty $I \subset \{0,\dotsc,m\}$ satisfying \eqref{disp:allowable_I_for_D_m}.  Define a homomorphism
\[
   \varepsilon\colon
   \xymatrix@R=0ex{
      \wt W_{D_m} \ar[r] & \ZZ/2\ZZ\\
      w \ar@{|->}[r] & \sum_{i = 1}^m \nu_0^w(i) \bmod 2.
   }
\]
Thus the affine Weyl group $W_{\aff,D_m} = Q^\vee_{D_m} \rtimes S_{2m}^\circ$ may be characterized as the subset of elements $w \in \wt W_{D_m}$ for which $\nu_0^w + (\nu_0^w)^* = \mathbf 0$ and $\varepsilon(w) = 0$.  
% Before continuing, let us pause to prove a little lemma which we'll use in the proof of \eqref{st:parahoric_W_faces_type_I} below.
% 
% \begin{lem}
% If $w \in \wt W_{D_m}$ stabilizes any one of the points $a_0$, $a_{0'}$, $a_m$, $a_{m'}$, then $\varepsilon(w) = 0$.
% \end{lem}
% 
% \begin{proof}
% Recall the vectors $\omega_0$ and $\omega_m$ \eqref{disp:omega_i}, and define
% \[
%    \omega_{0'} = a_{0'} = \bigl(-1,0^{(2m-2)},1\bigr)
%    \quad\text{and}\quad
%    \omega_{m'} = \bigl((- 1)^{(m-1)},0,-1,0^{(m-1)}\bigr).
% \]
% Then for all $k \in \{0,0',m,m'\}$, the stabilizer of $a_k$ in $\wt W_{D_m}$ equals the stabilizer of $\omega_k$; $\omega_k \in X_{*D_m}$; and $t_{\omega_k} \in \wt W_{D_m}$.  Since the stabilizer of $a_0 = \mathbf 0$ in $\wt W_{D_m}$ is $S_{2m}^\circ$, we conclude that the stabilizer of $a_k$ is $t_{\omega_k}S_{2m}^\circ t_{-\omega_k}$.  Now use that plainly $\varepsilon(S_{2m}^\circ) = 0$.
% % 
% %  Thus the stabilizer of $a_k$ is $t_{\omega_k}S_{2m}^\circ t_{-\omega_k}$.  Since the target of $\varepsilon$ is abelian, this renders the conclusion transparent.
% % 
% % The stabilizer of $a_0 = \mathbf 0$ in $\wt W_{D_m}$ is $S_{2m}^\circ$, and plainly $\varepsilon(S_{2m}^\circ) = 0$.  Now r
% \end{proof}
% 
% We now come to the key definition of this subsection.

\begin{defn}\label{def:F_I,D}
Let $I$ be a nonempty subset of $\{0,\dotsc,m\}$ (not necessarily satisfying \eqref{disp:allowable_I_for_D_m}).  We define the subset $\F_{I,D_m}$ of $\{\text{faces of type } I\} \times \ZZ/2\ZZ$ as follows.
If $0$, $m \in I$, then we define
\[
   \F_{I,D_m} := \bigl\{\, (\mathbf v, \gamma)  \bigm| 
      \mu_0^{\mathbf v} \equiv \mu_m^{\mathbf v} \bmod Q^\vee_{D_m} \text{ and } \gamma = \varepsilon(t_{\mu_0^{\mathbf v}})  \,\bigr\}.
\]
If $0 \in I$ and $m \notin I$ (resp.\ $m \in I$ and $0 \notin I$), then we define
\[
   \F_{I,D_m} := \bigl\{\, (\mathbf v, \gamma)  \bigm|
      \gamma = \varepsilon(t_{\mu_0^{\mathbf v}})  \,\bigr\}
      \quad
      \bigl(\text{resp.}\ \F_{I,D_m} := \bigl\{\, (\mathbf v, \gamma)  \bigm|
      \gamma = \varepsilon(t_{\mu_m^{\mathbf v}})  \,\bigr\}\bigr).
\]
If $0$, $m \notin I$, then we define
\[
   \F_{I,D_m} := \{\text{faces of type } I\} \times \ZZ/2\ZZ.
\]
\end{defn}

If $0$, $m \in I$ and $(\mathbf v, \gamma) \in \F_{I,D_m}$, then of course we also have $\gamma = \varepsilon (t_{\mu_m^{\mathbf v}})$.  The group $\wt W_{D_m}$ acts naturally on $\F_{I,D_m}$ via the rule $w \cdot (\mathbf v,\gamma) = \bigl(w\cdot \mathbf v, \varepsilon(w) + \gamma\bigr)$.  For a nonempty subset $J$ of $I$, the natural $\wt W_{D_m}$-equivariant map
\[
   \{\text{faces of type } I\} \to \{\text{faces of type } J\}
\]
induces a natural $\wt W_{D_m}$-equivariant map
\[
   \F_{I,D_m} \to \F_{J,D_m}.
\]
Furthermore, if exactly one of $0$, $m$ is in $I$ (resp.\ both $0$ and $m$ are in $I$), then the $\wt W_{D_m}$-equivariant projection
\[
   \F_{I,D_m} \to \{\text{faces of type $I$}\}
\]
is a bijection (resp.\ is an injection whose image consists of all faces
$\mathbf v$ such that $\mu_0^{\mathbf v} \equiv \mu_m^{\mathbf v} \bmod
Q^\vee_{D_m}$).

Now suppose that $I$ satisfies \eqref{disp:allowable_I_for_D_m}, let $w \in W_{I,D_m}$, and consider the standard face $\omega_I$ of type $(2m, I)$.  The affine Weyl group $W_{\aff,D_m}$ is naturally a subgroup of the affine Weyl group
\[
   \bigl\{\, v \in \ZZ^{2m} \bigm| \Sigma v = 0 \,\bigr\} \rtimes S_{2m}
\]
for the root datum attached to $GL_{2m}$, and the points $\omega_k \in \RR^{2m}$ for $k \in \ZZ$ are all contained in minimal facets of an alcove for $GL_{2m}$.  For $k \in I$, since $w$ fixes $a_k$ and $a_k$ is the midpoint of $\omega_k$ and $\omega_{-k}$, it follows that $w$ fixes $\omega_k$ and $\omega_{-k}$.  Hence $w$ fixes $\omega_I$.  Since moreover $\varepsilon(W_{I,D_m}) \subset \varepsilon(W_{\aff,D_m})= 0$, we get a well-defined $\wt W_{D_m}$-equivariant map
\[
   f\colon 
   \xymatrix@R=0ex{
      \wt W_{D_m}/W_{I,D_m} \ar[r] & \F_{I,D_m}\\
      w W_{I,D_m}  \ar@{|->}[r] & w \cdot (\omega_I,0).
   }
\]
Our main result for the subsection is the following.

\begin{lem}\label{st:parahoric_W_faces_type_I}
Let $I$ be a nonempty subset of $\{0,\dotsc,m\}$ satisfying \eqref{disp:allowable_I_for_D_m}.  Then the map $f \colon \wt W_{D_m}/W_{I,D_m} \to \F_{I,D_m}$ is a bijection.
\end{lem}

\begin{proof}
First suppose that $w \in \wt W_{D_m}$ fixes $(\omega_I,0)$.  Then $\varepsilon(w) = 0$; since $w$ fixes some vector in $\RR^{2m}$, $w$ must satisfy $\nu_0^w + (\nu_0^w)^* = \mathbf 0$; and, since $w$ fixes $\omega_k$ and $\omega_{-k}$ for all $k \in I$, $w$ fixes their midpoint $a_k$ for all $k \in I$.  Hence $w \in W_{I,D_m}$.  This proves that the map $f$ is injective.

To see that $f$ is surjective when $0$ or $m$ is in $I$, we need simply note that $\wt W_{D_m}$ acts transitively on the image of $\F_{I,D_m}$ under the injection
\[
   \F_{I,D_m} \inj \{\text{faces of type $I$}\}.
\]
If $0$, $m \notin I$, then we note that the element
\[
   w := \theta_0\theta_m = t_{(-1,0^{(2m-2)},1)}(1,2m)(m,m+1) \in \wt W_{D_m},
\]
where we use cycle notation for elements in $S_{2m}^\circ$, fixes $\omega_I$ and satisfies $\varepsilon(w) = 1$.  So the surjectivity of $f$ for such $I$ now follows from the fact that $\wt W_{D_m}$ acts transitively on faces of type $I$.
\end{proof}

\subsection{\texorpdfstring{$\mu$}{mu}-admissibility and \texorpdfstring{$\mu$}{mu}-spin-permissibility in the parahoric case}\label{ss:mu-adm_mu-sp-perm_parahoric}
In this subsection we generalize \eqref{st:sp-perm=adm_D} to the general parahoric case.  For subfacets $F$, $F'$ of $A_{D_m}$, we write
\[
   \Adm_{F',F,D_m}(\mu)
\]
for the $\mu$-admissible set in $W_{F',D_m} \bs \wt W_{D_m} / W_{F,D_m}$ \eqref{def:adm}.  When $F = F' = A_{D_m}$, we abbreviate this to $\Adm_{D_m}(\mu)$.

Recall $\theta_0$, $\theta_m \in \Aff(\A_{D_m})$ from \eqref{disp:thetas}, and let $\theta \in \{\theta_0,\theta_m,\theta_0\theta_m\}$.  Then $\theta \in \wt W_{2m}$ \eqref{disp:wtW_n}, and $\theta$ stabilizes the base alcove $A_{D_m}$.  Hence for any $\mu \in X_{*D_m}$, conjugation by $\theta$ induces an automorphism
\[
   W_{F',D_m} \bs \wt W_{D_m} / W_{F,D_m} \isoarrow
      W_{\theta F', D_m} \bs \wt W_{D_m} / W_{\theta F, D_m}
\]
identifying
\[
   \Adm_{F',F,D_m}(\mu) \isoarrow \Adm_{\theta F', \theta F, D_m}(\sigma\cdot\mu),
\]
where $\sigma$ denotes the linear part of $\theta$.  When $\mu$ is the cocharacter \eqref{disp:mu_D}, since $q < m$ we have $S_{2m}^\circ \mu = S_{2m}^*\mu$, and therefore the admissible set on the right-hand side of the last display equals $\Adm_{\theta F', \theta F,D_m}(\mu)$.

In the case that $F \preceq F' \preceq A_{D_m}$, we conclude that in studying admissible sets for this $\mu$, it is harmless to assume that $F$ and $F'$ satisfy property \eqref{disp:allowable_F_for_D_m}, since we can always apply a suitable $\theta \in \{\theta_0,\theta_m,\theta_0\theta_m\}$ to place ourselves in this situation.  Thus we shall assume that $W_{F,D_m}$ (resp.\ $W_{F',D_m}$) is of the form $W_{I,D_m}$ (resp.\ $W_{J,D_m}$) for some nonempty $I \subset J \subset \{0,\dotsc,m\}$ satisfying property \eqref{disp:allowable_I_for_D_m}.

Now recall that for $w \in \wt W_{D_m} / W_{I,D_m}$,  we get well-defined vectors $\mu_k^w = w\omega_k - \omega_k$ and $\nu_k^w = wa_k - a_k$ for $k \in 2m \pm I$, as discussed in the previous subsection.

\begin{defn}\label{def:mu-spin_perm_D_parahoric}
Let $\mu$ be the cocharacter \eqref{disp:mu_D}, and let $I \subset J$ be nonempty subsets of $\{0,\dotsc,m\}$ satisfying property \eqref{disp:allowable_I_for_D_m}.  We say that $w \in W_{J,D_m} \bs \wt W_{D_m}/W_{I,D_m}$ is \emph{$\mu$-spin-permissible} if $w \equiv t_\mu \bmod W_{\aff,D_m}$ and for one, hence any, representative $\wt w$ of $w$ in $\wt W_{D_m}/ W_{I,D_m}$, $\mu_k^w$ satisfies conditions \eqref{it:SP1}, \eqref{it:SP2}, and \eqref{it:SP3} in \eqref{def:spin-perm_D} for all $k \in I$.  We write
\[
   \SPerm_{J,I,D_m}(\mu)
\]
for the set of $\mu$-spin-permissible elements in $W_{J,D_m} \bs \wt W_{D_m} / W_{I,D_m}$.  We abbreviate this to $\SPerm_{I,D_m}(\mu)$ when $J = \{0,\dotsc,m\}$, and to $\SPerm_{D_m}(\mu)$ when $I = J = \{0,\dotsc,m\}$.
\end{defn}

We shall see in \eqref{st:sp-perm_double_cosets} below (with an assist from \eqref{st:SP_SP'_equiv}) that the definition really is independent of the choice of representative $\wt w \in \wt W_{D_m} / W_{I,D_m}$.  Let us first make some other remarks.

First, the condition $w \equiv t_\mu \bmod W_{\aff,D_m}$ is well-defined because $W_{I,D_m}$, $W_{J,D_m} \subset W_{\aff,D_m}$.  If $0$ (resp.\ $m$) is in $I$, then this condition is moreover redundant, since \eqref{it:SP3} requires $w \equiv t_{\mu_0^w} \equiv t_\mu \bmod W_{\aff, D_m}$ (resp.\ $w \equiv t_{\mu_m^w} \equiv t_\mu \bmod W_{\aff, D_m}$).  Of course, for any $I$ satisfying \eqref{disp:allowable_I_for_D_m}, $w$ is equivalently $\mu$-spin-permissible if $w \equiv t_\mu \bmod W_{\aff,D_m}$ and $\nu_k^{\wt w}$ satisfies \eqref{it:SP1'}, \eqref{it:SP2'}, and \eqref{it:SP3'} for all $k \in I$, for any representative $\wt w \in \wt W_{D_m}/W_{I,D_m}$, by \eqref{st:SP_SP'_equiv}.

To prepare for \eqref{st:sp-perm_double_cosets}, we first prove the following.

\begin{lem}\label{st:tau_A_k_B_k}
Let $x \in \wt W_{D_m}$, and write $x = t_{\nu_0^x} \tau$, with $\nu_0^x \in X_{*D_m}$ and $\tau$ in the Weyl group $S_{2m}^\circ$.  Suppose that $xa_k = a_k$ for some $k \in \{1,\dotsc, m-1\}$.  Then the sets $A_k$ and $B_k$ defined in \eqref{disp:A_i_B_i} are $\tau$-stable.
\end{lem}

\begin{proof}
Suppose by contradiction that $\tau$ carries some element $i \in A_k$ into $B_k$.  Since $a_k = \bigl((-\frac 1 2)^{(k)},0^{(2m-2k)},(\frac 1 2)^{(k)}\bigr)$, we obtain
\[
   (\tau a_k - a_k)\bigl(\tau(i)\bigr) = (\tau a_k)\bigl(\tau(i)\bigr) = a_k(i) = \pm \frac 1 2.
\]
But
\begin{equation}\label{disp:a_k=xa_k}
   a_k = xa_k = \nu_0^x + \tau a_k.
\end{equation}
Hence $\tau a_k - a_k = - \nu_0^x \in X_{*D_m} \subset \ZZ^{2m}$, a contradiction.
\end{proof}

\begin{lem}\label{st:sp-perm_double_cosets}
Let $w$, $w' \in \wt W_{D_m}/W_{I,D_m}$, and suppose that $w' = xw$ for some $x \in W_{I,D_m}$.  Then for $k \in I$, $\nu_k^w$ satisfies \eqref{it:SP1'} (resp.\ \eqref{it:SP2'}, resp.\ \eqref{it:SP3'}) $\iff$ $\nu_k^{w'}$ satisfies \eqref{it:SP1'} (resp.\ \eqref{it:SP2'}, resp.\ \eqref{it:SP3'}).  In particular, $w$ is $\mu$-spin-permissible $\iff$ $w'$ is $\mu$-spin-permissible.
\end{lem}

\begin{proof}
Let $k \in I$, and write $x = t_{\nu_0^x}\tau$ with $\tau$ in the Weyl group $S_{2m}^\circ$.  Then, using \eqref{disp:a_k=xa_k} to substitute for $a_k$ in the middle equality,
\[
   \nu_k^{w'} = xwa_k - a_k = \tau wa_k -\tau a_k = \tau\nu_k^w.
\]
So it is immediate that $\nu_k^{w'}$ satisfies \eqref{it:SP1'} (resp.\ \eqref{it:SP2'}) $\iff$ $\nu_k^{w}$  satisfies \eqref{it:SP1'} (resp.\ \eqref{it:SP2'}).  To consider the spin condition, we see from the display that $\nu_k^{w'} \in \ZZ^{2m}$ $\iff$ $\nu_k^w \in \ZZ^{2m}$.  Let us suppose this is so.  The display implies that $\nu_k^{w'} \equiv \nu_k^w \bmod Q^\vee_{D_m}$.  If these are both $\equiv \mu$, then they trivially satisfy the spin condition.  So suppose they are $\not\equiv \mu$.  Since $\nu_k^{w'} = \tau\nu_k^w$ and $\tau$ permutes the sets $A_k$ and $B_k$ separately by \eqref{st:tau_A_k_B_k}, we conclude at once that $\nu_k^{w'}$ satisfies the spin condition $\iff$ $\nu_k^w$ does, which completes the proof.
\end{proof}

Our main result for the subsection is the following.  We write $\Adm_{J,I,D_m}(\mu)$ for the set of $\mu$-admissible elements in $W_{J,D_m}\bs \wt W_{D_m} / W_{I,D_m}$.  We abbreviate this to $\Adm_{I,D_m}(\mu)$ when $J = \{0,\dotsc,m\}$, and to $\Adm_{D_m}(\mu)$ when $I = J = \{0,\dotsc,m\}$.

\begin{thm}\label{st:sp-perm=adm_D_parahoric}
Let $\mu$ be the cocharacter \eqref{disp:mu_D}, and let $I \subset J \subset \{0,\dotsc,m\}$ be nonempty sets satisfying \eqref{disp:allowable_I_for_D_m}.  Then
\[
   \Adm_{J,I,D_m}(\mu) = \SPerm_{J,I,D_m}(\mu).
\]
\end{thm}

\begin{proof}
The natural map $\wt W_{D_m} / W_{I,D_m} \surj W_{J,D_m} \bs \wt W_{D_m} / W_{I,D_m}$ induces a surjection
\[
   \Adm_{I,D_m}(\mu) \surj \Adm_{J,I,D_m}(\mu)
\]
by definition, and a surjection
\[
   \SPerm_{I,D_m}(\mu) \surj \SPerm_{J,I,D_m}(\mu)
\]
by \eqref{st:sp-perm_double_cosets}.  Thus it suffices to establish the equality $\Adm_{I,D_m}(\mu) = \SPerm_{I,D_m}(\mu)$ in $\wt W_{D_m} / W_{I,D_m}$.  When $I = \{0,\dotsc,m\}$, the result is given by \eqref{st:sp-perm=adm_D}.  For general $I$, we consider the natural map $\pi \colon \wt W_{D_m} \surj \wt W_{D_m} / W_{I,D_m}$.  Again by definition, $\pi$ induces a surjection $\Adm_{D_m}(\mu) \surj \Adm_{I,D_m}(\mu)$.  So the result follows from the fact that $\pi$ also induces a surjection $\SPerm_{D_m}(\mu) \surj \SPerm_{I,D_m}(\mu)$, which we prove in \eqref{st:spin-perm_surj} below.
\end{proof}

In fact we shall formulate \eqref{st:spin-perm_surj} in terms of faces of type $I$, or more precisely, in terms of the sets $\F_{I,D_m}$ \eqref{def:F_I,D} for varying $I$. Recall that in the previous subsection, for nonempty $I \subset \{0,\dotsc,m\}$ satisfying \eqref{disp:allowable_I_for_D_m}, we established a $\wt W_{D_m}$-equivariant bijection $\wt W_{D_m} / W_{I,D_m} \isoarrow \F_{I,D_m}$, compatible with the natural projections on both sides as we pass from $I$ to a subset of $I$.  To prove \eqref{st:spin-perm_surj}, it will be most natural to consider $\F_{I,D_m}$ for \emph{arbitrary} nonempty $I \subset \{0,\dotsc,m\}$.

For any nonempty $I$, $\F_{I,D_m}$ is a subset of $\{\text{faces of type}\ I\} \times \ZZ/2\ZZ$, and we say that $(\mathbf v,\gamma) \in \F_{I,D_m}$ is \emph{$\mu$-spin-permissible} if $\gamma = \varepsilon(\mu)$ and $\mu_k^{\mathbf v}$ satisfies \eqref{it:SP1}, \eqref{it:SP2}, and \eqref{it:SP3} for all $k \in I$.  If $I$ satisfies \eqref{disp:allowable_I_for_D_m}, then this definition of $\mu$-spin-permissibility coincides with the one in $\wt W_{D_m}/ W_{I,D_m}$ via the identification $\wt W_{D_m}/ W_{I,D_m} \isoarrow \F_{I,D_m}$.  For any $I$, we write $\SPerm_{\F_{I,D_m}}(\mu)$ for the set of $\mu$-spin-permissible elements in $\F_{I,D_m}$.

If $\mathbf v$ is a $2$-face of type $I$ and $k \in 2m\ZZ \pm I$, then let $l$ denote the unique element in $I$ that is congruent mod $2m$ to $k$ or $-k$.  Then by \eqref{disp:mu_per_cond} and \eqref{disp:mu_dlty_cond}, $\mu_k$ equals $\mu_l$ or $\mathbf 2 - \mu_l^*$.  Hence
\begin{itemize}
\item
   $\mu_k$ satisfies \eqref{it:SP1} $\iff$ $\mu_l$ does;
\item
   when these equivalent conditions hold, and using \eqref{st:c_equiv_defs} in the case that $\mu_k = \mathbf 2 - \mu_l^*$, $\mu_k$ satisfies \eqref{it:SP2} $\iff$ $\mu_l$ does; and
\item
   $\mu_k$ satisfies \eqref{it:SP3} $\iff$ $\mu_l$ does.
\end{itemize}
In particular, an element $\bigl(\mathbf v, \varepsilon(\mu)\bigr) \in \F_{I,D_m}$ is $\mu$-spin-permissible $\iff$ $\mu_k$ satisfies \eqref{it:SP1}, \eqref{it:SP2}, and \eqref{it:SP3} for all $k \in 2m\ZZ \pm I$.

% If $0$ or $m$ is in $I$, then $\F_{I,D_m}$ is a subset of faces of type $I$, and we say that such a face $(v_k)_k \in \F_{I,D_m}$ is \emph{$\mu$-spin-permissible} if $\mu_k$ satisfies \eqref{it:SP1}, \eqref{it:SP2}, and \eqref{it:SP3} for all $k \in I$.  If $0$, $m \notin I$, then $\F_{I,D_m} = \{\text{faces of type}\ I\} \times \ZZ/2\ZZ$, and we say that $\bigl((v_k)_k,\gamma \bigr) \in \F_{I,D_m}$ is \emph{$\mu$-spin-permissible} if $\gamma = \varepsilon(\mu)$ and $\mu_k$ satisfies \eqref{it:SP1}, \eqref{it:SP2}, and \eqref{it:SP3} for all $k \in I$.  If $I$ satisfies \eqref{disp:allowable_I_for_D_m}, then this definition of $\mu$-spin-permissibility coincides with the one in $\wt W_{D_m}/ W_{I,D_m}$ via the identification $\wt W_{D_m}/ W_{I,D_m} \isoarrow \F_{I,D_m}$.  For any $I$, we write $\SPerm_{I,D_m}(\mu)$ for the set of $\mu$-spin-permissible elements in $\F_{I,D_m}$.

\begin{lem}\label{st:spin-perm_surj}
Let $I \subset J$ be nonempty subsets of $\{0,\dotsc,m\}$.  Then the natural map $\F_{J,D_m} \to \F_{I,D_m}$ carries $\SPerm_{\F_{J,D_m}}(\mu)$ surjectively onto $\SPerm_{\F_{I,D_m}}(\mu)$.
\end{lem}

\begin{proof}
To prove the lemma, it suffices to assume that
\[
   2m\ZZ \pm J = (2m\ZZ \pm I) \amalg (2m\ZZ \pm \{i+1\})
\]
for some $i \in 2m\ZZ \pm I$.  Let $\bigl(\mathbf v, \varepsilon(\mu)\bigr) \in \SPerm_{\F_{I,D_m}}(\mu)$, with $\mathbf v = (v_k)_{k\in 2m\ZZ \pm I}$.  To lift this element to $\SPerm_{\F_{J,D_m}}(\mu)$, our task is to extend $\mathbf v$ to a face of type $J$ for which $\mu_k^{\mathbf v}$ satisfies \eqref{it:SP1}, \eqref{it:SP2}, and \eqref{it:SP3} for all $k \in J$.
% , for then $\bigl((v_k)_{k\in2m\ZZ \pm J},\varepsilon(\mu)\bigr)$ is the desired lift in $\SPerm_{J,D_m}(\mu)$.
Of course, once we have extended $\mathbf v$, it suffices to check just that $\mu_{i+1}^{\mathbf v}$ satisfies \eqref{it:SP1}, \eqref{it:SP2}, and \eqref{it:SP3}.

Let $j$ denote the minimal element in $2m\ZZ \pm I$ which is $> i$.  Then $j - i \geq 2$.  Let $\pi$ denote the canonical projection
\[
   \pi\colon \{\text{faces of type } J\} \to \{\text{faces of type } I\}.
\]
Kottwitz and Rapoport \cite{kottrap00}*{Lem.\ 10.3} show that there is a bijection%
\footnote{Here our description of the target set differs from that of Kottwitz--Rapoport due to our differing conventions in defining faces of type $I$.}
\[
   \xymatrix@R=0ex{
      \pi^{-1}(\mathbf v) \ar[r]^-{\sim}
         &  \{\, v\in \ZZ^{2m} \mid v_i \geq v \geq v_j \text{ and } \Sigma v = \Sigma v_i - 1 \,\} \\
      (v_k)_{k\in 2m\ZZ \pm J} \ar@{|->}[r]
         &  v_{i+1}.
   }
\]
Thus to extend $\mathbf v$ to a face of type $J$ is to define $v_{i+1}$ subject to the conditions in the definition of the target set in the display.  
% Let $l$ denote the unique element in $J$ that is congruent mod $2m$ to $i+1$ or to $-(i+1)$.  Then $J = I \amalg \{l\}$, and we must define $v_{i+1}$ such that furthermore $\mu_{l} = v_l - \omega_l$  satisfies \eqref{it:SP1}, \eqref{it:SP2}, and \eqref{it:SP3}.  By \eqref{disp:mu_per_cond} and \eqref{disp:mu_dlty_cond} we have $\mu_{i+1} = \mu_l$ or $\mu_{i+1} = \mathbf 2 - \mu_l^*$ 
% % according as $i+1 \equiv l$ or $i+1 \equiv -l \bmod 2m$ 
% (note that $\mathbf v$ is a $2$-face by $\mu$-spin-permissibility).  Thus it suffices to define $v_{i+1}$ such that $\mu_{i+1}$ itself satisfies \eqref{it:SP1}, \eqref{it:SP2}, and \eqref{it:SP3}; here, to formulate the spin condition for $\mu_{i+1}$, we define $A_{i+1} := A_l$ and $B_{i+1} := B_l$.
To ensure that moreover
\[
   \mu_{i+1} := v_{i+1} - \omega_{i+1}
\]
satisfies conditions \eqref{it:SP1}, \eqref{it:SP2}, and \eqref{it:SP3}, we shall largely follow G\"ortz's proof of surjectivity between $\mu$-permissible sets for $GL_{n}$ \cite{goertz05}*{Lem.\ 8}.  Indeed G\"ortz shows that one can always find a $v_{i+1}$ that extends $\mathbf v$ to a face of type $J$ such that $\mu_{i+1} \in \Conv(S_{2m}\mu)$; this last condition is equivalent, for our $\mu$, to $\mu_{i+1}$ satisfying \eqref{it:SP1} and \eqref{it:SP2}.  Unfortunately $\mu_{i+1}$ need not satisfy the spin condition, which will cause us some complications.  For any $k$, let $\wt k$ denote the unique element in $\{1,\dotsc, 2m\}$ congruent to $k \bmod 2m$.  We have that $v_j\bigl(\wt{i+1}\bigr)$ equals $v_i\bigl(\wt{i+1}\bigr)$ or $v_i\bigl(\wt{i+1}\bigr) - 1$, and we shall consider these cases separately.

\ssk
\noindent\emph{Case 1: $v_j\bigl(\wt{i+1}\bigr) = v_i\bigl(\wt{i+1}\bigr) - 1$.}
Let $v_{i+1} := v_i - e_{\wt{i+1}}$.  Then $v_{i+1}$ extends $\mathbf v$ to a face of type $J$, and $\mu_{i+1} = \mu_{i}^{\mathbf v}$.  Hence $\mu_{i+1}$ satisfies \eqref{it:SP1} and \eqref{it:SP2}.  If $\mu_{i+1}$ also satisfies \eqref{it:SP3} then we are done, so suppose $\mu_{i+1}$ fails \eqref{it:SP3}.  Then $\mu_{i+1}$ is self-dual and $\mu_{i+1} \not\equiv \mu \bmod Q^\vee_{D_m}$.  Let
\[
   c_i := \#\bigl\{\, k \bigm| \mu_i(k) = 2 \,\bigr\},
\]
and define $c_{i+1}$ likewise.  By \eqref{st:c_i_equiv_q} $c_i = c_{i+1} \not\equiv q \bmod 2$.  Hence $c_i < q$.

Now, since $\mu_i$ satisfies \eqref{it:SP3}, there must exist $l_1 \in A_i$ and $l_2 \in B_i$ such that $\mu_i(l_1) = \mu_i(l_2) = 1$.  And since $\mu_{i+1}$ fails \eqref{it:SP3}, we must have $l_1,\ l_2 \in A_{i+1}$\quad and $\mu_{i+1}(k) \in \{0,2\}$ for all $k \in B_{i+1}$, or the same statement with $A_{i+1}$ and $B_{i+1}$ interchanged.  For convenience, let us assume the former; one proceeds in an analogous way in case of the latter.  Then evidently $A_{i+1} = A_i \amalg \{l_2,l_2^*\}$; $\{l_2,l_2^*\} = \bigl\{\wt{i+1}, (\wt{i+1})^*\bigr\}$; and $B_{i+1} = B_i \smallsetminus \bigl\{\wt{i+1}, (\wt{i+1})^*\bigr\}$.  To proceed with our proof in Case 1, we shall consider separately the possibilities that $\mu_j^{\mathbf v} = \mu_{i+1}$ and $\mu_j^{\mathbf v} \neq \mu_{i+1}$.

Suppose that $\mu_j^{\mathbf v} \neq \mu_{i+1}$.  Then, since $\Sigma\mu_j^{\mathbf v} = \Sigma\mu_{i+1}$, there exists $l \in \{1,\dotsc,2m\}$ such that $\mu_j^{\mathbf v}(l) = \mu_{i+1}(l) - 1$.  This requires $l \notin \bigl\{\wt{i+1}, \wt{i+2},\dotsc,\wt \jmath\bigr\}$, whence $v_j(l) = v_{i+1}(l) - 1$.  We define $v_{i+1}' := v_i - e_l$.  Then $v_{i+1}'$ extends $\mathbf v$ to a face of type $J$, and we must show that $\mu_{i+1}' := v_{i+1}' - \omega_{i+1}$ satisfies \eqref{it:SP1}, \eqref{it:SP2}, and \eqref{it:SP3}.  All of this follows easily from the formula
\begin{equation}\label{disp:mu'_fmla}
   \mu_{i+1}' = \mu_i^{\mathbf v} + e_{\wt{i+1}} - e_l.
\end{equation}
Indeed, for \eqref{it:SP1}, we need note only that $\mu_i^{\mathbf v}$ and $\mu_j^{\mathbf v}$ satisfy \eqref{it:SP1} by hypothesis, that
\[
   \mu_{i+1}'\bigl(\wt{i+1}\bigr) = \mu_i^{\mathbf v}\bigl(\wt{i+1}\bigr) + 1 = 2,
\]
and that
\[
   \mu_{i+1}'(l) = \mu_i^{\mathbf v}(l) - 1 = \mu_{i+1}(l) - 1 = \mu_j^{\mathbf v}(l) \in \{0,1,2\}.
\]
For \eqref{it:SP2}, recalling from earlier that $c_i < q$, \eqref{disp:mu'_fmla} gives
\[
   c_{i+1}' := \#\bigl\{\, k \bigm| \mu_{i+1}'(k) = 2 \,\bigr\} \leq c_i + 1 \leq q.
\]
And for \eqref{it:SP3}, $\mu_{i+1}'$ is self-dual only if $l = (\wt{i+1})^*$, in which case $c_{i+1}' = c_i + 1$.  Then $\mu_{i+1}'\equiv \mu \bmod Q^\vee_{D_m}$ by \eqref{st:c_i_equiv_q}.

Now suppose that $\mu_j^{\mathbf v} = \mu_{i+1}$.  Then $v_j(k) = v_i(k) - 1$ for exactly the elements $k \in \{\wt{i+1}, \wt{i+2},\dotsc, \wt\jmath\}$.  Since $A_i \subset A_{i+1}$, we have $\wt{i+1} \in \{1,\dotsc,m\}$, and
\[
   B_{i+1} = \bigl\{\wt{i+2}, \wt{i+3},\dotsc, 2m - (\wt{i+1}) \bigr\}.
\]
(To be clear, we interpret this as $B_{i+1} = \emptyset$ if $\wt{i+1} = m$.)
If $\wt\jmath \in B_{i+1}$, then $B_j \subset B_{i+1}$, so that $\mu_j^{\mathbf v}(k) \in \{0,2\}$ for all $k \in B_j$.  Since $\mu_j^{\mathbf v} = \mu_{i+1} \not\equiv \mu \bmod Q^\vee_{D_m}$, this would violate the spin condition.  So $\wt\jmath \notin B_{j+1}$.  By minimality of $j \in 2m\ZZ \pm I$, we therefore have $\wt\jmath = 2m -(\wt{i+1}) + 1 = (\wt{i+1})^*$.  We now define
\[
   v_{i+1}' := v_i - e_{\wt\jmath} = v_i - e_{(\wt{i+1})^*}.
\]
Then $v_{i+1}'$ extends $\mathbf v$ to a face of type $J$, and one sees just as in the previous paragraph that $v_{i+1}' - \omega_{i+1}$ satisfies \eqref{it:SP1}, \eqref{it:SP2}, and \eqref{it:SP3}.  This completes the proof in Case 1.

\ssk
\noindent\emph{Case 2: $v_j\bigl(\wt{i+1}\bigr) = v_i\bigl(\wt{i+1}\bigr)$.}  Let $X := \{\, k \mid v_j(k) \neq v_i(k)\,\}$.  Working with respect to the Kottwitz--Rapoport conventions on faces of type $I$, G\"ortz shows in his proof of \cite{goertz05}*{Lem.\ 8}%
\footnote{For the reader's convenience, we note that G\"ortz takes permutations to act on vectors on the \emph{right,} so that $\sigma \cdot \sum_i b_ie_i = \sum_i b_i e_{\sigma^{-1}(i)}$ ($b_i \in \RR$).  It also appears that, in his notation, he means to define $m$ to equal $\sigma\bigl(\wt{\wt m}\bigr)$, not $\sigma(\wt m)$; the changes to his proof needed to accommodate this correction are trivial.}
that by taking $l$ to be any element of $X$ for which
\[
   \mu_i^{\mathbf v}(l) = \min\{\mu_i^{\mathbf v}(k)\}_{k\in X},
\]
the vector $v_{i+1} := v_i + e_l$ extends $\mathbf v$ to a face of type $J$ for which $\mu_{i+1}$ satisfies \eqref{it:SP1} and \eqref{it:SP2}.  Translating this to our conventions, we get that if $l$ is any element of $X$ for which
\[
   \mu_i^{\mathbf v}(l) = \max\{\mu_i^{\mathbf v}(k)\}_{k\in X},
\]
then the vector $v_{i+1} := v_i - e_l$ extends $\mathbf v$ to a face of type $J$ for which $\mu_{i+1}$ satisfies \eqref{it:SP1} and \eqref{it:SP2}.  If $\mu_{i+1}$ is not self-dual, or if $\mu_{i+1}$ is self-dual and $\equiv \mu \bmod Q^\vee_{D_m}$, then it trivially satisfies \eqref{it:SP3}.  So let us suppose $\mu_{i+1}$ is self-dual and $\not\equiv \mu \bmod Q^\vee_{D_m}$.

Before continuing, let us note that $\#X = j-i$ and $\wt{i+1} \notin X$.  Hence there exists $k' \in X \smallsetminus \bigl\{\wt{i+1}, \wt{i+2},\dotsc, \wt\jmath\bigr\}$.  Since $\mu_j^{\mathbf v}(k') = \mu_i^{\mathbf v}(k') - 1$ and $\mu_i^{\mathbf v}$ and $\mu_j^{\mathbf v}$ satisfy \eqref{it:SP1}, we conclude that $\max\{\mu_i^{\mathbf v}(k)\}_{k\in X}$ equals $1$ or $2$.

To proceed with the proof in Case 2, first suppose that $l$ is the unique element in $X$ such that $\mu_i^{\mathbf v}(l) \neq 0$.  We shall show that $\mu_{i+1}$ satisfies the spin condition.  Since $\mu_j^{\mathbf v}$ satisfies \eqref{it:SP1} and $\wt{i+1} \notin X$, we must have
\[
   X \smallsetminus \{l\} \subset \bigl\{ \wt{i+2}, \wt{i+3}, \dotsc, \wt\jmath \bigr\},
\]
and by counting cardinalities we see that this inclusion is an equality.  Hence $\mu_{i+1} = \mu_j^{\mathbf v}$, and $\mu_i^{\mathbf v}(k) = \mu_{i+1}(k) = \mu_j^{\mathbf v}(k) = 0$ for all $k \in \{\wt{i+2}, \wt{i+3}, \dotsc, \wt\jmath \}$.  Since $\mathbf v$ is a $2$-face and $\mu_j^{\mathbf v}$ is self-dual, we must have that $\{ \wt{i+2}, \wt{i+3}, \dotsc, \wt\jmath \}$ is contained in $\{1,\dotsc,m\}$ or in $\{m+1,\dotsc,2m\}$.  Let us assume the former; one proceeds in an analogous way in case of the latter.  Then $A_j = \{1,\dotsc,\wt\jmath,\wt\jmath^*,\dotsc,2m\}$ and $B_j = \{\wt\jmath + 1,\dotsc,2m-\wt\jmath\}$.  Since $\mu_j^{\mathbf v}$ satisfies the spin condition and $\mu_j^{\mathbf v} = \mu_{i+1} \not\equiv \mu \bmod Q^\vee_{D_m}$, there exist $k_1 \in A_j$ and $k_2 \in B_j$ such that $\mu_j^{\mathbf v}(k_1) = \mu_j^{\mathbf v}(k_2) = 1$.  Evidently $k_1 \in A_j \smallsetminus \{\wt{i+2}, \dotsc, \wt\jmath, \wt\jmath^*,\dotsc,(\wt{i+2})^*\}$.  It follows that $\wt{i+1} \in \{1,\dotsc,m\}$, that $A_{i+1} = \{1,\dotsc,\wt{i+1},(\wt{i+1})^*,\dotsc,m\}$, and that $B_{i+1} = \{\wt{i+2},\dotsc, 2m - (\wt{i+1})\}$.  Plainly $k_1 \in A_{i+1}$ and $k_2 \in B_{i+1}$, so that $\mu_{i+1}$ satisfies \eqref{it:SP3}.

Finally suppose that there exists another element $l' \in X$ distinct from $l$ and such that $\mu_i^{\mathbf v}(l') \neq 0$.  Of course if $\mu_{i+1}$ already satisfies \eqref{it:SP3} then there is nothing further to do, so let us suppose that it does not.  Let $v_{i+1}' := v_i - e_{l'}$ and $\mu_{i+1}' := v_{i+1}' - \omega_{i+1}$, and define $c_{i+1}$ and $c_{i+1}'$ as in Case 1.  Then $v_{i+1}'$ extends $\mathbf v$ to another face of type $J$, $c_{i+1} \not\equiv q \bmod 2$ by \eqref{st:c_i_equiv_q}, and $c_{i+1} < q$.  Our task is to show that $\mu_{i+1}'$ satisfies \eqref{it:SP1}, \eqref{it:SP2}, and \eqref{it:SP3}.  We have
\[
   \mu_{i+1}' = \mu_i^{\mathbf v} + e_{\wt{i+1}} - e_{l'} = \mu_{i+1} + e_l - e_{l'}.
\]
Since $\mu_i^{\mathbf v}$ and $\mu_{i+1}$ satisfy \eqref{it:SP1} and the elements $\wt{i+1}$, $l$, $l'$ are pairwise distinct, these expressions make clear that $0 \leq \mu_{i+1}'(k) \leq 2$ for all $k$ except possibly $k = \wt{l'}$; and since $\mu_i^{\mathbf v}(l') \in \{1,2\}$ by hypothesis, we have $\mu_{i+1}'(l') \in \{0,1\}$.  Hence $\mu_{i+1}'$ satisfies \eqref{it:SP1}.  Since $c_{i+1} < q$, we also read off from the display that
\[
   c_{i+1}' \leq c_i + 1 \leq q,
\]
so that $\mu_{i+1}'$ satisfies \eqref{it:SP2}.  It remains to see that $\mu_{i+1}'$ satisfies \eqref{it:SP3}.  But since $\mu_{i+1}$ is self-dual by assumption, $\mu_{i+1}'$ is self-dual only if $l' = l^*$.  In this case $c_{i+1}' = c_i \pm 1$ and $\mu_{i+1}' \equiv \mu \bmod Q^\vee_{D_m}$ by \eqref{st:c_i_equiv_q}, which completes the proof.
\end{proof}

\begin{rk}\label{rk:min_length_spin-perm}
For $I \subset J$ nonempty and satisfying \eqref{disp:allowable_I_for_D_m}, we have shown directly, in \eqref{st:sp-perm_double_cosets} and \eqref{st:spin-perm_surj}, that the map
\[
   \SPerm_{D_m}(\mu) \to \SPerm_{J,I,D_m}(\mu)
\]
is surjective.  This statement can be sharpened:  for $w \in \SPerm_{J,I,D_m}(\mu)$, let $\tensor[^J]w{^I}$ denote the minimal length representative of $w$ in $\wt W_{D_m}$.  Then $\tensor[^J]w{^I} \in \SPerm_{D_m}(\mu)$.  Indeed, since we've already shown that the displayed map is surjective, the sharpened statement is an immediate consequence of the fact that $\SPerm_{D_m}(\mu)$ is closed in the Bruhat order \eqref{st:SPerm_bruhat_closed}.

% Indeed, since we've already shown that the displayed map is surjective, the sharpened statement follows at once from the Iwahori case of \eqref{st:sp-perm=adm_D_parahoric}, which shows that $\SPerm_{D_m}(\mu)$ is closed in the Bruhat order.

% Say that if $w \in W_{J,D_m} \bs \wt W_{D_m} / W_{I,D_m}$ is $\mu$-spin-permissible, then the representative of minimal length $^J\!w^I$ is $\mu$-spin-permissible in $I$.  Can see this very easily using \eqref{st:sp-perm=adm_D_parahoric}, but it can also be deduced from the surjectivity of the map $\SPerm_{D_m}(\mu) \to \SPerm_{J,I,D_m}(\mu)$ and the easily proved fact that $\SPerm_{D_m}(\mu)$ is closed in the Bruhat order.
\end{rk}

\subsection{Vertexwise admissibility}\label{ss:vert-adm_D}
In this subsection we recall the notion of \emph{$\mu$-vertexwise admissibility} for any cocharacter $\mu$ in any based root datum, and we show that for $\mu$ the cocharacter \eqref{disp:mu_D} in $\R_{D_m}$, the notions of $\mu$-admissibility and $\mu$-vertexwise admissibility coincide.

Let \R be a based root datum, let $\wt W$ denote its extended affine Weyl group, and let $\mu$ be a cocharacter for \R.  Fix a base alcove $\mathbf a$, and let $\mathbf f \preceq \mathbf{f'}$ be subfacets of $\mathbf a$.  The choice of $\mathbf a$ endows $W_{\mathbf{f'}} \bs \wt W / W_{\mathbf f}$ with a Bruhat order, as in \s\ref{ss:bo}.  For any subfacet $\mathbf{v} \preceq \mathbf f$,  the natural map
\[
   \rho_{\mathbf{f'}, \mathbf f; \mathbf v} \colon W_{\mathbf{f'}} \bs \wt W / W_{\mathbf f} \to W_{\mathbf v} \bs \wt W / W_{\mathbf v}
\]
carries (by definition) the set $\Adm_{\mathbf{f'},\mathbf f}(\mu)$ surjectively to the set $\Adm_{\mathbf{v}, \mathbf v}(\mu)$.  We then have the following definition.

\begin{defn}[\cite{prs?}*{Def.\ 4.5.2}]\label{def:vert-adm}
For any subfacets $\mathbf f \preceq \mathbf{f'} \preceq \mathbf a$, the \emph{$\mu$-vertexwise admissible set} is the subset of $W_{\mathbf{f'}} \bs \wt W / W_{\mathbf f}$
\[
   \Adm_{\mathbf{f'},\mathbf f}^\vert(\mu) := 
     \bigcap_{\substack{\text{minimal}\\ \mathbf{v} \preceq \mathbf f}} \rho_{\mathbf{f'}, \mathbf f; \mathbf v}^{-1} \bigl( \Adm_{\mathbf v, \mathbf v}(\mu)\bigr).
\]
\end{defn}

In other words, an element in $W_{\mathbf{f'}} \bs \wt W / W_{\mathbf f}$ is $\mu$-vertexwise admissible if its image in $W_{\mathbf v} \bs \wt W / W_{\mathbf v}$ is $\mu$-admissible for all minimal subfacets $\mathbf v$ of $\mathbf f$.  When \R is semisimple, that is, when the root lattice has finite index in the character group, the minimal subfacets of $\mathbf f$ are vertices, i.e.\ points, which motivates the terminology.  Strictly speaking, vertexwise admissibility is only defined in \cite{prs?} when $\mathbf f = \mathbf{f'}$, but the added bit of generality we allow here will cause us no trouble, at least in the particular situations we consider later.

Since $\rho_{\mathbf{f'}, \mathbf f; \mathbf v}$ carries $\Adm_{\mathbf{f'}, \mathbf f}(\mu)$ to $\Adm_{\mathbf v, \mathbf v}(\mu)$, there is a tautological inclusion
\begin{equation}\label{disp:adm_incl}
   \Adm_{\mathbf{f'}, \mathbf f}(\mu) \subset \Adm_{\mathbf v, \mathbf v}^\vert(\mu).
\end{equation}
It is conjectured in \cite{prs?}*{Conj.\ 4.5.3} (at least for $\mathbf f = \mathbf{f'}$) that in many cases this inclusion is an equality.%
\footnote{We refer the reader to \cite{prs?} for the precise formulation of the conjecture, which is given in group-theoretic terms.}
Indeed no examples whatsoever are known where equality fails; and equality is known to hold in all situations arising from local models that have been studied,%
\footnote{In the situation arising from the local models 
% for even unitary groups considered earlier 
studied in this paper, the equality $\Adm_{\mathbf{f'}, \mathbf f}(\mu) = \Adm_{\mathbf{f'}, \mathbf f}^\vert(\mu)$ will be proved in \s\ref{ss:vert_adm_B}.}
and furthermore whenever $\mu$-admissibility and $\mu$-permissibility are equivalent in $W_{\mathbf{f'}} \bs \wt W / W_{\mathbf f}$.  In some sense the notion of $\mu$-vertexwise admissibility is intended as a replacement for $\mu$-permissibility (which of course is not equivalent to $\mu$-admissibility in general), in that $\mu$-vertexwise admissibility is a manifestly ``vertexwise'' condition which is supposed to characterize $\mu$-admissibility, at least in many cases.

Now let us specialize to the case $\R = \R_{D_m}$ and $\mu$ the cocharacter \eqref{disp:mu_D}.  We shall show that the inclusion \eqref{disp:adm_incl} is indeed an equality.  We write $\Adm_{F',F,D_m}^\vert(\mu)$ for the $\mu$-vertexwise admissible set in $W_{F',D_m} \bs \wt W_{D_m} / W_{F,D_m}$.

\begin{thm}\label{st:vert-adm_D}
Let $\mu$ be the cocharacter \eqref{disp:mu_D}, and let $F \preceq F' \preceq A_{D_m}$.  Then
\[
   \Adm_{F',F,D_m}(\mu) = \Adm_{F',F,D_m}^\vert(\mu).
\]
\end{thm}

\begin{proof}
After applying one or both of the automorphisms $\theta_0$ and $\theta_m$ \eqref{disp:thetas} as needed, we may assume that $F$ and $F'$ satisfy property \eqref{disp:allowable_F_for_D_m}, as discussed at the beginning of \s\ref{ss:mu-adm_mu-sp-perm_parahoric}.  Thus we may assume that $W_{F,B_m} = W_{I,B_m}$ and $W_{F',B_m} = W_{J,B_m}$ for some nonempty subsets $I \subset J \subset \{0,\dotsc,m\}$ satisfying property \eqref{disp:allowable_I_for_D_m}.  The proof will now follow fairly easily from our characterization of $\mu$-admissibility in terms of $\mu$-spin-permissibility \eqref{st:sp-perm=adm_D_parahoric}.

We must show that $\Adm_{F',F,D_m}^\vert(\mu) \subset \Adm_{F',F,D_m}(\mu)$.  So let $w \in \Adm_{F',F,D_m}^\vert(\mu)$.  Let $\wt w$ be any lift of $w$ to $\wt W_{D_m}$.  By \eqref{st:sp-perm=adm_D_parahoric}, we must show that $\wt w \equiv t_\mu \bmod W_{\aff,D_m}$ and that $\nu_k^{\wt w}$ satisfies \eqref{it:SP1'}, \eqref{it:SP2'}, and \eqref{it:SP3'} for all $k \in I$.  Of course $\mu$-vertexwise admissibility immediately implies 
%that $w \equiv t_\mu \bmod W_{\aff,D_m}$.
the first of these conditions.

Now, the minimal subfacets of $F$ are the lines
\begin{alignat*}{3}
   a_k + \RR\cdot (1,\dotsc, 1) & & \quad &\text{for} & \quad &
      k \in I \smallsetminus \{1,m-1\};\\
   a_{0'} + \RR\cdot (1,\dotsc, 1) & & \quad &\text{if} & \quad &
      1 \in I;\\
   a_{m'} + \RR\cdot (1,\dotsc, 1) & & \quad &\text{if} & \quad &
      m-1 \in I.
\end{alignat*}
So by vertexwise admissibility and \eqref{st:sp-perm=adm_D_parahoric} applied to  $W_{\{k\},D_m} \bs \wt W_{D_m} / W_{\{k\},D_m}$, we conclude that $\nu_k^{\wt w}$ satisfies \eqref{it:SP1'}, \eqref{it:SP2'}, and \eqref{it:SP3'} for $k \in I \smallsetminus \{1,m-1\}$.  It remains to consider the possibility that $1$, $m-1$, or both is in $I$.

If $1 \in I$, then $a_0$, $a_{0'} \in \ol F$.  Let $W_{\{0'\},D_m}$ denote the stabilizer of $a_{0'}$ in $W_{\aff,D_m}$.  Since $a_{0'}$ is a special vertex, we have
\[
   \wt W_{D_m} = X_{*D_m} \rtimes W_{\{0'\},D_m}
\]
and
\[
   W_{\{0'\},D_m} \bs \wt W_{D_m} / W_{\{0'\},D_m} \ciso W_{\{0'\},D_m} \bs X_{*D_m},
\]
which we identify with the set $X_{*D_m}^+$ of dominant elements in $X_{*D_m}$.  For $x = t_{\nu_0^x}\sigma$ with $\nu_0^x \in X_{*D_m}$ and $\sigma \in S_{2m}^\circ$, we have $t_{a_{0'}-\sigma a_{0'}}\sigma \in W_{\{0'\},D_m}$.  Hence $x = (t_{\nu_0^x + \sigma a_{0'} - a_{0'}}, t_{a_{0'}-\sigma a_{0'}}\sigma)$ according to the semidirect product decomposition $\wt W_{D_m} = X_{*D_m} \rtimes W_{\{0'\},D_m}$.  Hence the map
\[
   \wt W_{D_m} \to W_{\{0'\},D_m} \bs \wt W_{D_m} / W_{\{0'\},D_m} \ciso X_{*D_m}^+
\]
sends
\[
   x \mapsto \text{dominant form of } \nu_{0'}^x,
\]
where
\[
   \nu_{0'}^x := \nu_0^x + \sigma a_{0'} - a_{0'} = xa_{0'} - a_{0'}.
\]
The Bruhat order on $W_{\{0'\},D_m} \bs \wt W_{D_m} / W_{\{0'\},D_m}$ is just given by the dominance order on $X_{*D_m}^+$.  Thus our assumption that the image of $\wt w$ in $W_{\{0'\},D_m} \bs \wt W_{D_m} / W_{\{0'\},D_m}$ is $\mu$-admissible says exactly that $\nu_{0'}^{\wt w}$ satisfies \eqref{it:SP1'} and \eqref{it:SP2'} and $\nu_{0'}^{\wt w} \equiv \mu \bmod Q^\vee_{D_m}$.  This implies that $\nu_{0'}^{\wt w} \in \Conv(S_{2m}^\circ \mu)$.  Since we also have $0 \in I$ and we have already shown that $\nu_0^{\wt w}$ satisfies \eqref{it:SP1'}--\eqref{it:SP3'}, we conclude from \eqref{st:sickofnaminglemmas} that $\nu_1^{\wt w}$ satisfies \eqref{it:SP1'}--\eqref{it:SP3'}.

If $m - 1 \in I$, then one shows in a similar fashion that $\nu_{m-1}^{\wt w}$ satisfies \eqref{it:SP1'}--\eqref{it:SP3'}, using $a_m$ in place of $a_0$ and $a_{m'}$ in place of $a_{0'}$.
\end{proof}

\begin{rk}\label{rk:natural_def}
Our original definition of $\mu$-spin-permissibility \eqref{def:spin-perm_D}, and its parahoric variant \eqref{def:mu-spin_perm_D_parahoric}, is formulated in analogy with the definition of the Rapoport--Zink local model \cite{rapzink96}, which is in terms of totally ordered \emph{chains} (or more generally, multichains) of lattices.  However, from a root-theoretic perspective it is unnatural to emphasize the vectors $\mu_1^w$ and $\mu_{m-1}^w$, or what amounts to the same, the vectors $\nu_1^w$ and $\nu_{m-1}^w$, since the points $a_1$ and $a_{m-1}$ are not contained in minimal facets.  Instead it is more natural to consider the vectors $wa_k - a_k$ for all $k$ in some nonempty subset $I$ of $\{0,0',2,3,\dotsc,m-2,m,m'\}$.  We have just shown in the preceding theorem that if $F \preceq F'$ are the subfacets of $A_{D_m}$ whose closures contain the vertices $a_k$ for respective nonempty subsets $I \subset J$ of $\{0,0',2,3,\dotsc,m-2,m,m'\}$, then an element $w$ is in $\Adm_{F',F,D_m}(\mu)$ $\iff$ $w \equiv t_\mu \bmod W_{\aff,D_m}$ and $\wt wa_k - a_k$ satisfies \eqref{it:SP1'}, \eqref{it:SP2'}, and \eqref{it:SP3'} for all $k \in I$, for any representative $\wt w \in \wt W_{D_m}/W_{F,D_m}$; here for $k \in \{0,0',m,m'\}$, the spin condition \eqref{it:SP3'} just means $\wt wa_k - a_k \equiv \mu \bmod Q_{D_m}^\vee$ (which is equivalent to $w \equiv t_\mu \bmod W_{\aff,D_m}$).
\end{rk}

\subsection{\texorpdfstring{$\mu$}{mu}-permissibility in the parahoric case}\label{ss:mu-perm_parahoric}
In this final subsection of \s\ref{s:combinatorics_D}, we generalize \eqref{st:mu-perm_D} to give a characterization of $\mu$-permissibility in $\R_{D_m}$ in the general parahoric case.  See \eqref{def:mu-perm} for the general definition.  We write
\[
   \Perm_{F',F,D_m}(\mu)
\]
for the set of $\mu$-permissible elements in $W_{F',D_m} \bs \wt W_{D_m} / W_{F,D_m}$, and we abbreviate this to $\Perm_{D_m}(\mu)$ when $F = F' = A_{B_m}$.  As has been the case previously, we shall incur no loss of generality by restricting to subfacets of the base alcove $A_{D_m}$ satisfying property \eqref{disp:allowable_F_for_D_m}, so that the corresponding subgroups of $W_{\aff,D_m}$ are of the form $W_{I,D_m}$.

\begin{prop}\label{st:mu-perm_D_parahoric}
Let $\mu$ be the cocharacter \eqref{disp:mu_D}, let $I \subset J$ be nonempty subsets of $\{0,\dotsc,m\}$ satisfying \eqref{disp:allowable_I_for_D_m}, and let $w \in W_{J,D_m} \bs \wt W_{D_m} / W_{I,D_m}$.  Then $w$ is $\mu$-permissible $\iff$ $w \equiv t_\mu \bmod W_{\aff,D_m}$ and for one, hence any, representative $\wt w$ of $w$ in $\wt W_{D_m}/ W_{I,D_m}$, $\nu_k^{\wt w}$ satisfies \eqref{it:SP1'} and \eqref{it:SP2'} for all $k \in I$, and \eqref{it:SP3'} for all $k \in I \cap \{0,1,m-1,m\}$.
\end{prop}

\begin{proof}
As in the proof of \eqref{st:mu-perm_D}, this is obvious from \eqref{st:SP1-2'_conv_cond} and \eqref{st:sickofnaminglemmas}.
\end{proof}

We conclude by generalizing \eqref{st:exceptional_adm=perm} to the general parahoric case.

\begin{prop}\label{st:adm=perm_q=2}
Let $\mu$ be the cocharacter \eqref{disp:mu_D}, let $F \preceq F' \preceq A_{D_m}$, and suppose that $q \leq 2$.  Then
\[
   \Adm_{F',F,D_m}(\mu) = \Perm_{F',F,D_m}(\mu).
\]
\end{prop}

\begin{proof}
For $\theta \in \{\theta_0,\theta_m,\theta_0\theta_m\}$, conjugation by $\theta$ induces isomorphisms
\[
   \Adm_{F',F,D_m}(\mu) \isoarrow \Adm_{\theta F', \theta F,D_m}(\mu)
\]
and
\[
   \Perm_{F',F,D_m}(\mu) \isoarrow \Perm_{\theta F',\theta F, D_m}(\mu).
\]
So it suffices to prove the proposition in the case that $F$ and $F'$ satisfy \eqref{disp:allowable_F_for_D_m}.  This is done as in \eqref{st:exceptional_adm=perm}, using \eqref{st:sp-perm=adm_D_parahoric} in place of \eqref{st:sp-perm=adm_D}.
\end{proof}

\section{Type $B$ combinatorics}\label{s:combinatorics_B}

In this final section of the paper, we extend all of the combinatorial results from the previous section to type $B$.  The key result is \eqref{st:sp-perm=adm_B}, which characterizes $\mu$-admissibility in terms of a type $B$ analog of $\mu$-spin-permissibility, where $\mu$ is the cocharacter below; it is this result that is fed directly into the proofs of Theorems \ref{st:main_thm} and \ref{st:2nd_main_thm}, as explained earlier in the proof of \eqref{st:adm<=>perm_I}.

We fix an integer $m \geq 1$.  In \s\ref{ss:steinberg} we use the symbol $\mu$ to denote an arbitrary cocharacter, but from \s\ref{ss:R_B_m} on we take
\begin{equation}\label{disp:mu_B}
   \mu := \bigl(2^{q},1^{2m+1-2q},0^{(q)}\bigr), \quad 0 \leq q \leq m,
\end{equation}
except where specified to the contrary.  For $i \in \{1,\dotsc,2m+1\}$, we write $i^* = 2m + 2 - i$.  We denote by $e_1,\dotsc,$ $e_{2m+1}$ the standard basis in $\ZZ^{2m+1}$.

\subsection{Steinberg fixed-point root data}\label{ss:steinberg}

We begin \s\ref{s:combinatorics_B} with some generalities on Steinberg fixed-point root data.  We take as our main references the papers of Steinberg \cite{st68}, Kottwitz--Rapoport \cite{kottrap00}, and Haines--Ng\^o \cite{hngo02b}, especially \cite{hngo02b}*{\s9}.

Let us briefly recall what we need from the theory of Steinberg fixed-point root data.  Let $\R = (X^*,X_*,R,R^\vee,\Pi)$ be a reduced based root datum.%
\footnote{Strictly speaking, \cite{hngo02b} also assumes that \R is irreducible, but we shall not cite any facts from \cite{hngo02b} that require this assumption.}
Attached to \R are its Weyl group $W$, its affine Weyl group $W_\aff$, and its extended affine Weyl group $\wt W := X_* \rtimes W$, as in \s\ref{ss:I-W_gp}.  An \emph{automorphism} $\Theta$ of $\R$ is an automorphism $\Theta$ of the abelian group $X_*$ that stabilizes the set of coroots $R^\vee$ and such that the subsets $R$, $\Pi \subset X^*$ are stable under the dual automorphism $\Theta^*$ of $X^*$ induced by $\Theta$ and the perfect pairing $X^* \times X_* \to \ZZ$.  It follows that any automorphism of $\R$ induces automorphisms of $W$, $W_\aff$, and $\wt W$%, and that these induced automorphisms preserve the length functions on these groups
.  Attached to $\Theta$ is the \emph{Steinberg fixed-point root datum} $\R^{[\Theta]} = (X^{*[\Theta]}, X_*^{[\Theta]}, R^{[\Theta]}, R^{\vee [\Theta]}, \Pi^{[\Theta]})$; this is a reduced based root datum described explicitly in \cite{hngo02b}*{\s9}.  We systematically use a superscript $[\Theta]$ to denote the analogs for $\R^{[\Theta]}$ of all the objects defined for $\R$.  For our purposes, we shall just mention that $\wt W^{[\Theta]}$ is naturally a subgroup of $\wt W$, with $W^{[\Theta]}$ and $W_\aff^{[\Theta]}$ equal to the respective fixed-point subgroups $W^\Theta$ and $W^\Theta_\aff$, and that the cocharacter group $X_*^{[\Theta]}$ is the subgroup of $X_*$
\[
   X_*^{[\Theta]} := \bigl\{\, x \in X_* \bigm| \Theta(x) \equiv x \bmod Z \,\bigr\},
\]
where $Z := \{\, x \in X_* \mid \langle \alpha, x \rangle = 0\ \text{for all}\ \alpha \in R\,\}$.

Let $\A := X_* \otimes \RR$ and $\A^{[\Theta]} := X_*^{[\Theta]} \otimes \RR$ denote the respective apartments for \R and $\R^{[\Theta]}$.  For each facet $\mathbf f$ in \A, let
\[
   \mathbf f^{[\Theta]} := \mathbf f \cap \A^{[\Theta]}.
\]
We then have the following obvious generalization of \cite{hngo02b}*{Prop.\ 9.3}, which follows from the fact that the affine root hyperplanes for $\R^{[\Theta]}$ in $\A^{[\Theta]}$ are exactly the intersections of the affine root hyperplanes for \R with $\A^{[\Theta]}$.

\begin{lem}\label{st:steinberg_facets}\hfill
\begin{enumerate}
\renewcommand{\theenumi}{\roman{enumi}}
\item
For any facet $\mathbf f$ in \A, the set $\mathbf{f}^{[\Theta]}$ is either empty or a facet in $\A^{[\Theta]}$.
\item
Every facet in $\A^{[\Theta]}$ is of the form $\mathbf f^{[\Theta]}$ for a unique facet $\mathbf f$ in $\A$.
\item\label{it:whocares}
The rule $\Ba \mapsto \Ba^{[\Theta]}$ is a bijection from the set of alcoves in \A that meet $\A^{[\Theta]}$ to the set of alcoves in $\A^{[\Theta]}$.\qed
\end{enumerate}
\end{lem}

As something of a contrast to \eqref{it:whocares}, it can happen that non-minimal facets in \A intersect with $\A^{[\Theta]}$ to give minimal facets in $\A^{[\Theta]}$.

% Let \Ba be an alcove in \A meeting $\A^{[\Theta]}$, and let $\Omega_\Ba$ (resp. $\Omega_{\Ba^{[\Theta]}}$) be the stabilizer of \Ba (resp.\ $\Ba^{[\Theta]}$) in $\wt W$ (resp. $\wt W^{[\Theta]}$).  Then
% \[
%    \wt W = W_\aff \rtimes \Omega_\Ba
%    \quad\text{and}\quad
%    \wt W^{[\Theta]} = W_\aff^{[\Theta]} \rtimes \Omega_{\Ba^{[\Theta]}},
% \]
% and it is known that $\Omega_{\Ba^\Theta} = \Omega_\Ba \cap \wt W^{\Theta}$.

For any facet $\Bf$ in \A that meets $\A^{[\Theta]}$, it follows from the lemma that the stabilizer group $W_{\Bf^{[\Theta]}} \subset W_\aff^{[\Theta]} = W_\aff^\Theta$ is given by
\[
   W_{\Bf^{[\Theta]}} = W_\Bf^\Theta,
\]
where as in \s\ref{ss:I-W_gp} $W_\Bf$ denotes the stabilizer group of \Bf in $W_\aff$.  We then have the following.

\begin{lem}\label{st:double_cosets_inj}
Let $\mathbf{f_1}$ and $\mathbf{f_2}$ be facets in $\A$ that meet $\A^{[\Theta]}$.  Then the natural map
\[
   W_{\mathbf{f}_{\mathbf 1}^{[\Theta]}} \big\bs \wt W^{[\Theta]} \big/ W_{\mathbf{f}_{\mathbf 2}^{[\Theta]}} \to W_{\mathbf{f_1}} \bs \wt W / W_{\mathbf{f_2}}
\]
is an injection.
\end{lem}

Note that in the special case that $\mathbf{f_1} = \mathbf{f_2}$ is a special minimal facet in $\A$ meeting $\A^{[\Theta]}$, the displayed sets of double cosets identify with the sets of dominant cocharacters in $X_*^{[\Theta]}$ and $X_*$, respectively, relative to the choice of a positive Weyl chamber in $\A$ that meets $\A^{[\Theta]}$; and the displayed map is then the natural injection on dominant cocharacters.

\begin{proof}[Proof of \eqref{st:double_cosets_inj}]
Let us begin with a few generalities.  By a ``hyperplane'' in $\A$ we shall mean an ``affine root hyperplane for \R.''  For any subset $A$ in $\A$, we say that a hyperplane $H$ is an \emph{$A$-hyperplane} if $H$ contains $A$; we let $\H_A$ denote the (possibly empty) set of $A$-hyperplanes in $\A$; and we let $W_A$ denote the (possibly trivial) subgroup of $W_\aff$ generated by the reflections across the $A$-hyperplanes.  We refer to the connected components of the set $\A \smallsetminus \bigcup_{H \in \H_A}H$ as \emph{$A$-chambers}.  By general properties of reflection groups, $W_A$ acts simply transitively on the set of $A$-chambers.  If $A \subset \A^{[\Theta]}$, then $\Theta$ acts on $\H_A$, and hence on the set of $A$-chambers; an $A$-chamber $C$ meets $\A^{[\Theta]}$ if and only if $C$ is $\Theta$-stable; and $W_A^\Theta$ acts simply transitively on the set of $\Theta$-stable $A$-chambers.

To prove the lemma, our problem is to show that if $w' = x_1wx_2$ with $w$, $w' \in \wt W^{[\Theta]}$, $x_1 \in W_{\mathbf{f_1}}$, and $x_2 \in W_{\mathbf{f_2}}$, then there exist $y_1 \in W_{\mathbf{f_1}}^\Theta$ and $y_2 \in W_{\mathbf{f_2}}^\Theta$ such that $w' = y_1wy_2$.  Let $\mathbf a$ be an alcove in $\A$ meeting $\A^{[\Theta]}$ and such that $\mathbf{f_2} \preceq \mathbf a$.  Let $\Omega_{\mathbf a}$ (resp.\ $\Omega_{\Ba^{[\Theta]}}$) denote the stabilizer of $\mathbf a$ (resp.\ $\Ba^{[\Theta]}$) in $\wt W$ (resp.\ $\wt W^{[\Theta]}$), so that $\wt W = W_\aff \rtimes \Omega_{\mathbf a}$ and $\wt W^{[\Theta]} = W_\aff^\Theta \rtimes \Omega_{\Ba^{[\Theta]}}$.  Write $w = x\omega$ and $w' = x'\omega'$ with $x$, $x' \in W_\aff^\Theta$ and $\omega$, $\omega' \in \Omega_{\mathbf a^{[\Theta]}}$.  Then $\omega = \omega'$ and
\[
   x'\omega = w' = x_1wx_2 = x_1x\omega x_2\omega^{-1}\omega.
\]
We have $\omega x_2\omega^{-1} \in W_{\omega\mathbf{f_2}}$, and $\omega\mathbf{f_2}$ meets $\A^{[\Theta]}$ since $\omega \in \wt W^{[\Theta]}$.  Thus it suffices to solve the analogous problem (with $\omega \mathbf{f_2}$ in place of $\mathbf{f_2}$) for the equation $x' = x_1x(\omega x_2\omega^{-1})$.  So, since $\mathbf{f_2}$ is arbitrary, we might as well assume that $w$, $w' \in W_\aff^\Theta$ to begin with, which we shall do for the rest of the proof.

Consider the alcoves $x_1w\mathbf a$ and $w'\mathbf a$, which share the common subfacet $x_1w\mathbf{f_2}$.  Let $A$ denote the intersection of the $\mathbf{f_1}$-hyperplanes containing $x_1w\mathbf{f_2}$.  If $A$ is empty, which is to say that there are no $\mathbf{f_1}$-hyperplanes containing $x_1w\mathbf{f_2}$, then $x_1w\mathbf a$ and $w'\mathbf a$ must be contained in the same $\mathbf{f_1}$-chamber $C$, since any hyperplane that separates two facets in $\A$ must contain all the common subfacets of both.  Then $C$ and $x_1^{-1}C$ both meet $\A^{[\Theta]}$, since they respectively contain $w'\mathbf a$ and $w\mathbf a$.  Hence $x_1 \in W_{\mathbf{f_1}}^\Theta$, and we conclude immediately that $x_2 \in W_{\mathbf{f_2}}^\Theta$ as well.

If $A$ is nonempty, then tautologically $\mathbf{f_1} \subset A$, whence $W_A \subset W_{\mathbf{f_1}}$.  Let $z_1$ be the unique element in $W_A$ such that $z_1x_1w\mathbf a$ is contained in the same $A$-chamber as $w'\mathbf a$.  Since $x_1w\mathbf{f_2} \subset A$ and $A$ is fixed by the elements of $W_A$, we have $x_1w\mathbf{f_2} \preceq z_1x_1w\mathbf a$.  Since we also have $x_1w\mathbf{f_2} \preceq w'\mathbf a$, we conclude from the definition of $A$ that $z_1x_1w\mathbf a$ and $w'\mathbf a$ must be contained in the same $\mathbf{f_1}$-chamber.  As in the previous paragraph, it follows that $y_1 := z_1x_1 \in W_{\mathbf{f_1}}^\Theta$.  Since the alcoves that share the subfacet $x_1w\mathbf{f_2}$ with $y_1w\mathbf a$ are precisely those of the form $y_1wy\mathbf a$ for $y \in W_{\mathbf{f_2}}$, there exists $y_2 \in W_{\mathbf{f_2}}$ such that $y_1wy_2 = w'$, and it follows immediately that $y_2$ is $\Theta$-fixed as well.
\end{proof}

Now fix an alcove $\mathbf a$ in $\A$, and consider the associated Bruhat order $\leq$ and length function on $\wt W$, as in \s\ref{ss:bo}.  Let $\mathbf{f_1}$ and $\mathbf{f_2}$ be subfacets of $\mathbf a$.  Then every double coset $w \in W_{\mathbf{f_1}}\bs \wt W / W_{\mathbf{f_2}}$ contains a unique representative $\tensor[^{\mathbf{f_1}}]w{^{\mathbf{f_2}}}$ of minimal length in $\wt W$, which moreover has the property that $\tensor[^{\mathbf{f_1}}]w{^{\mathbf{f_2}}} \leq w'$ for all $w' \in w$.  In the Steinberg fixed-point setting, we then have the following result.

\begin{lem}\label{st:min_length_rep}
Let $\mathbf{f_1}$, $\mathbf{f_2} \preceq \Ba$, and suppose that all three meet $\A^{[\Theta]}$.  Let $w \in W_{\mathbf{f_1}} \bs \wt W / W_{\mathbf{f_2}}$ be a double coset contained in the image of $W_{\mathbf{f}_{\mathbf 1}^{[\Theta]}}\big\bs \wt W^{[\Theta]} \big/ W_{\mathbf{f}_{\mathbf 2}^{[\Theta]}}$.  Then  $\tensor[^{\mathbf{f_1}}]w{^{\mathbf{f_2}}} \in \wt W^{[\Theta]}$.
\end{lem}

\begin{proof}
This is proved for example in \cite{sm11d}*{Lem.\ 6.2.1}.
\end{proof}

We use the lemma to obtain the following double coset version of a result of Kottwitz--Rapoport and Haines--Ng\^o.

\begin{thm}\label{st:bruhat_inherited}
Let $\mathbf a$, $\mathbf{f_1}$, and $\mathbf{f_2}$ be as in \eqref{st:min_length_rep}, and let $\leq^{[\Theta]}$ denote the Bruhat order on $W_{\mathbf{f}_{\mathbf 1}^{[\Theta]}} \big\bs \wt W^{[\Theta]} \big/ W_{\mathbf{f}_{\mathbf 2}^{[\Theta]}}$ determined by the alcove $\mathbf a^{[\Theta]} \subset \A^{[\Theta]}$.  Regard $W_{\mathbf{f}_{\mathbf 1}^{[\Theta]}} \big\bs \wt W^{[\Theta]} \big/ W_{\mathbf{f}_{\mathbf 2}^{[\Theta]}}$ as a subset of $W_{\mathbf{f_1}} \bs \wt W / W_{\mathbf{f_2}}$ via \eqref{st:double_cosets_inj}, and let $w$, $x \in W_{\mathbf{f}_{\mathbf 1}^{[\Theta]}} \big\bs \wt W^{[\Theta]} \big/ W_{\mathbf{f}_{\mathbf 2}^{[\Theta]}}$.  Then
\[
   w \leq^{[\Theta]} x \text{ in } W_{\mathbf{f}_{\mathbf 1}^{[\Theta]}} \big\bs \wt W^{[\Theta]} \big/ W_{\mathbf{f}_{\mathbf 2}^{[\Theta]}}
   \iff
   w \leq x \text{ in } W_{\mathbf{f_1}}\bs \wt W / W_{\mathbf{f_2}}.
\]
\end{thm}

\begin{proof}
By definition of the Bruhat order on $W_{\mathbf{f_1}}\bs \wt W / W_{\mathbf{f_2}}$,
\[
   w \leq x \text{ in } W_{\mathbf{f_1}}\bs \wt W / W_{\mathbf{f_2}} 
   \iff
   \tensor[^{\mathbf{f_1}}]w{^{\mathbf{f_2}}} \leq \tensor[^{\mathbf{f_1}}]x{^{\mathbf{f_2}}} \text{ in } \wt W.
\]
By \eqref{st:min_length_rep}, $\tensor[^{\mathbf{f_1}}]w{^{\mathbf{f_2}}}$, $\tensor[^{\mathbf{f_1}}]x{^{\mathbf{f_2}}} \in \wt W^{[\Theta]}$.  Hence by the Iwahori version of the theorem \cite{hngo02b}*{Prop.\ 9.6} (which generalizes the version of the theorem for affine Weyl groups proved by Kottwitz--Rapoport \cite{kottrap00}*{Prop.\ 2.3}), the right-hand condition in the display holds
\[
   \iff
   \tensor[^{\mathbf{f_1}}]w{^{\mathbf{f_2}}}
   \leq^{[\Theta]} 
   \tensor[^{\mathbf{f_1}}]x{^{\mathbf{f_2}}} \text{ in } \wt W^{[\Theta]}.
\]
By the Iwahori version of the theorem again, $\tensor[^{\mathbf{f_1}}]w{^{\mathbf{f_2}}}$ and $\tensor[^{\mathbf{f_1}}]x{^{\mathbf{f_2}}}$ are the minimal length representatives in $\wt W^{[\Theta]}$ of their respective $(W_{\mathbf f_{\mathbf 1}^{[\Theta]}}, W_{\mathbf f_{\mathbf 2}^{[\Theta]}})$-cosets.  Hence the last display holds
\[
   \iff
   w \leq^{[\Theta]} x \text{ in } W_{\mathbf{f}_{\mathbf 1}^{[\Theta]}} \big\bs \wt W^{[\Theta]} \big/ W_{\mathbf{f}_{\mathbf 2}^{[\Theta]}}. \qedhere
\]
\end{proof}

For use later on when we come to $\mu$-permissibility, we cite the following result, which is proved for example in \cite{sm11d}*{Lem.\ 6.5.1}.

\begin{lem}\label{st:convex_hulls}
Let $\mu$ be any cocharacter in $X_*^{[\Theta]}$.  Then
\begin{flalign*}
   \phantom{\qed}& & \Conv(W^{[\Theta]} \mu) = \Conv(W\mu) \cap \A^{[\Theta]}. & & \qed
\end{flalign*}
\end{lem}

\subsection{The root datum \texorpdfstring{$\R_{B_m}$}{R\_$\{$B\_m$\}$}}\label{ss:R_B_m}
In this subsection we define the root datum $\R_{B_m}$.  Let
\begin{align*}
   X_{*B_m} &:= X_{*2m+1}\\
    &\phantom{:}= \biggl\{\,(x_1,\dotsc,x_{2m+1}) \in \ZZ^{2m+1} \biggm|
      \begin{varwidth}{\textwidth}
         \centering
         $x_1+x_{2m+1} = x_2+x_{2m} = \dotsb =$\\
         $= x_m+x_{m+2} = 2x_{m+1}$
      \end{varwidth}\,\biggr\},
\end{align*}
where we recall $X_{*2m+1}$ from \eqref{disp:X_*n},
and
\[
   X^*_{B_m} := \ZZ^{2m+1}\big/
       \bigl\{(x_1,\dotsc,x_{m},\textstyle{-2\sum_{i=1}^{m}x_i}, x_{m},\dotsc,x_1)\bigr\}.
\]
Then the standard dot product on $\ZZ^{2m+1}$ induces a perfect pairing $X^*_{B_m} \times X_{*B_m} \to \ZZ$.  For $1 \leq i,j \leq 2m+1$ with $j \neq i$, $i^*$, let
\[
   \begin{matrix}
   \alpha_{i,j}\colon
   \xymatrix@R=0ex{
      X_{*B_m} \ar[r] & \ZZ\\
      (x_1,\dotsc,x_{2m+1}) \ar@{|->}[r] & x_i - x_j
   }
   \end{matrix}
\]
and
\[
   \alpha_{i,j}^\vee := 
   \begin{cases}
      e_i-e_j+e_{j^*}-e_{i^*},  & i,j \neq m+1;\\
      2e_i - 2e_{i^*},  & j = m+1;\\
      -2e_j + 2e_{j^*},  & i = m+1.
   \end{cases}
\]
Then we may regard $\alpha_{i,j}$ as an element of $X^*_{B_m}\ciso \Hom(X_{*B_m},\ZZ)$, and we let
\begin{equation}\label{disp:roots_B}
   \Phi_{B_m}:= \{\alpha_{i,j}\}_{j\neq i,i^*} \subset X^*_{B_m}
   \quad\text{and}\quad
   \Phi^\vee_{B_m} := \{\alpha^\vee_{i,j}\}_{j\neq i,i^*} \subset X_{*B_m}.
\end{equation}
The objects so defined form a root datum
\[
   \R_{B_m} := (X^*_{B_m},X_{*B_m},\Phi_{B_m},\Phi^\vee_{B_m}),
\]
which is the root datum for split $GO_{2m+1}$.  We take as simple roots
\[
   \alpha_{1,2}, \alpha_{2,3},\dotsc,\alpha_{m,m+1}.
\]

The Weyl group of $\R_{B_m}$ identifies canonically with $S_{2m+1}^*$.  As discussed in \s\ref{ss:I-W_gp}, we then have the extended affine Weyl group
\[
   \wt W_{B_m} := X_{*B_m} \rtimes S_{2m+1}^*.
\]
The affine Weyl group $W_{\aff,B_m}$ is the subgroup $Q^\vee_{B_m} \rtimes S_{2m+1}^* \subset \wt W_{B_m}$, where $Q^\vee_{B_m} \subset X_{*B_m}$ is coroot lattice.  Explicitly,
\[
   Q^\vee_{B_m} = \biggl\{\,(x_1,\dotsc,x_{2m+1}) \in X_{*B_m} \biggm|
   \begin{varwidth}{\textwidth}
      \centering
      $x_i + x_{i^*} = 0$ for all $i$ and\\
      $x_1 + \dotsb + x_m$ is even
   \end{varwidth}\,\biggr\}.
\]

Let
\[
   \A_{B_m}:= X_{*B_m} \otimes_\ZZ \RR \subset \RR^{2m+1}.
\]
We take as positive Weyl chamber the chamber in $\A_{B_m}$ on which the simple roots are all positive.
We take as our base alcove
\begin{equation}\label{disp:A_B_m}
   A_{B_m} := \biggl\{\,(x_1,\dotsc,x_{2m+1}) \in \RR^{2m+1} \biggm|
      \begin{varwidth}{\textwidth}
         \centering
         $x_1 + x_{2m+1} = \dotsb = x_m + x_{m+2} = 2x_{m+1}$\\
         and $x_1, x_{2m+1}-1 < x_2 < x_3 < \dotsb < x_{m+1}$
      \end{varwidth}\,\biggr\};
\end{equation}
note that this is the unique alcove contained in the Weyl chamber \emph{opposite} the positive chamber and whose closure contains the origin.  The minimal facets of $A_{B_m}$ are the lines
\[
   a + \RR\cdot \bigl(1,\dotsc,1\bigr)
\]
for $a$ one of the vertices
\begin{equation}\label{disp:B_m_verts}
\begin{aligned}
   a_k &:= \bigl((-\tfrac 1 2)^{(k)},0^{(2m + 1 - 2k)},(\tfrac 1 2)^{(k)}\bigr)
   \quad\text{for}\quad k = 0,2,3,\dotsc,m,\\
   a_{0'} &:= \bigl(-1,0^{{2m-1}},1\bigr).
\end{aligned}
\end{equation}
As discussed in \s\ref{ss:bo}, our choice of $A_{B_m}$ endows $W_{\aff,B_m}$ and $\wt W_{B_m}$ with Bruhat orders.  We also define
\[
   a_1 := \bigl(-\tfrac1 2, 0^{(2m-1)}, \tfrac 1 2\bigr).
\]

Note that in terms of our earlier notation \eqref{disp:wtW_n}, we have $\wt W_{B_m} = \wt W_{2m}$ on the nose.  Thus for each $w \in \wt W_{B_m}$ we get we get a face $\mathbf v$ of type $(2m+1,\{0,\dotsc,m\})$ by letting $w$ act on the standard face $\omega_{\{0,\dotsc,m\}}$, as in \s\ref{ss:faces_type_I}.  We write $\mu_k^w$ for the vector $\mu_k^{\mathbf{v}}$ attached to this face,
\[
   \mu_k^w := w\omega_k - \omega_k
   \quad\text{for}\quad
   k \in \ZZ.
\]
The rule $w \mapsto w\cdot \omega_{\{0,\dotsc,m\}}$ identifies $\wt W_{B_m}$ with the faces of type $(2m+1,\{0,\dotsc,m\})$.

\subsection{Relation to \texorpdfstring{$\R_{D_{m+1}}$}{R\_$\{$D\_m+1$\}$}}\label{ss:relation_to_D_m}
Consider the root datum $\R_{D_{m+1}}$ defined in \s\ref{ss:type_D_variant}, and let $\Theta := \theta_{m+1}$ be the automorphism on $X_{*D_{m+1}}$ defined in \eqref{disp:thetas} (with $m$ replaced by $m+1$),
\[
   \Theta \colon (x_1,\dotsc,x_{2m+2}) \mapsto (x_1,\dotsc,x_m,x_{m+2},x_{m+1}, x_{m+3},x_{m+4},\dotsc,x_{2m+2}).
\]
Then $\Theta$ is an automorphism of $\R_{D_{m+1}}$ in the sense of \s\ref{ss:steinberg}, where we take the roots \eqref{disp:simple_roots_D} (with $m$ replaced by $m+1$) as simple roots in $\Phi_{D_{m+1}}$.
The map on cocharacters
\[
   \iota\colon
   \xymatrix@R=0ex{
      X_{*B_m} \ar[r]  & X_{*D_{m+1}}\\
      (x_1,\dotsc,x_{2m+1}) \ar@{|->}[r]  & (x_1,\dotsc,x_m,x_{m+1},x_{m+1},x_{m+2},\dotsc,x_{2m+1})
   }
\]
identifies $X_{*B_m} \isoarrow X_{*D_{m+1}}^{[\Theta]}$, and we leave it to the reader to verify that $\iota$ induces $\R_{B_m}\iso \R_{D_{m+1}}^{[\Theta]}$.

Regarding $\A_{B_m} = X_{*B_m}\otimes \RR$ as a subspace of $\A_{D_{m+1}} = X_{*D_{m+1}}\otimes \RR$ via $\iota \otimes \id_\RR$, note that $A_{B_m} = A_{D_{m+1}} \cap \A_{B_m}$.  Moreover, $\iota \otimes \id_\RR$ identifies each of the points $a_0$, $a_{0'}$, $a_2,\dotsc,$ $a_m$ defined in \eqref{disp:B_m_verts} with the point in $\A_{D_{m+1}}$ denoted by the same symbol, as defined in \eqref{disp:D_m_verts} and \eqref{disp:a_k} (with $m$ replaced by $m+1$).  
% Thus there is no inconsistency in our notation for these points.

The $\Theta$-fixed Weyl group $(S_{2m+2}^\circ)^{[\Theta]} = (S_{2m+2}^\circ)^\Theta$ consists of the $\sigma \in S_{2m+2}^\circ$ such that $\sigma(m+1) \in \{m+1,m+2\}$.  It identifies with $S_{2m+1}^*$ via restriction to $X_{*B_m}$ along $\iota$.  Let us describe the induced embedding $S_{2m+1}^* \inj S_{2m+2}^\circ$ more explicitly.  Given $\sigma \in S_{2m+1}^*$, $\sigma$ necessarily fixes $m+1$, and we view $\sigma$ as a permutation on $\{1,\dotsc,m,m+3,\dotsc,2m+2\}$ by identifying
\[
   \xymatrix@R=0ex{
      \{1,\dotsc,m,m+2,\dotsc,2m+1\} \ar[r]^-\sim
         & \{1,\dotsc,m,m+3,\dotsc,2m+2\}\\
      i \ar@{|->}[r]  
         & {\begin{cases} i, & i \leq m;\\ 
                          i+1, & m+2 \leq i. \end{cases}}
   }
\]
To define $\sigma$ on all of $\{1,\dotsc,2m+2\}$, we then declare that $\sigma$ fixes or interchanges $m+1$ and $m+2$ as needed to make $\sigma$ even in $S_{2m+2}$.

Writing $\iota'$ for the embedding $S_{2m+1} \inj S_{2m+2}^\circ$ just described, $\iota$ and $\iota'$ combine to specify the embedding
\[
   \phi\colon
   \wt W_{B_m} = X_{*B_m} \rtimes S_{2m+1}^* \xinj{\iota \rtimes \iota'}
   X_{*D_{m+1}} \rtimes S_{2m+2}^\circ = \wt W_{D_{m+1}}.
\]
Let $w \in \wt W_{B_m}$ and $0 \leq k \leq m$.  In terms of our identifications of $\wt W_{B_m}$ and $\wt W_{D_{m+1}}$ with faces of type $(2m+1,\{0,\dotsc,m\})$ and $(2m+2,\{0,\dotsc,m+1\})$, respectively, $\phi$ sends
\begin{equation}\label{disp:mu_k-vects_map}
\begin{split}
   \mu_k^w = (x_1,\dotsc, x_{2m+1}&) \mapsto\\
       & (x_1,\dotsc,x_{m},x_{m+1},x_{m+1},x_{m+2},\dotsc,x_{2m+1}) = \mu_k^{\phi(w)}.
\end{split}
\end{equation}
In this way $\phi$ identifies $\wt W_{B_m}$ with the set of $x \in \wt W_{D_m}$ such that $\mu_k^x(m+1) = \mu_k^x(m+2)$ for all $0 \leq k \leq m$.  (Note however that for such an $x$, one may have $\mu_{m+1}^x(m+1) \neq \mu_{m+1}^x(m+2)$.)

\subsection{Parahoric variants}
We continue to embed
\[
   X_{*B_m} \subset X_{*D_{m+1}},\quad
   \A_{B_m} \subset \A_{D_{m+1}},
   \quad\text{and}\quad
   \wt W_{B_m} \subset \wt W_{D_{m+1}}
\]
as in the previous subsection.
Let $F$ be a subfacet of the base alcove $A_{B_m}$, and let $W_{F, B_m}$ denote the common stabilizer and pointwise fixer of $F$ in $W_{\aff,B_m}$.  By \eqref{st:steinberg_facets} $F$ is of the form $\wt F \cap \A_{B_m}$ for a uniquely determined subfacet $\wt F$ of $A_{D_{m+1}}$, and
\[
   W_{F,B_m} = W_{\wt F,D_{m+1}}^\Theta = W_{\wt F,D_{m+1}} \cap W_{\aff,B_m} = W_{\wt F, D_{m+1}} \cap \wt W_{B_m}.
\]
The affine automorphism $\theta_0$ \eqref{disp:thetas} on $\A_{D_{m+1}}$ stabilizes $\A_{B_m}$ and interchanges $a_0$ and $a_{0'}$, while fixing $a_k$ for all $2 \leq k \leq m$.  In analogy with \s\ref{ss:parahoric_D}, by applying $\theta_0$ as needed, we shall therefore incur no loss of generality by working only with facets $F$ with the property that if $a_{0'} \in \ol F$, then $a_0 \in \ol F$.

Continuing the analogy with \s\ref{ss:parahoric_D}, given a subfacet $F \preceq A_{B_m}$ such that $a_0$ is a vertex of $F$ if $a_{0'}$ is, let
\[
   I := \bigl\{\, k \in \{0,\dotsc,m\} \bigm| a_k \in \ol F \,\bigr\}
\]
and
\begin{equation}\label{disp:W_I,B_m}
   W_{I,B_m} := \{\, w \in W_{\aff,B_m} \mid wa_k = a_k \text{ for all } k \in I \,\}.
\end{equation}
Then, arguing as in \s\ref{ss:parahoric_D}, the association $F \mapsto I$ gives a bijection from the set of such subfacets $F$ to the set of nonempty subsets $I$ of $\{0,\dotsc,m\}$ with the property that
\begin{equation}\label{disp:I_cond_B}
   1 \in I \implies 0 \in I;
\end{equation}
and for $F$ corresponding to $I$, we have
\[
   W_{F,B_m} = W_{I,B_m}.
\]
In terms of our notation \eqref{disp:W_I,D_m} (with $m$ replaced by $m+1$), we have
\[
   W_{I,B_m} = W_{\wt I,D_{m+1}}^\Theta = W_{\wt I,D_{m+1}} \cap \wt W_{B_m},
\]
where $\wt I$ is the nonempty subset of $\{0,\dotsc,m+1\}$
\begin{equation}\label{disp:wtI}
   \wt I :=
   \begin{cases}
      I \cup \{m+1\}, & m\in I;\\
      I, & m \notin I.
   \end{cases}
\end{equation}

By the same reasoning as in the paragraph before \eqref{st:parahoric_W_faces_type_I}, the elements of $W_{I,B_m}$ fix the standard face $\omega_I$ of type $(2m+1,I)$.  Thus to any $w \in \wt W_{B_m}/W_{I,B_m}$ we may attach a well-defined face $w \cdot \omega_I$ of type $(2m+1,I)$, and we write
\[
   \mu_k^w := w \omega_k - \omega_k
   \quad\text{for}\quad
   k \in I.
\]
The embedding $\phi\colon \wt W_{B_m} \subset \wt W_{D_{m+1}}$ induces an embedding
\[
   \wt W_{B_m}/W_{I,B_m} \subset \wt W_{D_{m+1}}/ W_{\wt I,D_{m+1}},
\]
and of course the corresponding map on $\mu_k$-vectors is again given by \eqref{disp:mu_k-vects_map}, for $k \in I$.  By definition of $W_{I,B_m}$, for each $w \in \wt W_{B_m}/W_{I,B_m}$ we also get a well-defined vector
\[
   \nu_k^w := wa_k - a_k
   \quad\text{for}\quad
   k \in I.
\]
If $1 \in I$, then the elements of $W_{I,B_m}$ also fix $a_{0'}$ (because $a_{0'}$ is a vertex of the facet containing $a_1$), and we define
\[
   \nu_{0'}^w := wa_{0'} - a_{0'}
\]
for $w \in \wt W_{B_m}/ W_{I,B_m}$.

\subsection{\texorpdfstring{$\mu$}{mu}-spin-permissibility}

In this subsection we define $\mu$-spin-permissibility in $\wt W_{B_m}$ and its parahoric variants.  We continue with the notation of the previous subsections.  Recall the set $A_k \subset \{1,\dotsc,2m+1\}$ from \eqref{disp:A_i_B_i}.

\begin{defn}\label{def:spin-permissible_B_m}
Let $\mu$ be the cocharacter \eqref{disp:mu_B}, and let $I \subset J$ be nonempty subsets of $\{0,\dotsc,m\}$ satisfying property \eqref{disp:I_cond_B}.  We say that $w \in W_{J,B_m} \bs \wt W_{B_m}/W_{I,B_m}$ is \emph{$\mu$-spin-permissible} if $w \equiv t_\mu \bmod W_{\aff,B_m}$ and for one, hence any, representative $\wt w$ of $w$ in $\wt W_{B_m}/W_{I,B_m}$, $\wt w$ satisfies the following conditions for all $k \in I$.
\begin{enumerate}
\renewcommand{\theenumi}{sp\arabic{enumi}}
\item\label{it:sp1}
   $\mu_k^{\wt w} + (\mu_{-k}^{\wt w})^* = \mathbf 2$ and $\mathbf 0 \leq \mu_k^{\wt w} \leq \mathbf{2}$.
\item\label{it:sp2}
   $\#\bigl\{\, j\bigm| \mu^{\wt w}_k(j) = 2\,\bigr\} \leq q$.
\item\label{it:sp3}
   (spin condition) If $\mu_k^{\wt w}$ is self-dual and $\mu_k^{\wt w} \not\equiv \mu \bmod Q^\vee_{B_m}$, then there exists $j \in A_k$ such that $\mu_k^{\wt w}(j) = 1$.
\end{enumerate}
We write
\[
   \SPerm_{J,I,B_m}(\mu)
\]
for the set of $\mu$-spin-permissible elements in $W_{J,B_m} \bs \wt W_{B_m} / W_{I,B_m}$.  We abbreviate this to $\SPerm_{I,B_m}(\mu)$ when $J = \{0,\dotsc,m\}$, and to $\SPerm_{B_m}(\mu)$ when $I = J = \{0,\dotsc,m\}$.
\end{defn}

In analogy with the type $D$ situation, we shall verify in a moment that
% , in analogy with the type $D$ setting, 
the definition does not depend on the choice of representative $\wt w$.  Let us first make a couple of other remarks.

First, the condition $w \equiv t_\mu \bmod W_{\aff,B_m}$ is well-defined because $W_{J,B_m}$, $W_{I,B_m} \subset W_{\aff,B_m}$.  And, as in the type $D$ setting, it is implied by the spin condition when $0 \in I$, since $A_0 = \emptyset$.  

Second, in contrast with the spin condition \eqref{it:SP3} for $\R_{D_m}$, the set $B_k$ does not appear in \eqref{it:sp3}.  Indeed if \eqref{it:sp1} holds, then necessarily $\mu_k^{\wt w}(m+1) = 1$ by \eqref{rk:whatevs}. Since $m + 1 \in B_k$ for $k$ in the range $0 \leq k \leq m$, the existence of a $j' \in B_k$ such that $\mu_k^{\wt w}(j') = 1$ is then automatic.  This point can also be understood via the following lemma.

\begin{lem}\label{st:spin-perm_intersection}
Let $I$ be a nonempty subset of $\{0,\dotsc,m\}$ satisfying property \eqref{disp:I_cond_B}, let  $w \in \wt W_{B_m} / W_{I,B_m}$, and consider the embedding
\[
   \phi\colon \wt W_{B_m} / W_{I,B_m} \inj \wt W_{D_{m+1}}/ W_{\wt I, D_{m+1}},
\]
with $\wt I$ as in \eqref{disp:wtI}.  Then
\begin{enumerate}
\renewcommand{\theenumi}{\roman{enumi}}
\item\label{it:W_aff_equiv}
   $w \equiv t_\mu \bmod W_{\aff,B_m} \iff \phi(w) \equiv t_{\iota(\mu)} \bmod W_{\aff,D_{m+1}}$;
\item\label{it:sp1_equiv}
   for $k \in I$, $\mu_k^w$ satisfies \eqref{it:sp1} $\iff$ $\mu_k^{\phi(w)}$ satisfies \eqref{it:SP1}; and
\item\label{it:sp2-3_equivs}
   for $k \in I$, if $\mu_k^w$ satisfies \eqref{it:sp1}, then $\mu_k^w$ satisfies \eqref{it:sp2} (resp.\ \eqref{it:sp3}) relative to $\mu$ $\iff$ $\mu_k^{\phi(w)}$ satisfies \eqref{it:SP2} (resp.\ \eqref{it:SP3}) relative to $\iota(\mu)$.
\end{enumerate}
Furthermore $\phi$ gives an identification
\[
   \SPerm_{I,B_m}(\mu) = \SPerm_{\wt I,D_{m+1}}\bigl(\iota(\mu)\bigr) \cap \wt W_{B_m}/W_{I,B_m}.
\]
\end{lem}

\begin{proof}
Part \eqref{it:W_aff_equiv} is obvious because, by general properties of fixed-point root data, $W_{\aff,B_m} = W_{\aff,D_{m+1}} \cap \wt W_{B_m}$.
Parts \eqref{it:sp1_equiv} and \eqref{it:sp2-3_equivs} are obvious from the explicit form of the map on $\mu_k$-vectors \eqref{disp:mu_k-vects_map}.
Thus the real content of the lemma is contained in the final asserted equality,  which is an immediate consequence of \eqref{it:W_aff_equiv}--\eqref{it:sp2-3_equivs} and the following lemma.
\end{proof} 

\begin{lem}\label{st:automatic_spin_cond}
Under the assumptions of the previous lemma, suppose that $m \in I$, $\mu_m^w$ satisfies \eqref{it:sp1} and \eqref{it:sp2}, and $w \equiv t_\mu \bmod W_{\aff,B_m}$.  Then $\mu_m^w$ also satisfies \eqref{it:sp3}, and $\mu_{m+1}^{\phi(w)}$ satisfies \eqref{it:SP1}--\eqref{it:SP3}.
\end{lem}

\begin{proof}
Before getting started, we note that since $\mu_m^w$ satisfies \eqref{it:sp1}, we have
\begin{equation}\label{disp:whocares}
   \mu_m^{\phi(w)}(m+1) = \mu_m^{\phi(w)}(m+2) = 1.
\end{equation}
   
Let us first show that $\mu_m^w$ satisfies \eqref{it:sp3}, or, equivalently, that $\mu_m^{\phi(w)}$ satisfies \eqref{it:SP3}.  Recall the integer $c_m^{\phi(w)}$ from \eqref{def:c_i^w}.  If $c_m^{\phi(w)} \equiv q \bmod 2$, then $\mu_m^{\phi(w)}$ trivially satisfies \eqref{it:SP3} by \eqref{st:c_i_equiv_q}.  If $c_m^{\phi(w)} \not\equiv q \bmod 2$, then $c_m^{\phi(w)} < q$ since $\mu_m^{\phi(w)}$ satisfies \eqref{it:SP2}.  And since $q \leq m$, $\# A_m = 2m$, and $\mu_k^{\phi(w)}$ satisfies \eqref{it:SP1}, there must exist $j \in A_m$ such that $\mu_k^{\phi(w)}(j) = 1$.  This combines with \eqref{disp:whocares} to show that $\mu_{m}^{\phi(w)}$ satisfies \eqref{it:SP3}.
   
It remains to show that $\mu_{m+1}^{\phi(w)}$ satisfies \eqref{it:SP1}--\eqref{it:SP3}.
It follows from the form of the embedding $\iota'\colon S_{2m+1}^* \inj S_{2m+2}^\circ$ in \s\ref{ss:relation_to_D_m} and the recursion relation \eqref{disp:mu_recursion} that
\[
   \mu_{m+1}^{\phi(w)} = \mu_{m}^{\phi(w)}
   \quad\text{or}\quad
   \mu_{m+1}^{\phi(w)} = \mu_m^{\phi(w)} + e'_{m+1} - e'_{m+2},
\]
where $e'_1,\dotsc,e'_{2m+2}$ denotes the standard basis in $\ZZ^{2m+2}$.  Since $\mu_m^{\phi(w)}$ satisfies \eqref{it:SP1}, we conclude, with the help of \eqref{disp:whocares} in the second case, that $\mu_{m+1}^{\phi(w)}$ satisfies \eqref{it:SP1}.  Furthermore, since $x \equiv t_{\mu_{m+1}^x} \bmod W_{\aff,D_{m+1}}$ for any $x \in \wt W_{D_{m+1}}/W_{\wt I,D_{m+1}}$, we find that $t_{\mu_{m+1}^{\phi(w)}} \equiv \phi(w) \equiv t_{\iota(\mu)} \bmod W_{\aff,D_{m+1}}$.  Hence $\mu_{m+1}^{\phi(w)}$ automatically satisfies \eqref{it:SP3}, and by \eqref{st:c_i_equiv_q}
\[
   c_{m+1}^{\phi(w)} \equiv q \bmod 2.
\]
Since $\mu_m^{\phi(w)}$ satisfies \eqref{it:SP2}, we have
\[
   c_m^{\phi(w)} \leq q,
\]
and by \eqref{st:c_i_change_lem}
\[
   c_{m+1}^{\phi(w)} \leq c_m^{\phi(w)} + 1.
\]
These last three displays imply $c_{m+1}^{\phi(w)} \leq q$, which completes the proof.
\end{proof}

Using \eqref{st:spin-perm_intersection}, it is now easy to see that for $w \in W_{J,B_m} \bs \wt W_{B_m}/ W_{I,B_m}$, conditions \eqref{it:sp1}--\eqref{it:sp3} are independent of the choice of representative $\wt w \in \wt W_{B_m}/ W_{I,B_m}$.  Indeed, $\phi(W_{J,B_m}) \subset W_{\wt{J},D_{m+1}}$, and we then need note only that conditions \eqref{it:SP1}--\eqref{it:SP3}, for $k \in \wt I$, are independent of the choice of element in $W_{\wt{J},D_{m+1}}\phi(\wt w) \subset \wt W_{D_{m+1}}/ W_{\wt I,D_{m+1}}$ by \eqref{st:SP_SP'_equiv} and \eqref{st:sp-perm_double_cosets}.

Let us conclude the subsection by reformulating conditions \eqref{it:sp1}--\eqref{it:sp3} in terms of the vector $\nu_k^{w}$.  Let $w \in \wt W_{B_m} / W_{I,B_m}$, and consider the following conditions for $k \in I$.

\begin{enumerate}
\renewcommand{\theenumi}{sp\arabic{enumi}$'$}
\item\label{it:sp1'}
   $\nu_k^w + (\nu_k^w)^* = \mathbf 2$ and $\mathbf 0 \leq \nu_k^w \leq \mathbf 2$.
\item\label{it:sp2'}
   $\#\{\, j\mid \nu^w_k(j) = 2\,\} + \#\{\, j \mid \nu^w_k(j) \notin \ZZ\, \}/4 \leq q$.
\item\label{it:sp3'}
   (spin condition) If $\nu_k^w \in \ZZ^{2m}$ and $\nu_k^w \not\equiv \mu \bmod Q^\vee_{B_m}$, then there exists $j \in A_k$ such that $\nu_k^w(j) = 1$.
\end{enumerate}

The following equivalences are all proved as in \eqref{st:SP_SP'_equiv}, or can be deduced from \eqref{st:SP_SP'_equiv} via the embedding $\wt W_{B_m} \inj \wt W_{D_{m+1}}$.

\begin{lem}\label{st:sp-sp'_equivs}
Let $w \in \wt W_{B_m} / W_{I,B_m}$ and $k \in I$
\begin{enumerate}
\renewcommand{\theenumi}{\roman{enumi}}
\item\label{it:sp1_equiv'}
   $\mu_k^w$ satisfies \eqref{it:sp1} $\iff$ $\nu_k^w$ satisfies \eqref{it:sp1'}.
\item
   Suppose that the equivalent conditions in \eqref{it:sp1_equiv'} hold.  Then $\mu_k^w$ satisfies \eqref{it:sp2} $\iff$ $\nu_k^w$ satisfies \eqref{it:sp2'}.
\item
   $\mu_k^w$ satisfies \eqref{it:sp3} $\iff$ $\nu_k^w$ satisfies \eqref{it:sp3'}.\qed
\end{enumerate}
\end{lem}

\subsection{\texorpdfstring{$\mu$}{mu}-spin-permissibility and \texorpdfstring{$\mu$}{mu}-admissibility}
In this subsection we prove the equivalence of $\mu$-spin-permissibility and $\mu$-admissibility for $\R_{B_m}$, for $\mu$ the cocharacter \eqref{disp:mu_B}, in analogy with the result for $\R_{D_m}$ \eqref{st:sp-perm=adm_D_parahoric}.  We continue with the notation of the previous subsections.  We write $\Adm_{J,I,B_m}(\mu)$ for the set of $\mu$-admissible elements in $W_{J,B_m}\bs \wt W_{B_m} / W_{I,B_m}$%.  We abbreviate this to $\Adm_{I,B_m}(\mu)$ when $J = \{0,\dotsc,m\}$, and to $\Adm_{B_m}(\mu)$ when $I = J = \{0,\dotsc,m\}$.
, and we abbreviate this to $\Adm_{B_m}(\mu)$ when $I = J = \{0,\dotsc,m\}$.

\begin{thm}\label{st:sp-perm=adm_B}
Let $\mu$ be the cocharacter \eqref{disp:mu_B}, and let $I \subset J$ be nonempty subsets of $\{0,\dotsc,m\}$ satisfying property \eqref{disp:I_cond_B}. Then
\[
   \Adm_{J,I,B_m}(\mu) = \SPerm_{J,I,B_m}(\mu).
\]
\end{thm}

\begin{proof}
This will follow easily from our results in type $D$ and from \eqref{st:bruhat_inherited}.  For notational convenience, we regard $X_{*B_m}$ as a subgroup of $X_{*D_{m+1}}$, and $\wt W_{B_m}$ as a subgroup of $\wt W_{D_{m+1}}$.  The inclusion $\Adm_{J,I,B_m}(\mu) \subset \SPerm_{J,I,B_m}(\mu)$ is an immediate consequence of \eqref{st:sp-perm=adm_D_parahoric}, \eqref{st:bruhat_inherited}, and \eqref{st:spin-perm_intersection}.

To prove the reverse inclusion, let $w \in \SPerm_{J,I,B_m}(\mu)$.  Then
\[
   w \in \SPerm_{\wt J, \wt I, D_{m+1}}(\mu)
\]
by \eqref{st:spin-perm_intersection}.  Let $\tensor[^{\wt J}]w{^{\wt I}}$ denote the minimal length representative of $w$ in $\wt W_{D_{m+1}}$.  Then
\[
   \tensor[^{\wt J}]w{^{\wt I}} \in \SPerm_{D_{m+1}}(\mu) \cap \wt W_{B_m}
\]
by \eqref{rk:min_length_spin-perm} and \eqref{st:min_length_rep}.  By applying \eqref{st:big_D_prop}\eqref{it:part_ii} repeatedly, we eventually find a translation element $t_{\nu} \in \SPerm_{D_{m+1}}(\mu)$ such that $\tensor[^{\wt J}]x{^{\wt I}} \leq t_\nu$ in $\wt W_{D_{m+1}}$ and $\nu(m+1) = \nu(m+2) = 1$ (regarding $\nu$ as a $(2m+2)$-tuple), i.e. $t_\nu \in \wt W_{B_m}$.  Hence $t_\nu \in \SPerm_{B_m}(\mu)$ by \eqref{st:spin-perm_intersection}, and $\tensor[^{\wt J}]w{^{\wt I}} \leq t_\nu$ in $\wt W_{B_m}$ by \eqref{st:bruhat_inherited}.  Since the $\mu$-spin-permissible translation elements in $\wt W_{B_m}$ are $\mu$-admissible by the same argument as in \eqref{st:transl_adm}, we conclude that $w \in \Adm_{J,I,B_m}(\mu)$.
\end{proof}

The key point in the above proof, besides the fact \eqref{st:bruhat_inherited} that the Bruhat order for $\R_{B_m}$ is inherited from the Bruhat order for $\R_{D_{m+1}}$, is that given $w \in \wt W_{B_m}$ such that $w \leq t_\nu$ in $\wt W_{D_{m+1}}$ for some $\nu \in S_{2m+2}^\circ\mu$, there in fact exists $\nu' \in S_{2m+1}^*\mu$ such that $w \leq t_{\nu'}$; or in other words,
% the key point is 
that $\Adm_{B_m}(\mu) = \wt W_{B_m} \cap \Adm_{D_{m+1}}(\mu)$.  We conjecture that the same holds for any Steinberg fixed-point root datum, as follows.

\begin{conj}\label{conj}
Let $\R = (X^*,X_*,R,R^\vee,\Pi)$ be a based root datum and $\Theta$ an automorphism of \R, as in \s\ref{ss:steinberg}.  Let $\A$ denote the apartment of \R, let $\mathbf a$ be an alcove in $\A$ meeting $\A^{[\Theta]}$, and let $\mathbf{f_1}$, $\mathbf{f_2} \preceq \mathbf a$ be subfacets also meeting $\A^{[\Theta]}$.
Then for any cocharacter $\mu \in X_*^{[\Theta]}$,
\[
   \Adm_{\mathbf{f}_{\mathbf 1}^{[\Theta]}, \mathbf f_{\mathbf 2}^{[\Theta]}, \R^{[\Theta]}}(\mu) 
      = \Adm_{\mathbf{f_1},\mathbf{f_2}, \R}(\mu) \cap \bigl(W_{\mathbf{f}_{\mathbf 1}^{[\Theta]}} \bs \wt W^{[\Theta]} / W_{\mathbf f_{\mathbf 2}^{[\Theta]}}\bigr),
\]
where the Bruhat orders in the admissible sets are specified by $\Ba^{[\Theta]}$ and \Ba, respectively.
\end{conj}

The content of the conjecture lies in the inclusion $\supset$, as the forward inclusion $\subset$ is a trivial consequence of \eqref{st:bruhat_inherited}.  It suffices to prove the conjecture just in the Iwahori case $\mathbf {f_1} = \mathbf{f_2} = \mathbf a$, since the sets $\Adm_{\R^{[\Theta]}}(\mu)$ and $\Adm_\R(\mu) \cap \wt W^{[\Theta]}$ respectively surject onto the left-hand and right-hand sets in the display.  If \R is irreducible of type $A$, then the conjecture has been proved by Haines and Ng\^o for any $\mu$ \cite{hngo02b}*{Th.\ 1, Prop.\ 9.7}.  The conjecture also holds whenever $\mu$-admissibility and $\mu$-permissibility are equivalent in $\wt W$ and in $\wt W^{[\Theta]}$, since it is an easy consequence of \eqref{st:convex_hulls} that $\Perm_\R(\mu) \cap \wt W^{[\Theta]} \subset \Perm_{\R^{[\Theta]}}(\mu)$.  In particular, the conjecture is known for \emph{minuscule $\mu$} when the root datums involve only types $A$, $B$, $C$, and $D$; see \cite{prs?}*{Prop.\ 4.4.5(iii)}.  As we remarked above, we proved the conjecture for $\R = \R_{D_{m+1}}$, $\R^{[\Theta]} = \R_{B_m}$, and $\mu$ the cocharacter \eqref{disp:mu_B} in the course of proving \eqref{st:sp-perm=adm_B}.
Our proof of this relied crucially on \eqref{st:big_D_prop}\eqref{it:part_ii}.  If however the conjecture were already known, then it would have obviated our need for \eqref{st:big_D_prop}\eqref{it:part_ii} in the first place.

\subsection{Vertexwise admissibility}\label{ss:vert_adm_B}
In this subsection we prove the equivalence of $\mu$-admissibility and $\mu$-vertexwise admissibility in $\R_{B_m}$ for the cocharacter $\mu$ \eqref{disp:mu_B}, in analogy with \eqref{st:vert-adm_D}.  See \eqref{def:vert-adm} for the general definition of $\mu$-vertexwise admissibility.  We write $\Adm_{F',F,B_m}^\vert(\mu)$ for the $\mu$-vertexwise admissible set in $W_{F',B_m} \bs \wt W_{B_m} / W_{F,B_m}$.

\begin{thm}\label{st:vert-adm_B}
Let $\mu$ be the cocharacter \eqref{disp:mu_B}, and let $F \preceq F' \preceq A_{B_m}$.  Then
\[
   \Adm_{F',F,B_m}(\mu) = \Adm_{F',F,B_m}^\vert(\mu).
\]
\end{thm}

\begin{proof}
The argument is very similar to the proof of \eqref{st:vert-adm_D}.  After applying the automorphism $\theta_0$ if necessary, we may assume that if $a_{0'}$ is a vertex of $F$, then so is $a_0$, and likewise for $F'$.  Thus we may assume that $W_{F,B_m} = W_{I,B_m}$ and $W_{F',B_m} = W_{J,B_m}$ for some nonempty subsets $I \subset J \subset \{0,\dotsc,m\}$ satisfying property \eqref{disp:I_cond_B}.

We must show that $\Adm_{F',F}^\vert(\mu) \subset \Adm_{F',F}(\mu)$.  If $w \in \Adm_{F',F}^\vert(\mu)$ and $\wt w$ is any representative of $w$ in $\wt W_{B_m}/W_{F,B_m}$, then we see right away that $w \equiv t_\mu \bmod W_{\aff,B_m}$ and, by \eqref{st:sp-perm=adm_B} applied with $I = J = \{k\}$ for $k \in I \smallsetminus \{1\}$, that $\mu_k^{\wt w}$ satisfies \eqref{it:sp1}--\eqref{it:sp3} for $k \in I \smallsetminus \{1\}$.  Thus, by \eqref{st:sp-perm=adm_B} again, it remains only to show that $\mu_1^{\wt w}$ satisfies \eqref{it:sp1}--\eqref{it:sp3} in case $1 \in I$.  This is done by using that the image of $w$ in $W_{\{0'\},B_m} \bs \wt W_{B_m} / W_{\{0'\},B_m}$ is $\mu$-admissible, as in the proof of \eqref{st:vert-adm_D}, where $W_{\{0'\},B_m}$ denotes the stabilizer in $W_{\aff,B_m}$ of the facet $a_{0'} + \RR \cdot (1,\dotsc,1)$.
\end{proof}

\begin{rk}
Upon identifying $\wt W_G$ with $\wt W_{B_m}$ as in \eqref{st:wtW_G->wtW_B_m}, the theorem was conjectured by Pappas and Rapoport in   
\cite{paprap09}*{Conj.\ 4.2}.
\end{rk}

\begin{rk}
In analogy with \eqref{rk:natural_def}, from a root-theoretic point of view it would be more natural to define $\mu$-spin-permissibility in terms of the vectors $\mu_k^{w}$, or equivalently $\nu_k^w$, for $k \in \{0,0',2,3,\dotsc,m\}$, since the points $a_k$ for such $k$ are vertices of the base alcove.  The theorem shows that $\Adm_{F',F,B_m}(\mu)$ consists of all elements $w$ such that $w \equiv t_{\mu} \bmod W_{\aff,B_m}$ and for one, hence any, representative $\wt w$ of $w$ in $\wt W_{B_m}/ W_{F,B_m}$, $\wt w a_k - a_k$ satisfies \eqref{it:sp1'}--\eqref{it:sp3'} for all $a_k$ which are vertices of $F$.  Here, analogously to \eqref{rk:natural_def}, the spin condition \eqref{it:sp3'} for $\wt w a_{0'} - a_{0'}$ means that $\wt w a_{0'} - a_{0'} \equiv \mu \bmod Q_{B_m}^\vee$ (which is equivalent to $w \equiv t_\mu \bmod W_{\aff,B_m}$).
\end{rk}

\subsection{\texorpdfstring{$\mu$}{mu}-permissibility}
In this final subsection of the paper, we shall characterize $\mu$-permissibility in $\R_{B_m}$, for $\mu$ the cocharacter \eqref{disp:mu_B}, in terms of the conditions \eqref{it:sp1}--\eqref{it:sp3}.  Everything we shall do will be in close analogy with \s\s\ref{ss:mu-perm} and \ref{ss:mu-perm_parahoric}.  We write
\[
   \Perm_{F',F,B_m}(\mu)
\]
for the set of $\mu$-permissible elements in $W_{F',B_m} \bs \wt W_{B_m} / W_{F,B_m}$, and we abbreviate this to $\Perm_{B_m}(\mu)$ when $F = F' = A_{B_m}$.
As usual, it is harmless to restrict to restrict to subfacets $F$ of $A_{B_m}$ with the property that $a_0$ is a vertex of $F$ if $a_{0'}$ is.

\begin{prop}\label{st:mu-perm_B}
Let $\mu$ be the cocharacter \eqref{disp:mu_B}, let $I \subset J$ be nonempty subsets of $\{0,\dotsc,m\}$ satisfying property \eqref{disp:I_cond_B}, and let $w \in W_{J,B_m} \bs \wt W_{B_m} / W_{I,B_m}$.  Then $w$ is $\mu$-permissible $\iff$ $w \equiv t_\mu \bmod W_{\aff,B_m}$ and for one, hence any, representative $\wt w$ of $w$ in $\wt W_{B_m}/ W_{I,B_m}$, $\nu_k^{\wt w}$ satisfies \eqref{it:sp1'} and \eqref{it:sp2'} for all $k \in I$, and \eqref{it:sp3'} for all $k \in I \cap \{0,1\}$.
\end{prop}

\begin{proof}
This is obvious from the following lemma.
\end{proof}

\begin{lem}\label{st:blah}
Keep the notation of \eqref{st:mu-perm_B}, and let $\Conv(S_{2m+1}^*\mu)$ denote the convex hull in $\A_{B_m}$ of the Weyl orbit $S_{2m+1}^*\mu$.
\begin{enumerate}
\renewcommand{\theenumi}{\roman{enumi}}
\item
   Let $k \in I$.  Then $\nu_k^{\wt w}$ satisfies \eqref{it:sp1'} and \eqref{it:sp2'} $\iff$ $\nu_k^{\wt w} \in \Conv(S_{2m+1}^*\mu)$.
\item
   Suppose that $1 \in I$.  Then $\nu_0^{\wt w}$ and $\nu_1^{\wt w}$ satisfy \eqref{it:sp1'}--\eqref{it:sp3'} $\iff$ $w \equiv t_\mu \bmod W_{\aff,B_m}$ and $\nu_0^{\wt w}$, $\nu_{0'}^{\wt w} \in \Conv(S_{2m+1}^*\mu)$.
\end{enumerate}
\end{lem}

\begin{proof}
Upon identifying $\R_{B_m}$ with $\R_{D_{m+1}}^{[\Theta]}$ as in \s\ref{ss:relation_to_D_m}, everything follows from the analogous results \eqref{st:SP1-2'_conv_cond} and \eqref{st:sickofnaminglemmas} for $\R_{D_{m+1}}$, the equivalence of the (sp) and (SP) conditions, and the fact that $\Conv(S_{2m+1}^*\mu) = \Conv(S_{2m+1}^\circ \mu) \cap \A_{B_m}$ by \eqref{st:convex_hulls}.
\end{proof}

\begin{eg}\label{eg:perm_neq_adm_B}
Recall the element $w \in \wt W_{D_4}$ from \eqref{eg:not_mu-adm_D} which is contained in $\Perm_{D_4}(\mu)$ but not in $\Adm_{D_4}(\mu)$, for $\mu$ the cocharacter $\bigl(2^{(2)},1,1,0^{(3)}\bigr) \in X_{*D_4}$.  Upon embedding $\wt W_{B_3} \subset \wt W_{D_4}$, we find that $w$, $t_\mu \in \wt W_{B_3}$, and that $w \in \Perm_{B_m}(\mu)$ but $w \notin \Adm_{B_m}(\mu)$.  As in \eqref{eg:not_mu-adm_D}, this example easily generalizes to show that $\Adm_{B_m}(\mu) \neq \Perm_{B_m}(\mu)$ whenever $m$, $q \geq 3$.  We remark that Haines and Ng\^o \cite{hngo02b}*{Th.\ 3} have previously shown that in any irreducible root datum \R not of type $A$ and of rank $\geq 4$, there exist cocharacters $\mu$ for which $\Adm_\R(\mu) \neq \Perm_\R(\mu)$.
\end{eg}

In general, $\mu$-permissible sets need not be well-behaved with regard to Steinberg fixed-point root data, in the sense that for an arbitrary cocharacter $\mu$ in $\R^{[\Theta]}$, the intersection $\Perm_\R(\mu) \cap \wt W^{[\Theta]}$ may only be properly contained in $\Perm_{\R^{[\Theta]}}(\mu)$, not equal.  However, for $\mu$ the cocharacter \eqref{disp:mu_B}, such bad behavior does not arise.

\begin{prop}\label{st:perm_intersect}
Let $F \preceq F'$ be subfacets of $A_{B_m}$, and let $\wt F$ and $\wt{F'}$ be the respective facets in $\A_{D_m}$ containing each of them.  Then
\[
\Perm_{F',F,B_m}(\mu) = \Perm_{\wt{F'},\wt F, D_m}(\mu) \cap
                        \bigl(W_{F',B_m} \bs \wt W_{B_m}/ W_{F,B_m}\bigr).
\]
\end{prop}

\begin{proof}
The inclusion $\supset$ holds by general properties of Steinberg fixed-point root data.  To prove the forward inclusion $\subset$, by applying the automorphism $\theta_0$ if necessary, we may assume that $a_{0'}$ is a vertex of $F$ if $a_0$ is, and likewise for $F'$.  The conclusion then follows at once from \eqref{st:spin-perm_intersection}, \eqref{st:automatic_spin_cond}, and our explicit descriptions of $\Perm_{\wt{F'},\wt F, D_m}(\mu)$ \eqref{st:mu-perm_D_parahoric} and $\Perm_{F',F,B_m}(\mu)$ \eqref{st:mu-perm_B}.
\end{proof}

As a consequence of \eqref{st:perm_intersect}, \eqref{st:adm=perm_q=2}, and the fact that we have proved \eqref{conj} in the case at hand, we conclude the following.

\begin{cor}
Let $\mu$ be the cocharacter \eqref{disp:mu_B}, let $F \preceq F' \preceq A_{B_m}$, and suppose that $q \leq 2$.  Then
\begin{flalign*}
   \phantom{\qed}& & \Adm_{F',F,B_m}(\mu) = \Perm_{F',F,B_m}(\mu). & & \qed
\end{flalign*}
\end{cor}

\end{document}